\documentclass[a4paper, reqno, 10pt]{amsart}

\usepackage[usenames,dvipsnames]{color}

\usepackage{amsthm,amsfonts,amssymb,amsmath,amsxtra,amsrefs}
\usepackage[all]{xy}
\SelectTips{cm}{}
\usepackage{xr-hyper}
\usepackage[colorlinks=
   citecolor=Black,
   linkcolor=Red,
   urlcolor=Blue]{hyperref}
\usepackage{verbatim}

\usepackage[margin=1.25in]{geometry}

\usepackage{tikz}

\usepackage{mathrsfs}

\RequirePackage{xspace}
\RequirePackage{etoolbox}
\RequirePackage{varwidth}
\RequirePackage{enumitem}
\RequirePackage{tensor}
\RequirePackage{mathtools}
\RequirePackage{longtable}
\RequirePackage{multirow}

\newcommand{\sD}{\ensuremath{\mathscr{D}}\xspace}

\newcommand{\sH}{\ensuremath{\mathscr{H}}\xspace}

\newcommand{\sL}{\ensuremath{\mathscr{L}}\xspace}

\newcommand{\sP}{\ensuremath{\mathscr{P}}\xspace}

\newcommand{\fkm}{\ensuremath{\mathfrak{m}}\xspace}

\newcommand{\nat}{{\natural}}

\newcommand{\BA}{\ensuremath{\mathbb {A}}\xspace}
\newcommand{\BB}{\ensuremath{\mathbb {B}}\xspace}
\newcommand{\BC}{\ensuremath{\mathbb {C}}\xspace}

\newcommand{\BJ}{\ensuremath{\mathbb {J}}\xspace}
\newcommand{\BK}{\ensuremath{\mathbb {K}}\xspace}

\newcommand{\BO}{\ensuremath{\mathbb {O}}\xspace}
\newcommand{\BP}{\ensuremath{\mathbb {P}}\xspace}
\newcommand{\BQ}{\ensuremath{\mathbb {Q}}\xspace}

\newcommand{\BZ}{\ensuremath{\mathbb {Z}}\xspace}

\newcommand{\CC}{\ensuremath{\mathcal {C}}\xspace}

\newcommand{\CI}{\ensuremath{\mathcal {I}}\xspace}

\newcommand{\CM}{\ensuremath{\mathcal {M}}\xspace}

\newcommand{\CO}{\ensuremath{\mathcal {O}}\xspace}

\newcommand{\CW}{\ensuremath{\mathcal {W}}\xspace}

\newcommand{\cA}{\ensuremath{\mathcal {A}}\xspace}

\newcommand{\cC}{\ensuremath{\mathcal {C}}\xspace}

\newcommand{\cE}{\ensuremath{\mathcal {E}}\xspace}
\newcommand{\cF}{\ensuremath{\mathcal {F}}\xspace}
\newcommand{\cG}{\ensuremath{\mathcal {G}}\xspace}
\newcommand{\cH}{\ensuremath{\mathcal {H}}\xspace}
\newcommand{\cI}{\ensuremath{\mathcal {I}}\xspace}
\newcommand{\cJ}{\ensuremath{\mathcal {J}}\xspace}
\newcommand{\cK}{\ensuremath{\mathcal {K}}\xspace}
\newcommand{\cL}{\ensuremath{\mathcal {L}}\xspace}
\newcommand{\cM}{\ensuremath{\mathcal {M}}\xspace}
\newcommand{\cN}{\ensuremath{\mathcal {N}}\xspace}
\newcommand{\cO}{\ensuremath{\mathcal {O}}\xspace}
\newcommand{\cP}{\ensuremath{\mathcal {P}}\xspace}
\newcommand{\cQ}{\ensuremath{\mathcal {Q}}\xspace}

\newcommand{\cS}{\ensuremath{\mathcal {S}}\xspace}

\newcommand{\cW}{\ensuremath{\mathcal {W}}\xspace}

\newcommand{\cY}{\ensuremath{\mathcal {Y}}\xspace}
\newcommand{\cZ}{\ensuremath{\mathcal {Z}}\xspace}

\newcommand{\Ad}{{\mathrm{Ad}}}

\DeclareMathOperator{\Aut}{Aut}

\newcommand{\Ch}{{\mathrm{Ch}}}

\DeclareMathOperator{\coker}{coker}

\newcommand{\cl}{{\mathrm{cl}}}

\newcommand{\codim}{{\mathrm {codim}}}

\DeclareMathOperator{\diag}{diag}

\newcommand{\Div}{{\mathrm{Div}}}
\renewcommand{\div}{{\mathrm{div}}}

\DeclareMathOperator{\Eis}{Eis}
\DeclareMathOperator{\End}{End}

\DeclareMathOperator{\Fr}{Fr}

\DeclareMathOperator{\Gal}{Gal}
\newcommand{\GL}{\mathrm{GL}}

\DeclareMathOperator{\Hom}{Hom}

\newcommand{\id}{\ensuremath{\mathrm{id}}\xspace}
\let\Im\relax
\DeclareMathOperator{\Im}{Im}
\newcommand{\Ind}{{\mathrm{Ind}}}

\newcommand{\inv}{{\mathrm{inv}}}

\DeclareMathOperator{\Ker}{Ker}

\DeclareMathOperator{\Nm}{Nm}

\DeclareMathOperator{\ord}{ord}

\DeclareMathOperator{\Quot}{Quot}

\newcommand{\PGL}{{\mathrm{PGL}}}
\DeclareMathOperator{\Pic}{Pic}

\newcommand{\red}{\ensuremath{\mathrm{red}}\xspace}

\DeclareMathOperator{\Res}{Res}

\newcommand{\Sat}{{\mathrm{Sat}}}

\newcommand{\SL}{{\mathrm{SL}}}
\DeclareMathOperator{\Spec}{Spec}

\DeclareMathOperator{\St}{St}

\DeclareMathOperator{\Supp}{Supp}

\DeclareMathOperator{\Tr}{Tr}

\DeclareMathOperator{\vol}{vol}

\newcommand{\Bun}{{\mathrm{Bun}}}

\newcommand{\Sht}{{\mathrm{Sht}}}

\newcommand{\Hk}{{\mathrm{Hk}}}

\newcommand{\Gr}{\mathrm{Gr}}

\newcommand{\Iw}{\mathrm{Iw}}

\newcommand{\Mat}{\mathrm{Mat}}

\newcommand{\wt}{\widetilde}
\newcommand{\wh}{\widehat}

\newcommand{\pair}[1]{\langle {#1} \rangle}

\newcommand{\ov}{\overline}

\newcommand{\incl}{\hookrightarrow}
\newcommand{\lra}{\longrightarrow}

\newcommand{\bs}{\backslash}

\newcommand{\ep}{\varepsilon}

\renewcommand\AA{\mathbb{A}}
\renewcommand\CC{\mathbb{C}}

\newcommand\FF{\mathbb{F}}
\newcommand\GG{\mathbb{G}}

\newcommand\II{\mathbb{I}}
\newcommand\JJ{\mathbb{J}}
\newcommand\KK{\mathbb{K}}

\newcommand\OO{\mathbb{O}}
\newcommand\PP{\mathbb{P}}
\newcommand\QQ{\mathbb{Q}}

\newcommand\ZZ{\mathbb{Z}}

\newcommand\bR{\mathbf{R}}

\newcommand\bJ{\mathbf{J}}

\newcommand\frA{\mathfrak{A}}

\newcommand\frD{\mathfrak{D}}

\newcommand\frK{\mathfrak{K}}

\newcommand\frS{\mathfrak{S}}
\newcommand\frT{\mathfrak{T}}

\newcommand\frX{\mathfrak{X}}

\renewcommand\a\alpha
\renewcommand\b\beta
\newcommand\g\gamma
\renewcommand\d\delta
\newcommand\D\Delta

\renewcommand{\k}{\kappa}
\renewcommand{\th}{\theta}
\newcommand{\vth}{\vartheta}
\newcommand{\Th}{\Theta}
\newcommand{\ph}{\varphi}
\newcommand{\s}{\sigma}
\newcommand{\Sig}{\Sigma}
\renewcommand{\t}{\tau}
\renewcommand\r{\rho}
\newcommand\io{\iota}

\newcommand{\y}{\eta}
\newcommand{\z}{\zeta}
\newcommand{\vp}{\varpi}
\renewcommand{\l}{\lambda}
\renewcommand{\L}{\Lambda}
\newcommand{\om}{\omega}
\newcommand{\Om}{\Omega}

\renewcommand\i{\imath}
\renewcommand\j{\jmath}

\newcommand\xr{\xrightarrow}
\newcommand{\isom}{\stackrel{\sim}{\to}}
\newcommand{\surj}{\twoheadrightarrow}

\renewcommand{\c}{\circ}
\newcommand\ha{\frac{1}{2}}

\newcommand{\leftexp}[2]{{\vphantom{#2}}^{#1}{#2}}
\newcommand{\pH}{\leftexp{p}{\textup{H}}}
\newcommand{\ptau}{\leftexp{p}{\tau}}

\newcommand{\Ql}{\QQ_{\ell}}
\newcommand{\Qlbar}{\overline{\QQ}_\ell}
\newcommand{\kbar}{\overline{k}}

\newcommand{\twtimes}[1]{\stackrel{#1}{\times}}

\newcommand{\htimes}{\widehat{\times}}
\newcommand{\jiao}[1]{\langle{#1}\rangle}
\newcommand\un{\underline}
\newcommand\one{\mathbf{1}}
\newcommand\op{\oplus}
\newcommand\ot{\otimes}
\newcommand\vn{\varnothing}
\newcommand{\sss}{\subsubsection}
\newcommand\mat[4]{\left[\begin{array}{cc} #1 & #2 \\ #3 & #4 \end{array}\right]}  

\newcommand{\cohog}[2]{\textup{H}^{#1}({#2})}     
\newcommand{\cohoc}[2]{\textup{H}_{c}^{#1}({#2})}     
\newcommand{\hBM}[2]{\textup{H}^{\textup{BM}}_{#1}({#2})}  

\newcommand{\oll}[1]{\overleftarrow{#1}}
\newcommand{\orr}[1]{\overrightarrow{#1}}
\newcommand{\olr}[1]{\overleftrightarrow{#1}}

\renewcommand\div{\textup{div}}
\newcommand\ev{\textup{ev}}
\newcommand\AJ{\textup{AJ}}
\newcommand\AL{\textup{AL}}
\newcommand\add{\textup{add}}
\newcommand{\Gm}{\GG_m}
\newcommand{\Ga}{\GG_a}
\newcommand\hX{\wh{X}}

\newcommand\mult{\textup{mult}}
\newcommand\pr{\textup{pr}}
\newcommand\pt{\textup{pt}}
\newcommand\Prym{\textup{Prym}}
\newcommand\Perv{\textup{Perv}}
\newcommand\sep{\textup{sep}}

\newcommand\inst{\textup{inst}}
\newcommand\ind{\textup{ind}}
\newcommand\Sect{\textup{Sect}}

\newcommand\hs{\heartsuit}
\newcommand\dm{\diamondsuit}

\newcommand\da{\dagger}
\newcommand\sh{\sharp}
\newcommand\na{\natural}
\newcommand\fl{\flat}

\newcommand{\sqR}{\sqrt{R}}
\newcommand{\sqx}{\sqrt{x}}

\newcommand{\Si}{\Sigma_{\infty}}
\newcommand{\Sf}{\Sigma_{f}}
\newcommand{\Sii}{\Sig; \Sigma_{\infty}}
\newcommand{\frSi}{\frS_{\infty}}
\newcommand{\Di}{D_{\infty}}
\newcommand{\mi}{\mu_{\infty}}
\newcommand{\mf}{\mu_{f}}

\newcommand{\bsq}{\square}
\newcommand\bx{\mathbf{x}}

\newcommand\ba{\mathbf{a}}

\newtheorem{theorem}{Theorem}
\newtheorem{prop}[theorem]{Proposition}
\newtheorem{lem}[theorem]{Lemma}
\newtheorem{lemma}[theorem]{Lemma}
\newtheorem{conj}[theorem]{Conjecture}
\newtheorem{cor}[theorem]{Corollary}
\newtheorem{thm}[theorem]{Theorem}
\newtheorem*{claim}{Claim}

\theoremstyle{definition}
\newtheorem{defn}[theorem]{Definition}

\newtheorem{remark}[theorem]{Remark}

\newenvironment{altenumerate}
   {\begin{list}
      {\textup{(\theenumi)} }
      {\usecounter{enumi}
       \setlength{\labelwidth}{0pt}
       \setlength{\labelsep}{0pt}
       \setlength{\leftmargin}{0pt}
       \setlength{\itemsep}{\the\smallskipamount}
       \renewcommand{\theenumi}{\roman{enumi}}
      }}
   {\end{list}}

\newcommand{\matrixx}[4]
{\left[ \begin{array}{cc}
  #1 &  #2  \\
  #3 &  #4  \\
 \end{array}\right]}

 \newcommand{\smat}[4]
{\bigl[\begin{smallmatrix}#1 & #2\\ #3& #4\end{smallmatrix}\bigr]}

\numberwithin{equation}{section}
\numberwithin{theorem}{section}

\setitemize[0]{leftmargin=*,itemsep=\the\smallskipamount}
\setenumerate[0]{leftmargin=*,itemsep=\the\smallskipamount}

\renewcommand{\to}{%
   \ifbool{@display}{\longrightarrow}{\rightarrow}%
   }
\let\shortmapsto\mapsto
\renewcommand{\mapsto}{%
   \ifbool{@display}{\longmapsto}{\shortmapsto}%
   }
\newlength{\olen}
\newlength{\ulen}
\newlength{\xlen}
\newcommand{\xra}[2][]{%
   \ifbool{@display}%
      {\settowidth{\olen}{$\overset{#2}{\longrightarrow}$}%
       \settowidth{\ulen}{$\underset{#1}{\longrightarrow}$}%
       \settowidth{\xlen}{$\xrightarrow[#1]{#2}$}%
       \ifdimgreater{\olen}{\xlen}%
          {\underset{#1}{\overset{#2}{\longrightarrow}}}%
          {\ifdimgreater{\ulen}{\xlen}%
             {\underset{#1}{\overset{#2}{\longrightarrow}}}
             {\xrightarrow[#1]{#2}}}}%
      {\xrightarrow[#1]{#2}}
   }
\makeatother
\newcommand{\xyra}[2][]{%
   \settowidth{\xlen}{$\xrightarrow[#1]{#2}$}%
   \ifbool{@display}%
      {\settowidth{\olen}{$\overset{#2}{\longrightarrow}$}%
       \settowidth{\ulen}{$\underset{#1}{\longrightarrow}$}%
       \ifdimgreater{\olen}{\xlen}%
          {\mathrel{\xymatrix@M=.12ex@C=3.2ex{\ar[r]^-{#2}_-{#1} &}}}%
          {\ifdimgreater{\ulen}{\xlen}%
             {\mathrel{\xymatrix@M=.12ex@C=3.2ex{\ar[r]^-{#2}_-{#1} &}}}
             {\mathrel{\xymatrix@M=.12ex@C=\the\xlen{\ar[r]^-{#2}_-{#1} &}}}}}%
      {\mathrel{\xymatrix@M=.12ex@C=\the\xlen{\ar[r]^-{#2}_-{#1} &}}}%
   }
\makeatletter
\newcommand{\xla}[2][]{%
   \ifbool{@display}%
      {\settowidth{\olen}{$\overset{#2}{\longleftarrow}$}%
       \settowidth{\ulen}{$\underset{#1}{\longleftarrow}$}%
       \settowidth{\xlen}{$\xleftarrow[#1]{#2}$}%
       \ifdimgreater{\olen}{\xlen}%
          {\underset{#1}{\overset{#2}{\longleftarrow}}}%
          {\ifdimgreater{\ulen}{\xlen}%
             {\underset{#1}{\overset{#2}{\longleftarrow}}}
             {\xleftarrow[#1]{#2}}}}%
      {\xleftarrow[#1]{#2}}
   }
\newcommand{\isoarrow}{%
   \ifbool{@display}{\overset{\sim}{\longrightarrow}}{\xrightarrow\sim}%
   }
\renewcommand{\lra}{%
   \ifbool{@display}{\longleftrightarrow}{\leftrightarrow}%
   }

\begin{document}

\thanks{Research of Z.Yun partially supported by the Packard Foundation and the NSF grant DMS 1302071/1736600. Research of W. Zhang partially supported by the NSF grant DMS-1601144 and a Simons fellowship. }

\title[Shtukas and the Taylor expansion (II)]{Shtukas and the Taylor expansion of $L$-functions (II)} 
\author{Zhiwei Yun}
\address{Zhiwei Yun: Massachusetts Institute of Technology, Department of Mathematics, 77 Massachusetts Avenue, Cambridge, MA 02139, USA}
\email{zyun@mit.edu}
\author{Wei Zhang}
\address{Wei Zhang: Massachusetts Institute of Technology, Department of Mathematics, 77 Massachusetts Avenue, Cambridge, MA 02139, USA}
\email{weizhang@mit.edu}




\date{\today}

\begin{abstract}For arithmetic applications, we extend and refine our results in
\cite{YZ} to allow ramifications in a minimal way. Starting with a possibly ramified quadratic extension $F'/F$ of function fields over a finite field in odd characteristic, and a finite set of places $\Sig$ of $F$ that are unramified in $F'$, we define a collection of Heegner--Drinfeld cycles on the moduli stack of $\PGL_{2}$-Shtukas with $r$-modifications and Iwahori level structures at places of $\Sig$. For a cuspidal automorphic representation $\pi$ of $\PGL_{2}(\AA_{F})$ with square-free level $\Sig$, and $r\in\ZZ_{\ge0}$ whose parity matches the root number of $\pi_{F'}$, we prove a series of identities between
\begin{enumerate}
\item The product of the central derivatives of the normalized $L$-functions 
$$\sL^{(a)}(\pi, \ha)\sL^{(r-a)}(\pi\otimes\y, \ha),$$ 
where $\y$ is the quadratic id\`ele class character attached to $F'/F$, and $0\le a\le r$;
\item The self intersection number of a linear combination of Heegner--Drinfeld cycles.
\end{enumerate}
In particular, we can now obtain global $L$-functions with odd vanishing orders.  These identities are function-field analogues of the formulae of Waldspurger and Gross--Zagier for higher derivatives of $L$-functions.
\end{abstract}

\maketitle

\tableofcontents

\section{Introduction}
\subsection{Main results}

Let $X$ be a smooth projective and geometrically connected curve over a finite field $k=\FF_{q}$ of characteristic $p\ne 2$. Let $F=k(X)$ be the function field of $X$ and $\AA_{F}$ be the ring of ad\`eles of $F$. Let $G=\PGL_{2}$. Let $\pi$ be a cuspidal automorphic representation of $G(\AA_{F})$.   Let $X'$ be another smooth projective and geometrically connected curve over $k$ together with a double cover $\nu: X'\to X$. 

In \cite{YZ}, under the assumption that both $\pi$ and $\nu$ are everywhere unramified, we proved an analogue of the formulae of Waldspurger \cite{W} and Gross--Zagier \cite{GZ} for higher order central derivatives of the base change $L$-function $L(\pi_{F'}, s)$. Our formula reads 
\begin{equation}\label{HGZ unram}
\frac{|\om_{X}|}{2(\log q)^{r}L(\pi,\Ad,1)}\sL^{(r)}(\pi_{F'}, \ha)= \Big([\Sht^{\mu}_{T}]_{\pi}, [\Sht^{\mu}_{T}]_{\pi}\Big)_{\Sht'^{r}_{G}}.
\end{equation}
Here $r\ge0$ is an {\em even} integer. This formula relates the $r$-th central derivative of a certain normalization \footnote{In \cite{YZ}, the definition of $\sL(\pi_{F'}, s)$ included the denominator $L(\pi,\Ad,1)$; in the current paper, we separate $L(\pi,\Ad,1)$ from $\sL(\pi_{F'},s)$.} $\sL(\pi_{F'}, s)$  of the  $L$-function of the base change $\pi_{F'}$ to the self-intersection number of a certain algebraic cycle $[\Sht^{\mu}_{T}]_{\pi}$ on the moduli stack of $G$-Shtukas $\Sht'^{r}_{G}$ with $r$ modifications.  The cycles $[\Sht^{\mu}_{T}]_{\pi}$ are analogous to the Heegner points on modular curves.

In this paper, we generalize the formula \eqref{HGZ unram} to the case where the double cover $\nu$ is allowed to be ramified and the automorphic representation $\pi$ is allowed to have square-free level. Moreover, we refine the formula \eqref{HGZ unram} to give a geometric expression of central derivatives of the form $\sL^{(a)}(\pi,\ha)\sL^{(b)}(\pi\ot\y,\ha)$. Below we set up some notation for the statement of our main results.

\sss{Ramifications of the automorphic representation}
Let $\Sig$ be a finite set of closed points of $X$. Let $\pi$ be a cuspidal automorphic representation of $G(\AA)$ which is isomorphic to an unramified twist of the Steinberg representation at each $x\in \Sig$, and unramified away from $\Sig$. Let $N=\deg \Sig$.

Let $R$ be the ramification locus of the double cover $\nu$, and let $\r=\deg R$. Then the genus $g'$ of $X'$ and the genus  $g$ of $X$ are related by $g'-1=2(g-1)+\r/2$. Let $\y=\y_{F'/F}: F^{\times}\bs\AA^{\times}_{F}\to \{\pm1\}$ be the id\`ele class character corresponding to the extension $F'/F$. 
 
We assume 
$$
\text{\it The sets $R$ and $\Sig$ are disjoint}. 
$$The normalized $L$-functions
\begin{eqnarray*}
\sL(\pi, s+\ha)&=&q^{(2g-2+N/2)s}L(\pi,s+\ha)\\
\sL(\pi\ot\y, s+\ha)&=&q^{(2g-2+\r+N/2)s}L(\pi\otimes\eta,s+\ha)
\end{eqnarray*}
are either even or odd functions in $s$ depending on the root numbers of $\pi$ and $\pi\otimes\y$. We define a normalized $L$-function in two variables 
\begin{equation*}
\sL_{F'/F}(\pi, s_{1},s_{2}):=\sL(\pi,s_1+s_2+\ha)\sL(\pi\otimes\eta,s_1-s_2+\ha)\end{equation*}
so that its specialization to $s_{1}=s,s_{2}=0$ gives the normalized base change $L$-function $\sL(\pi_{F'}, s+\ha)$. Then $\sL_{F'/F}(\pi, s_{1},s_{2})$ satisfies a function equation
\begin{equation*}
\sL_{F'/F}(\pi, s_{1},s_{2})=(-1)^{r(\pi_{F'})}\sL_{F'/F}(\pi, -s_{1},-s_{2})
\end{equation*}
where $(-1)^{r(\pi_{F'})}$ is the root number for the base change $\pi_{F'}$, and
\begin{equation*}
r(\pi_{F'})=\#\Big\{x\in \Sig\bigm| x \mbox{ is inert in } X'\Big\}.
\end{equation*}

For $r_{+}, r_{-}\in \ZZ_{\ge0}$, we define
\begin{equation*}
\sL^{(r_{+},r_{-})}_{F'/F}(\pi):=\left(\frac{\partial}{\partial s_{1}}\right)^{r_{+}}\left(\frac{\partial}{\partial s_{2}}\right)^{r_{-}}\sL_{F'/F}(\pi,s_{1}, s_{2})\Big|_{s_{1}=s_{2}=0}.
\end{equation*}
From the functional equation of $\sL_{F'/F}(\pi,s_{1},s_{2})$, we see that $\sL^{(r_{+},r_{-})}_{F'/F}(\pi)=0$ unless
\begin{equation*}
r_{+}+r_{-}\equiv r(\pi_{F'})\mod2.
\end{equation*}

\sss{The moduli of Shtukas with Iwahori level structure}
On the geometric side, we will consider the moduli stack of $G$-Shtukas with Iwahori level structures. The points with Iwahori level structure come in two kinds: those resembling the finite primes dividing the level $N$ for a modular curve $X_{0}(N)$ and those resembling the Archimedean place. In fact, starting with a  finite subset $\Sig\subset |X|$ together with a disjoint union decomposition $\Sig=\Sf\sqcup \Si$ and a non-negative integer $r$ such that $r\equiv \#\Si\mod 2$, we will define in \S\ref{sss:mu more} and \S\ref{sss:ShtG} a moduli stack $\Sht^{r}_{G}(\Sii)$ equipped with a map
\begin{equation*}
\Pi^{r}_{G}\colon \Sht^{r}_{G}(\Sii) \to  X^{r}\times \frSi,
\end{equation*}
where $\frSi=\prod_{x\in \Si}\Spec k(x)$. Then  $\Sht^{r}_{G}(\Sii)$ is a smooth $2r$-dimensional DM stack locally of finite type over $k$ (see Proposition \ref{p:ShtG}).
We will also consider the base change
\begin{equation*}
\Sht'^{r}_{G}(\Sii):=\Sht^{r}_{G}(\Sii)\times_{(X^{r}\times\frSi)}(X'^{r}\times\frSi'),
\end{equation*}
where $\frSi'=\prod_{x'\in \Si'}\Spec k(x')$, and $\Si'=\nu^{-1}(\Si)$. 
If we base change $\Sht'^{r}_{G}(\Sii)$ to $\kbar$, it decomposes as
\begin{equation*}
\Sht'^{r}_{G}(\Sii)\ot\kbar=\coprod_{\xi}\Sht'^{r}_{G}(\Sig;\xi),
\end{equation*}
where $\xi=(\xi_{x'})_{x'\in \Si'}$ runs over the choices of a $\kbar$-point $\xi_{x'}$ over each $x'\in \Si'$. We fix such a $\xi$.

There is an action of the spherical Hecke algebra $\sH^{\Sig}_{G}=\otimes_{x\in|X|-\Sig}\sH_{x}$ on the cohomology groups $\cohoc{*}{\Sht'^{r}_{G}(\Sig;\xi),\Ql}$, which is infinite-dimensional in the middle degree. We have an Eisenstein ideal $\cI_{\Eis}\subset \sH^{\Sig}_{G}$ defined in the same way as in \cite[\S4.1]{YZ}. We prove a spectral decomposition similar to the unramified case.

\begin{theorem}\label{th:intro spec decomp} There is a canonical decomposition of $\sH^{\Sig}_{G}$-modules
\begin{equation}\label{coho' decomp}
\cohoc{2r}{\Sht'^{r}_{G}(\Sig;\xi),\Qlbar}=\big(\bigoplus_{\fkm}V'(\xi)_{\fkm}\big)\oplus V'(\xi)_{\Eis},
\end{equation}
where
\begin{itemize}
\item  $\fkm$ runs over a finite set of maximal ideals of $\sH^{\Sig}_{G}$ which do not contain the Eisenstein ideal, and $V'(\xi)_{\fkm}$ is the generalized eigenspace of the $\sH^{\Sig}_{G}$-action on $\cohoc{2r}{\Sht'^{r}_{G}(\Sig;\xi),\Qlbar}$ corresponding to $\fkm$. Moreover,  $V'(\xi)_{\fkm}$ is finite-dimensional over $\Qlbar$.

\item  $V'(\xi)_{\Eis}$ is a finitely generated $\sH^{\Sig}_{G}$-module on which the action of $\sH^{\Sig}_{G}$ factors through $\sH^{\Sig}_{G}/\cI^{m}_{\Eis}$ for some $m>0$.
\end{itemize}
\end{theorem}

Using the cup product, we have a perfect pairing
\begin{equation}\label{Vs pairing}
(\cdot, \cdot)_{\Sht'^{r}_{G}(\Sig;\xi)}: V'(\xi)_{\fkm}\times V'(\xi)_{\fkm}\to \Qlbar.
\end{equation}

\sss{The Heegner--Drinfeld cycle}
We make the following assumptions which are analogous to the Heegner hypothesis:
\begin{eqnarray}
\label{Sf split}&&\mbox{\it All places in $\Sf$ are split in $X'$;}\\
\label{Si inert}&&\mbox{\it All places in $\Si$ are inert in $X'$.}
\end{eqnarray}
By considering rank one Shtukas on $X'$, we obtain a moduli stack $\Sht^{\un\mu}_{T}(\mi\cdot\Si')$ that depends on the data $\un\mu\in \{\pm1\}^{r}$ and $\mi\in\{\pm1\}^{\Si}$. The stack $\Sht^{\un\mu}_{T}(\mi\cdot\Si')$ is a finite \'etale cover of $X'^{r}\times\frSi'$.

To map $\Sht^{\un\mu}_{T}(\mi\cdot\Si')$ to $\Sht'^{r}_{G}(\Sii)$ we need an extra choice $\mf$, which is a section to the two-to-one map $\Sf':=\nu^{-1}(\Sf)\to \Sf$. Altogether we have chosen an element
\begin{equation}\label{mu in T}
\mu=(\un\mu,\mf,\mi)\in \frT_{r,\Sig}:=\{\pm1\}^{r}\times\Sect(\Sf'/\Sf)\times\{\pm1\}^{\Si}.
\end{equation}
From this choice we have a map (cf. \S\ref{sss HD})
\begin{equation*}
\th'^{\mu}: \Sht^{\un\mu}_{T}(\mi\cdot\Si')\to \Sht'^{r}_{G}(\Sii).
\end{equation*}
Base-changing to $\kbar$ and taking the $\xi$-component, we get a map
\begin{equation*}
\th'^{\mu}_{\xi}: \Sht^{\un\mu}_{T}(\mi\cdot\xi)\to \Sht'^{r}_{G}(\Sig;\xi).
\end{equation*}

We define the {\em Heegner--Drinfeld cycle} to be the algebraic cycle with proper support
\begin{equation*}
\cZ^{\mu}(\xi):=\th'^{\mu}_{\xi,*}[\Sht^{\un\mu}_{T}(\mi\cdot\xi)]\in \Ch_{c,r}(\Sht'^{r}_{G}(\Sig;\xi))_{\QQ}.
\end{equation*}
Its cycle class in cohomology is denoted by
\begin{equation*}
Z^{\mu}(\xi):=\cl(\cZ^{\mu}(\xi))\in \cohoc{2r}{\Sht'^{r}_{G}(\Sig;\xi),\Ql}.
\end{equation*}

\sss{Main result}

Our main theorem is the following.
\begin{theorem}[Main result, first formulation]\label{th:main} Let $\pi$ be a cuspidal automorphic representation of $G(\AA_{F})$ ramified at a finite set of places $\Sig$. Assume
\begin{itemize}
\item For each $x\in \Sig$, $\pi_{x}$ is isomorphic to an unramified twist of the Steinberg representation;
\item The ramification locus $R$ of the double cover $\nu:X'\to X$ is disjoint from $\Sig$. 
\end{itemize}
We decompose $\Sig$ as $\Sf\sqcup \Si$ in a unique way so that the conditions \eqref{Sf split} and \eqref{Si inert} hold.  Let $r$ be a non-negative integer such that
\begin{equation*}
r\equiv\#\Si\mod2.
\end{equation*}
Let $\mu,\mu'\in\frT_{r,\Sig}$. Let
\begin{equation*}
r_{+}=\{1\le i\le r\mid\mu_{i}=\mu_{i}'\}, \quad r_{-}=\{1\le i\le r\mid\mu_{i}\ne\mu_{i}'\}.
\end{equation*}
Then
\begin{equation}\label{main formula}
\frac{|\om_{X}|q^{\r/2-N}\ep_{-}(\pi\ot\y)}{2(-\log q)^{r}L(\pi,\Ad,1)}\sL^{(r_{+},r_{-})}_{F'/F}(\pi)
=\big(Z^{\mu}_{\pi}(\xi), Z^{\mu'}_{\pi}(\xi)\big)_{\Sht'^{r}_{G}(\Sig;\xi)}.
\end{equation}
Here, 
\begin{itemize}
\item $|\om_{X}|=q^{-(2g-2)}$.
\item $\ep_{-}(\pi\ot\y)\in\{\pm1\}$ is the product of the Atkin--Lehner eigenvalues of $\pi\ot\y$ at $x\in \Sig_{-}(\mu,\mu')$, where $\Sig_{-}(\mu,\mu')\subset \Sig$ is defined in \eqref{Sig -}. 
\item The automorphic representation $\pi$ gives a character $\l_{\pi}$ of $\sH^{\Sig}_{G}$ which does not factor through the Eisenstein ideal; we denote by $V'(\xi)_{\pi}$ the direct summand in \eqref{coho' decomp} corresponding to the maximal ideal $\fkm_{\pi}=\ker(\l_{\pi})$ and let $Z^{\mu}_{\pi}(\xi)$ be the projection of $Z^{\mu}(\xi)$ to $V'(\xi)_{\pi}$. 
\item The pairing $(\cdot,\cdot)_{\Sht'^{r}_{G}(\Sig;\xi)}$ on the right side of \eqref{main formula} is \eqref{Vs pairing}.
\end{itemize}
\end{theorem}

The Galois involution for the double cover $X'/X$ induces an action of $(\ZZ/2\ZZ)^{r}$ on $X'^{r}$, hence on $\Sht'^{r}_{G}(\Sig;\xi)$ by acting only on the $X'^{r}$-factor. Let $\s_{i}\in(\ZZ/2\ZZ)^{r}$ be the element with only the $i$-th coordinate nontrivial. For $0\le r_{1}\le r$, we define an idempotent in the group algebra $\QQ[(\ZZ/2\ZZ)^{r}]$ by
\begin{equation*}
\ep_{r_{1}}=\prod_{i=1}^{r_{1}}\frac{1+\s_{i}}{2}\prod_{j=r_{1}+1}^{r}\frac{1-\s_{i}}{2}.
\end{equation*}

\begin{theorem}[Main result, second formulation]\label{th:int Sht}
Keep the same assumptions as Theorem \ref{th:main}. Let $0\le r_{1}\le r$ be an integer, and $\mu\in \frT_{r,\Sig}$. Then
\begin{equation*}
\frac{|\om_{X}|q^{\r/2-N}}{2(-\log q)^{r}L(\pi,\Ad,1)}\sL^{(r_{1})}(\pi,\ha)\sL^{(r-r_{1})}(\pi\otimes\y, \ha)=\big(\ep_{r_{1}}Z^{\mu}_{\pi}(\xi), \ep_{r_{1}}Z^{\mu}_{\pi}(\xi)\big)_{\Sht'^{r}_{G}(\Sig;\xi)}.
\end{equation*}
\end{theorem}

In the special case $r_{1}=r$, we may further reformulate the theorem as follows. 

\begin{cor} Keep the same assumptions as Theorem \ref{th:main}.  Let $Y^{\mu}_{\pi}(\xi)\in \cohoc{2r}{\Sht^{r}_{G}(\Sig;\xi),\Qlbar}$ be the class of the push-forward of $Z^{\mu}_{\pi}(\xi)$ to $\Sht^{r}_{G}(\Sig;\xi)=\Sht^{r}_{G}(\Sii)\times_{\frSi}\xi$. Then $Y^{\mu}_{\pi}(\xi)$ depends only on $(r,\mf,\mi)$, and
\begin{equation*}
\frac{2^{r-1}|\om_{X}|q^{\r/2-N}}{(-\log q)^{r}L(\pi,\Ad,1)}\sL^{(r)}(\pi,\ha)\sL(\pi\otimes\y, \ha)=\big(Y^{\mu}_{\pi}(\xi), Y^{\mu}_{\pi}(\xi)\big)_{\Sht^{r}_{G}(\Sig;\xi)}.
\end{equation*}
\end{cor}

\begin{remark} Consider the case where $\Si$ consists of a single place $\infty$, $r=1$, and $\mu=\mu'$. In this case the moduli stack $\Sht^{1}_{G}(\Sii)$ over  $X$ is closely related to the moduli space of elliptic modules originally defined by Drinfeld \cite{Dr ell} (see the discussion in \S\ref{sss:rel DrMod}), the latter being a perfect analogue of a semistable integral model for modular curves $X_{0}(N)$.  In this special case, Theorem \ref{th:main} reads
\begin{equation}\label{intro GZ}
-\frac{|\om_{X}|q^{\r/2-N}}{2\log q\cdot L(\pi,\Ad,1)}\sL'(\pi_{F'},\ha)=\left(Z^{\mu}_{\pi}(\xi),Z^{\mu}_{\pi}(\xi) \right)_{\Sht'^{1}_{G}(\Sig;\xi)}.
\end{equation}
This is a direct analogue of the Gross-Zagier formula for modular curves \cite{GZ}. We understand that D. Ulmer has an unpublished proof of a formula similar to \eqref{intro GZ}. The method of our proof is quite different from that in \cite{GZ} in that we do not need to explicitly compute either side of the formula.
\end{remark}

\subsection{What's new} We highlight both the new results and new techniques in this paper compared to the unramified case treated in \cite{YZ}.

\sss{} First we compare our results with our previous ones in \cite{YZ}. 
Theorems \ref{th:main} and \ref {th:int Sht} have much wider applicability than the ones in \cite{YZ}. In particular, for a non-isotrivial elliptic curve $E$ over $F$ with semistable reductions, its $L$-function $L(E,s)$ is the automorphic $L$-function $L(\pi, s+1/2)$ for some $\pi$ satisfying the conditions of our theorems. Therefore, our results in this paper give a geometric interpretation of Taylor coefficients of $L$-functions of semistable elliptic curves over function fields. For potential applications to the arithmetic of elliptic curves, see the discussion in \S\ref{ss:BSD}.

 In addition, in this paper we study the intersection of different Heegner--Drinfeld cycles by varying the discrete datum $\mu$. As a result we get products of derivatives of $\sL(\pi,s)$ and $\sL(\pi\ot\y,s)$, as opposed to just the derivatives of their product $\sL(\pi_{F'},s)$. So Theorems \ref{th:main} and \ref {th:int Sht} are new even in the unramified case.

\sss{} Next we comment on the proof. To prove Theorem \ref{th:main}, we follow the general strategy of relative trace formulae comparison as in \cite{YZ}. In this paper, we have tried to avoid repeating similar arguments from \cite{YZ} and only write new arguments in detail. Here are some highlights of the new techniques compared to the unramified case.

%
%

The key identity between relative traces takes the form
\begin{equation*}
\left(\frac{\partial}{\partial s_{1}}\right)^{r_{+}}\left(\frac{\partial}{\partial s_{2}}\right)^{r_{-}}(q^{N_{+}s_{1}+N_{-}s_{2}}\JJ(f',s_{1},s_{2}))\Big|_{s_{1}=s_{2}=0}=\left(Z^{\mu}(\xi), f*Z^{\mu'}(\xi)\right)_{\Sht'^{r}_{G}(\Sig;\xi)}
\end{equation*}
where $f\in \sH^{\Sig\cup R}_{G}$ and $f'\in C_{c}(G(\AA))$ is a ``matching function'', and $N_{\pm}=\deg\Sig_{\pm}(\mu,\mu')$ (see \eqref{Sig +} and \eqref{Sig -}). In the unramified case, we simply took $f'=f$. At places $x\in \Sig$, the corresponding factors of $f'$ are not surprising: they are essentially characteristic functions of the Iwahori. However, it is not obvious what to put at places $x\in R$ (where $R$ is the ramification locus of $F'/F$). This is one of the main difficulties of this work.

In \S\ref{sss:hx at R} we give a somewhat surprising formula for the test function $h_{x}^{\bsq}$ to be put at $x\in R$ in $f'$. The discovery of the function $h_{x}^{\bsq}$ was guided by the geometric interpretation of orbital integrals. We wanted a moduli space $\cN_{d}$ which looked like the counterpart of $\cM_{d}$ (see Definition \ref{defn Md}) for a split quadratic extension $F\times F$ but somehow remembers the ramification locus $R$. Once we realized the correct candidate for $\cN_{d}$ (see Definition \ref{defn Nd}), the formula for $h^{\bsq}_{x}$ fell out quite naturally as counting points on $\cN_{d}$. From the spectral calculation, we get another characterization of $h^{\bsq}_{x}$ (see \S\ref{sss:fx at R}), which justifies its canonicity from a different perspective. The idea should be applicable to other situations of relative trace formulae where one needs explicit {\it ramified} test functions. We hope to return to this topic in the future.

The presence of Iwahori structures makes the geometry of the horocycles in $\Sht'^{r}_{G}(\Sii)$ much more complicated than in the unramified case, which explains the length of \S\ref{ss:horo}. The study of the horocycles is needed in order to establish a cohomological spectral decomposition. Also, the proof of the key finiteness results leading to the cohomological spectral decomposition in \S\ref{ss:coho spec decomp} uses a new strategy: we introduce ``almost isomorphisms'' between ind-perverse sheaves (i.e., we work with a quotient category of ind-perverse sheaves). Compared to our approach in \cite{YZ}, this strategy is more robust in showing qualitative results for spaces of infinite type, and should work for the cohomological spectral decomposition for higher rank groups.

\subsection{Potential arithmetic applications}\label{ss:BSD} 

\sss{Determinant of the Frobenius eigenspace}
 Let $\pi$ be a cuspidal automorphic representation  of $G(\BA)$ as in Theorem \ref{th:main}.  By the global Langlands correspondence proved by Drinfeld \cite{Dr ICM}, there is a rank two irreducible $\Qlbar$-local system  $\rho_\pi$  attached to $\pi$ over an open subset of $X$. Our convention is that $\det(\rho_{\pi})\cong \Qlbar(-1)$; in particular, $\rho_{\pi}$ is pure of weight $1$.  Let $j_{!\ast}\rho_\pi$ be the middle extension of $\rho_{\pi}$ to the complete curve $X$. The base change $\pi_{F'}$ corresponds to the local system $\nu^{*}\r_{\pi}$ on an open subset of $X'$, and we denote by $j'_{!*}\nu^{*}\r_{\pi}$ its middle extension to $X'$. Let 
 $$
W'_\pi:=\cohog{1}{ X'\otimes \kbar,\,  j'_{!\ast}\nu^{*}\rho_\pi}.
$$
This is a $\ov\BQ_\ell$-vector space with the geometric Frobenius automorphism $\Fr$ of weight $2$. The $L$-function $L(\pi_{F'},s)$ is related to $\nu^{*}\rho_{\pi}$ by
$$
L(\pi_{F'},s-\ha)=\det\left(1-q^{-s}\Fr\bigm| W_\pi'\right).
$$
 
Let $\Pi^{r}_{G}\colon\Sht^{r}_{G}(\Sig)\to X^{r}\times \frSi$ be the projection map. It is expected that under the $\sH^{\Sig}_{G}$-action, the $\l_{\pi}$-isotypical component of the complex ${\bf R}\Pi^{r}_{G,!}\ov\BQ_\ell$ on $X^r\times \frSi$ takes the form
\begin{equation}\label{RP decomp}
({\bf R}\Pi^{r}_{G,!}\ov\BQ_\ell)_{\pi}=\pi^K\otimes\Big(\underbrace{ j_{!\ast}\rho_\pi[-1] \boxtimes\cdots\boxtimes j_{!\ast}\rho_\pi[-1]}_{r\text{\, times}}\Big)\boxtimes\Big(\boxtimes_{x\in \Si}\r_{\pi,x}^{I_{x}}\Big)
\end{equation}
where $K=\prod_{x\notin \Sig}G(\cO_{x})\times \prod_{x\in \Sig}\Iw_{x}$, and $\r_{\pi,x}$ is the restriction of $\r_{\pi}$ to $\Spec F_{x}$ and $I_{x}<\Gal(F^{\sep}_{x}/F_{x})$ is the inertial group at $x$.  Pulling back to $X'^{r}\times \frSi'$, \eqref{RP decomp} implies that the generalized eigenspace $V'(\xi)_{\pi}:=V'(\xi)_{\ker(\l_{\pi})}$ in \eqref{coho' decomp} should take the form
\begin{equation*}
V'(\xi)_\pi\cong \pi^K\otimes W'^{\otimes r}_\pi \otimes \ell_{\pi,\xi}
\end{equation*}
where $\ell_{\pi,\xi}$ is the geometric stalk of $\boxtimes_{x\in \Si}\r_{\pi,x}^{I_{x}}$ at $\xi$. Note that both $\pi^{K}$  and $\ell_{\pi,\xi}$ are one-dimensional since $\pi$ is an unramified twist of the Steinberg representation at $x\in \Sig$.

Then the cohomology class of the Heegner--Drinfeld cycle gives rise to an element in $Z_\pi^\mu(\xi)\in \pi^K\otimes W'^{\otimes r}_\pi\otimes\ell_{\pi,\xi}$. It can be shown that $Z^{\mu}_{\pi}(\xi)$ is an eigenvector for the operator $\id\otimes \Fr^{\otimes r}\otimes \id$, with eigenvalue $q^r$.
Our main result (Theorem \ref{th:main}) together with the super-positivity proved in \cite[Theorem B.2]{YZ} shows that $Z^{\mu}_{\pi}(\xi)$ does not vanish when $r\geq \ord _{s=1/2}L(\pi_{F'},s)$, provided that $L(\pi_{F'},s)$ is not a constant (i.e., $2(4g-4+N+\rho)>0$).

Partly motivated by the standard conjecture about Frobenius semi-simplicity, we propose

\begin{conj}\label{conj basis}
Let $r=\ord _{s=1/2}L(\pi_{F'},s)$ (i.e., $r$ is the dimension of the {\em generalized} eigenspace of $\Fr$ on $W'_{\pi}$ with eigenvalue $q$) and $\mu\in \frT_{r,\Sig}$. Then the class $Z_\pi^\mu(\xi)$ belongs to $\pi^K\otimes \wedge^r\left(W'^{\Fr=q}_\pi\right)\ot\ell_{\pi,\xi}$.

In particular, for the eigenvalue $q$, the generalized eigenspace of the $\Fr$-action on $W'_{\pi}$ coincides with the eigenspace, and $Z_\pi^\mu(\xi)$ gives a basis of the line $\pi^K\otimes \wedge^r\left(W'^{\Fr=q}_\pi\right)\ot\ell_{\pi,\xi}$.
\end{conj}

In a forthcoming work, the authors plan to prove (assuming that \eqref{RP decomp} holds):
\begin{altenumerate}
\item  If $r_0\geq 0$ is the smallest integer $r$ such that $Z_\pi^\mu\neq 0$ for some $\mu\in\{\pm1\}^{r}$, then $\dim W'^{\Fr=q}_\pi=r_0$ and the class $Z_\pi^\mu(\xi)$ gives a basis of the line $\pi^K\otimes \wedge^{r_0}\left(W'^{\Fr=q}_\pi\right)\ot\ell_{\pi,\xi}$.

\item {\em $\ord _{s=1/2}L(\pi_{F'},s)=1$ if and only if $\dim W'^{\Fr=q}_\pi=1$}. Moreover, if $\ord _{s=1/2}L(\pi_{F'},s)=3$, then $\dim W'^{\Fr=q}_\pi=3$.
\end{altenumerate}

\sss{Elliptic curves} Let $E$ be a non-isotrivial semistable elliptic curve over $F$. Attached to $E$ is a cuspidal automorphic representation $\pi$ of $G(\AA_{F})$ such that $\r_{\pi}\cong V_{\ell}(E)^{*}\ot_{\Ql}\Qlbar$ as representations of $\Gal(F^{\sep}/F)$. In particular, $L(E,s)=L(\pi,s-\ha)$, and  $L(E_{F'},s)=L(\pi_{F'},s-\ha)$. Moreover, after choosing a semistable model $\cE'$ over $X'$, we may identify $W'_{\pi}$ with a subquotient of  $\cohog{2}{\cE'\ot\kbar, \Qlbar}$, and think of it as the $\ell$-adic Selmer group of $E$. The function-field analogue of the conjecture of Birch and Swinnerton-Dyer, as formulated by Artin and Tate \cite{T}, predicts that the $q$-eigenspace of $\Fr$ on $W'_{\pi}$ is the same as the generalized eigenspace, and is spanned by classes of sections of $\cE'$. 
The expected result (ii) above would imply that if $\ord_{s=1}L(E_{F'},s)=3$, then the $q$-eigenspace of $\Fr$ on $W'_{\pi}$ is the same as the generalized eigenspace.

While it is difficult to construct algebraic cycles on $\cE'$ spanning $W_{\pi}'^{\Fr=q}$, it may be easier to construct a basis of the line $\wedge^{r}(W_{\pi}'^{\Fr=q})$. Conjecture \ref{conj basis} proposes a candidate generator for $\wedge^{r}(W_{\pi}'^{\Fr=q})$, namely the cycle  $Z^{\mu}_{\pi}(\xi)$. It is not clear though how to relate the ambient space of $Z^{\mu}_{\pi}(\xi)$, namely $\Sht'^{r}_{G}(\Sig;\xi)$,  to powers of $\cE'$.

\subsection{Notations} 
\sss{Function field notation} Throughout this paper, we fix a finite field $k=\FF_{q}$ of characteristic $p\ne 2$. We fix a smooth, projective and geometrically connected curve $X$ over $k$. Let $F=k(X)$ be the function field of $X$.  Let $|X|$ denote the set of closed points of $X$. 

For $x\in |X|$, let $\cO_{x}$ (resp. $F_{x}$) denote the completed local ring of $X$ at $x$ (resp. the fraction field of $\cO_{x}$).  Let $\fkm_{x}\subset \cO_{x}$ be the maximal ideal and we typically denote a uniformizer of $\cO_{x}$ by $\vp_{x}$. Let $\AA_{F}$ denote the ring of ad\`eles of $F$, and let $\OO=\prod_{x\in|X|}\cO_{x}$. Let $k(x)$ denote the residue field of $\cO_{x}$ and let
\begin{equation*}
d_{x}=[k(x):k], \quad q_{x}=q^{d_{x}}=\#k(x).
\end{equation*}
Let  $v_{x}: F^{\times}_{x}\to \ZZ$ be the valuation normalized by $v_{x}(\vp_{x})=1$.

We will also consider a double covering $\nu: X'\to X$ where $X'$ is also a smooth, projective and geometrically connected curve $X$ over $k$. The function field of $X'$ is denoted by $F'$. Other notations for $X$ extend to their counterparts for $X'$.

\sss{Group-theoretic notation} Except for \S\ref{ss:Bun Iw} and \S\ref{ss:Sht Iw}, the letter $G$ always denotes the algebraic group $\PGL_{2}$ over $k$. Let $A\subset G$ be the diagonal torus. For $x\in |X|$, the standard Iwahori subgroup $\Iw_{x}$ of $G(F_{x})$ is the image of the following subgroup of $\GL_{2}(\cO_{x})$
\begin{equation*}
\wt\Iw_{x}=\left\{\mat{a}{b}{c}{d}\in \GL_{2}(\cO_{x})\biggm| c\in\fkm_{x}\right\}.
\end{equation*}
For an algebraic group $H$ over $F$, we denote
\begin{equation*}
[H]:=H(F)\bs H(\AA).
\end{equation*}

\sss{Algebro-geometric notation} Most of the algebraic stacks that appear in this paper are over the finite field $k$ (with exceptions of affine $\Ql$-schemes appearing in Theorem \ref{th:spec decomp}), and the product $S\times S'$ (without subscript) of such stacks $S$ and $S'$ is always understood to be the fiber product of $S$ and $S'$ over $\Spec k$. 

For any stack $S$ over $k$, $\Fr_{S}:S\to S$ denotes the $k$-linear Frobenius which raises functions to the $q$-th power.

For an $S$-point $x:S\to X$, we denote by $\Gamma_{x}\subset X\times S$ the graph of $x$, which is a Cartier divisor of $X\times S$.

We fix a prime $\ell$ different from $p$, and an algebraic closure $\Qlbar$ of $\Ql$. The \'etale cohomology groups in this paper are with $\Ql$ or $\Qlbar$ coefficients.

\sss*{Acknowledgement}
The authors would like to thank Benedict Gross for useful discussions and encouragement.  They also thank an anonymous referee for useful suggestions on the presentation.

\section{The analytic side: relative trace formula}\label{s:RTF}

We extend the results in \cite[\S2, \S4]{YZ} on Jacquet's RTF \cite{J86} to our current setting.  Since most arguments in \cite{YZ} extend without any difficulty, we will not repeat them, but simply indicate the necessary changes. 

A new phenomenon is that we need to choose a new test function at the places where $F'/F$ is ramified. This is done in \S\ref{ss:local test}, and is the most non-obvious point of the analytic part of this paper.

By convention, the automorphic representations we consider in this section are on $\CC$-vector spaces.

\subsection{Jacquet's RTF}

For $f\in C_c^\infty(G(\BA))$, we
consider the automorphic kernel function
\begin{align}\label{eqn kernel}
\BK_{f}(g_1,g_2)=\sum_{\gamma\in G(F)} f(g_1^{-1}\gamma g_2),\quad g_1,g_2\in G(\BA).\index{$\BK_f$}%
\end{align}
Let $A\subset G$ be the diagonal torus, and we define a distribution given by a regularized integral, for $(s_1,s_2)\in\BC^2$
\begin{align}\label{eq: J(f,s1,s2)}
\BJ(f,s_1,s_2)=\int^{{\rm reg}}_{[A]\times [A]}\BK_{f}(h_1,h_2)\lvert h_1  \rvert^{s_1+s_2}\lvert h_2  \rvert^{s_1-s_2} \eta(h_2)\,dh_1\,dh_2.\index{$\BJ(f,s)$}%
\end{align}
Here the measure on $[A]=A(F)\bs A(\AA)$ is induced from the Haar measure on $A(\AA)$ such that $\vol(A(\OO))=1$.

The regularization is the same as in \cite[\S2.2--\S2.5]{YZ}, i.e., as the limit of the integral over a certain sequence of increasing bounded subsets that cover $[A]\times [A]$. Moreover, we define a two-variable orbital integral
\begin{equation*}
\JJ(\g, f, s_{1},s_{2})=\int_{A(\AA)\times A(\AA)} f(h_1^{-1} \g h_2)|h_1h_2|^{s_{1}}|h_1/h_2|^{s_{2}}\y(h_2)\,dh_1\,dh_2.
\end{equation*}
Recall the function $\inv: G(F)\to \PP^{1}(F)-\{1\}$  defined in \cite[(2.1)]{YZ}. When $u=\inv(\gamma)\in\BP^1(F)\setminus \{0,1,\infty\}$, the integral $\JJ(\g, f, s_{1},s_{2})$ is absolutely convergent.
When $u=\inv(\gamma)\in\{0,\infty\}$, the integral defining $\JJ(\g, f, s_{1},s_{2})$ requires regularization as in \cite[\S2.5]{YZ}, and the proof in {\em loc. cit.} goes through in our two-variable setting. 
 
Now $\BJ(f,s_1,s_2)$ and $\JJ(\g, f, s_{1},s_{2})$ are in $\BC[q^{\pm s_1}, q^{\pm s_2}]$, i.e., each of them is a {\em finite} sum of the form
  $$\sum_{(n_1,n_2)\in\BZ^2}a_{n_1,n_2}\, q^{n_1s_1+n_2s_2},\quad a_{n_1,n_2}\in\BC.$$ 
We have an expansion of $\BJ(f,s_{1},s_{2})$ into a sum of orbital integrals 
\begin{align}\label{J(f)=Orb}
\BJ(f,s_1,s_2)=\sum_{\gamma\in A(F)\bs G(F)/ A(F)}\BJ(\gamma,f,s_1,s_2),
\end{align}
We also define 
\begin{align}\label{orb J(u,s)}
\BJ(u,f,s_1,s_2)=\sum_{\gamma\in A(F)\bs G(F)/A(F),\,\inv(\gamma)=u}\BJ(\gamma,f,s_1,s_2),\quad u\in \BP^1(F)-\{1\}.\index{$\BJ(u,f,s)$}%
\end{align}

\subsection{The Eisenstein ideal}

For $x\in |X|$, let $\sH_{x}=C_{c}(G(\cO_{x})\bs G(F_{x})/G(\cO_{x}))$ be the spherical Hecke algebra of $G(F_{x})$. 
For a finite set $S$ of closed points of $X$ , define $\sH^{S}_{G}=\otimes_{x\in|X|-S}\sH_{x}$.
In \cite[\S4.1]{YZ} we defined the Eisenstein ideal $\cI_{\Eis}\subset \sH_{G}$ for the full spherical Hecke algebra $\sH_{G}=\otimes_{x\in |X|}\sH_{x}$, as the kernel of the composition of ring homomorphisms
\begin{equation*}
a_{\Eis}\colon\sH_G\xrightarrow{\Sat}\QQ[\Div(X)]\surj \QQ[\Pic_{X}(k)] .
\end{equation*}
Here the first map $\Sat$ is the tensor product of Satake transforms $\sH_{x}\to \QQ[t_{x}, t^{-1}_{x}]$. We restrict the homomorphism to the subalgebra $\sH^{S}_{G}$
\begin{equation*}
a^S_{\Eis}\colon\sH_G^S\xrightarrow{\Sat}\QQ[\Div(X-S)]\to \QQ[\Pic_{X}(k)]
\end{equation*}
and define
\begin{equation*}
\cI_{\Eis}^S\colon=\Ker\left(a^S_{\Eis}:\sH_G^S\to \QQ[\Pic_{X}(k)]\right).
\end{equation*}

Recall from \cite[4.1.2]{YZ} that the image of $a_{\Eis}$, hence that of $a_{\Eis}^{S}$ lies in $\QQ[\Pic_{X}(k)]^{\io_{\Pic}}$ for an involution $\io_{\Pic}$ on $\QQ[\Pic_{X}(k)]$. We have the following analogue of \cite[Lemma 4.2]{YZ} with the same proof.

\begin{lemma}\label{l:aEis surj} The map $a^{S}_{\Eis}:\sH_G^S\to \QQ[\Pic_{X}(k)]^{\io_{\Pic}}$ is surjective.
\end{lemma}

We have a generalization of \cite[Theorem 4.3]{YZ}.

\begin{thm}\label{thm K eis=0} Let $f^S\in\CI^S_{\Eis}$ and let $f_{S}\in C_{c}^{\infty}(G(\AA_{S}))$ be left invariant under the Iwahori $\Iw_{S}=\prod_{x\in S}\Iw_{x}$. Then for $f=f_S\otimes f^S\in C_c^\infty(G(\BA))$ we have
\[
\BK_f=\BK_{f,{\rm cusp}}+\BK_{f,{\rm sp}}.
\]
Here $\BK_{f,{\rm cusp}}$ (resp. $\BK_{f,{\rm sp}}$) is the projection of $\KK_{f}$ to the cuspidal spectrum (resp. residual spectrum, i.e., one-dimensional representations), see \cite[\S4.2]{YZ}.
\end{thm}
\begin{proof}
We indicate how to the modify the proof of \cite[Theorem 4.3]{YZ}. Let $K^S=\prod_{x\notin S}G(\CO_x)$, and let $K=K_S\cdot K^S $ be a compact open subgroup of $G(\BA)$ such that $K_{S}\subset \Iw_{S}$ and that $f$ is bi-$K$-invariant. The analogue of equation \cite[(4.9)]{YZ} now reads
\begin{align}\label{eqn K eis chi 1}
\BK_{f,\Eis,\chi}(x,y)=\frac{\log q}{2\pi i}\sum_{\a,\b}\int_0^{\frac{2\pi i}{\log q}} (\rho_{\chi,u}(f)\phi_{\a},\phi_{\b})\, E(x,\phi_{\a},u,\chi)\ov {E(y,\phi_{\b},u,\chi)}\,du,
\end{align}
where $\{\phi_{\a}\}$ is an orthonormal basis of $V_\chi^K$.  Since $f$ is left invariant under the Iwahori $\Iw_{S}\times K^{S}$, $(\rho_{\chi,u}(f)\phi_{\a},\phi_{\b})=0$ unless the $\Iw_{S}\times K^{S}$-average  of $\phi_{\b}$ is nonzero; i.e.,  $(\rho_{\chi,u}(f)\phi_{\a},\phi_{\b})=0$ unless $V_{\chi}^{\Iw_{S}\times K^{S}}\ne0$ which happens if and only if $\chi$ is everywhere unramified.  When $\chi$ is everywhere unramified,  we have
$$
(\rho_{\chi,u}(f)\phi_{\a},\phi_{\b})=\chi_{u+1/2}(a^{S}_{\Eis}(f^S)) (\rho_{\chi,u}(f_S\otimes 1_{K^S})\phi_{\a},\phi_{\b}).
$$
In particular, if $f^S$ lies in the Eisenstein ideal, then $a^{S}_{\Eis}(f^S)=0$, and hence the integrand in \eqref{eqn K eis chi 1} vanishes.
This completes the proof.
\end{proof}

\subsection{The spherical character: global and local}
\sss{Global spherical characters and period integral}
We first recall from \cite[\S4.3]{YZ} the global  spherical character. Let $\pi$ be a cuspidal automorphic representation of $G(\BA)$, endowed with the natural Hermitian form given by the Petersson inner product: $\pair{\phi,\phi'}$ for $\phi,\phi'\in\pi$.

For a character $\chi:F^\times\bs\BA^\times\to \BC^\times$, the $(A,\chi)-$period integral for $\phi\in \pi$ is defined as
\begin{align}\label{eqn P-chi}
\sP_\chi(\phi,s):=\int_{[A]}\phi(h)\chi(h)\big|h\big |^s\, dh.
\end{align}
We simply write $\sP(\phi,s)$ if $\chi=\one$ is trivial. The global spherical character (relative to $(A\times A, 1\times \eta)$) associated to $\pi$ is a distribution on $G(\BA)$ defined by
\begin{align}\label{eqn dist J pi}
\BJ_{\pi}(f,s_1,s_2)=\sum_{\{\phi\}}\frac{ \sP(\pi(f)\phi,s_1+s_2)\sP_\eta(\ov{\phi},s_1-s_2)}{\pair{\phi,\phi}},\quad f\in C_c^\infty(G(\BA)),
\end{align}
where the sum runs over an {\em orthogonal} basis $\{\phi\}$  of $\pi$. This expression is independent of the choice of the measure on $G(\BA)$ as long as we use the same measure to define the operator $\pi(f)$ and the Petersson inner product.  The function $\BJ_{\pi}(f,s_1,s_2)$ defines an element in  $\BC[q^{\pm s_1}, q^{\pm s_2}]$.

Using Theorem \ref{thm K eis=0}, the same argument of \cite[Lemma 4.4]{YZ} proves the following Lemma.
\begin{lem}\label{l:J Eis ideal} Let $f$ be the same as in Theorem \ref{thm K eis=0}. Then
$$
\BJ(f,s_1,s_2)=\sum_{\pi }\,\BJ_\pi(f,s_1,s_2),
$$
where the sum runs over all {\em cuspidal} automorphic representations $\pi$ of $G(\BA)$ and the summand $\BJ_\pi(f,s)$ is zero for all but finitely many $\pi$.
\end{lem}

\sss{Local spherical characters}
We now recall the factorization of the global spherical character \eqref{eqn dist J pi} into a product of local spherical characters. For unexplained notation and convention we refer to the proof of \cite[Prop. 4.5]{YZ}.

Let $\psi:F\bs \AA\to \CC^{\times}$ be a nontrivial character, and let $\psi_{x}$ be its restriction to $F_{x}$. For the discussion of the local spherical characters,  we will use Tamagawa measures on various groups, which differ from our earlier convention. Strictly speaking, as in {\em loc. cit.}, the measure on $A(\BA)=\BA^\times$ is not the Tamagawa measure, but an unnormalized (decomposable) one $\prod_{x\in|X|} d^\times t_x$ where  $d^\times t_x=\zeta_x(1)\frac{dt_x}{|t_x|}$ for the self-dual measure $dt_x$ (with respect to $\psi_{x}$). In particular, we have $\vol(\cO_{x}^{\times})=1$ when $\psi_x$ is unramified (i.e., the conductor of $\psi_x$ is $\cO_{x}$). Similar remark applies to the measure $G(\BA)$, cf.  \cite[p.804]{YZ}.

We consider the Whittaker model of $\pi_x$ with respect to the character $\psi_x$, denoted by $\CW_{\psi_x}(\pi_x)$. For $\phi=\otimes_{x\in|X|}\phi_x\in \pi=\otimes'_{x\in|X|}\pi_x$, the $\psi$-Whittaker coefficient $W_\phi$ decomposes as a product $\otimes_{x\in|X|} W_x$, where $W_x\in \CW_{\psi_x}(\pi_x)$.  We define a normalized  linear functional
\begin{align*}
\lambda^\nat_x(W_x,\eta_x,s):=\frac{1}{L(\pi_x\otimes\eta_x,s+1/2)}\int_{F_x^\times}W_x\left(\matrixx{a}{}{}{1}\right)\eta_x(a)|a|^s\,d^\times a.
\end{align*}
We define a local (invariant) inner product $\theta^\nat_x$ on the Whittaker model $\CW_{\psi_x}(\pi_x)$ 
\begin{align*}
\theta^\nat_x(W_x,W_x'):=\frac{1}{L(\pi_x\times\wt\pi_x,1)}\int_{F_x^\times}W_x\left(\matrixx{a}{}{}{1}\right)\ov {W'_{x}}\left(\matrixx{a}{}{}{1}\right)\,d^\times a.
\end{align*}
Now we define  the local spherical character as
\begin{align}\label{eqn J loc}
\BJ_{\pi_x}(f_x,s_1,s_2):=
\sum_{\{W_{i}\}}
\frac{\lambda_x^\nat( \pi_{x}(f_{x})W_{i}, {\bf 1}_x, s_1+s_2)\lambda_x^\nat(\ov{W_{i}}, \eta_x, s_1-s_2)}{\theta_x^\nat(W_{i},W_{i})}.
\end{align}
where the sum runs over an {\em orthogonal} basis $\{W_{i}\}$ of $\CW_{\psi_x}(\pi_{x})$. 
By the product decomposition of the period integrals \eqref{eqn P-chi} and the Petersson inner product (cf. the proof of \cite[Prop. 4.5]{YZ}), the global spherical character decomposes into a product of local ones (cf. \cite[(4.16)]{YZ}):
\begin{align}\label{char g2l}
\BJ_{\pi}(f,s_1,s_2)=\lvert \omega_X\rvert^{-1}\,\frac{L(\pi,s_{1}+s_{2}+\ha)L(\pi\otimes\y, s_{1}-s_{2}+\ha)}{2\,L(\pi,\Ad,1)}\prod_{x\in |X|}\BJ_{\pi_x}(f_x,s_1,s_2).
\end{align}
We note that the factor  $\lvert \omega_X\rvert^{-1}$ is due to the fact that in our earlier definition \eqref{eq: J(f,s1,s2)} of $\BJ(f,s_1,s_2)$, the measure on $A(\BA)$ gives $\vol(A(\BO))=1$, while the (unnormalized) Tamagawa measure gives $\vol(A(\BO))=\lvert \omega_X\rvert^{1/2}$.

\subsection{Local test functions}\label{ss:local test}

Out test function $f\in C_{c}^{\infty}(G(\AA))$ will be a pure tensor $f=\otimes_{x\in |X|}f_{x}$ where $f_{x}\in \sH_{x}$ is in the spherical Hecke algebra for $x\notin \Sig\cup R$. Below we define the local components $f_{x}$ for $x\in R$ (in \S\ref{sss:hx at R}-\ref{sss:fx at R}) and for $x\in \Sigma$ (in \S\ref{sss:fx at Sig}).

For any place $x\in |X|$, let $p_{x}: \GL_{2}(F_{x})\to G(F_{x})$ be the projection. The fibers of $p_{x}$ are torsors under $F^{\times}_{x}$ and are equipped with $F^{\times}_{x}$-invariant measures such that any $\cO^{\times}_{x}$-orbit has volume 1. Let $p_{x,*}: C_c^\infty(\GL_2(F_x))\to C_c^\infty(G(F_x))$ be the map defined by integration along the fibers of $p_{x}$ with the above-defined measure.

\sss{The function $h^{\bsq}_{x}$}\label{sss:hx at R} For $a\in \cO_{x}$, we denote $\ov a$ its image in $k(x)$. For any $n\in\ZZ$, let $\Mat_{2}(\cO_{x})_{v_{x}(\det)=n}$ be the set of $2$-by-$2$ matrices $M$ with entries in $\cO_{x}$ such that $v_{x}(\det(M))=n$.

At $x\in R$, the character $\eta_{x}|_{\cO^{\times}_{x}}$ factors through the unique nontrivial character $\ov\eta_{x}: k(x)^{\times}\to\{\pm1\}$. We also denote by $\ov \eta_{x}: k(x)\to \{0,\pm1\}$ its extension by zero to the whole $k(x)$. 

When $x\in R$, let $\wt h^{\bsq}_{x}\in  C_c^\infty(\GL_2(F_x))$ be the function supported on $\Mat_{2}(\cO_{x})_{v_{x}(\det)=1}$ given by 
\begin{equation}\label{eqn h bsq}
\wt h^{\bsq}_{x}((a_{ij}))=\begin{cases} \frac{1}{2}\prod_{i,j\in\{1,2\}}(1+\ov\eta_{x}(\ov a_{ij})) & \textup{if }a_{ij}\in \cO^{\times}_{x}, \forall i,j\in\{1,2\} ;\\
\prod_{i,j\in\{1,2\}}(1+\ov\eta_{x}(\ov a_{ij})) & \textup{otherwise.}
\end{cases}
\end{equation}
Define
\begin{equation*}
h^{\bsq}_{x}=p_{x,*}\wt h^{\bsq}_{x}\in C_c^\infty(G(F_x)). 
\end{equation*} 

We give an interpretation of the formula \eqref{eqn h bsq} as counting the number of certain ``square-roots" of $(a_{ij})$. Let $\Xi_{x}$ be the set of pairs of matrices $\left(\bigl[\begin{smallmatrix}a_{11} & a_{12}\\ a_{21}& a_{22}\end{smallmatrix}\bigr] , \bigl[\begin{smallmatrix}\a_{11} & \a_{12}\\\a_{21}& \a_{22}\end{smallmatrix}\bigr] \right) \in \Mat_{2}(\cO_{x})\times\Mat_{2}(k(x))$ such that
\begin{enumerate}
\item for $1\le i,j\le 2$, $\a_{ij}^{2}=\ov a_{ij}$, the image of $a_{ij}$ in $k(x)$;
\item $\det(\a_{ij})=0$ ;
\item $v_{x}(\det(a_{ij}))=1$. 
\end{enumerate}

\begin{lem} \label{lem mu Xix} Let $\mu_{x}: \Xi_{x}\to \Mat_{2}(\cO_{x})$ be the projection to the first factor $(a_{ij})$.  We have
\begin{equation}\label{one Xix}
\wt h^{\bsq}_{x}=\mu_{x,*}\one_{\Xi_{x}}.
\end{equation}
\end{lem}
\begin{proof} Let $(a_{ij})\in\Mat_{2}(\cO_{x})_{v_{x}(\det)=1}$ be such that all $a_{ij}$ are squares. Then its preimage in $\Xi_{D,x}$ consists of $(\a_{ij})\in\Mat_{2}(k(x))$ where $\a_{ij}$ is a square root of $\ov a_{ij}$, such that $\det(\a_{ij})=0$. If all $a_{ij}$ are units, among the $\prod_{i,j}(1+\ov\eta_x(\ov a_{ij}))=2^{4}=16$ choices of $(\a_{ij})$, only half of them satisfy $\det(\a_{ij})=0$. Hence the preimage of such $(a_{ij})$ in $\Xi_{x}$ consists of $8$ elements. If at least one of $a_{ij}$  is non-unit, then the condition $v_{x}(\det(a_{ij}))=1$ implies $\det(\a_{ij})=0$. Therefore, the preimage of such $(a_{ij})\in \Mat_{2}(\cO_{x})_{v_{x}(\det)=1}$ in $\Xi_{x}$ has cardinality given by $\prod_{i,j}(1+\ov\eta_x(\ov a_{ij}))$, as desired by \eqref{eqn h bsq}. 
\end{proof}

\sss{The function $f^{\bsq}_{x}$}\label{sss:fx at R} We introduce another test function, closely related to $h^{\bsq}_{x}$, which will be useful in the calculation of its action on representations.

For $x\in R$, let $\wt f^{\bsq}_{x}$ be the function supported on $\Mat_{2}(\cO_{x})_{v_{x}(\det)=1}$ given by the formula
\begin{equation*}
\wt f^{\bsq}_{x}((a_{ij}))=\begin{cases}\ov\eta_{x}(\ov a_{11}\ov a_{12}) & \textup{if }a_{11}, a_{12}\in \cO^{\times}_{x};\\
\ov\eta_{x}(\ov a_{21}\ov a_{22}) & \textup{if }a_{21}, a_{22}\in \cO^{\times}_{x};\\
0 & \textup{otherwise.}
\end{cases}
\end{equation*}
Note that the first two cases above are not mutually exclusive, but when all $a_{ij}\in \cO^{\times}_{x}$ we have $\ov\eta_{x}(\ov a_{11}\ov a_{12})=\ov\eta_{x}(\ov a_{21}\ov a_{22})$ because the rank of $(\ov a_{ij})\in \Mat_2(k(x))$ is one.

We then define
\begin{equation*}
f^{\bsq}_{x}=p_{x*}\wt f^{\bsq}_{x}\in C_c^\infty(G(F_x)).
\end{equation*}

\begin{lemma} The function $\wt f^{\bsq}_{x}$ is characterized up to a scalar by the following three properties:
\begin{enumerate}
\item Its support is contained in $\Mat_{2}(\cO_{x})_{v_{x}(\det)=1}$;
\item It is left invariant under $\GL_{2}(\cO_{x})$;
\item Under the action of the diagonal torus $\wt A(\cO_{x})$ by right multiplication, it is an eigenfunction with eigencharacter $\diag(a,d)\mapsto \ov\eta_{x}(a/d)$.
\end{enumerate}
Furthermore, we have
\begin{align}\label{wt f bsq}
\wt f^{\bsq}_{x}=\sum_{u\in k(x)^\times}\ov\eta_x( u)\cdot {\bf 1}_{ \GL_2(\cO_x) \smat{1}{u}{}{\varpi_x}  }.
\end{align}
\end{lemma}
\begin{proof} Let $\cF$ be the space $\CC$-valued functions satisfying the above conditions. The coset space $\GL_{2}(\cO_{x})\bs \Mat_{2}(\cO_{x})_{v_{x}(\det)=1}$ has representatives given by 
\begin{equation*}
\mat{\vp_{x}}{0}{0}{1}, \quad \mat{1}{u}{0}{\vp_{x}}, u\in k(x).
\end{equation*}
We have a bijection $\GL_{2}(\cO_{x})\bs \Mat_{2}(\cO_{x})_{v_{x}(\det)=1}\cong \PP^{1}(k(x))=k(x)\cup\{\infty\}$ by sending $\smat{\vp_{x}}{0}{0}{1}$ to $\infty$ and $\smat{1}{u}{0}{\vp_{x}}$ to $u$. The right multiplication of $ \wt A(\cO_{x})$ on $\GL_{2}(\cO_{x})\bs \Mat_{2}(\cO_{x})_{v_{x}(\det)=1}$ factors through $\wt A(\cO_{x})\to \wt A(k(x))$, and $\diag(a,d)$ acts as $u\mapsto (\ov d/\ov a)\cdot u$ ($u\in \PP^{1}(k(x))$). Therefore $\cF$ is isomorphic to the $\ov\y_{x}$-eigenspace of $\wt A(k(x))$ on $C(\PP^{1}(k(x)))$ under right translation. The latter space is one-dimensional and is spanned by $f_{\eta}: u\mapsto \ov\eta_{x}(u)$ for $u\in k(x)^{\times}$ and zero for $u=0$ or $\infty$. Hence $\dim_{\CC}\cF=1$. 

The RHS of the expression \eqref{wt f bsq} is the function in $\cF$ corresponding to $f_{\eta}$, therefore it is a constant multiple of $\wt f^{\bsq}_{x}$. But both sides take value $1$ at $\smat{1}{1}{0}{\vp_{x}}$, they must be equal. This proves the lemma.

\end{proof}

We compare the test functions $h_x^\bsq$ and $f_x^\bsq$.

\begin{lem}\label{lem h-f}
The difference $h_x^\bsq-f_x^\bsq$ is a sum of two functions, one is invariant under the right translation by $A(\cO_x)$, and the other is $\eta$-eigen under the left translation by $A(\cO_x)$.
\end{lem}
\begin{proof}
The function $\wt h_x^\bsq$ can be written as
 $$\wt h_x^\bsq=\Phi_0- \frac{1}{2}\Phi_1,$$ where both $\Phi_0$ and $\Phi_1$ are supported on $\Mat_{2}(\cO_{x})_{v_{x}(\det)=1}$: 
\begin{equation*}
\Phi_0((a_{ij}))=\prod_{i,j\in\{1,2\}}(1+\ov\eta_{x}(\ov a_{ij})) 
\end{equation*}
and
\begin{equation*}
\Phi_1((a_{ij}))=\begin{cases}\prod_{i,j\in\{1,2\}}(1+\ov\eta_{x}(\ov a_{ij})) & \textup{if } a_{ij}\in \cO^{\times}_{x}, \forall  i,j\in\{1, 2\};\\
0 & \textup{otherwise.}
\end{cases}
\end{equation*}

For any subset $S\subset \{(1,1),(1,2),(2,1),(2,2)\}$, define the following functions supported on $\Mat_{2}(\cO_{x})_{v_{x}(\det)=1}$:
\begin{eqnarray*}
\wt\delta_{0,S}((a_{ij}))&\colon=&\prod_{(i,j)\in S}\ov\eta_x(\ov a_{ij}),\\
\wt\delta_{1,S}((a_{ij}))&\colon=&\begin{cases}\prod_{(i,j)\in S}\ov\eta_x(\ov a_{ij}), & a_{ij}\in  \cO^{\times}_{x}, \forall  i,j\in\{1, 2\};\\ 0 & \textup{otherwise.}
\end{cases}
\end{eqnarray*}
Then
\begin{equation*}
\Phi_0=\sum_{S}\wt\delta_{0,S}, \quad \Phi_{1}=\sum_{S}\wt\delta_{1,S},
\end{equation*}
hence
\begin{equation}\label{hsq in del}
\wt h_x^\bsq=\sum_{S}\wt\d_{0,S}-\ha\sum_{S}\wt\d_{1,S}.
\end{equation}

On the other hand, let $S_{1\ast}=\{(1,1),(1,2)\}$ (entries in the first row) and  $S_{2\ast}=\{(2,1),(2,2)\}$ (entries in the second row). From the definition of $\wt f_{x}^{\bsq}$, we have
\begin{equation}\label{fsq in del}
\wt f_x^\bsq=\wt\delta_{0,S_{1\ast}}+\wt\delta_{0,S_{2\ast}}-\frac{1}{2} \left(\wt\delta_{1,S_{1\ast}}+\wt\delta_{1,S_{2\ast}}\right).
\end{equation}
In fact, the only non-obvious part of the equality is when all four entries are units, in which cases all four functions $\wt\delta_{0,S_{1\ast}}$, $\wt\delta_{0,S_{2\ast}}$, $\wt\delta_{1,S_{1\ast}}$ and $\wt\delta_{1,S_{2\ast}}$  take the same value. Comparing \eqref{hsq in del} and \eqref{fsq in del}, we see that $\wt h_{x}^{\bsq}- \wt f^{\bsq}_{x}$ is a linear combination of $\wt \d_{0,S}$ and $\wt \d_{1,S}$ for $S$ in one of the three cases
\begin{enumerate}
\item $|S|$ is odd;
\item $S$ is either a column, or contains every entry;
\item $S$ is one of the two diagonals.
\end{enumerate}
Therefore $h_{x}^{\bsq}-f^{\bsq}_{x}$ is a linear combination of $\d_{0,S}=p_{x*}\wt\d_{0,S}$ and $\d_{1,S}=p_{x*}\wt\d_{1,S}$ for $S$ in one of the above three cases.

In case (1), $\wt\d_{0,S}$ and $\wt\d_{1,S}$ are eigenfunctions under the translation by scalar matrices in $\cO_{x}^{\times}$ with nontrivial eigenvalue $\y_{x}$, therefore $\d_{0,S}=\d_{1,S}=0$.

In case (2), $\wt\d_{0,S}$ and $\wt\d_{1,S}$ are right invariant under $\wt A(\cO_{x})$. Therefore $\d_{0,S}$ and $\d_{1,S}$ are right invariant under $A(\cO_{x})$.

In case (3), $\wt\d_{0,S}$ and $\wt\d_{1,S}$ are eigen under the left translation by $\wt A(\cO_x)$ with respect to the character $\diag(a,d)\mapsto \y_{x}(a/d)$, and hence $\delta_{0,S}$ and $\d_{1,S}$ are $\eta_{x}$-eigen under the left translation by $A(\cO_x)$.

Combining these calculations, we have proved the lemma.
\end{proof}

\sss{}\label{sss:fx at Sig} We fix a decomposition 
\begin{equation}\label{Sig pm}
\Sig=\Sig_{+}\sqcup\Sig_{-}.
\end{equation}
Let $N_{\pm}=\deg \Sig_{\pm}$. Later such a decomposition will come from a pair $\mu,\mu'\in \frT_{r,\Sig}$ (see \eqref{Sig +}, \eqref{Sig -}).

For each $x\in \Sig$, we define a subset $\bJ_{x}\subset G(\cO_{x})$ by
\begin{equation}\label{def bJx}
\bJ_{x}=\begin{cases} \left\{g\in G(\cO_{x})| g\equiv\mat{*}{*}{0}{*}\mod \fkm_{x} \right\}=\Iw_{x}, & \textup{ if }x\in \Sig_{+},\\ \\
\left\{g\in G(\cO_{x})| g\equiv\mat{*}{*}{*}{0}\mod \fkm_{x} \right\}=\Iw_{x} \cdot w, & \textup{ if }x\in \Sig_{-}.\end{cases}
\end{equation}
Here $w=\smat{}{1}{-1}{}$ is the Weyl element. The local component $f_{x}$ of our test function $f$ at $x\in \Sigma$ will be the characteristic function of $\bJ_{x}$.

\subsection{Calculations of local spherical characters}
In this subsection we compute the local distributions $\JJ_{\pi_{x}}(f_{x}, s_{1},s_{2})$ for certain pairs $(\pi_{x}, f_{x})$. We always assume that the additive character $\psi_x$ is unramified. It follows that our measure $d^\times t_x=\zeta_x(1)\frac{dt_x}{|t_x|}$ on $A(F_{x})=F_x^\times$ gives $\vol(\CO_x^\times)=1$.

\sss{The case $x\in R$ and $\pi_x$ unramified} We consider the test function introduced in \S\ref{sss:fx at R}
$$
\wt f_{x}=\wt f^{\bsq}_{x},\quad f_{x}=f^\bsq_{x}.
$$
We need an equivalent expression of the local spherical character \eqref{eqn J loc}:
\begin{align}\label{J f vee}
\BJ_{\pi_{x}}(f_{x},s_1,s_2)= \sum_{\{W_{i}\}}  \frac{\lambda^\nat_{x}( W_{i}, {\bf1}, s_1+s_2)\lambda^\nat_{x}\left(\ov{\pi_{x}(f^\vee_{x})   W_{i}}, \eta_{x}, s_1-s_2\right)}{\theta^\nat_{x}(W_{i},W_{i})},
\end{align}
where 
$$
f^\vee_{x}(g)\colon =\ov{f_{x}(g^{-1})}.
$$
Similar definition applies to the test function $\wt f_{x}$ on $\GL_2(F_{x})$.  By \eqref{wt f bsq}, we have
$$
\wt f^\vee_{x}=\sum_{u\in k(x)^\times}\ov\eta_{x}(u)\cdot {\bf 1}_{ \smat{1}{u}{}{\varpi_{x}}^{-1} \GL_2(\cO_{x})  }.
$$

\begin{lem} \label{lem f bsq W}
Let $\pi_{x}$ be unramified and $K_{x}=G(\CO_{x})$. Let $W_0\in \CW_{\psi_x}(\pi_{x})^{K_{x}}$ be the unique element such that $W_0(1_2)=1$.
Then 
$$
\pi(f^\vee_{x})W_0\left( \matrixx{a}{}{}{1} \right)=\begin{cases}\vol(K_{x})\eta_{x}(-a) \cdot q^{1/2}_{x}\epsilon(\eta_{x},1/2,\psi_{x}),&v_{x}(a)=-1,\\
0,&\text{otherwise}.
\end{cases}
$$
Here the local $\epsilon$-factor for the quadratic character $\eta_{x}$ is given by
$$
\epsilon(\eta_{x},1/2,\psi_{x})=q^{-1/2}_{x}\sum_{u\in k(x)^\times}\eta_{x}(a'u)\psi_{x}(a'u)
$$
where $a'\in F^\times_{x}$ is any element with $v_{x}(a')=-1$.

\end{lem}

\begin{proof} Let $\smat{\alpha}{}{}{\beta}\in\SL_2(\BC)$ (i.e., $\alpha\beta=1$) be the Satake parameter of $\pi$. By Casselman--Shalika formula, 
we have
$$
W_0\left( \matrixx{\varpi_{x}^n}{}{}{1} \right)=\begin{cases}
q_{x}^{-n/2}\frac{\alpha^{n+1}-\beta^{n+1}}{\alpha-\beta},& n\geq 0,\\
0,& n<0.
\end{cases}
$$

On the other hand, we have
\begin{align*}
&\pi_{x}\left(\matrixx{1}{u}{}{\varpi_{x}}^{-1}\right)W_0\left( \matrixx{a}{}{}{1} \right)=W_0\left( \matrixx{a}{}{}{1}\matrixx{1}{u}{}{\varpi_{x}}^{-1} \right)
\\&=W_0\left(\mat{1}{-au}{}{1}\mat{a}{}{}{\vp^{-1}_{x}} \right)=\psi_{x}(-ua)W_0\left( \matrixx{a\varpi_{x}}{}{}{1} \right).
\end{align*}
It follows that
\begin{align*}
\pi_{x}(f^\vee_{x})W_0\left( \matrixx{a}{}{}{1} \right)&= \vol(K_{x})\left(\sum_{u\in k(x)^{\times}  } \ov \eta_{x}(u)\psi_{x}(-ua)\right)W_0 \left( \matrixx{a\varpi_{x}}{}{}{1} \right).
\end{align*}
By the support of $W_0$, the second factor in the RHS vanishes if $v_{x}(a)\leq -2$. Since $\psi_{x}$ is unramified,   the first factor in the RHS vanishes if $v_{x}(a)\geq 0$. When $v_{x}(a)=-1$, we have
\begin{align*}
\pi_{x}(f_{x}^\vee)W_0\left( \matrixx{a}{}{}{1} \right)= &\vol(K_{x})\left(\sum_{u\in k(x)^\times  }   \ov\eta_{x}(u)\psi_{x}(-au)\right) \\
=&\vol(K_{x})\eta_{x}(-a) \left(\sum_{u\in k(x)^\times  }   \eta_{x}(-au)\psi_{x}(-au)\right)\\
=&\vol(K_{x})\eta_{x}(-a) \cdot q^{1/2}_{x}\epsilon(\eta_{x},1/2,\psi_{x}).
\end{align*}
This completes the proof.
\end{proof}

\begin{prop}\label{prop R}
Let $\pi_{x}$ be unramified, and $F'_{x}/F_{x}$ ramified. 
Then
$$
\BJ_{\pi_x}(h^\bsq_{x},s_{1},s_{2})=\BJ_{\pi_x}(f^\bsq_{x},s_{1},s_{2})=\vol(G(\cO_{x}))\z_{x}(2)\cdot \eta_{x}(-1) \epsilon(\eta_{x},1/2,\psi_{x})\cdot q_{x}^{s_1-s_2+1/2}.
$$
\end{prop}

\begin{proof}We use the formula \eqref{J f vee}  for the local spherical character evaluated at $f_{x}=f_{x}^\bsq$.
Now we note that $f^\vee_{x}$ is right invariant under $K_{x}=G(\cO_{x})$. Therefore we may simplify 
 the sum into one term involving only the spherical vector $W_0\in \CW_{\psi_x}(\pi_{x})^{K_{x}}$ (normalized so that $W_{0}(1_{2})=1$):
\begin{align}\label{J unram R}
\BJ_{\pi_{x}}(f_{x},s_1,s_2)=\frac{\lambda^\nat_{x}( W_{0}, {\bf1}, s_1+s_2)\lambda^\nat_{x}\left(\ov{\pi(f^\vee)  W_{0}}, \eta_{x}, s_1-s_2\right)}{\theta^\nat_{x}(W_0,W_0)}.
\end{align}

Since $\pi_{x}$ is unramified, we have 
\begin{equation}\label{lam unram 1}
\lambda^\nat_{x}( W_{0}, {\bf1}, s)=1.
\end{equation}
The quadratic character $\eta_{x}$ is ramified and hence
$$
L(\pi_{x}\otimes\eta_{x},s)=1.
$$ 
Using this and Lemma \ref{lem f bsq W}, we get
\begin{equation}\label{lam unram eta}
\lambda^\nat_{x}\left(\ov{\pi_{x}(f_{x}^\vee)  W_{0}}, \eta_{x}, s\right)=\vol(K_{x})\y_{x}(-1)q_{x}^{1/2}\ep(\y_{x}, 1/2, \psi_{x})\cdot q_{x}^{s}.
\end{equation}
Again since $\pi_{x}$ is unramified (and $\psi_x$ unramified), we have 
\begin{equation}\label{th unram}
\theta^\nat_{x}(W_0,W_0)=1-q^{-2}_{x}=\z_{x}(2)^{-1}.
\end{equation}
Plugging \eqref{lam unram 1}, \eqref{lam unram eta} and \eqref{th unram} into \eqref{J unram R}, we get the desired formula for $\BJ_{\pi_x}(f^\bsq_{x},s_1,s_2)$. 

To show $\BJ_{\pi_x}(h^\bsq_{x},s_1,s_2)=\BJ_{\pi_x}(f^\bsq_{x},s_1,s_2)$, by Lemma \ref{lem h-f}, it suffices to show that $\BJ_{\pi_x}(f,s_1,s_2)=0$ when $f$ is 
either
\begin{enumerate}
\item
 invariant under right translation by $A(\cO_{x})$, or
 \item $\eta_{x}$-eigen under left translation by $A(\cO_{x})$.
\end{enumerate}
In the first case,  $f^\vee$ is invariant under the left translation by $A(\cO_{x})$. The desired vanishing follows from the formula \eqref{J f vee}, and the fact that the linear functional $\l^{\nat}_{x}(-,\y_{x}, s)$ of $\pi_{x}$ is $\eta_{x}$-eigen under $A(\cO_{x})$. In the second case, the desired vanishing follows from the formula \eqref{eqn J loc}, and the fact that the linear functional $\l^{\nat}_{x}(-,\one, s)$ of $\pi_{x}$ is  invariant under $A(\cO_{x})$.
\end{proof}

\sss{The case $x\in \Sig$ and $\pi_x$ a  twisted Steinberg.}

Let $\St$ be the Steinberg representation of $G(F_{x})$.

\begin{prop}\label{prop St}
Let $\pi_{x}={\rm St}_\chi=\St\otimes\chi$ be an unramified twist of Steinberg representation, where $\chi$ is an unramified quadratic character of $F^\times_{x}$. Then we have
\begin{eqnarray}
\label{J+}\BJ_{\pi_{x}}(1_{\Iw_{x}},s_1,s_2)&=&\vol(G(\cO_{x}))\z_{x}(2)\cdot q_{x}^{-1},\\
\label{J-}\BJ_{\pi_{x}}(1_{\Iw_{x}\cdot w},s_1,s_2)&=&\vol(G(\cO_{x}))\z_{x}(2)\cdot \epsilon(\pi_{x}\otimes\y_{x},1/2,\psi_{x})q^{s_1-s_2-1}_{x}.
\end{eqnarray}
\end{prop}
\begin{proof} We first prove \eqref{J+}. By \eqref{eqn J loc}, the local spherical character evaluated at $f=1_{\Iw_{x}}$ simplifies into one term
\begin{align}\label{pre J+}
\BJ_{\pi_{x}}(1_{\Iw_{x}},s_1,s_2)=\vol(\Iw_{x})\,
  \frac{\lambda^\nat_{x}( W_{0}, {\bf1}, s_1+s_2)\lambda^\nat_{x}\left(  W_{0}, \eta_{x}, s_1-s_2\right)}{\theta_{x}^\nat(W_0,W_0)},
\end{align}
where $W_0$ is any nonzero element in  the line $\CW_{\psi_x}(\St_\chi)^{\Iw_{x}}$. We normalized $W_{0}$ so that $W_0(1_2)=1$, then we have explicitly 
$$
W_0\left( \matrixx{a}{}{}{1} \right)=\begin{cases}
\chi(a)|a|,& v_{x}(a)\geq 0,\\
0,& v_{x}(a)<0.
\end{cases}
$$
For any unramified character $\chi': F^\times_{x} \to\BC^\times$, we have 
\begin{align}\label{eq lambda W0}
\lambda^\nat_{x}(W_0,\chi',s)=1.
\end{align}
We compute the inner product $\theta^\nat_{x}(W_0,W_0)$. First we note
\begin{align*}
\int_{F^\times_{x}}W_0\left( \matrixx{a}{}{}{1} \right)\ov W_0\left( \matrixx{a}{}{}{1} \right)d^\times a
=\sum_{i=0}^\infty q_{x}^{-2i}=(1-q^{-2}_{x})^{-1}.
\end{align*}
For $\pi_{x}=\St_\chi$, the local L-factor 
$$
L(\pi_{x}\times \wt\pi_{x},s)=(1-q_{x}^{-1-s})^{-1}(1-q_{x}^{-s})^{-1}.
$$
It follows that the normalized inner product 
\begin{align*}
\theta^\nat_{x}(W_0,W_0)=1-q^{-1}_{x}.
\end{align*}
Finally we note
\begin{equation*}
\vol(\Iw_{x})=(1+q_{x})^{-1}\vol (G(\CO_{x})).
\end{equation*}
Hence
\begin{equation}\label{theta vol}
\vol(\Iw_{x})\theta^\nat_{x}(W_0,W_0)^{-1}=\vol(G(\cO_{x}))\z_{x}(2)q_{x}^{-1}.
\end{equation}
Plugging \eqref{eq lambda W0}, \eqref{theta vol} into \eqref{pre J+}, we get \eqref{J+}.

Now we prove \eqref{J-}. By definition, we have
\begin{eqnarray*}
\BJ_{\pi_{x}}(1_{\Iw_{x}\cdot w},s_1,s_2)&=& \sum_{\{W_{i}\}}  \frac{\lambda^\nat_{x}( \pi_{x}(1_{\Iw_{x}\cdot w})W_{i}, {\bf1}, s_1+s_2)\lambda^\nat_{x}\left(\ov W_{i}, \eta, s_1-s_2\right)}{\theta_{x}^\nat(W_{i},W_{i})}\\
&=&\sum_{\{W_{i}\}}  \frac{\lambda^\nat_{x}( \pi_{x}(1_{\Iw_{x}})\pi_{x}(w)W_{i}, {\bf1}, s_1+s_2)\lambda^\nat_{x}\left(\ov{\pi_{x}(w)\pi_{x}(w)W_{i}}, \eta_{x}, s_1-s_2\right)}{\theta^\nat_{x}(\pi_{x}(w)W_{i},\pi_{x}(w)W_{i})}
\end{eqnarray*}
Note that $\{\pi(w)W_{i}\}$ is another orthogonal basis for $\CW_{\psi_x}(\St_{\chi})$, therefore we may rename it by $\{W_{i}\}$ and rewrite the above as
\begin{equation*}
\BJ_{\pi_{x}}(1_{\Iw_{x}\cdot w},s_1,s_2)=\sum_{\{W_{i}\}}  \frac{\lambda^\nat_{x}( \pi_{x}(1_{\Iw_{x}})W_{i}, {\bf1}, s_1+s_2)\lambda^\nat_{x}\left(\ov{\pi_{x}(w)W_{i}}, \eta_{x}, s_1-s_2\right)}{\theta_{x}^\nat(W_{i},W_{i})}
\end{equation*}
which again simplifies into one single term corresponding to the unique $W_0\in  \CW_{\psi_x}(\St_\chi)^{\Iw_{x}}$ with $W_0(1_2)=1$
\begin{align}\label{pre J-}
\BJ_{\pi_{x}}(1_{\Iw_{x}\cdot w},s_1,s_2)=\vol(\Iw_{x})\,
  \frac{\lambda^\nat_{x}( W_{0}, {\bf1}, s_1+s_2)\lambda^\nat_{x}\left( \ov{ \pi_{x}(w)W_{0}}, \eta_{x}, s_1-s_2\right)}{\theta^\nat_{x}(W_0,W_0)},
\end{align}
We have an explicit formula
\begin{equation*}
(\pi_{x}(w)W_{0})\left(\mat{a}{}{}{1}\right)=W_{0}\left(\mat{}{a}{-1}{}\right)=\begin{cases} -q^{-1}_{x}\chi(a)|a|, & v_{x}(a)\ge-1 \\ 0, & v_{x}(a)\le-2 .\end{cases}
\end{equation*}
Using this we can calculate
\begin{equation}\label{lam wW0}
\lambda^\nat_{x}( \ov{ \pi_{x}(w)W_{0}}, \eta_{x}, s)=-(\chi\eta_{x})(\vp_{x})q^{s}_{x}.
\end{equation}
Plugging \eqref{theta vol}, \eqref{eq lambda W0} and \eqref{lam wW0} into \eqref{pre J-}, we get
\begin{equation}\label{J- chi eta}
\BJ_{\pi_{x}}(1_{\Iw_{x}\cdot w},s_1,s_2)=-\vol(G(\cO_{x}))\z_{x}(2)(\chi\eta_{x})(\vp_{x})q_{x}^{s_{1}-s_{2}-1}.
\end{equation}
Finally recall the $\ep$-factor for the twisted Steinberg $\pi_{x}\otimes \y_{x}=\St\otimes\chi\y_{x}$ and the unramified $\psi_{x}$ is the Atkin--Lehner eigenvalue
\begin{equation*}
\epsilon(\pi_{x}\otimes \y_{x},1/2, \psi_{x})=\ep(\St\ot\chi\y_{x},1/2, \psi_{x})= -(\chi\y_{x})(\varpi_{x}).
\end{equation*}
Using this we can rewrite \eqref{J- chi eta} in the form of  \eqref{J-}.

\end{proof}

\subsection{The global spherical character for our test functions}

\sss{Assumptions on $\pi$}\label{sss:ass pi} Let $\pi=\otimes'_{x\in|X|}\pi_x$ be a cuspidal automorphic representation of $G(\BA)$ which is ramified exactly at the set $\Sig$. Assume that $\pi_{x}$ is isomorphic to an unramified twist of the Steinberg representation at each $x\in \Sig$. 

Recall that $R\subset |X|$ is the ramification locus of the double cover $\nu: X'\to X$. Assume $\Sig\cap R=\vn$. Let $\Sig=\Sf\sqcup \Si$ be the decomposition determined by the conditions \eqref{Sf split} and \eqref{Si inert}.

The degrees of the $L$-functions 
$L(\pi,s)$ and $L(\pi\otimes\eta,s)$ as a polynomials of $q^{-s}$ are
$$
\deg L(\pi,s)= 4g-4+N,\quad \deg  L(\pi\otimes\eta,s)= 4g-4+2\rho+N.
$$
We set
\begin{eqnarray*}
&&\sL_{F'/F}(\pi,s_1,s_2)\\
&:=& q^{(2g-2+N/2)(s_1+s_2)+(2g-2+\rho+N/2)(s_1-s_2) } \frac{L(\pi,s_1+s_2+\ha)L(\pi\otimes\eta,s_1-s_2+\ha)}{L(\pi,\Ad,1)}\\
&=&|\om_{X}|^{-2s_{1}}q^{\r(s_1-s_2) } q^{Ns_{1}}\frac{L(\pi,s_1+s_2+\ha)L(\pi\otimes\eta,s_1-s_2+\ha)}{L(\pi,\Ad,1)}.
\end{eqnarray*}
Then we have 
$$
\sL_{F'/F}(\pi,s_1,s_2)=(-1)^{\#\Si}\sL_{F'/F}(\pi,-s_1,-s_2).
$$
Indeed, the sign that appears above is the root number of the base change $L$-function $L(\pi_{F'},s)$, which is the parity of the number of places in $F'$ at which the base change of $\pi_{x}$ is the Steinberg representation. If $x\in \Sf$, $x$ is split in $F'$, its contribution to the root number is always $+1$; if $x\in \Si$, $x$ is inert in $F'$, the base change of $\pi_{x}$ is always the Steinberg representation, hence it contributes $-1$ to the root number.

Recall that in \eqref{Sig pm} we have a decomposition $\Sig=\Sig_{+}\sqcup \Sig_{-}$ (right now arbitrary). We set
$$
\epsilon_-(\pi\otimes\y)\colon=\prod_{x\in\Sig_-}\epsilon(\pi_x\otimes\y_{x},1/2).
$$
Note that this is the Atkin--Lehner eigenvalue at the set of places $\Sigma_-$.

For each $f\in \sH^{\Sig\cup R}_{G}$, we define
\begin{equation}\label{def f Sig pm}
f^{\Sig_{\pm}}=f\otimes \left(\bigotimes_{x\in R}h^{\bsq}_{x}\right)\otimes \left(\bigotimes_{x\in \Sig}\one_{\bJ_{x}}\right)\in C_c^\infty(G(\BA)).
\end{equation}

\begin{prop}\label{p:global char} Let $\pi$ be a cuspidal automorphic representation of $G(\BA)$ satisfying the conditions in \S\ref{sss:ass pi}. Let $\lambda_\pi: \sH^{\Sigma\cup R}_{G}\to \BC$ be the homomorphism associated to $\pi$. Then for $f\in \sH^{\Sig\cup R}_{G}$, we have
$$
q^{N_{+}s_{1}+N_{-}s_{2}}\BJ_\pi(f^{\Sig_{\pm}},s_1,s_2)=\ha\lambda_\pi(f) \cdot \epsilon_-(\pi\otimes\y)\cdot |\om_{X}|q^{\r/2-N} \sL_{F'/F}(\pi,s_1,s_2).
$$ 
\end{prop}

\begin{proof}
We choose  a nontrivial $\psi:F\bs\BA\to\BC^\times$. Such a character $\psi$ comes from a nonzero rational differential form $c$ on $X$, so that the conductor of $\psi_{x}$ is $\fkm_{x}^{v_{x}(c)}$ where $v_{x}(c)$ is the order of $c$ at $x$. We choose such a $c$ so that $c$ has no zeros or poles at $\Sig\cup R$, so that $\psi_{x}$ is unramified at $x\in \Sig\cup R$.

When $x\notin  \Sigma\cup R$, $f_x$ is in the spherical Hecke algebra $\sH_x$, therefore
\begin{equation*}
\BJ_{\pi_x}(f_x,s_1,s_2)=\lambda_{\pi_x}(f_x)\vol(G(\cO_{x}))\frac{\l^{\nat}_{x}(W_{0},\one, s_{1}+s_{2})\l^{\nat}_{x}(\ov W_{0},\y_{x}, s_{1}-s_{2})}{\th^{\nat}_{x}(W_{0},W_{0})}
\end{equation*}
for $W_{0}\in \cW_{\psi_{x}}(\pi_{x})^{G(\cO_{x})}$ normalized by $W_{0}(1_{2})=1$. By the same proof as \cite[Lemma 4.6]{YZ}, we obtain 
\begin{align*}
\frac{\l^{\nat}_{x}(W_{0},\one, s_{1}+s_{2})\l^{\nat}_{x}(\ov W_{0},\y_{x}, s_{1}-s_{2})}{\th^{\nat}_{x}(W_{0},W_{0})}=\eta_{x}(c)|c|_{x}^{-2s_1+1/2}\z_{x}(2).
\end{align*}
Therefore
\begin{equation}\label{J unram psi ram}
\BJ_{\pi_x}(f_x,s_1,s_2)=\vol(G(\cO_{x}))\z_{x}(2)\cdot \eta_{x}(c)|c|_{x}^{-2s_1+1/2} \lambda_{\pi_x}(f_x).
\end{equation}

Now we use the calculation of local spherical characters at $x\in \Sigma\cup R$ given in Prop. \ref{prop R} and \ref{prop St} together with \eqref{J unram psi ram}, and plug them into \eqref{char g2l} to obtain
\begin{eqnarray}\label{J pi CL}
&&\BJ_\pi(f^{\Sig_{\pm}},s_1,s_2)\\
\notag &=&|\om_{X}|^{-1} C_{\vol} C_{0}C_{\Sig_{+}} C_{\Sig_{-}} C_{R} \frac{L(\pi,s_1+s_2+\ha)L(\pi\otimes\eta,s_1-s_2+\ha)}{2\,L(\pi,\Ad,1)}
\end{eqnarray}
where
\begin{eqnarray}
\notag C_{\vol}&=&\prod_{x\in |X|}\vol(G(\cO_{x}))\z_{x}(2)=\vol(G(\OO))\z_{F}(2)=|\om_{X}|^{3/2},\\
\label{C0}C_{0}&=&\l_{\pi}(f)\prod_{x\notin R\cup\Sig}\eta_x(c)|c|_{x}^{1/2-2s_{1}}=\l_{\pi}(f)|\om_{X}|^{1/2-2s_{1}}\prod_{x\notin R\cup\Sig}\eta_{x}(c),\\
\notag C_{\Sig_{+}}&=&\prod_{x\in \Sig_{+}}q_{x}^{-1}=q^{-N_{+}},\\
\notag C_{\Sig_{-}}&=&\prod_{x\in \Sig_{-}}\ep(\pi_{x}\otimes\y_{x},1/2,\psi_{x})q_{x}^{s_{1}-s_{2}-1}=\ep_{-}(\pi\otimes\y)q^{N_-(s_1-s_2)-N_{-}},\\
\label{CR} C_{R}&=&\prod_{x\in R}\eta_x(-1)\epsilon(\eta_x,1/2,\psi_x)q_{x}^{s_{1}-s_{2}+1/2}=q^{\r(s_{1}-s_{2})+\r/2}\prod_{x\in R}\epsilon(\eta_x,1/2,\psi_x).
\end{eqnarray}
Here, in \eqref{C0} we used that $c$ is a differential form with no zeros or poles along $\Sig\cup R$; in \eqref{CR} we have used $\prod_{x\in R}\eta_x(-1)=\eta(-1)=1$ since $\y_{x}(-1)$ is trivial for $x\notin R$. Taking the product and using \eqref{J pi CL} we get
\begin{eqnarray}\label{J pi Cy}
&&\BJ_\pi(f^{\Sig_{\pm}},s_1,s_2)\\
\notag&=&\ha \lambda_\pi(f) |\om_{X}|\ep_-(\pi\otimes\y) \cdot C_{\y}\cdot |\om_{X}|^{-2s_{1}}q^{\r(s_1-s_2)+\r/2}q^{-N}q^{N_-(s_1-s_2)} \\
\notag&&\times
 \frac{L(\pi,s_1+s_2+\ha)L(\pi\otimes\eta,s_1-s_2+\ha)}{L(\pi,\Ad,1)}
\end{eqnarray}
where
\begin{equation*}
C_{\y}=\prod_{x\in R}\epsilon(\eta_x,1/2,\psi_x)\prod_{x\notin R\cup\Sig}\eta_x(c).
\end{equation*}
We claim that $C_{\y}=1$. In fact,  for $x\notin R$ we have
$$
\epsilon(\eta_x,1/2,\psi_x)=\eta_x(c).
$$ It follows that
$$
C_{\y}=\epsilon(\eta,1/2,\psi).
$$
Recall that $\epsilon(\eta,s)=\epsilon(\eta,s,\psi)=\prod_{x\in |X|}\epsilon(\eta_x,s,\psi_x)$ is the $\epsilon$-factor in the functional equation $L(\eta,s)=\epsilon(\eta,s)L(\eta,1-s)$. It follows from the expression $L(\eta,s)=\frac{\zeta_{F'}(s)}{\zeta_F(s)}$ that $\epsilon(\eta,1/2)=1$. This proves $C_{\y}=1$. Comparing the other terms in \eqref{J pi Cy} and in the definition of $\sL_{F'/F}(\pi,s_{1},s_{2})$, we get
\begin{equation*}
\BJ_\pi(f^{\Sig_{\pm}},s_1,s_2)=\ha\l_{\pi}(f)\ep_{-}(\pi\otimes\y)|\om_{X}|q^{\r/2-N}q^{N_{-}(s_{1}-s_{2})-Ns_{1}}\sL_{F'/F}(\pi,s_{1},s_{2}). 
\end{equation*}
Multiplying both sides by $q^{N_{+}s_{1}+N_{-}s_{2}}$, the proposition follows.
\end{proof}

\section{Shtukas with Iwahori level structures}
In this section we will define various moduli stacks of Shtukas with Iwahori level structure and ``supersingular legs'' at $\infty$. We study the geometric properties of such moduli stacks, and establish the spectral decomposition of their cohomology under the action of the Hecke algebra.  

\subsection{Bundles with Iwahori level structures}\label{ss:Bun Iw}

Let $n$ be a positive integer. Let $G=\PGL_{n}$.  Let $\Sig\subset |X|$ be  finite set of closed points of $X$.

\begin{defn} Let $\Bun_{n}(\Sig)$ be the moduli stack whose $S$-points is the groupoid of
\begin{equation*}
\cE^{\da}=\left(\cE, \{\cE(-\frac{j}{n}x)\}_{1\le j\le n-1, x\in \Sig}\right)
\end{equation*}
where
\begin{itemize}
\item $\cE$ is a rank $n$ vector bundle over $X\times S$;
\item For each $x\in\Sig$, $\{\cE(-\frac{j}{n}x)\}_{1\le j\le n-1}$ form a chain of coherent subsheaves of $\cE$ such that
\begin{equation*}
\cE\supset \cE(-\frac{1}{n}x)\supset\cE(-\frac{2}{n}x)\supset\cdots\supset \cE(-\frac{n-1}{n}x)\supset \cE(-x)=\cE\otimes_{\cO_{X}}\cO_{X}(-x)
\end{equation*}
and that the quotient $\cE(-\frac{j-1}{n}x)/\cE(-\frac{j}{n}x)$ is scheme theoretically supported at $\{x\}\times S=\Spec k(x)\times S$ and is locally free of rank one on $\{x\}\times S$.
\end{itemize}
\end{defn}

The Picard stack $\Pic_{X}$ acts on $\Bun_{n}(\Sig)$ by tensoring on both $\cE$ and the $\cE(-\frac{j}{n}x)$'s. We define
\begin{equation*}
\Bun_{G}(\Sig):=\Bun_{n}(\Sig)/\Pic_{X}.
\end{equation*}

\sss{Fractional twists}\label{sss:1/n twist} Let $\cE^{\da}=(\cE; \{\cE(-\frac{j}{n} x)\}_{x\in \Sig})\in\Bun_{n}(\Sig)(S)$. For any rational divisor
\begin{equation*}
D=\sum_{x\in|X|}c_{x}\cdot x
\end{equation*}
on $X$ satisfying
\begin{equation}\label{D 1/n}
c_{x}\in \frac{1}{n}\ZZ \textup{ for } x\in \Sig; \quad c_{x}\in \ZZ \textup{ otherwise,}
\end{equation} 
we may define a vector bundle $\cE(D)$ in the following way. There is a unique way to write  $D=D_{0}-D_{1}$ where $D_{0}\in\Div(X)$ and $D_{1}=\sum_{x\in \Sig}\frac{i_{x}}{n}x$ for integers $0\le i_{x}\le n-1$. We define $\cE(-D_{1})\subset\cE$ to be the kernel of the sum of projections
$$\cE\to \bigoplus_{x\in\Sig}\cE/\cE(-\frac{i_{x}}{n}x).$$
Then we define $\cE(D)=\cE(-D_{1})\otimes_{X}\cO_{X}(D_{0})$. It is easy to check that $\cE(D+D')=(\cE(D))(D')$ whenever both $D$ and $D'$ satisfy \eqref{D 1/n}.

\sss{Variant of fractional twists}\label{sss:Bun Sig} Now suppose $\Sig$ is decomposed into a disjoint union of two subsets
\begin{equation}\label{def Sig}
\Sig=\Si\coprod \Sf.
\end{equation}
Let 
\begin{equation*}
\frSi=\prod_{x\in \Si}\Spec k(x) \quad \mbox{(product over $k$)}.
\end{equation*}
We now consider the base change 
\begin{equation*}
\Bun_n(\Sig)\times \frSi.
\end{equation*}

An $S$-point of $\frSi$ is a collection $\{x^{(1)}\}_{x\in \Si}$  where $x^{(1)}: S\to \Spec k(x)\incl X$, for each $x\in \Si$. It will be convenient to introduce $x^{(i)}$ for all integers $i$ inductively such that
\begin{equation}\label{x(i)}
x^{(i)}=x^{(i-1)}\circ\Fr_{S}: S\xr{\Fr_{S}}S\xr{x^{(i-1)}} \Spec k(x)\incl X, \quad i\in\ZZ.
\end{equation}
Clearly we have $x^{(i)}=x^{(i+d_{x})}$, where $d_{x}=[k(x):k]$.

For each $x\in \Si$, we have a canonical point
\begin{equation*}
\bx^{(1)}: \frS_{\infty}\to\Spec k(x)\to X
\end{equation*}
given by projection to the $x$-factor.  We define $\bx^{(i)}$ as in \eqref{x(i)} with $S$ replaced by $\frS_{\infty}$. Then the graph $\Gamma_{\bx^{(i)}}$ of $\bx^{(i)}$ is a divisor in $X\times\frS_{\infty}$. We abuse the notation to abbreviate $\Gamma_{\bx^{(i)}}$ by $\bx^{(i)}$. Then we have a decomposition
\begin{equation*}
\{x\}\times \frS_{\infty}=\Spec k(x)\times \frS_{\infty}=\coprod_{i=1}^{d_{x}} \bx^{(i)}.
\end{equation*}

Now let $\{x^{(1)}\}_{x\in\Si}$  be an $S$-point of $\frS_{\infty}$, then the graphs of $x^{(i)}$ ($x\in\Si, 1\le i\le d_{x}$) are divisors in $X\times S$  pulled back from the divisors $\bx^{(i)}$ on $X\times \frS_{\infty}$. For $\cE^{\da}\in \Bun_{n}(\Sig)(S)$, the quotient $\cE/\cE(-\frac{i}{n}x)$ then splits as a direct sum $\oplus_{j=1}^{d_{x}}\cQ^{(j)}_{i}$ where $\cQ^{(j)}_{i}$ is supported on $\Gamma_{x^{(j)}}$ (with rank $i$). We define $\cE(-\frac{i}{n}x^{(j)})$ to be the kernel
\begin{equation*}
\cE(-\frac{i}{n}x^{(j)}):=\ker\left(\cE\to \cE/\cE(-\frac{i}{n}x)\surj \cQ^{(j)}_{i}\right).
\end{equation*}
In other words, $\{\cE(-\frac{i}{n}x^{(j)})\}_{1\le i\le n-1}$ give an Iwahori level structure of $\cE$ at $x^{(j)}$. With these definitions, for $\cE^{\da}\in \Bun_{n}(\Sig)(S)$, the construction in \S\ref{sss:1/n twist} then allows us to make sense of $\cE(D)$ where $D$ is a divisor on $X\times \frS_{\infty}$ of the form
\begin{equation}\label{D for twist}
D=\sum_{x\in \Si, 1\le j\le d_{x}}c^{(j)}_{x}\bx^{(j)}+\sum_{x\in |X|-\Si}c_{x}(\{x\}\times\frSi)
\end{equation}
where
\begin{eqnarray*}
c^{(j)}_{x}&\in& \frac{1}{n}\ZZ, \textup{ for } x\in \Si, 1\le j\le d_{x};\\
c_{x}&\in& \frac{1}{n}\ZZ, \textup{ for } x\in \Sf;\\
c_{x}&\in& \ZZ, \textup{ otherwise.}
\end{eqnarray*}
More precisely, we can uniquely write $D=D_{0}-D_{1}$ where $D_{0}\in \Div(X\times \frSi)$ has $\ZZ$-coefficients and $D_{1}$ is of the form $D_{1}=\sum_{x\in \Si, 1\le j\le d_{x}}\frac{i^{(j)}_{x}}{n}\bx^{(j)}+\sum_{x\in \Sf}\frac{i_{x}}{n}(\{x\}\times\frSi)$, and the coefficients $\frac{i^{(j)}_{x}}{n}$ (for $x\in \Si$) and $\frac{i_{x}}{n}$ (for $x\in \Sf$) are in $\{\frac{1}{n}, \frac{2}{n},\cdots, \frac{n-1}{n}\}$. We define $\cE(-D_{1})$ to be the kernel of the sum of the projections
\begin{equation*}
\cE\to \left(\bigoplus_{x\in \Si, 1\le j\le d_{x}}\cE/\cE(-\frac{i_{x}^{(j)}}{n}x^{(j)})\right)\bigoplus \left(\bigoplus_{x\in \Sf}\cE/\cE(-\frac{i_{x}}{n}x)\right).
\end{equation*}
Finally let $\cE(D):=\cE(-D_{1})\otimes_{\cO_{X\times \frSi}}\cO_{X\times\frSi}(D_{0})$.

\begin{defn} Let $D$ be a $\QQ$-divisor of $X\times \frSi$ satisfying the conditions as in \eqref{D for twist}. The {\em Atkin--Lehner automorphisms} for $\Bun_{n}(\Sig)$ and $\Bun_{G}(\Sig)$ are maps
\begin{eqnarray*}
\wt\AL(D): \Bun_{n}(\Sig)\times\frSi\to \Bun_{n}(\Sig),\\
\AL(D):\Bun_{G}(\Sig)\times\frSi\to \Bun_{G}(\Sig)
\end{eqnarray*}
sending $\cE^{\da}=(\cE;\{\cE(-\frac{j}{n}x)\}_{x\in \Sig}; \{x^{(1)}\}_{x\in \Si})$ to 
\begin{equation*}
\cE^{\da}(D)=\left(\cE(D); \{\cE(D-\frac{j}{n}(\{x\}\times\frSi))\}_{x\in \Sig}\right)
\end{equation*}
which makes sense by the discussion in \S\ref{sss:Bun Sig}.
\end{defn}

The maps $\wt\AL(D)$ and $\AL(D)$ are analogous to the Atkin--Lehner automorphisms on the modular curves, hence their name. From the definition we see that $\AL(D)$ depends only on $\Di\mod \ZZ$.

\sss{}\label{sss:mu} Let $r\ge0$ be an integer and $\un\mu=(\mu_1,...,\mu_r)\in \{\pm1\}^{r}$. We define the Hecke stack with Iwahori level structures.

\begin{defn}\label{def:Hk} Let $\Hk^{\un\mu}_{n}(\Sig)$ be the stack whose $S$-points is the groupoid of the following data: 
\begin{itemize}
\item A sequence of $S$-points $\cE^{\da}_{i}=(\cE_{i};\{\cE_{i}(-\frac{j}{n} x)\}_{x\in \Sig})\in \Bun_{n}(\Sig)(S)$ for $i=0,1,\cdots, r$;
\item Morphisms $x_i: S\to X$ for $i=1,\dotsc, r$, with graphs $\Gamma_{x_{i}}\subset X\times S$;
\item Isomorphisms of vector bundles
\begin{equation}\label{fi}
f_i: \cE_{i-1}|_{X\times S-\Gamma_{x_i}}\isom \cE_{i}|_{X\times S-\Gamma_{x_i}}, \quad i=1,2,\dotsc,r,
\end{equation}
\end{itemize}
These data are required to satisfy the following conditions
\begin{enumerate}
\item If $\mu_{i}=1$, then $f_{i}$ extends to an injective map $\cE_{i-1}\to \cE_{i}$ whose cokernel is an  invertible sheaf on the graph $\Gamma_{x_i}$. Moreover, $f_{i}$ sends $\cE_{i-1}(-\frac{j}{n} x)$ to $\cE_{i}(-\frac{j}{n}x)$ for all $x\in \Sig$ and $1\le j\le n-1$.
\item If $\mu_{i}=-1$, then $f^{-1}_{i}$ extends to an injective map $\cE_{i}\to \cE_{i-1}$ whose cokernel  is an  invertible sheaf on the graph $\Gamma_{x_i}$. Moreover, $f^{-1}_{i}$ sends $\cE_{i}(-\frac{j}{n} x)$ to $\cE_{i-1}(-\frac{j}{n}x)$ for all $x\in \Sig$ and $1\le j\le n-1$.
\end{enumerate}
\end{defn}

We have a morphism $\pi^{\un\mu}_{\Hk}: \Hk^{\un\mu}_{n}(\Sig)\to X^{r}$ recording the points $x_{1},\dotsc, x_{r}$ in the above definition.   For $0\le i\le r$, let
\begin{equation*}
\wt p_{i}:\Hk^{\un\mu}_{n}(\Sig)\to\Bun_{n}(\Sig)
\end{equation*}
be the morphism recording the $i$-th point $\cE^{\da}_{i}\in \Bun_{n}(\Sig)$. 

There is an action of $\Pic_{X}$ on $\Hk^{\un\mu}_{n}(\Sig)$ by tensoring. We form the quotient
\begin{equation*}
\Hk^{\un\mu}_{G}(\Sig)=\Hk^{\un\mu}_{n}(\Sig)/\Pic_{X}
\end{equation*}
with maps recording $\cE^{\da}_{i}$
\begin{equation*}
p_{i}:\Hk^{\un\mu}_{G}(\Sig)\to\Bun_{G}(\Sig), \quad i=0,\dotsc, r.
\end{equation*}

\begin{prop}\label{p:Hk smooth} 
\begin{enumerate}
\item For $0\le i\le r$, the morphism $\wt p_{i}: \Hk^{\un\mu}_{n}(\Sig)\to\Bun_{n}(\Sig)$ is smooth of relative dimension $rn$.
\item For $0\le i\le r$, the morphism $(\wt p_{i},\pi^{\un\mu}_{\Hk}): \Hk^{\un\mu}_{n}(\Sig)\to\Bun_{n}(\Sig)\times X^{r}$ is smooth of relative dimension $r(n-1)$ when restricted to $\Bun_{n}(\Sig)\times (X-\Sig)^{r}$.
\item For $0\le i\le r$, the morphism $(\wt p_{i},\pi^{\un\mu}_{\Hk}): \Hk^{\un\mu}_{n}(\Sig)\to\Bun_{n}(\Sig)\times X^{r}$ is flat of relative dimension $r(n-1)$.
\item The statements of (1)-(3) hold when $\Hk^{\un\mu}_{n}(\Sig)$ is replaced with $\Hk^{\un\mu}_{G}(\Sig)$ and $\Bun_{n}(\Sig)$ is replaced with $\Bun_{G}(\Sig)$.
\end{enumerate}
\end{prop}
\begin{proof}
We first make some reductions. Once (1)-(3) are proved, (4) follows by dividing out by $\Pic_{X}$. By the iterative nature of $\Hk^{\un\mu}_{n}(\Sig)$, it is enough to treat the case $r=1$. We consider the case $\un\mu=1$ and $i=1$; the other cases can be treated similarly. We also base change the situation to $\kbar$ without changing notation (i.e., $X$ now means $X_{\kbar}$, $\Sig$ means $\Sig(\kbar)$, and the products are over $\kbar$, etc).  Moreover, if $x\in \Sig$ and $\Sig^{x}=\Sig-\{x\}$, we observe that over $X-\Sig^{x}$ there is an isomorphism $\Hk^{1}_{n}(\Sig)|_{X-\Sig^{x}}\cong (\Hk^{\un\mu}_{n}(\{x\})|_{X-\Sig^{x}})\times_{\Bun_{n}(\{x\})}\Bun_{n}(\Sig)$ such that the projection $p_{1}$ is the projection to the second factor. Therefore to show the statements over $X-\Sig^{x}$, it suffices to show the same statements for $\Sig=\{x\}$. Since the $X-\Sig^{x}$ cover $X$ as $x$ runs over $\Sig$, we reduce to the case where $\Sig$ is a singleton $\{x\}$. In other words, we are concerned with the map
\begin{equation*}
(\wt p_{1}, \pi^{1}_{\Hk}): \Hk^{1}_{n}(\{x\})\to \Bun_{n}(\{x\})\times X.
\end{equation*}

(2) Since $\Hk^{1}_{n}(\{x\})|_{X-\{x\}}\cong (\Hk^{1}_{n}|_{X-\{x\}})\times_{\Bun_{n}}\Bun_{n}(\{x\})$, we have a Cartesian diagram
\begin{equation*}
\xymatrix{\Hk^{1}_{n}(\{x\})|_{X-\{x\}}\ar[r]^-{\wt p_{1}|_{X-\{x\}}}\ar[d] & \Bun_{n}(\{x\})\times(X-\{x\})\ar[d]\\
\Hk^{1}_{n}\ar[r]^-{\wt p_{1}} & \Bun_{n}\times X
}
\end{equation*}
Since the bottom horizontal map $\Hk^{1}_{n}\to \Bun_{n}\times X$ is the projectivization of the universal bundle over $\Bun_{n}\times X$, it is smooth of relative dimension $n-1$. Therefore the same is true for the top  horizontal map.

(1) and (3).  Let $S=\Spec R$ where $R$ is a local $\kbar$-algebra.  Let $\cE^{\da}\in \Bun_{n}(\{x\})(S)$. For an $S$-scheme $S'$, the $S'$-points of the fiber $\wt p^{-1}_{1}(\cE^{\da})$ are $\cF^{\da}\in\Bun_{n}(\{x\})(S')$ such that for each $0\le i\le n-1$, $\cF(-\frac{i}{n}x)\subset \cE(-\frac{i}{n}x)$ with quotients an invertible sheaf supported on the graph of some map $y: S'\to X$.  Such $\cF(-\frac{i}{n}x)$ are classified by the projectivization $\PP(\cE(-\frac{i}{n}x))$ over $X\times S$. The fiber $\wt p^{-1}_{1}(\cE^{\da})$ is then a closed subscheme of
\begin{equation*}
\PP(\cE)\times_{X\times S} \PP(\cE(-\frac{1}{n}x))\times_{X\times S}\cdots\times \PP(\cE(-\frac{n-1}{n}x)).
\end{equation*}
We will write down defining equations of this closed subscheme. Let $U_{x}\subset X$ be an open affine neighborhood of $x$, and let $t\in \cO(U_{x})$ be a coordinate at $x$. Shrinking $U_{x}$ we may assume $t$ only vanishes at $x$. Since we know (2) already, to show (1) and (3), it is enough to show the corresponding statements over $U_{x}$.

After \'etale localizing $S$, we may assume that $\cE^{\da}$ is trivialized on $U_{x}\times S$. Thus we fix a trivialization $\io: \cE|_{U_{x}\times S}\isom \cO_{U_{x}\times S}^{n}$ so that
\begin{equation}\label{coor Ei}
\io(\cE(-\frac{i}{n}x)|_{U_{x}\times S})=t\cO_{U_{x}\times S}\op\cdots\op t\cO_{U_{x}\times S}\op \cO_{U_{x}\times S}\op\cdots\op \cO_{U_{x}\times S}
\end{equation}
where the first $i$ summands are $t\cO_{U_{x}\times S}$ and the last $n-i$ are $\cO_{U_{x}\times S}$. Using the decomposition \eqref{coor Ei}, we may canonically identify $\PP(\cE(-\frac{i}{n}x))|_{U_{x}\times S}\cong \PP^{n-1}\times U_{x}\times S$. Let $S'=\Spec R'$ where $R'$ is a local $R$-algebra.  Then an $R'$ point in $\wt p^{-1}(\cE^{\da})|_{U_{x}\times S}$ may be expressed using homogeneous coordinates $a^{(i)}=[a^{(i)}_{0},\dotsc, a^{(i)}_{n-1}]\in \PP^{n-1}(R')$ for $i=0, \dotsc, n-1$ (which gives $\cF(-\frac{i}{n}x)$) and a point $y\in U_{x}(R)$. The superscripts and subscripts of $a^{(i)}_{j}$ are understood as elements in $\ZZ/n\ZZ$, so $a^{(0)}_{j}=a^{(n)}_{j}$ etc..

The condition $\cF(-\frac{i}{n}x)\subset \cF(-\frac{i-1}{n}x)$ means that the following diagram can be completed into a commutative diagram by a choice of $\l\in R'$
\begin{equation*}
\xymatrix{\cE(-\frac{i}{n}x)\ar[r]^-{\ev_{y}}\ar@{^{(}->}[d] & R'^{n}\ar[d]^{\tau_{i-1}:=\diag(1,\cdots, t(y),\cdots, 1)}\ar[rrr]^-{a^{(i)}} &&& R'\ar@{-->}[d]^{\l}\\
\cE(-\frac{i-1}{n}x)\ar[r]^-{\ev_{y}} & R'^{n}\ar[rrr]^-{a^{(i-1)}} &&& R'}
\end{equation*}
where the middle vertical map $\tau_{i-1}$ is the diagonal matrix with $t(y)\in R'$ on the $(i,i)$-entry and $1$ elsewhere on the diagonal (so $\tau_{i-1}(a^{(i-1)})$ multiplies $a_{i-1}^{(i-1)}$ by $t(y)$ and leaves the other coordinates unchanged). This gives the closed condition
\begin{equation}\label{shift line}
\mbox{$\tau_{i-1}(a^{(i-1)})$ is in the line spanned by $a^{(i)}$}.
\end{equation}

We study the special fiber of $\wt p_{i}$ over $(\cE^{\da},x)$. Fix a $\kbar$-point of $\cF^{\da}\in \wt p_{1}^{-1}(\cE^{\da})$ over $y=x$ with coordinates $\ba^{(i)}=[\ba^{(i)}_{0},\dotsc,\ba^{(i)}_{n-1}],  i\in\ZZ/n\ZZ$. Let $[e_{i}]\in \PP^{n-1}$ be the coordinate line where only the $i$-th coordinate can be nonzero. Define
\begin{equation*}
I=\{i\in \ZZ/n\ZZ|\ba^{(i)}=[e_{i}]\}.
\end{equation*}
It is easy to see from the condition \eqref{shift line} that $I\ne\vn$. 
The points in $I$ cut the cyclically ordered set $\ZZ/n\ZZ$ into intervals (think about the $n$-th roots of unity on the unit circle). For neighboring $i_{1}, i_{2}\in I$, we have an interval $(i_{1}, i_{2}]$ (excluding $i_{1}$ and containing $i_{2}$ and not containing any other elements in $I$). When $I$ is a singleton $i_{1}$, we understand $(i_{1},i_{1}]$ to be the whole $\ZZ/n\ZZ$. These intervals give a partition of $\ZZ/n\ZZ$. By \eqref{shift line},  the homogeneous coordinates $[\ba^{(i)}_{0},\dotsc,\ba^{(i)}_{n-1}]$ for $\cF(-\frac{i}{n}x)$ satisfy
\begin{equation*}
\mbox{If $i$ is in the interval $(i_{1},i_{2}]$, then $\ba^{(i)}_{j}=0$ unless $j\in [i, i_{2}]$.}
\end{equation*}
Moreover, by the definition of $I$, $\ba^{(i)}_{i}$ is nonzero when $i\in I$. The relation \eqref{shift line} implies that whenever $i\in (i_{1}, i_{2}]$, where $i_{1},i_{2}\in I$ are neighbors, $\ba^{(i)}_{i_{2}}$ is nonzero. 

Now we give equations defining  $\wt p_{1}^{-1}(\cE^{\da})$ near the point $\cF^{\da}$. Let $a^{(i)}=[a_{0}^{(i)},\dotsc, a^{(i)}_{n-1}], 0\le i\le n-1$ be the coordinates of such an $R'$-valued point that specializes to $\cF^{\da}$. For an interval $(i_{1},i_{2}]$ and $i\in (i_{1},i_{2}]$, since $\ba^{(i)}_{i_{2}}\ne0$, $a^{(i)}_{i_{2}}$ is interval in $R'$, therefore we may assume $a^{(i)}_{i_{2}}=1$ for $i\in (i_{1},i_{2}]$. We now use the following affine coordinates: for any interval $(i_{1}, i_{2}]$ formed by neighboring elements $i_{1},i_{2}\in I$, we consider
\begin{equation}\label{list var}
a^{(i_{1}+1)}_{i_{1}+1},\cdots, a^{(i_{1}+1)}_{i_{2}-1}, \textup{ and } a^{(i_{1})}_{i_{2}}.
\end{equation}
There are $n$ such variables. The condition \eqref{shift line} implies that 
\begin{equation}\label{prod ty}
\prod_{i_{1}\in I}a^{(i_{1})}_{i_{2}}=t(y)
\end{equation}
where $i_{1}$ runs over $I$ and $i_{2}$ is its immediate successor. It turns out that the other coordinates can be uniquely determined by the ones in \eqref{list var} using the condition \eqref{shift line}, and that \eqref{prod ty} is the only relation implied by \eqref{shift line}. From this description we conclude that \'etale locally near $\cF^{\da}$, $\wt p^{-1}_{1}(\cE^{\da})|_{U_{x}}$ is isomorphic to $\AA^{n}_{S}$ with the map $\wt p^{-1}_{1}(\cE^{\da})|_{U_{x}}\xr{\pi^{1}_{\Hk}} U_{x}\times S\xr{t}\AA^{1}_{S}$ corresponding to $\AA^{n}_{S}\to \AA^{1}_{S}$ given by the product of a subset of coordinates. Therefore (1) and (3) follow.
\end{proof}


\subsection{Shtukas with Iwahori level structures}\label{ss:Sht Iw}

\sss{Moduli of rank $n$ Shtukas with Iwahori level structures}\label{sss:mu more} Let $\un\mu\in\{\pm1\}^{r}$.  Fix a $\QQ$-divisor $\Di$ on $X\times \frS_{\infty}$ supported at $\Si\times\frS_{\infty}$ of the form
\begin{equation}\label{D infty}
\Di=\sum_{x\in \Si, 1\le i\le d_{x}} c_{x}^{(i)}\bx^{(i)}, \quad c^{(i)}_{x}\in\frac{1}{n}\ZZ.
\end{equation}
We assume that $\un\mu$  satisfies the following condition
\begin{equation}\label{comp mu Di}
\sum_{i=1}^{r}\mu_i=\sum_{x\in \Si, 1\le i\le d_{x}}nc^{(i)}_{x}=n\deg \Di.
\end{equation}

\begin{defn} We define the stack $\Sht^{\un\mu}_{n}(\Sig;\Di)$ by the Cartesian diagram
\begin{equation}\label{Shtn}
\xymatrix{\Sht_{n}^{\un\mu}(\Sig;\Di)\ar[rr]\ar[d]&&
\Hk^{\un\mu}_{n}(\Sig)\times\frSi \ar[d]^{(\wt p_{0},\wt\AL(-\Di)\circ (\wt p_{r}\times\id_{\frSi}))}\\
\Bun_{n}(\Sig)\ar[rr]^-{(\id, \Fr)}&& \Bun_{n}(\Sig)\times \Bun_{n}(\Sig)}
\end{equation}
\end{defn}

Concretely, for a $k$-scheme $S$, an $S$-point of $\Sht_{n}^{\un\mu}(\Sig;\Di)$ consists of the following data:
\begin{itemize}
\item For each $0\le i\le r$, a point $\cE^{\da}_{i}=(\cE_{i};\{\cE_{i}(-\frac{j}{n} x)\}_{x\in \Sig})\in\Bun_{n}(\Sig)(S)$;
\item For each $x\in \Si$, a morphism $x^{(1)}: S\to \Spec k(x)$;
\item For each $1\le i\le r$,  a morphism $x_{i}:S\to X$;
\item Maps $f_{1},\dotsc, f_{r}$ as in the definition of $\Hk^{\un\mu}_{n}(\Sig)$;
\item An isomorphism $\io: \cE_{r}\cong (\leftexp{\tau}{\cE}_{0})(\Di)$ (first pullback by Frobenius, then fractional twist by $\Di$) respecting the Iwahori level structures at all $x\in \Sig$.
\end{itemize}

By definition, we have a morphism recording $x_{i}$ and $\{x^{(1)}\}_{x\in \Si}$ in the definition above
\begin{equation}\label{def Pi mu n}
\Pi^{\un\mu}_{n,\Di}: \Sht_{n}^{\un\mu}(\Sig;\Di)\to X^{r}\times \frSi.
\end{equation}

\begin{lemma}\label{l:sum Di}
Let $\Di$  be a $\QQ$-divisor of the form \eqref{D infty}. Then up to canonical isomorphisms, $\Sht^{\un\mu}_{n}(\Sig;\Di)$ depends only on the sum $\sum_{1\le i\le d_{x}}c_{x}^{(i)}$ for each $x\in \Si$. 
\end{lemma}
\begin{proof}
Let $\Di'=\sum_{x\in \Si}(\sum_{1\le i\le d_{x}}c_{x}^{(i)})\bx^{(1)}$. It suffices to give a canonical isomorphism $\a: \Sht^{\mu}_{n}(\Sig;\Di)\isom \Sht^{\mu}_{n}(\Sig;\Di')$. Let $(\cE_{i}^{\da}; x_{i}; \{x^{(1)}\}; \iota)$ be an $S$-point of $\Sht^{\mu}_{n}(\Sig;\Di)$. For $0\le i\le r$, let
\begin{equation*}
\cF_{i}^{\da}=\cE_{i}^{\da}\left(-\sum_{2\le j\le j'\le d_{x}}c_{x}^{(j')}\bx^{(j)}\right).
\end{equation*}
One checks that $\iota$ induces an isomorphism $\iota': \cF_{r}^{\da}\cong {}^{\tau}\cF_{0}^{\da}(\Di')$. Define $\a(\cE_{i}^{\da}; x_{i}; \{x^{(1)}\}; \iota)=(\cF_{i}^{\da}; x_{i}; \{x^{(1)}\}; \iota')$, which is easily seen to be an isomorphism.
\end{proof}

\sss{The case $r=0$} When $r=0$, $\Sht^{\vn}_{n}(\Sii)$ is zero dimensional.  We describe the groupoid of $\kbar$-points of $\Sht^{\vn}_{n}(\Sii)$. For any $\xi: \frSi\to \kbar$ (which amounts to choosing a $\kbar$-point $x^{(1)}$ over each $x\in \Si$), let $\Sht^{\vn}_{n}(\Sig;\xi)$ be the fiber of $\Sht^{\vn}_{n}(\Sii)$ over $\xi$.

Let $B$ be the central simple algebra over $F$ of dimension $n^{2}$, which is split at points away from $\Si$, and  has Hasse invariant $\inv_{x}(B)=\sum_{1\le i\le d_{x}}c_{x}^{(i)}$ for $x\in\Si$. Since $\sum_{x\in \Si}\sum_{1\le i\le d_{x}}c_{x}^{(i)}=0$ by \eqref{comp mu Di}, such a central simple algebra $B$ exists. Let $B^{\times}$ denote the algebraic group over $F$ of the multiplicative group of units in $B$. For $x\in\Sig$, let $K_{x}\subset B^{\times}(F_{x})$ be a minimal parahoric subgroup (so for $x\in\Sf$, $K_{x}$ is an Iwahori subgroup of $B^{\times}(F_{x})\cong \GL_{n}(F_{x})$). For $x\in |X-\Sig|$, let $K_{x}$ be a maximal parahoric of $B^{\times}(F_{x})\cong\GL_{n}(F_{x})$ such that almost all of them come from an integral model of $B$ over $X$. Then we have an isomorphism of groupoids
\begin{equation*}
\Sht^{\vn}_{n}(\Sig;\xi)(\kbar)\cong B^{\times}(F)\bs B^{\times}(\AA_{F})/\prod_{x\in |X|}K_{x}.
\end{equation*}

\sss{The case $r=1$ and Drinfeld modules}\label{sss:rel DrMod} We consider the special case where $r=1$, $\mu=-1$, $\Si$ consists of a single point $\infty$, and $\Di=-\frac{1}{n}\Gamma_{\mathbf{\infty}^{(1)}}$. In this case the stack $\Sht^{\mu}_{n}(\Sig; \Di)$ is closely related to the moduli stack $\textup{DrMod}_{n}(\Sf)$ of Drinfeld $A=\Gamma(X-\{\infty\}, \cO_{X})$-modules with Iwahori level structure at $\Sf$. In fact, in \cite[Theorem 3.1.4]{BS} it is shown that $\textup{DrMod}_{n}(\Sf)$ can be identified with the open and closed substack of $\Sht^{\mu}_{n}(\Sig; \Di)|_{X-\{\infty\}}$ consisting of those $(\cE^{\da}_{i}; \dotsc)$ where $\cE_{0}$ has degree $n(g-1)+1$. This implies an isomorphism over $X-\{\infty\}$:
\begin{equation*}
\textup{DrMod}_{n}(\Sf)/\Pic^{0}_{X}(k)\cong \Sht^{\mu}_{G}(\Sig; \Di)|_{X-\{\infty\}}.
\end{equation*}

\sss{Relation with the usual Shtukas}
We explain how $\Sht^{\un\mu}_{n}(\Sig;\Di)$ is related to the Shtukas in the sense of \cite{Va}. Let $\Si=\{y_{1},\dotsc, y_{s}\}$, and  $d_{i}=[k(y_{i}):k]$. Let $r'=r+\sum_{i=1}^{s}d_{i}$. For each $c\in \frac{1}{n}\ZZ$ we have a unique coweight $\un\mu(c)=(a_{1},\dotsc, a_{n})\in\ZZ^{n}$ of $\GL_{n}$ such that $\sum_{i}a_{i}=nc$ and $a_{n}\le a_{n-1}\le \dotsc\le a_{1}\le a_{n}+1$ (in other words $\mu(c)$ is a minuscule coweight).  Let $\Di$ take the form \eqref{D infty}. Let
\begin{equation*}
\un\mu'=(\mu_{1},\dotsc, \mu_{r}, \mu(c^{(1)}_{y_{1}}), \dotsc, \mu(c^{(d_{1})}_{y_{1}}), \mu(c^{(1)}_{y_{2}}), \dotsc, \mu(c^{(d_{2})}_{y_{2}}), \dotsc, \mu(c^{(1)}_{y_{s}}), \dotsc, \mu(c^{(d_{s})}_{y_{s}})).
\end{equation*}
This is an $r'$-tuple of minuscule dominant coweights of $\GL_{n}$.
We consider the stack $\Sht^{\un\mu'}_{n}(\Sig)$ of rank $n$ Shtukas with modification types given by $\un\mu'$ and Iwahori level structure at $\Sig$: it is given by the Cartesian diagram
\begin{equation*}
\xymatrix{   \Sht^{\un\mu'}_{n}(\Sig)\ar[d]\ar[r]     & \Hk^{\un\mu'}_{n}(\Sig)\ar[d]^{(\wt p_{0},\wt p_{r'})}   \\
\Bun_{n}(\Sig)\ar[r]^-{(\id, \Fr)} & \Bun_{n}(\Sig)\times \Bun_{n}(\Sig)}
\end{equation*}
where $\Hk^{\un\mu'}_{n}(\Sig)$ is defined similarly as $\Hk^{\un\mu}_{n}(\Sig)$. There is a natural map $\pi^{\un\mu'}_{n}: \Sht^{\un\mu'}_{n}(\Sig)  \to X^{r'}$. We have a map
\begin{equation*}
e_{\Si}: X^{r}\times \frSi\mapsto X^{r'}
\end{equation*}
given by sending $(x_{1},\dotsc, x_{r}, y_{1}^{(1)},\dotsc, y_{s}^{(1)})$ to $(x_{1},\dotsc, x_{r}, y_{1}^{(1)},\dotsc, y_{1}^{(d_{1})}, y_{2}^{(1)}, \dotsc, y_{s}^{(d_{s})})$.

\begin{lemma}\label{l:more legs}
There is a canonical closed embedding $\wt e: \Sht^{\un\mu}_{n}(\Sig;\Di)\incl \Sht^{\un\mu'}_{n}(\Sig)$ making the following diagram commutative
\begin{equation*}
\xymatrix{  \Sht^{\un\mu}_{n}(\Sig;\Di)\ar[r]^-{\wt e}\ar[d]^{\Pi^{\un\mu}_{n,\Di}} & \Sht^{\un\mu'}_{n}(\Sig)    \ar[d]^{\pi^{\un\mu'}_{n}}     \\
X^{r}\times \frSi\ar[r]^-{e_{\Si}} & X^{r'}
}
\end{equation*}
\end{lemma}
\begin{proof}
The map $\wt e$ is defined by sending $(\cE^{\da}_{i}, f_{i}, \io)\in \Sht^{\un\mu}_{n}(\Sig;\Di)$ over $(x_{1},\dotsc, x_{r}, y^{(1)}_{1},\dotsc, y^{(1)}_{s})\in X^{r}\times \frSi$ to the following point $(\cF^{\da}_{i}, f'_{i}, \io')$ over $e_{\Si}(x_{1},\dotsc, x_{r}, y^{(1)}_{1},\dotsc, y^{(1)}_{s})$. We define
\begin{equation*}
\cF^{\da}_{i}=\begin{cases}\cE^{\da}_{i}, \textup{ if }  0\le i\le r;\\
({}^{\tau}\cE^{\da}_{0})(\Di-\sum_{h=1}^{j_{1}}c^{(h)}_{y_{h}}y_{1}^{(h)}), \\
\quad \textup{ if } i=r+j_{1}, 1\le j_{1} \le d_{1};  \\
({}^{\tau}\cE^{\da}_{0})(\Di-\sum_{h=1}^{d_{1}}c^{(h)}_{y_{1}}y_{1}^{(h)}-\sum_{h=1}^{j_{2}}c^{(h)}_{y_{2}}y_{2}^{(h)}),\\
\quad \textup{ if }  i=r+d_{1}+j_{2}, 1\le j_{2} \le d_{2};\\
\cdots \\
({}^{\tau}\cE^{\da}_{0})(\sum_{h=j_{s}+1}^{d_{s}}c^{(h)}_{y_{s}}y_{s}^{(h)}),\\
\quad \textup{ if }  i=r+d_{1}+\cdots +d_{s-1}+j_{s}, 1\le j_{s} \le d_{s}.
\end{cases}
\end{equation*}
The map $f'_{r}$ is $\cE^{\da}_{r}\xr{\io}({}^{\tau}\cE^{\da}_{0})(\Di)\dashrightarrow ({}^{\tau}\cE^{\da}_{0})(\Di-c^{(1)}_{y_{1}}y_{1})$, and the other maps $f'_{i}, \io'$ are the obvious ones. The above equation for $\cF_{i}^{\da}$ gives a closed condition on $\Sht^{\un\mu'}_{n}(\Sig)$ without changing automorphisms, realizing $\Sht^{\un\mu}_{n}(\Sig;\Di)$ as a closed substack of $\Sht^{\un\mu'}_{n}(\Sig)$.
\end{proof}

\sss{$\Sht^{\un\mu}_{G}(\Sig; D_{\infty})$ and its geometric properties}
The groupoid $\Pic_{X}(k)$ acts on $\Sht_{n}^{\un\mu}(\Sig;\Di)$ by tensoring. We define the quotient (see \cite[5.2.1]{YZ} for the explanation why the quotient makes sense as a stack)
\begin{equation*}
\Sht^{\un\mu}_{G}(\Sig;\Di):=\Sht_{n}^{\un\mu}(\Sig;\Di)/\Pic_{X}(k).
\end{equation*}
We have a Cartesian diagram
\begin{equation}\label{Cart ShtG Di}
\xymatrix{\Sht_G^{\un\mu}(\Sig;\Di)\ar[rr]\ar[d]^{\om_{0}}&&
\Hk^{\un\mu}_{G}(\Sig)\times\frSi \ar[d]^{(p_{0},\AL(-\Di)\circ (p_{r}\times\id_{\frSi}))}\\
\Bun_G(\Sig)\ar[rr]^-{(\id, \Fr)}&& \Bun_G(\Sig)\times \Bun_G(\Sig)}
\end{equation}

The map $\Pi^{\un\mu}_{n,\Di}$ in \eqref{def Pi mu n} induces a map 
\begin{equation}\label{Pi mu}
\Pi^{\un\mu}_{G,\Di}=(\pi^{\un\mu}_{G}, \pi_{G,\infty}): \Sht^{\un\mu}_{G}(\Sig;\Di)\to X^{r}\times \frSi.
\end{equation}

Since the action $\AL(\Di)$ on $\Bun_{G}(\Sig)$ depends only on $\Di\mod \ZZ$, combined with Lemma \ref{l:sum Di} we conclude that

\begin{lemma}\label{l:Di mod Z}
The moduli stack $\Sht^{\un\mu}_{G}(\Sig;\Di)$ depends only on the image of $\Di$ in $\Div(\Si)\otimes_{\ZZ}(\frac{1}{n}\ZZ/\ZZ)$.
\end{lemma}

\begin{prop}\label{p:ShtG} 
\begin{enumerate}
\item The stack $\Sht^{\un\mu}_{G}(\Sig;\Di)$ is a smooth DM stack of dimension $rn$.

\item The morphism $\Pi^{\un\mu}_{G,\Di}$ is separated, and is smooth of relative dimension $r(n-1)$ over $(X-\Sigma)^{r}\times\frSi$.

\end{enumerate}
\end{prop}
\begin{proof} To show the smoothness statements in (1) and (2), we adapt the argument of \cite[Prop. 2.11]{VL} and apply \cite[Lemme 2.13]{VL} to the diagram \eqref{Cart ShtG Di}. Without giving all the details, the same argument of \cite[Prop. 2.11]{VL}  shows that after an \'etale base change, the fibration $p_{r}: \Hk^{\un\mu}_{G}(\Sig)\to \Bun_{G}(\Sig)$ can be trivialized. Therefore the same is true for $q_{r}:=\AL(-\Di)\circ (p_{r}\times\id_{\frSi}):  \Hk^{\un\mu}_{G}(\Sig)\times\frSi\to\Bun_{G}(\Sig)$ because  $\AL(-\Di)$ is \'etale. Then \cite[Lemme 2.13]{VL} applied to the diagram \eqref{Cart ShtG Di} implies that $\Sht^{\un\mu}_{G}(\Sig;\Di)$ is  \'etale locally isomorphic to a fiber of $q_{r}$. More precisely, for a fixed choice of $\cE^{\da}\in \Bun_{G}(\Sig)(k)$ (for example the trivial bundle with any Iwahori level structure at $\Sig$), there exists an \'etale covering $\{U_{i}\}$ of $\Sht^{\un\mu}_{G}(\Sig;\Di)$ together with \'etale maps $U_{i}\to q_{r}^{-1}(\cE^{\da})$ over $X^{r}\times \frSi$. 

Since $p_{r}$ is smooth of relative dimension $rn$ by Prop. \ref{p:Hk smooth}(1), so is $q_{r}$ and hence $q_{r}^{-1}(\cE^{\da})$ is smooth over $k$ of dimension $rn$. This implies that $\Sht^{\un\mu}_{G}(\Sig;\Di)$ is smooth of dimension $rn$.

By Prop. \ref{p:Hk smooth}(2), $p_{r}^{-1}(\cE^{\da})|_{(X-\Sig)^{r}}$ is smooth of relative of dimension $r(n-1)$ over $(X-\Sig)^{r}$. Therefore the same is true for $q_{r}^{-1}(\cE^{\da})|_{(X-\Sig)^{r}\times\frSi}$. By the discussion in the first paragraph, this implies that $\Sht^{\un\mu}_{G}(\Sig;\Di)|_{(X-\Sig)^{r}\times\frSi}$ is smooth over $(X-\Sig)^{r}\times\frSi$ of relative dimension $r(n-1)$.

We now show that $\Sht^{\un\mu}_{G}(\Sig;\Di)$ is DM. By Lemma \ref{l:more legs}, $\Sht^{\un\mu}_{G}(\Sig;\Di)$ is a closed substack of $\Sht^{\un\mu'}_{G}(\Sig):=\Sht^{\un\mu'}_{n}(\Sig)/\Pic_{X}(k)$. The map $\Sht^{\un\mu'}_{G}(\Sig)\to \Sht^{\un\mu'}_{G}$ (forgetting the level structure) is clearly representable. By \cite[Prop. 2.16(a)]{Va}, $\Sht^{\un\mu'}_{G}$ is DM, hence so are $\Sht^{\un\mu'}_{G}(\Sig)$ and its closed substack $\Sht^{\un\mu}_{G}(\Sig;\Di)$.

Finally we show that  $\Pi^{\un\mu}_{G,\Di}$ is separated. The map $\Sht^{\un\mu'}_{G}\to X^{r'}$ is separated, as can be seen from the same argument following \cite[Theorem 5.4]{YZ}. This implies  that $\pi^{\un\mu'}_{n}:\Sht^{\un\mu'}_{G}(\Sig)\to X^{r'}\times\frSi$ is also separated as $\Sht^{\un\mu'}_{G}(\Sig)\to \Sht^{\un\mu'}_{G}$ is proper. Since $\Sht^{\un\mu}_{G}(\Sig;\Di)$ is a closed substack of $\Sht^{\un\mu'}_{G}(\Sig)$, $\Pi^{\un\mu}_{G,\Di}$ is also separated.
\end{proof}

\sss{The base-change  situation}\label{sss:bc} Now let  $X'$ be another smooth, projective curve over $k$ with a map $\nu: X'\to X$ satisfying
\begin{equation}\label{R Sig disjoint}
\mbox{The map $\nu$ is unramified over $\Sig$.}
\end{equation}
Let
\begin{equation*}
\frSi'=\prod_{x'\in \nu^{-1}(\Si)}\Spec k(x').
\end{equation*}
Then we have a natural map induced by $\nu$
\begin{equation}\label{nu r}
\nu'^{r}: X'^{r}\times \frSi'\to X^{r}\times \frSi.
\end{equation}
Define the base change of $\Sht^{\un\mu}_{G}(\Sig;\Di)$
\begin{equation*}
\Sht'^{\un\mu}_{G}(\Sig;\Di):=\Sht^{\un\mu}_{G}(\Sig;\Di)\times_{(X^{r}\times \frSi)}(X'^{r}\times \frSi').
\end{equation*}

\begin{prop}\label{p:bc smooth} Under the assumption \eqref{R Sig disjoint}, the stack $\Sht'^{\un\mu}_{G}(\Sig;\Di)$ is a smooth DM stack of dimension $rn$.
\end{prop}
\begin{proof}
Only the smoothness of $\Sht'^{\un\mu}_{G}(\Sig;\Di)$ requires an argument. Let $\Hk'^{\un\mu}_{G}(\Sig)=\Hk^{\un\mu}_{G}(\Sig)\times_{X^{r}}X'^{r}$. As in the proof of Prop. \ref{p:ShtG}(1), we reduce to showing that  $p'_{r}: \Hk'^{\un\mu}_{G}(\Sig)\to \Bun_{G}(\Sig)$ is smooth of relative dimension $rn$. As in the proof of Prop. \ref{p:Hk smooth}, it suffices to treat the case $r=1$ and $\un\mu=1$. 

Let $R'$ be the ramification locus of $\nu$. Then $\Hk'^{\un\mu}_{G}(\Sig)|_{X'-R'}\to \Hk^{\un\mu}_{G}(\Sig)$ is \'etale, hence by Prop. \ref{p:Hk smooth}(1), $p'_{1}: \Hk'^{\un\mu}_{G}(\Sig)\to \Bun_{G}(\Sig)$ is smooth of relative dimension $n$ when restricted to $\Hk'^{\un\mu}_{G}(\Sig)|_{X'-R'}$. On the other hand, let $\Sig'=\nu^{-1}(\Sig)$. By Prop. \ref{p:Hk smooth}(2), $(p_{1}, \pi^{\un\mu}_{\Hk}):\Hk^{\un\mu}_{G}(\Sig)|_{X-\Sig}\to \Bun_{G}(\Sig)\times (X-\Sig)$ is smooth of relative dimension $n-1$. Base change along $\nu|_{X'-\Sig'}:X'-\Sig'\to X-\Sig$, we see that $\Hk^{\un\mu}_{G}(\Sig)|_{X'-\Sig'}\to \Bun_{G}(\Sig)\times (X'-\Sig')$ is smooth of relative dimension $n-1$, hence $p'_{1}$ is smooth of relative dimension $n$ when restricted to $\Hk'^{\un\mu}_{G}(\Sig)|_{X'-\Sig'}$. By assumption \eqref{R Sig disjoint}, $R'\cap \Sig'=\vn$ hence $X'-\Sig'$ and $X'-R'$ cover $X'$, we conclude that $p'_{1}$ is smooth of relative dimension $n$, which finishes the proof.
\end{proof}

\sss{Atkin--Lehner for $\Sht^{\un\mu}_{G}(\Sig;\Di)$}\label{sss:AL Sht}
For $x\in \Sig$, fractional twisting by $\frac{1}{n}x$ gives an automorphism of $\Bun_{n}(\Sig)$ and $\Hk^{\un\mu}_{n}(\Sig)$. By the diagram \eqref{Shtn}, we have an induced automorphism on $\Sht^{\un\mu}_{n}(\Sig;\Di)$
\begin{equation*}
\wt\AL_{\Sht,x}: \Sht^{\un\mu}_{n}(\Sig;\Di)\to \Sht^{\un\mu}_{n}(\Sig;\Di)
\end{equation*}
sending $(\cE^{\da}_{i},x_{i},\dotsc)$ to $(\cE^{\da}_{i}(-\frac{1}{n}x),x_{i},\dotsc)$. This also induces an automorphism on $\Sht^{\un\mu}_{G}(\Sig;\Di)$
\begin{equation*}
\AL_{\Sht,x}: \Sht^{\un\mu}_{G}(\Sig;\Di)\to \Sht^{\un\mu}_{G}(\Sig;\Di).
\end{equation*}

\sss{The case $n=2$ and a specific choice of $\Di$}\label{sss:ShtG}
We specialize to the case $n=2$ and hence $G=\PGL_{2}$.  Let $\sD_{\infty}$ be the group of $\ZZ$-valued divisors on $X\times\frSi$ supported on $\Si\times\frSi$, which is the union of the graphs of  $\bx^{(i)}$ for $x\in \Si$ and $1\le i\le d_{x}$. Let $\ha\sD_{\infty}=\ha\ZZ\ot_{\ZZ}\sD_{\infty}$. Then $\Sht^{\un\mu}_{G}(\Sig;\Di)$ is defined for $\Di\in  \ha\sD_{\infty}$ satisfying \eqref{comp mu Di} for $n=2$. As in \cite[Lemma 5.5]{YZ}, one can show that $\Hk^{\un\mu}_{G}(\Sig)$ is canonically independent of $\un\mu$. In this case we denote $\Hk^{\un\mu}_{G}(\Sig)$ by $\Hk^{r}_{G}(\Sig)$. This implies

\begin{lemma}\label{l:ShtG indep mu} For fixed $r$ and $\Di\in\ha\sD_{\infty}$, and any two $\un\mu, \un\mu'\in\{\pm1\}^{r}$ satisfying the same condition \eqref{comp mu Di}, there is a canonical isomorphism of stacks $\Sht^{\un\mu}_{G}(\Sig;\Di)\cong \Sht^{\un\mu'}_{G}(\Sig;\Di)$ over $X^{r}$.
\end{lemma}

Lemma \ref{l:Di mod Z} implies that $\Sht^{\un\mu}_{G}(\Sig;\Di)$ depends only on the image of $\Di$ in $\Div(\Si)\otimes\frac{1}{2}\ZZ/\ZZ$. We consider the following specific choice of $\Di$
\begin{equation*}
\Di^{(1)}=\sum_{x\in \Si}\ha \bx^{(1)}.
\end{equation*}

\begin{defn}\label{def:ShtGr} Assume $r$ satisfies the parity condition
\begin{equation}\label{parity}
r\equiv \#\Si\mod 2.
\end{equation}
Let $\un\mu=(\mu_{1},\dotsc,\mu_{r})\in\{\pm1\}^{r}$. For any $\Di\in\ha\sD_{\infty}$ such that
\begin{equation}\label{mu Di n=2}
\Di\equiv \Di^{(1)}\mod \sD_{\infty}, \quad \textup{ and }
\sum_{i=1}^{r}\mu_{i}=2\deg \Di,
\end{equation}
we define
\begin{equation*}
\Sht^{r}_{G}(\Sii):=\Sht^{\un\mu}_{G}(\Sig;\Di).
\end{equation*}
By Lemma \ref{l:ShtG indep mu} and Lemma \ref{l:Di mod Z}, this is independent of the choice of $\un\mu$ and $\Di$ satisfying the condition \eqref{mu Di n=2}, justifying the notation.
\end{defn}

We remark that the parity condition \eqref{parity} guarantees that for any $\un\mu\in\{\pm1\}^{r}$, the $\Di\in \ha\sD_{\infty}$ satisfying \eqref{mu Di n=2} exists.

We denote $\AL(-\Di^{(1)})$ simply by
\begin{equation}\label{ALG infty}
\AL_{G,\infty}:=\AL(-\Di^{(1)}): \Bun_{G}(\Sig)\times\frSi\to \Bun_{G}(\Sig).
\end{equation}
Then the diagram \eqref{Cart ShtG Di} becomes
\begin{equation}\label{Cart ShtG}
\xymatrix{\Sht_G^{r}(\Sii)\ar[rr]\ar[d]^{\om_{0}}&&
\Hk^{r}_{G}(\Sig)\times\frSi \ar[d]^{(p_{0},\AL_{G,\infty}\circ (p_{r}\times\id_{\frSi}))}\\
\Bun_G(\Sig)\ar[rr]^-{(\id, \Fr)}&& \Bun_G(\Sig)\times \Bun_G(\Sig)}
\end{equation}

For $\Di$ satisfying \eqref{mu Di n=2}, we denote the morphism $\Pi^{\un\mu}_{G,\Di}$ in \eqref{Pi mu} by
\begin{equation*}
\Pi^{r}_{G}=(\pi^{r}_{G}, \pi_{G,\infty}): \Sht^{r}_{G}(\Sii)\to X^{r}\times \frSi.
\end{equation*}

\subsection{Hecke symmetry}\label{ss:Hecke}

For the rest of the paper, we will use $G$ to denote $\PGL_{2}$. We will focus on the the stack $\Sht^{r}_{G}(\Sii)$ for $r$ and $\Si$ satisfying the parity condition \eqref{parity}.

\sss{Hecke correspondence}\label{sss:Hk action on Sht}
For $x\in |X-\Sig|$ let $\sH_{x}$ be the spherical Hecke algebra
\begin{equation*}
\sH_{x}=C_{c}(G(\cO_{x})\bs G(F_{x})/G(\cO_{x}),\QQ).
\end{equation*}
Let $\sH^{\Sig}_{G}=\otimes_{x\in |X-\Sig|}\sH_{x}$. Then $\sH^{\Sig}_{G}$ has a $\QQ$-basis $\{h_{D}\}$ indexed by effective divisors $D\in\Div^{+}(X-\Sig)$, where $h_{D}$ is defined in \cite[\S3.1]{YZ}.

Let $D$ be an effective divisor on $X-\Sig$. For $\un\mu\in\{\pm1\}^{r}$ and $\Di=\sum_{x\in\Si}c_{x}\bx^{(1)}$ as in Definition \ref{def:ShtGr}, we define a stack  $\Sht^{\un\mu}_{2}(\Sig;\Di;h_{D})$ whose  $S$-points classify the data
\begin{itemize}
\item Two objects $(\cE^{\da}_{i}, f_{i}, \io, \dotsc)$ and $(\cE'^{\da}_{i}, f'_{i}, \io',\dotsc)$ of $\Sht^{\un\mu}_{2}(\Sig;\Di)(S)$ which map to the same $S$-point of $(x_{1},\dotsc, x_{r}, \{x^{(1)}\})\in (X^{r}\times \frSi)(S)$;
\item For each $i = 0,1,\dotsc, r$, an embedding of coherent sheaves $\ph_{i}: \cE_{i}\to \cE'_{i}$ compatible with the Iwahori level structures,  such that $\det(\ph_{i}) :
\det(\cE_{i})\to\det(\cE'_{i})$ has divisor $D\times S\subset X\times S$, and such that the  following diagram is commutative
\begin{equation}\label{Sht hD diag}
\xymatrix{\cE_{0}\ar@{-->}[r]^{f_{1}}\ar[d]^{\ph_{0}} & \cE_{1}\ar@{-->}[r]^{f_{2}}\ar[d]^{\ph_{1}} & \cdots\ar@{-->}[r]^{f_{r}} & \cE_{r}\ar[d]^{\ph_{r}} \ar[r]^-{\io} & ({}^{\tau}\cE_{0})(\Di)\ar[d]^{{}^{\tau}\ph_{0}}       \\
\cE'_{0}\ar@{-->}[r]^{f'_{1}} & \cE'_{1}\ar@{-->}[r]^{f'_{2}} & \cdots\ar@{-->}[r]^{f'_{r}} & \cE'_{r} \ar[r]^-{\io'} & ({}^{\tau}\cE'_{0})(\Di)}
\end{equation}
\end{itemize} 
Let $\Sht^{r}_{G}(\Sii;h_{D})=\Sht^{\un\mu}_{2}(\Sig;\Di;h_{D})/\Pic_{X}(k)$, which is independent of the choice of $(\un\mu, \Di)$ as it is for $\Sht^{r}_{G}(\Sii)$.  Then $\Sht^{r}_{G}(\Sii;h_{D})$ can be viewed as a self-correspondence of $\Sht^{r}_{G}(\Sii)$ over $X^{r}\times\frSi$
\begin{equation}\label{ShthDpp}
\xymatrix{\Sht^{r}_{G}(\Sii) & \Sht^{r}_{G}(\Sii;h_{D})\ar[r]^{\orr{p}}\ar[l]_{\oll{p}} & \Sht^{r}_{G}(\Sii)}
\end{equation}
where the maps $\oll{p}$ and $\orr{p}$ record the first and the second row of \eqref{Sht hD diag}.

\begin{lemma}\label{l:fiber ShthD} Let $D$ be an effective divisor on $X-\Sig$.
\begin{enumerate}
\item The two maps $\oll{p},\orr{p}: \Sht^{r}_{G}(\Sii;h_{D})\to \Sht^{r}_{G}(\Sii)$ are representable and proper.
\item The restrictions of  $\oll{p}$ and $\orr{p}$ over $(X-D)^{r}$ are finite \'etale.
\item The fibers of $\Pi^{r}_{G}(h_{D}): \Sht^{r}_{G}(\Sii;h_{D})\to X^{r}\times\frSi$ all have dimension $r$.
\end{enumerate}
\end{lemma}
\begin{proof}
(1) For a rank two vector bundle $\cE$ over $X\times S$, let $\Quot^{D}_{X\times S/S}(\cE)$ be the $S$-scheme classifying quotients $\cE\surj \cQ$, flat over $S$ and with divisor $D$ (namely for every geometric point $s\in S$, $\cQ|_{s}$ is a torsion sheaf on $X\times s$ with length $n_{x}$ at $x\times s$ for any $x\in |X|$, where $n_{x}$ is the coefficient of $x$ in $D$). Then $\Quot^{D}_{X\times S/S}(\cE)$ is a closed subscheme of the Quot-scheme of $\cE$, hence projective over $S$. The fiber of $\orr{p}$ over any point $(\cE'^{\da}_{i}, x_{i}, f'_{i},\io')\in \Sht^{r}_{G}(\Sii)(S)$ is a closed subscheme of $\Quot^{D}_{X\times S/S}(\cE'_{1})\times_{S}\Quot^{D}_{X\times S/S}(\cE'_{2})\times\cdots\times_{S}\Quot^{D}_{X\times S/S}(\cE'_{r})$, hence projective over $S$. This shows that $\orr{p}$ are representable and proper. The argument for $\oll{p}$ is similar.

(2) When $(\cE^{\da}_{i}, x_{i}, f_{i},\io)\in \Sht^{r}_{G}(\Sii)(S)$ and $x_{i}$ are disjoint from $D$ (which is assumed to be disjoint from $\Sig$), the restriction $\cE|_{D\times S}$ carries a Frobenius structure $\io|_{D\times S}: \cE|_{D\times S}\isom {}^{\tau}\cE|_{D\times S}$, and hence descends to a $G_{D}$-torsor $\cE_{D}$ over $S$, with $G_{D}=\Res^{\cO_{D}}_{k}G$ the Weil restriction. Recording this $G_{D}$-torsor  defines a map
\begin{equation*}
\om_{D}: \Sht^{r}_{G}(\Sii)|_{(X-D)^{r}}\to  \BB G_{D}.
\end{equation*}
Let $\wt L_{D}$ be the moduli stack whose $S$-points are triples $(\cF_{D}, \cF'_{D}, \ph_{D})$ where $\cF_{D}, \cF'_{D}$ are lisse sheaves over $S$ that are locally free $\cO_{D}$-modules of rank two, and $\ph_{D}: \cF_{D}\to \cF'_{D}$ is an $\cO_{D}$-linear map whose cokernel at each geometric point of $S$ has divisor $D$ (i.e., if $D=\sum_{x} n_{x}x$, then the cokernel as an $\cO_{D}$-module has length $n_{x}$ when localized at  $x$). Let $L_{D}=\wt L_{D}/\BB\cO^{\times}_{D}$ where $\BB\cO^{\times}_{D}$ acts by simultaneously tensoring. The stack $L_{D}$ itself is the quotient of a finite discrete scheme over $k$ by a finite group, hence is finite \'etale over $k$, and it has two maps to $\BB G_{D}$ recording $\cF_{D}$ and $\cF'_{D}$
\begin{equation*}
\xymatrix{ \BB G_{D} &   L_{D}\ar[l]_{ \oll{\ell} }\ar[r]^{\orr{\ell}} &  \BB G_{D} }
\end{equation*}
which are also finite \'etale.

There is a natural map
\begin{equation*}
\wt\om_{D}: \Sht^{r}_{G}(\Sii;h_{D})|_{(X-D)^{r}}\to L_{D}.
\end{equation*}
In fact, each point $(\cE^{\da}_{i}, x_{i},\dotsc, \cE'^{\da}_{i},\dotsc, \ph_{i})\in \Sht^{r}_{G}(\Sii;h_{D})(S)$  gives a pair of $G_{D}$-torsors $\cE_{D}$ and $\cE'_{D}$  over $S$.  If we lift $\cE_{i}$ and $\cE'_{i}$ to rank two vector bundles on $X\times S$, $\cE_{D}$ and $\cE'_{D}$ have associated  $\cO_{D}^{\oplus2}$-torsors $\cF_{D}$ and $\cF'_{D}$ over $S$, well-defined up to simultaneous twisting by $\cO_{D}^{\times}$-torsors on $S$. The $\ph_{i}$  then induces an $\cO_{D}$-linear map $\ph_{D}: \cE_{D}\to \cE'_{D}$ whose cokernel has divisor $D$. 

When the points $x_{i}$ are disjoint from $D$, knowing the top row (or the bottom row) of \eqref{Sht hD diag} and any of the vertical arrows recovers the whole diagram. Any vertical arrow $\ph_{i}: \cE_{i}\to \cE'_{i}$ is in turn determined by $\cE_{i}$ (or $\cE'_{i}$) together with its image in $L_{D}$. Therefore,  the whole diagram is uniquely determined by  the top row (or the bottom row) and its image in $L_{D}$. Moreover, since $D$ is disjoint from $\Sig$, the level structures of the top row determines that of the bottom row, and vice versa. This shows the two squares below are Cartesian
\begin{equation*}
\xymatrix{ \Sht^{r}_{G}(\Sii)|_{(X-D)^{r}}     \ar[d]^{\om_{D}}  &    \Sht^{r}_{G}(\Sii;h_{D})|_{(X-D)^{r}}\ar[l]_{\oll{p}}\ar[r]^{\orr{p}} \ar[d]^{\wt\om_{D}} & \Sht^{r}_{G}(\Sii)|_{(X-D)^{r}}\ar[d]^{\om_{D}}\\
 \BB G_{D} &   L_{D}\ar[l]_{ \oll{\ell} }\ar[r]^{\orr{\ell}} &  \BB G_{D}
}
\end{equation*}
This implies that both $\oll{p}$ and $\orr{p}$ are finite \'etale, because the maps $\oll{\ell}$ and $\orr{\ell}$ are.

(3) The argument is similar to that of \cite[Lemma 5.9]{YZ}, so we only give a sketch.

Fix a geometric point $\un x=(x_{1},\dotsc, x_{r})\in X^{r}$, and we will show that the fiber $\Sht^{r}_{G}(\Sii;h_{D})_{\un x}$ has dimension $r$.  We introduce the moduli stack $H_{D}(\Sig)$ classifying $(\cE^{\da},\cE'^{\da}, \ph)$ up to the action of $\Pic_{X}$, where $\cE^{\da}, \cE'^{\da}\in \Bun_{2}(\Sig)$ and $\ph:\cE\to\cE'$ is an injective map with divisor $D$. Let $\Hk^{r}_{H, D}(\Sig)$ classify diagrams
\begin{equation}\label{EE'ph}
\xymatrix{\cE_{0}\ar@{-->}[r]^{f_{1}}\ar[d]^{\ph_{0}} & \cE_{1}\ar@{-->}[r]^{f_{1}}\ar[d]^{\ph_{1}}  & \cdots \ar@{-->}[r]^{f_{r}} &  \cE_{r}\ar[d]^{\ph_{r}}\\
\cE_{0}'\ar@{-->}[r]^{f'_{1}} & \cE_{1}'\ar@{-->}[r]^{f'_{2}} & \cdots \ar@{-->}[r]^{f'_{r}}& \cE_{r}'
}
\end{equation}
satisfying the same conditions as the diagram \eqref{Sht hD diag} without the last column, modulo simultaneous tensoring by $\Pic_{X}$. We have a Cartesian diagram
\begin{equation*}
\xymatrix{     \Sht^{r}_{G}(\Sii;h_{D})_{\un x}\ar[rr]\ar[d]  &   & \Hk^{r}_{H, D}(\Sig)_{\un x}\ar[d]^{(p_{0}, p_{r})}\\
H_{D}(\Sig)\times\frSi\ar[rr]^-{(\id, \AL_{H,\infty}\circ\Fr)} && H_{D}(\Sig)\times H_{D}(\Sig)}
\end{equation*}
Here $\AL_{H,\infty}: H_{D}(\Sig)\times\frSi\to H_{D}(\Sig)$ is given by applying $\AL_{G,\infty}$ to both $\cE^{\da}$ and $\cE'^{\da}$. The stacks $H_{D}(\Sig)$ and $\Hk^{r}_{H,D}(\Sig)$ will turn out to be fibers of the stacks $H_{d}(\Sig)$ and $\Hk^{r}_{H,d}(\Sig)$ over $D\in X_{d}$, to be introduced in \S\ref{sss:Hd}.

We introduce the analog $H^{\na}_{D}(\Sig)$ of the $H_{D,D}$ introduced in \cite[6.4.4]{YZ}, which is an open substack of $H_{D}(\Sig)$ where $\ph$ does not land in $\cE'(-x)$ for any $x\in D$. We claim that the map $H^{\na}_{D}(\Sig)\to \Bun_{G}(\Sig)$  sending $(\cE^{\da},\cE'^{\da}, \ph)$ to $\cE'^{\da}$ is smooth. Indeed, its fiber over $\cE'^{\da}\in \Bun_{G}(\Sig)(S)$ is $\Res^{D\times S}_{S}(\PP_{D\times S}(\cE'_{D\times S}))$, the restriction of scalars of the projectivization of the rank two bundle $\cE'_{D\times S}$ over $D\times S$ (the $\Sig$-level structure on $\cE^{\da}$ is automatically inherited from $\cE'^{\da}$, since $D$ is disjoint from $\Sig$). In particular, $H^{\na}_{D}(\Sig)$ is smooth over $k$.

Similarly we introduce the open substack $\Hk^{r,\na}_{H,D}(\Sig)_{\un x}\subset \Hk^{r}_{H,D}(\Sig)_{\un x}$ by requiring each column of \eqref{EE'ph} to be in $H^{\na}_{D}(\Sig)$. We define  the open substack $\Sht^{r,\na}_{G}(\Sii;h_{D})_{\un x}\subset \Sht^{r}_{G}(\Sii;h_{D})_ {\un x}$ to fit into a Cartesian diagram
\begin{equation*}
\xymatrix{     \Sht^{r,\na}_{G}(\Sii;h_{D})_{\un x}\ar[rr]\ar[d]  &   & \Hk^{r,\na}_{H,D}(\Sig)_{\un x}\ar[d]^{(p_{0}, p_{r})}\\
H^{\na}_{D}(\Sig)\times\frSi\ar[rr]^-{(\id, \AL_{H,\infty}\circ\Fr)} && H^{\na}_{D}(\Sig)\times H^{\na}_{D}(\Sig)}
\end{equation*}
As in \cite[6.4.4]{YZ}, it suffices to show that $\dim\Sht^{r,\na}_{G}(\Sii;h_{D})_{\un x}=r$. As in the case without level structures, $p_{r}: \Hk^{r,\na}_{H,D}(\Sig)_{\un x}\to H^{\na}_{D}(\Sig)$ is an \'etale locally trivial fibration. Using a slight variant of \cite[Lemme 2.13]{VL}, $\Sht^{r,\na}_{G}(\Sii;h_{D})_{\un x}$ is \'etale locally isomorphic to a fiber of $p_{r}$. It remains to show that the geometric fibers of $p_{r}$ have dimension $r$. The iterative nature of $\Hk^{r,\na}_{H,D}(\Sig)_{\un x}$ allows us to reduce to the case $r=1$. 

First consider the case $x_{1}\notin D$. Then the diagram \eqref{EE'ph} is determined by its top row and the last column, which means that the fibers of $p_{1}$ are the same as the fibers of $p_{1}: \Hk^{1}_{G}(\Sig)_{x_{1}}\to \Bun_{G}(\Sig)$, which are $1$-dimensional  by Prop. \ref{p:Hk smooth}(3). 

Next consider the case $x_{1}\in D$. Since $\Sig$ is disjoint from $D$, the  Iwahori level structures along $\Sig$ of $\cE_{1}$ and  $\cE'_{1}$ uniquely determine the Iwahori level structures along $\Sig$ of all bundles in the diagram \eqref{EE'ph}. Thus the fibers of $p_{1}$ are the same as the fibers of $p_{1}:\Hk^{1,\na}_{H,D,\un x}\to  H^{\na}_{D}$ (the version without level structure); this latter map was denoted $\Hk^{1}_{D,D,\un x}\to H_{D,D}$ in \cite[6.4.4]{YZ} and in the last paragraph of \cite[6.4.4]{YZ} it was shown that its fibers are $1$-dimensional. We are done. 
\end{proof}

\sss{Hecke symmetry on the Chow group}\label{sss:Hk Chow}

Using the dimension calculation in Lemma \ref{l:fiber ShthD}, the same argument as in \cite[Prop 5.10]{YZ} proves the following result.

\begin{prop}\label{p:Hk Chow action} The assignment
\begin{equation*}
h_{D}\mapsto (\oll{p}\times\orr{p})_{*}[\Sht^{r}_{G}(\Sii;h_{D})]
\end{equation*}
extends linearly to a ring homomorphism
\begin{eqnarray*}
\sH^{\Sig}_{G}&\to& {}_{c}\Ch_{2r}(\Sht^{r}_{G}(\Sii)\times\Sht^{r}_{G}(\Sii))_{\QQ}.
\end{eqnarray*}
In particular, we get an action of $\sH^{\Sig}_{G}$ on the Chow group of proper cycles $\Ch_{c,*}(\Sht^{r}_{G}(\Sii))_{\QQ}$.
\end{prop}

\sss{Hecke symmetry on cohomology}\label{sss:Hk coho} 
We shall define an action of $\sH^{\Sig}_{G}$ on $\cohoc{*}{\Sht^{r}_{G}(\Sii)\ot\kbar,\Ql}$ following the strategy in \cite[7.1.4]{YZ}. For this we need a presentation of $\Sht^{r}_{G}(\Sii)$ as an increasing union of open substacks of finite type. Here we are satisfied with a minimal version of such a presentation, and we postpone a more refined version to \S\ref{ss:horo}. For $N\ge0$ we define ${}^{\le N}\Sht$ to be the open substack of $\Sht^{r}_{G}(\Sii)$ consisting of those $(\cE^{\da}_{i};\dotsc)$ where $\inst(\cE_{0})\le N$. Since the forgetful map $\Sht^{r}_{G}(\Sii)\to \Bun_{G}$ recording $\cE_{0}$ is of finite type, ${}^{\le N}\Sht$ is of finite type over $k$. As $N$ increases, $\Sht^{r}_{G}(\Sii)$ is the union of the increasing sequence of open substacks ${}^{\le N}\Sht$. 

With the finite-type open substacks ${}^{\le N}\Sht$, we can copy the construction of \cite[7.1.4]{YZ} by first defining the action of $h_{D}$ as a map $\bR\pi_{\le N,!}\Ql\to \bR\pi_{\le N',!}\Ql$ (where $\pi_{\le N}: {}^{\le N}\Sht\to X^{r}\times \frSi)$ for $N'-N\ge\deg D$, and then pass to cohomology and pass to inductive limits. Using the dimension calculation in Lemma \ref{l:fiber ShthD}(3), the same argument as in \cite[Prop. 7.1]{YZ} shows

\begin{prop}\label{p:Hk coho action} The assignment $h_{D}\mapsto C(h_{D})$, extended linearly, defines an action of $\sH^{\Sig}_{G}\ot\Ql$ on $\cohoc{i}{\Sht^{r}_{G}(\Sii)\ot\kbar,\Ql}$ for each $i\in\ZZ$.
\end{prop}

The following two results are analogues of \cite[Lemma 5.12, Lemma 7.2, and Lemma 7.3]{YZ}, with the same proofs.

\begin{lemma}\label{l:self adjoint} Let $f\in \sH^{\Sig}_{G}$. Then the action of $f$ on the Chow group $\Ch_{c,*}(\Sht^{r}_{G}(\Sii))_{\QQ}$ (resp. on the cohomology $\cohoc{2r}{\Sht^{r}_{G}(\Sii)\ot\kbar,\Ql}(r)$) is self-adjoint with respect to the intersection pairing (resp. cup product pairing).  
\end{lemma}

\begin{lemma}\label{l:Chow coho equiv} The cycle class map
\begin{equation*}
\cl: \Ch_{c,i}(\Sht^{r}_{G}(\Sii))_{\QQ}\to \cohoc{4r-2i}{\Sht^{r}_{G}(\Sii)\ot\kbar,\Ql}(2r-i)
\end{equation*}
is equivariant under the $\sH^{\Sig}_{G}$-actions for all $i$.
\end{lemma}

\sss{The base-change situation}\label{sss:bc Hk action}
Consider another curve $X'$ as in \S\ref{sss:bc}. Let 
\begin{equation*}
\Sht'^{r}_{G}(\Sii)=\Sht^{r}_{G}(\Sii)\times_{(X^{r}\times \frSi)}(X'^{r}\times \frSi').
\end{equation*}
We may define the Hecke correspondence $\Sht'^{r}_{G}(\Sii;h_{D})$ for $\Sht'^{r}_{G}(\Sii)$ as the base change of $\Sht^{r}_{G}(\Sii)$ from $X^{r}\times \frSi$ to $X'^{r} \times \frSi'$. The smoothness of $\Sht'^{r}_{G}(\Sii)$ proved in Prop. \ref{p:bc smooth} allows to apply the formalism of correspondences acting on Chow groups, see \cite[A.1.6]{YZ}. The same argument as in \cite[Prop. 5.10]{YZ} gives an analogue of Prop. \ref{p:Hk Chow action}: there is an action of $\sH^{\Sig}_{G}$ on the Chow group of proper cycles $\Ch_{c,*}(\Sht'^{r}_{G}(\Sii))_{\QQ}$, where $h_{D}$ acts via the fundamental class of $\Sht'^{r}_{G}(\Sii;h_{D})$.

Similarly, with the smoothness of $\Sht'^{r}_{G}(\Sii)$ proved in Prop. \ref{p:bc smooth}, analogues of Prop. \ref{p:Hk coho action}, Lemma \ref{l:self adjoint} and Lemma \ref{l:Chow coho equiv} make sense and continue to hold true for $\Sht'^{r}_{G}(\Sii)$ in place of $\Sht^{r}_{G}(\Sii)$.

\begin{remark}\label{r:AL Sht} Besides the action of $\sH^{\Sig}_{G}$, the Atkin--Lehner involutions $\AL_{\Sht,x}$ for $x\in\Sig$ (see \S\ref{sss:AL Sht}) also act on $\Sht^{r}_{G}(\Sii)$ and $\Sht'^{r}_{G}(\Sii)$. Therefore they induce involutions on the Chow groups and cohomology groups of $\Sht^{r}_{G}(\Sii)$ and $\Sht'^{r}_{G}(\Sii)$, which we still denote by $\AL_{\Sht, x}$. These involutions commute with the action of $\sH^{\Sig}_{G}$.
\end{remark}

\subsection{Horocycles}\label{ss:horo} This subsection studies the geometry of $\Sht^{r}_{G}(\Sii)$ ``near infinity''. It serves as technical preparation for the proof of the spectral decomposition in the next subsection.  

To alleviate notation, in this subsection, we introduce the notation
\begin{equation*}
\Sht:=\Sht^{r}_{G}(\Sii)\ot\kbar.
\end{equation*}

\sss{Index of instability} Let us first introduce the notion of instability for points in $\Bun_{2}(\Sig)$. For a rank two bundle $\cE$ on $X$, $\inst(\cE)\in \ZZ$ is defined as in \cite[\S7.1.1]{YZ}: it is the maximum of $2\deg\cL-\deg\cE$ when $\cL$ runs over line subbundles of $\cE$. For a geometric point $\cE^{\da}=(\cE, \{\cE(-\ha x)\}_{x\in \Sig})\in \Bun_{2}(\Sig)(K)$, we have a bundle  $\cE(\ha D)$ for any divisor $D\subset X_{K}$ supported in $\Sig(K)$.  We call $\cE^{\da}$ {\em purely unstable} if $\inst(\cE(\ha D))>0$ for all $D\le\Sig(K)$. Note that the condition $\inst(\cE(\ha D))>0$ depends only on the class of $D$ modulo $2$, i.e., we may think of $D$ as an element in $\ZZ/2\ZZ[\Sig(K)]$, the free $\ZZ/2\ZZ$-module with basis given by $\Sig(K)$. Define
\begin{equation*}
\inst(\cE^{\da}):=\min\left\{\inst(\cE(\ha D)); D\in\ZZ/2\ZZ[\Sig(K)]\right\}.
\end{equation*}
Both the notion of pure instability and the number $\inst(\cE^{\da})$ depends only on the image of  $\cE^{\da}$ in $\Bun_{G}(\Sig)$.

Suppose $\cF\in\Bun_{2}(K)$ is unstable, with maximal line bundle $\cL$ and quotient $\cM:=\cF/\cL$.   For any effective divisor $D'$, we denote by $\cF\lrcorner_{D'}$ the resulting rank two bundle by pushing out the exact sequence $0\to\cL\to\cF\to\cM\to 0$ along $\cL\incl \cL(D')$.  Similarly let $\ulcorner_{D'}\cF$ be the pullback of the same exact sequence along $\cM(-D')\incl \cM$. Note that we have a canonical isomorphism $\cF\lrcorner_{D'}\cong (\ulcorner_{D'}\cF)(D')$, which means that $\cF\lrcorner_{D'}$ and $\ulcorner_{D'}\cF$ have the same image in $\Bun_{G}$.

\begin{lemma}\label{l:purely unst} Let $K$ be an algebraically closed field containing $k$, and $\cE^{\da}\in \Bun_{G}(\Sig)(K)$ be purely unstable.
\begin{enumerate}
\item There is a unique $D\in\ZZ/2\ZZ[\Sig(K)]$ such that $\inst(\cE^{\da})=\inst(\cE(\ha D))$. (Note that $\cE(\ha D)$ is a well-defined point of $\Bun_{G}(\Sig)$ when $D\in\ZZ/2\ZZ[\Sig(K)]$.)
\item The point $\cE^{\da}$ is uniquely determined by $\cE(\ha D)$ ($D$ as in (1)) in the following way: for any effective divisor $D'$ supported on $\Sig(K)$, $\cE(\ha D+\ha D')=\cE(\ha D)\lrcorner_{D'}$. Moreover,  we have
\begin{equation}\label{inst diff}
\inst(\cE(\ha D+\ha D'))=\inst(\cE^{\da})+|D'|
\end{equation}
where $|D'|=\#\{x\in \Sig(K)|\mbox{$x$ has nonzero coefficient in $D'$}\}$.
\end{enumerate}
\end{lemma}
\begin{proof} We prove all statements simultaneously. Let $D\in \ZZ/2\ZZ[\Sig(K)]$ be some divisor such that  $\inst(\cE^{\da})=\inst(\cE(\ha D))$ (we do not assume $D$ is unique for now). Write $\cF=\cE(\ha D)$. For any $x\in \Sig(K)$, we have $\inst(\cF(\ha x))=\inst(\cF)\pm1$. Since $\cF$ achieves the minimal index of instability, we must have $\inst(\cF(\ha x))=\inst(\cF)+1$. This means that $\cF(\ha x)=\cF\lrcorner_{x}$. For any effective $D'$ supported on $\Sig(K)$ and {\em multiplicity-free}, $\cF(\ha D')$ is the union of $\cF(\ha x)$ for $x\in D'$, we get $\cF(\ha D')=\cF\lrcorner_{D'}$. This implies that 
\begin{equation}\label{diff inst}
\inst(\cF(\ha D'))=\inst(\cF)+\deg D'=\inst(\cF)+|D'\mod2|.
\end{equation}
Since the set of points $\{\cF(\ha D')\}_{D'\le \Sig(K)}$, as points of $\Bun_{G}(\Sig)$, is exactly $\{\cE(\ha D')\}_{D'\le \Sig(K)}$, we see that $\inst(\cE(\ha D'))$ achieves its minimum exactly when $D'=D$ and nowhere else. The equality \eqref{inst diff} follows from \eqref{diff inst}.
\end{proof}

By the above lemma, for a purely unstable $\cE^{\da}\in \Bun_{G}(\Sig)(K)$, we may define an invariant
\begin{equation*}
\k(\cE^{\da})=(D, \inst(\cE^{\da}))\in \ZZ/2\ZZ[\Sig(K)]\times\ZZ_{>0}.
\end{equation*}
where $D\in\ZZ/2\ZZ[\Sig(K)]$ is the unique element such that $\inst(\cE^{\da})=\inst(\cE(\ha D))$.

\sss{Strata in $\Bun_{G}(\Sig)$} 
For $N>0$, we also denote by ${}^{N}\Bun_{G}$ the locally closed substack of $\Bun_{G}$ whose geometric points are exactly those $\cE$ with $\inst(\cE)=N$. 

For any field $K$ containing $\kbar$, we have a canonical bijection $\Sig(\kbar)\isom\Sig(K)$. For $\k\in\ZZ/2\ZZ[\Sig(\kbar)]\times\ZZ_{>0}$, there is a locally closed substack ${}^{\k}\Bun_{G}(\Sig)\subset \Bun_{G}(\Sig)\ot\kbar$ whose geometric points are exactly those geometric points $\cE^{\da}$ with $\k(\cE^{\da})=\k$ (under the identification $\Sig(\kbar)\isom\Sig(K)$).

We define a partial order on $\ZZ/2\ZZ[\Sig(\kbar)]\times\ZZ$ by saying that $\k=(D,N)\le \k'=(D',N')$ if and only if
\begin{equation*}
N'-N\ge |D-D'|.
\end{equation*}
For $\k=(D,N)\in\ZZ/2\ZZ[\Sig(\kbar)]\times\ZZ_{>0}$, let ${}^{\le\k}\Bun_{G}(\Sig)\subset \Bun_{G}(\Sig)\ot\kbar$ be the open substack consisting of $\cE^{\da}$ such that for any $D'\in\ZZ/2\ZZ[\Sig(\kbar)]$, $\inst(\cE(\ha D'))\le N+|D'-D|$. We see that ${}^{\k}\Bun_{G}(\Sig)\subset{}^{\le\k'}\Bun_{G}(\Sig)$ if and only if $\k\le \k'$. Moreover,  ${}^{\k}\Bun_{G}(\Sig)$ is closed in ${}^{\le\k}\Bun_{G}(\Sig)$, with open complement denoted by ${}^{<\k}\Bun_{G}(\Sig)$.

\begin{cor}[of Lemma \ref{l:purely unst}]\label{c:kBunSig} For $\k=(D, N)\in \ZZ/2\ZZ[\Sig(\kbar)]\times\ZZ_{>0}$, the map $\cE^{\da}\mapsto \cE(\ha D)$ gives an isomorphism of $\kbar$-stacks
\begin{equation*}
{}^{\k}\Bun_{G}(\Sig)\isom {}^{N}\Bun_{G}\ot\kbar.
\end{equation*}
\end{cor}

\sss{Elementary modifications} Next we study how the invariant $\k$ changes under an elementary modification of bundles.  Recall the stack $\Hk^{1}_{G}(\Sig)$ classifying $(\cE^{\da}, \cF^{\da}, y, \ph)$ modulo tensoring with line bundles, where $\cE^{\da}, \cF^{\da}\in \Bun_{2}(\Sig)$ and $\ph: \cE\incl \cF$ is an injective map compatible with Iwahori structures whose cokernel is an invertible sheaf on the graph of $y:S\to X$. Recording $y$ gives a map $\pi^{1}_{\Hk}: \Hk^{1}_{G}(\Sig)\to X$.

For two elements $\k=(D,N),\k'=(D',N')\in \ZZ/2\ZZ[\Sig(\kbar)]\times\ZZ_{>0}$ we define
\begin{equation*}
|\k-\k'|:=|D-D'|+|N-N'|\in\ZZ_{\ge0}
\end{equation*}
with $|D-D'|$ defined in Lemma \ref{l:purely unst}(3).

\begin{lemma}\label{l:kappa mod} Suppose $(\cE^{\da},\cF^{\da},y, \ph)\in \Hk^{1}_{G}(\Sig)(K)$ (where $K$ is an algebraically closed field, $\cE^{\da},\cF^{\da}$ are lifted to $\Bun_{2}(\Sig)(K)$, $\ph: \cE\incl \cF$ and $y$ is the support of $\coker(\ph)$), and  $\cE^{\da}$ and $\cF^{\da}$ are both purely unstable. Write $\k(\cE^{\da})=(D,N)$, $\k(\cF^{\da})=(D',N')$. 
\begin{enumerate}
\item $|\k(\cE^{\da})-\k(\cF^{\da})|=1$.
\item If  $N=N'$, then $D$ and $D'$ differ at a unique point $x\in \Sig(K)$, and we have $y=x$. The points $\cE^{\da}$ and $\cF^{\da}$ are uniquely determined by the triple $(\cE(\ha D), \cF(\ha D'), \a)$ where $\a$ is an isomorphism of $G$-bundles
\begin{equation*}
\a: \cE(\ha D)\lrcorner_{x}\cong \cF(\ha D')\lrcorner_{x}. 
\end{equation*}

\item If $N=N'-1$, then $D=D'$, and $\cE^{\da}$ and $\cF^{\da}$ are determined by the single bundle $\cE(\ha D)$ in the following way: $\cE^{\da}$ is determined by $\cE(\ha D)$ as in Lemma \ref{l:purely unst}(2); $\cF(\ha D)=\cE(\ha D)\lrcorner_{y}$ and $\cF^{\da}$ is determined by $\cF(\ha D)$ again by Lemma \ref{l:purely unst}(2).
\item If $N=N'+1$, then $D=D'$, and $\cE^{\da}$ and $\cF^{\da}$ are determined by the single bundle $\cF(\ha D)$ in the following way: $\cF^{\da}$ is determined by $\cF(\ha D)$ as in Lemma \ref{l:purely unst}(2); $\cE(\ha D)=\ulcorner_{y}(\cF(\ha D))$ and $\cE^{\da}$ is determined by $\cE(\ha D)$ again by Lemma \ref{l:purely unst}(2). 
\end{enumerate}
\end{lemma}
\begin{proof} For any $D''\in\ZZ/2\ZZ[\Sig(\kbar)]$, we have $\inst(\cE(\ha D''))=\inst(\cF(\ha D''))\pm1$, therefore $N-N'\in\{0,1,-1\}$. 

When $N-N'=-1$, $\cE(\ha D)$ achieves the minimal index of instability among all the bundles $\{\cE(\ha D''), \cF(\ha D'')\}_{D''\in\ZZ/2\ZZ[\Sig(K)]}$. Since $\inst(\cF(\ha D))=\inst(\cE(\ha D))\pm1$, we must have $\inst(\cF(\ha D))=N+1$, therefore $\inst(\cF(\ha D))=N'$ and $D'=D$.  The same argument as Lemma \ref{l:purely unst}(2) shows that $\cF(\ha D)$ is determined by $\cE(\ha D)$. This proves (3).

The analysis of the case $N-N'=1$ is similar, which takes care of (4).

Finally consider the case $N=N'$. Since $\inst(\cF(\ha D))=\inst(\cE(\ha D))\pm1$ and $\inst(\cF(\ha D))\ge N'=N=\inst(\cE(\ha D))$, we must have $\inst(\cF(\ha D))=N+1$. On the other hand, we have $\inst(\cF(\ha D'))=N'=N$ by definition. By Lemma \ref{l:purely unst}(3), we have $|D-D'|=(N+1)-N=1$, i.e., $D'$ and $D$ differ by one point $x\in \Sig(K)$.  We show that $y$ must be equal to $x$. Suppose not, consider the bundle $\cG=\cF(\ha D)$ (represented by a rank two bundle on $X_{K}$) with subsheaves $\cG(-\ha y):=\cE(\ha D)$ and $\cG(-\ha x):=\cF(\ha D-\ha x)$. Then $\cG^{\da}:=(\cG, \cG(-\ha y), \cG(-\ha x))$ defines a point in $\Bun_{2}(\{x,y\})(K)$. Note that $\inst(\cG(-\ha y))=N$ by definition and $\inst(\cG(-\ha x))=\inst(\cF(\ha D-\ha x))=\inst(\cF(\ha D'))=N$; also $\inst(\cG)=N+1$ and $\inst(\cG(-\ha x-\ha y))=\inst(\cE(\ha D-\ha x))=N+1$. It follows that $\cG^\da$ is purely unstable. This contradicts Lemma \ref{l:purely unst}(1) because both $\cG(-\ha x)$ and $\cG(-\ha y)$ achieve the minimal index of instability. This contradiction proves $y=x$.  The isomorphism $\a$ comes from the fact that $\cG(-\ha y)\lrcorner_{x}=\cG=\cG(-\ha x)\lrcorner_{x}$. The triple $(\cE(\ha D), \cF(\ha D'), \a)$ first determines $\cE^{\da}$ and $\cF^{\da}$ by Lemma \ref{l:purely unst}(2). Now we represent $D$ and $D'$ by multiplicity-free effective divisors on $\Sig(K)$. When $D'=D+x$, the map $\a$ then determines the injective map $\psi: \cE(-\ha D)\incl \cF(-\ha D')\lrcorner_{x}$ which then gives $\ph: \cE=\cE(-\ha D)\lrcorner_{D}\xr{\psi}\cF(-\ha D')\lrcorner_{x+D}=\cF(-\ha D')\lrcorner_{D'}=\cF$. When $D'=D-x$, the map $\a$ gives the injective map $\psi: \cE(-\ha D)\lrcorner_{x}\incl\cF(-\ha D')$ which then gives $\ph: \cE=\cE(-\ha D)\lrcorner_{D}=(\cE(-\ha D)\lrcorner_{x})\lrcorner_{D'}\xr{\psi}\cF(-\ha D')\lrcorner_{D'}=\cF$. Part (2) is proved.

All three cases above satisfy $|\k(\cE^{\da})-\k(\cF^{\da})|=1$, which verifies (1).
\end{proof}

For $\k=(D,N)$ and $\k'=(D',N')$ in  $\ZZ/2\ZZ[\Sig(\kbar)]\times\ZZ_{>0}$, let ${}^{\k,\k'}\Hk^{1}_{G}(\Sig)$ be the locally closed substack of $\Hk^{1}_{G}(\Sig)\ot\kbar$ whose geometric points are exactly those $(\cE^{\da},\cF^{\da},y,\ph)$ such that $\k(\cE^{\da})=\k$ and $\k(\cF^{\da})=\k'$. 

\begin{cor}[of Lemma \ref{l:kappa mod}]\label{c:Mod kappa}
\begin{enumerate}
\item The stack ${}^{\k,\k'}\Hk^{1}_{G}(\Sig)$ is empty unless  $|\k-\k'|=1$.
\item When $N=N'$ and $D$ and $D'$ differ only at $x\in \Sig(\kbar)$, the map $\pi^{1}_{\Hk}$ maps ${}^{\k,\k'}\Hk^{1}_{G}(\Sig)$ to a single point $x$, and there is an isomorphism
\begin{equation*}
{}^{\k,\k'}\Hk^{1}_{G}(\Sig)\isom({}^{N}\Bun_{G} \times_{{}^{N+1}\Bun_{G}}{}^{N}\Bun_{G})\otimes\kbar
\end{equation*}
where with both maps ${}^{N}\Bun_{G} \to {}^{N+1}\Bun_{G}$ given by $(-)\lrcorner_{x}$. The above isomorphism is given by  $(\cE^{\da}, \cF^{\da},x,\ph)\mapsto (\cE(\ha D), \cF(\ha D'), \a)$ as in Lemma \ref{l:kappa mod}(2).
\item When $N=N'-1$ and $D=D'$, we have an isomorphism
\begin{equation*}
{}^{\k,\k'}\Hk^{1}_{G}(\Sig)\isom ({}^{N}\Bun_{G}\times X)\otimes\kbar
\end{equation*}
given by $(\cE^{\da}, \cF^{\da},y,\ph)\mapsto (\cE(\ha D), y)$.
\item When $N=N'+1$ and $D=D'$, we have an isomorphism
\begin{equation*}
{}^{\k,\k'}\Hk^{1}_{G}(\Sig)\isom ({}^{N'}\Bun_{G}\times X)\otimes\kbar
\end{equation*}
given by $(\cE^{\da}, \cF^{\da},y,\ph)\mapsto (\cF(\ha D'), y)$. 
\end{enumerate}
\end{cor}

\begin{defn}\label{def: horo}
Let $\un\k=(\k_{0},\k_{1},\dotsc, \k_{r})$ be a sequence of elements in $\ZZ/2\ZZ[\Sig(\kbar)]\times \ZZ_{>0}$. 
\begin{enumerate}
\item The {\em horocycle of type $\un\k$} of $\Sht$ is the locally closed substack ${}^{\un\k}\Sht\subset \Sht$ whose geometric points are exactly those $(\cE^{\da}_{i};\dotsc)\in \Sht$ such that each $\cE^{\da}_{i}$ is purely unstable with $\k(\cE_{i}^{\da})=\k_{i}$, for $i=0,1,\dotsc, r$.
\item The {\em truncation up to $\un\k$} of $\Sht$ is the open substack of $\Sht$ consisting of $(\cE^{\da}_{i};\dotsc)$ such that $\cE^{\da}_{i}\in {}^{\le \k_{i}}\Bun_{G}(\Sig)$ for all $0\le i\le r$. 
\end{enumerate}
\end{defn}

Then ${}^{\un\k}\Sht$ is closed in ${}^{\le\un\k}\Sht$ and we denote its open complement by ${}^{<\un\k}\Sht$.

\sss{The index set for horocycles} 
Above we defined horocycles for any $r$-tuple of elements $\un\k$ in $\ZZ/2\ZZ[\Sig(\kbar)]\times \ZZ_{>0}$. However, for many such $\un\k$, ${}^{\un\k}\Sht$ turns out to be empty. 

\begin{lemma}\label{l:nonempty horo} Let $\un\k=(\k_{0},\k_{1},\dotsc, \k_{r})$ be a sequence of elements in $\ZZ/2\ZZ[\Sig(\kbar)]\times \ZZ_{>0}$. If ${}^{\un\k}\Sht$ is non-empty, then
\begin{enumerate}
\item For each $i=1,\dotsc, r$, $|\k_{i-1}-\k_{i}|=1$;
\item If we write $\k_{i}=(D_{i}, N_{i})$, then $N_{0}=N_{r}$, and $\Fr(D_{0})$ (applying the arithmetic Frobenius to each point appearing $D_{0}$) and $D_{r}$ differ at exactly one $\kbar$-point above each place of $\Si$ and nowhere else.  
\end{enumerate}
\end{lemma}
\begin{proof}
Suppose $(\cE^{\da}_{i},\dotsc)\in {}^{\un\k}\Sht$ is a geometric point over $\{x^{(1)}\}_{x\in \Si}\in \frSi$, then $|\k_{i-1}-\k_{i}|=1$ by Corollary \ref{c:Mod kappa}(1). The isomorphism $\cE_{r}\cong ({}^{\tau}\cE_{0})(\ha\sum_{x\in \Si}x^{(1)})$ implies $N_{0}=N_{r}$ and $\Fr(D_{0})+\sum_{x\in \Si}x^{(1)}\equiv D_{r}\mod 2$, which implies the second condition.
\end{proof}

\begin{defn}\label{defn:Kr} Let $\frK_{r}$ be the set of $\un\k=(\k_{0},\k_{1},\dotsc, \k_{r})$, where each $\k_{i}\in\ZZ/2\ZZ[\Sig(\kbar)]\times \ZZ$, satisfying the two conditions in Lemma \ref{l:nonempty horo}. (For technical reasons we do not impose $\k_{i}>0$ in the definition of $\frK_{r}$.)
\end{defn}
From the definition and Lemma \ref{l:nonempty horo} we see that 
\begin{equation*}
\Sht=\bigcup_{\un\k\in \frK_{r}}{}^{\le\un\k}\Sht.
\end{equation*}
The partial order on $\ZZ/2\ZZ[\Sig(\kbar)]\times\ZZ_{>0}$ extends to one on $\frK_{r}$: we say that $(\k_{0},\dotsc, \k_{r})\le (\k'_{0},\dotsc, \k'_{r})$ if and only if $\k_{i}\le\k'_{i}$ for all $0\le i\le r$. Then it is easy to check that, for $\un\k,\un\k'\in \frK_{r}$, ${}^{\un\k}\Sht\subset {}^{\le\un\k'}\Sht$ if and only if $\un\k\le \un\k'$.

For $\un\k\in\frK_{r}$ and $N\in \ZZ$, we write $\un\k>N$ if $N_{i}(\un\k)>N$  for all $0\le i\le r$ ($N_{i}(\un\k)$ denotes the $\ZZ_{>0}$-part of the $i$-th component of $\un\k$).

\sss{ $I(\un\k)$ and $X(\un\k)$} For $\un\k=(\k_{0},\dotsc, \k_{r})\in \frK_{r}$ with $\un\k_{i}=(D_{i}, N_{i})$, we define the subset $I(\un\k)\subset \{1,2,\dotsc, r\}$ as
\begin{equation*}
I(\un\k)=\{1\le i\le r| N_{i-1}\ne N_{i}\}.
\end{equation*}
For $i\in\{1,2,\dotsc, r\}-I(\un\k)$, there is a unique point $x\in \Sig(\kbar)$ such that $D_{i-1}$ and $D_{i}$ differ at $x$. We denote this point $x$ by $x_{i}(\un\k)$. Also, by the second condition on $\un\k$ above, the difference between $D_{r}$ and $\Fr(D_{0})$ consists of a $\kbar$-point $x^{(1)}(\un\k)$ over each $x\in\Si$.

For $i\in I(\un\k)$ we have $N_{i}=N_{i-1}\pm1$. Since $N_{r}=N_{0}$, we see that $\#I(\un\k)$ is even.

We define $X(\un\k)\subset (X^{r}\times\frSi)\ot\kbar$ to be the coordinate subspace
\begin{equation*}
X(\un\k)=\{(x_{1},\dotsc, x_{r}, \{x^{(1)}\}_{x\in\Si})| x_{i}=x_{i}(\un\k) \mbox{ for all }i\notin I(\un\k); x^{(1)}=x^{(1)}(\un\k)\mbox{ for all }x\in \Si\}.
\end{equation*}
The projection to the $I(\un\k)$-coordinates gives an isomorphism
\begin{equation*}
X(\un\k)\isom X^{I(\un\k)}\ot\kbar.
\end{equation*}

Viewing $\ZZ/2\ZZ[\Sig]$ as a subgroup of $\ZZ/2\ZZ[\Sig(\kbar)]$ by $\Sig\ni x\mapsto \sum_{\Sig(\kbar)\ni \ov x\mapsto x}\ov x$, there is an action of $\ZZ/2\ZZ[\Sig]$ on $\ZZ/2\ZZ[\Sig(\kbar)]$ by translation. This induces a diagonal action of $\ZZ/2\ZZ[\Sig]$ on $\frK_{r}$ by acting only on the divisor parts of each $\k_{i}$. For $\un\k,\un\k'\in\frK_{r}$, we say $\un\k\sim\un\k'$ if the divisor parts of $\un\k$ and $\un\k'$ are in the same $\ZZ/2\ZZ[\Sig]$-orbit (no other condition on the $\ZZ$-factors). This defines an equivalence relation on $\frK_{r}$. Let $[\frK_{r}]$ be the quotient 
\begin{equation*}
[\frK_{r}]:=\frK_{r}/\sim.
\end{equation*}

The following lemma is a direct calculation.
\begin{lemma}\label{l:supp} The map 
\begin{eqnarray*}
X(\cdot): \frK_{r} & \to & \{\mbox{subschemes of $(X^{r}\times\frSi)\ot\kbar$}\}\\
\un\k &\mapsto & X(\un\k)
\end{eqnarray*}
factors through $[\frK_{r}]$, and  induces an injective map
\begin{equation*}
X(\cdot): [\frK_{r}]\incl \{\mbox{subschemes of $(X^{r}\times\frSi)\ot\kbar$}\}.
\end{equation*}
\end{lemma}
By the above lemma, for $\s\in[\frK_{r}]$, we may write
\begin{equation*}
X(\s), \quad I(\s)
\end{equation*}
for $X(\un\k)$ and $I(\un\k)$, where $\un\k$ is any element in the orbit $\s$.

\begin{cor}[of Lemma \ref{l:kappa mod} and Corollary \ref{c:Mod kappa}]  For $\un\k\in\frK_{r}$ and $\un\k>0$, the restriction of the map $\Pi^{r}_{G}: \Sht\to X^{r}\times \frSi$ to ${}^{\un\k}\Sht$ has image in $X(\un\k)$. We denote the resulting map by
\begin{equation*}
\pi_{\un\k}: {}^{\un\k}\Sht\to X(\un\k).
\end{equation*}
\end{cor}

\sss{Geometry of horocycles}
For any $N>0$, we have a map
\begin{equation*}
\D: {}^{N}\Bun_{G}\to \Pic_{X}^{N}
\end{equation*}
sending $\cE$ to the line bundle $\D(\cE)=\cL\ot\cM^{-1}$ of degree $N$ on $X$, where $\cL\subset \cE$ is the maximal line subbundle and $\cM=\cE/\cL$. 

Now if $\un\k\in\frK_{r}$ and $\un\k>0$, for $(\cE^{\da}_{i};\dotsc)\in {}^{\un\k}\Sht$, we have a sequence of line bundles $\D_{i}:=\D(\cE_{i}(\ha D_{i}))$ by the above construction applied to $\cE_{i}(\ha D_{i})\in {}^{N_{i}}\Bun_{G}$ (recall $\k_{i}=(D_{i}, N_{i})$, so $\cE_{i}(\ha D_{i})$ has the smallest index of instability among all fractional twists of $\cE_{i}$). By Lemma \ref{l:kappa mod}, these line bundles are related by canonical isomorphisms
\begin{equation*}
\D_{i}\cong \begin{cases} \D_{i-1} & \mbox{ if } N_{i}=N_{i-1}; \\ 
\D_{i-1}(x_{i}) & \mbox{ if } N_{i}=N_{i-1}+1;\\
\D_{i-1}(-x_{i}) & \mbox{ if } N_{i}=N_{i-1}-1.
\end{cases}
\end{equation*}
Finally $\D_{r}\cong {}^{\tau}\D_{0}$. Therefore $\D=(\D_{0},\dotsc, \D_{r})$ together with the above isomorphisms give a point in $\Sht^{N(\un\k)}_{1}$, the moduli of rank one Shtukas $(\cL_{0},\cL_{1},\dotsc, \cL_{r})$ over $X$ with $\deg(\cL_{i})=N_{i}$ (when $N_{i-1}=N_{i}$ we have an isomorphism $\cL_{i-1}\isom \cL_{i}$). This gives a morphism
\begin{equation*}
q_{\un\k}: {}^{\un\k}\Sht\to \Sht^{N(\un\k)}_{1}\ot\kbar.
\end{equation*}
through which  the canonical map $\Pi^{r}_{G}: {}^{\un\k}\Sht_{G}^{r}(\Sii)\to X(\un\k)\cong X^{I(\un\k)}\ot\kbar$ factors.

\begin{lemma}\label{l:horo fib} Suppose $\un\k\in \frK_{r}$ and $\un\k>\max\{2g-2,0\}$. Then the map $q_{\un\k}$ is smooth of relative dimension $r-\#I(\un\k)/2$. The geometric fibers of $q_{\un\k}$ are isomorphic to $[\Ga^{r-\#I(\un\k)/2}/Z]$ for some finite \'etale group scheme $Z$ acting on $\Ga^{r-\#I(\un\k)/2}$ via a homomorphism $Z\to \Ga^{r-\#I(\un\k)/2}$.
\end{lemma}
\begin{proof}
The argument is similar to \cite[Lemma 7.5]{YZ}, so we only sketch the difference with the situation without level structures. We define ${}^{\un\k}\Hk^{r}_{G}(\Sig)\subset\Hk^{r}_{G}(\Sig)\ot\kbar $ to be the locally closed substack where $\k(\cE^{\da}_{i})=\k_{i}$ for $0\le i\le r$. Then ${}^{\un\k}\Hk^{r}_{G}(\Sig)$ is the iterated fiber product of ${}^{\k_{i-1},\k_{i}}\Hk^{1}_{G}(\Sig)$.  By definition, we have a Cartesian diagram
\begin{equation}\label{pre Cart Sht kappa}
\xymatrix{ {}^{\un\k}\Sht^{r}_{G}(\Sii)\ar[r]\ar[d]^{p_{0}} & {}^{\un\k}\Hk^{r}_{G}(\Sig)\ar[d]^{(p_{0}, \AL_{G,\infty}\circ p_{r})}\\
{}^{\k_{0}}\Bun_{G}(\Sig)\ar[r]^-{(\id, \Fr_{/\kbar})} & {}^{\k_{0}}\Bun_{G}(\Sig)\times{}^{\Fr(\k_{0})}\Bun_{G}(\Sig) }
\end{equation}
where the map $\Fr_{/\kbar}: {}^{\k_{0}}\Bun_{G}(\Sig)\to {}^{\Fr(\k_{0})}\Bun_{G}(\Sig)$ is the restriction of the $\kbar$-linear Frobenius $\Fr\times\id_{\kbar}: \Bun_{G}(\Sig)\otimes\kbar\to \Bun_{G}(\Sig)\otimes\kbar$ to the stratum ${}^{\k_{0}}\Bun_{G}(\Sig)$. Using Corollary \ref{c:kBunSig}, we may replace the bottom row by $(\id, \Fr\times\id_{\kbar}):{}^{N_{0}}\Bun_{G}\otimes\kbar\to ({}^{N_{0}}\Bun_{G}\otimes\kbar)\times_{\kbar}({}^{N_{0}}\Bun_{G}\otimes\kbar)$. The diagram \eqref{pre Cart Sht kappa} now reads
\begin{equation}\label{Cart Sht kappa}
\xymatrix{{}^{\un\k}\Sht^{r}_{G}(\Sii)\ar[rr]\ar[d]^{h_{0}} && {}^{\un\k}\Hk^{r}_{G}(\Sig)\ar[d]^{(h_{0}, h_{r})}\\
{}^{N_{0}}\Bun_{G}\ot\kbar\ar[rr]^-{(\id, \Fr\times\id_{\kbar})} && ({}^{N_{0}}\Bun_{G}\ot\kbar)\times_{\kbar}({}^{N_{0}}\Bun_{G}\ot\kbar)}
\end{equation}
where $h_{i}: {}^{\un\k}\Hk^{r}_{G}(\Sig)\to {}^{N_{i}}\Bun_{G}\otimes\kbar$ is the composition of $p_{i}$ with the isomorphism ${}^{\k_{i}}\Bun_{G}(\Sig)\isom{}^{N_{i}}\Bun_{G}\otimes\kbar$ in Corollary \ref{c:kBunSig}.

Let $S$ be a $\kbar$-algebra.  Fix an $S$-point $\un y=(y_{1},\dotsc, y_{r})\in X(\un\k)$, denote ${}^{\un\k}\Hk^{r}_{G}(\Sig)_{\un y}$ the fiber over $\un y$. Let ${}^{N}\Bun_{G,S}$ be the base change of ${}^{N}\Bun_{G}$ from $\Spec k$ to $S$. 

For $1\le i\le r$, let
\begin{equation*}
M_{i}=\min\{N_{i-1},N_{i}\}+1.
\end{equation*}
Then using the description of ${}^{\k_{i-1},\k_{i}}\Hk^{1}_{G}(\Sig)$ in Corollary \ref{c:Mod kappa}, we get an isomorphism
\begin{equation}\label{Hky}
{}^{\un\k}\Hk^{r}_{G}(\Sig)_{\un y}\cong {}^{N_{0}}\Bun_{G,S}\times_{{}^{M_{1}}\Bun_{G,S}}{}^{N_{1}}\Bun_{G,S}\times_{{}^{M_{2}}\Bun_{G,S}}{}^{N_{2}}\Bun_{G,S}\times\cdots\times_{{}^{M_{r}}\Bun_{G,S}}{}^{N_{r}}\Bun_{G,S}
\end{equation}
where the maps ${}^{N_{i-1}}\Bun_{G,S}\to {}^{M_{i}}\Bun_{G,S}$ and ${}^{N_{i}}\Bun_{G,S}\to {}^{M_{i}}\Bun_{G,S}$ are either the identity map or the pushout $\lrcorner_{y_{i}}$. 

There is a map $\D_{\Hk,\un y}: {}^{\un\k}\Hk_{G}^{r}(\Sig)_{\un y}\to \Pic^{N_{0}}_{X,S}\times\Pic_{X,S}^{N_{r}}$, which is induced by the map $\D:{}^{N_{i}}\Bun_{G}\to \Pic_{X}^{N_{i}}$ on each factor in \eqref{Hky}. Now we fix an $S$-point $\un\D=(\D_{0}, \D_{1},\dotsc, \D_{r})\in \Sht^{N(\un\k)}_{1}(S)$ over $\un y$, namely $\deg \D_{i}=N_{i}$ and $\D_{i}=\D_{i-1}((N_{i}-N_{i-1})y_{i})$ for $1\le i\le r$. Let $E_{i}\subset {}^{N_{i}}\Bun_{G,S}$ be the preimage of $\D_{i}\in \Pic_{X}^{N_{i}}(S)$ under $\D$ (so $E_{i}$ is an $S$-stack). Since $N_{i}>\max\{2g-2,0\}$, we have $E_{i}\cong \BB H_{i}$ is the classifying space of  the vector bundle $H_{i}=p_{S*}\D_{i}$ over $S$ (where $p_{S}:X\times S\to S$).  Similarly, we let $C_{i}\subset {}^{M_{i}}\Bun_{G,S}$ be the preimage of the following line bundle under $\D$
\begin{equation*}
\D'_{i}:=\begin{cases}\D_{i}(y_{i}) & \textup{if }N_{i}=N_{i+1},\\
\D_{i} & \textup{if } N_{i}=N_{i-1}+1,\\
\D_{i-1} & \textup{if } N_{i}=N_{i-1}-1.\end{cases}
\end{equation*}
We have $C_{i}\cong\BB J_{i}$ for the vector bundle $J_{i}=p_{S*}\D'_{i}$ over $S$. The canonical embeddings $\D_{i-1},\D_{i}\incl \D'_{i}$ induce embeddings  $H_{i-1}\incl J_{i}$ and $H_{i}\incl J_{i}$, hence maps $E_{i-1}\to C_{i}$ and $E_{i}\to C_{i}$ for $1\le i\le r$. By \eqref{Hky}, the preimage of $\un\D$ under  $\D_{\Hk,\un y}$ is
\begin{equation*}
E_{0}\times_{C_{1}}E_{1}\times_{C_{2}}\cdots\times_{C_{r}}E_{r}
\end{equation*}
which is isomorphic to the stack over $S$
\begin{equation*}
H_{0}\bs J_{1}\twtimes{H_{1}}J_{2}\twtimes{H_{2}}\cdots\twtimes{H_{r-1}}J_{r}/H_{r}
\end{equation*}
which is the quotient of $J_{1}\times \cdots J_{r}$ (product over $S$) by the action of $H_{0}$ on $J_{1}$, the diagonal action of $H_{1}$ on $J_{1}$ and $J_{2}$,..., the diagonal action of $H_{i}$ on $J_{i}$ and $J_{i+1}$,..., and the action of $H_{r}$ on $J_{r}$.   

Using the Cartesian diagram \eqref{Cart Sht kappa}, we get
\begin{equation*}
q_{\un\k}^{-1}(\un\D)\cong (J_{1}\twtimes{H_{1}}J_{2}\twtimes{H_{2}}\cdots\twtimes{H_{r-1}}J_{r})/H_{0}
\end{equation*}
where the action of $H_{0}$ is by translation on $J_{1}$ and on $J_{r}$,  via composing with the relative Frobenius $\Fr_{H_{0}/S}: H_{0}\to H_{r}$ and the $H_{r}$-translation on $J_{r}$. This presentation shows that $q_{\un\k}^{-1}(\un\D)$ is smooth over $S$. Hence $q_{\un\k}$ is smooth. 

To calculate the relative dimension of $q_{\un\k}$, we take $S=\Spec K$ to be a geometric point, and
\begin{equation*}
\dim q_{\un\k}^{-1}(\un\D)=\sum_{i=1}^{r}\dim J_{i}-\sum_{i=0}^{r-1}\dim H_{i}.
\end{equation*}
Since
\begin{equation*}
\dim J_{i}-\dim H_{i-1}=\dim\cohog{0}{X_{K}, \D'_{i}}-\dim\cohog{0}{X_{K}, \D_{i-1}}=\begin{cases} 1 & \textup{ if } N_{i}=N_{i-1} \textup{ or } N_{i}=N_{i-1}-1,\\
0 & \textup{ if } N_{i}=N_{i-1}+1,\end{cases}
\end{equation*}
we see that
\begin{equation*}
\dim q_{\un\k}^{-1}(\un\D)=r-\#\{1\le i\le r|N_{i}=N_{i-1}-1\}=r-\#I(\un\k)/2.
\end{equation*}
This proves the dimension part of the statement. The rest of the argument is the same as the last part of the proof of \cite[Lemma 7.5]{YZ}, using the fact that the translation of $H_{0}$ on $J_{1}$ induces a free action on the vector space $J_{1}\twtimes{H_{1}}J_{2}\twtimes{H_{2}}\cdots\twtimes{H_{r-1}}J_{r}$.
\end{proof}

\begin{cor}[of Lemma \ref{l:horo fib}]\label{c:Pk perv} Suppose  $\un\k\in \frK_{r}$ and $\un\k>\max\{2g-2,0\}$. Let $\pi_{1}^{N(\un\k)}: \Sht^{N(\un\k)}_{1}\ot\kbar \to X(\un\k)$ be the projection. Then we have a canonical isomorphism
\begin{equation*}
\bR \pi_{\un\k,!}\Ql\cong \bR\pi_{1,!}^{N(\un\k)}\Ql[-2r+\#I(\un\k)](-r+\#I(\un\k)/2).
\end{equation*}
In particular, $\bR \pi_{\un\k,!}\Ql$ is a local system shifted in degree $2r-\#I(\un\k)$, and
\begin{equation}\label{defn Pk}
P_{\un\k}:=\bR \pi_{\un\k,!}\Ql[2r](r)\in D^{b}(X(\un\k),\Ql)
\end{equation}
is a perverse sheaf on $X(\un\k)$ with full support and pure of weight 0.
\end{cor}

\sss{When  is $\Sht^{r}_{G}(\Sii)$ of finite type?}\label{sss:ft} 
Let
\begin{eqnarray*}
\frK^{\sh}_{r}&=&\{\un\k\in \frK_{r}|\un\k>\max\{2g-2,0\}\}\\
{}^{\sh}\Sht_{\kbar}&=&\cup_{\un\k\in\frK^{\sh}_{r}}{}^{\un\k}\Sht.
\end{eqnarray*}
Then ${}^{\sh}\Sht$ consists of $(\cE^{\da}_{i};\dotsc)$  where all $\inst(\cE^{\da}_{i})> \max\{2g-2,0\}$, therefore it is a closed substack of $\Sht$.  Let ${}^{\fl}\Sht=\Sht-{}^{\sh}\Sht$ be its open complement. 

\begin{lemma}\label{l:ft}
The substack ${}^{\fl}\Sht$ is of finite type over $\kbar$. 
\end{lemma}
\begin{proof}
Let $(\cE^{\da}_{i};\dotsc) $ be a geometric point of ${}^{\fl}\Sht$. Then for some $i_{0}$, $\inst(\cE^{\da}_{i_{0}})\le \max\{2g-2,0\}$, hence $\inst(\cE_{i_{0}})\le \max\{2g-2, 0\}+\deg\Sig$. Since  $\cE_{0}$ is related to $\cE_{i_{0}}$ by at most $r$ steps of elementary modifications, we have $\inst(\cE_{0})\le r+\max\{2g-2,0\}+\deg\Sig=:c$ for any $i$. Then  ${}^{\fl}\Sht$ is contained in the preimage of ${}^{\le c}\Bun_{G}$ under the map $\un p_{0}: \Sht\to \Bun_{G}$ (recording only $\cE_{0}$). Since $\un p_{0}$ is of finite type and ${}^{\le c}\Bun_{G}$ is of finite type over $k$,  so is ${}^{\fl}\Sht$.
\end{proof}

\begin{cor}[of Lemma \ref{l:ft} and Lemma \ref{l:horo fib}] The stack $\Sht^{r}_{G}(\Sii)$ is of finite type over $k$ if and only if $r<\#\Si$.
\end{cor}
\begin{proof}
If $r<\#\Si$, then the set $\frK_{r}$ is empty. In fact, if $\un\k=(\k_{0},\dotsc, \k_{r})\in \frK_{r}$, then the first condition defining $\frK_{r}$  implies $|D_{r}-D_{0}|\le  r$ ($D_{i}$ is the divisor part of $\k_{i}$), while the second condition implies that for each $x\in \Si$, $D_{0}$ and $D_{r}$ must differ at a geometric point above $x$, hence $|D_{r}-D_{0}|\ge \#\Si$. Therefore $\Sht={}^{\fl}\Sht$, which is of finite type over $\kbar$ by Lemma \ref{l:ft}. This implies that $\Sht$ is of finite type over $k$.

Conversely,  if $r\ge\#\Si$, then the set $\frK^{\sh}_{r}$ is infinite as can be seen in the following way. Write $\Si=\{x_{1},\cdots, x_{m}\}$, and fix $x^{(1)}_{i}\in X(\kbar)$ above each $x_{i}$. Let $D_{0}=0$, $D_{i}=x^{(1)}_{1}+\cdots+x^{(1)}_{i}$ for $1\le i\le m$,  and $D_{m}=D_{m+1}=\cdots=D_{r}$. Then take $N_{0}=\cdots=N_{m}$ and $N_{j}=N_{j-1}\pm1$ for $m<j\le r$ such that $N_{r}=N_{m}$ and $N_{i}>\max\{2g-2,0\}$ for all $0\le i\le r$ (there are infinitely many such sequences $(N_{i})$). Let $\k_{i}=(D_{i}, N_{i})$, then $\un\k=(\un\k_{1},\cdots,\un\k_{r})\in\frK^{\sh}_{r}$.  For each $\un\k\in\frK^{\sh}_{r}$ ,  ${}^{\un\k}\Sht$ is non-empty by Lemma \ref{l:horo fib}. Therefore $\Sht$ is not of finite type over $\kbar$ in this case.
\end{proof}

\subsection{Cohomological spectral decomposition}\label{ss:coho spec decomp}
In this subsection, we continue to use the abbreviations $\Sht, {}^{\un\k}\Sht$ as in \S\ref{ss:horo}. Let
\begin{equation*}
V=\cohoc{2r}{\Sht,\Ql}(r).
\end{equation*}
Since $\Sht$ is the union of open substacks ${}^{\le \un\k}\Sht$ for $\un\k\in\frK_{r}$, we have by definition
\begin{equation*}
V=\varinjlim_{\un\k\in\frK_{r},\un\k>0}\cohoc{2r}{{}^{\le \un\k}\Sht,\Ql}(r).
\end{equation*}

For $\un\k\in \frK_{r},\un\k>0$, let $\pi_{\le\un\k}: {}^{\le\un\k}\Sht\to (X^{r}\times\frSi)\ot\kbar$ be the restriction of $\Pi^{r}_{G}$. Let
\begin{equation*}
K_{\le \un\k}=\bR\pi_{\le\un\k,!}\Ql[2r](r)\in D^{b}((X^{r}\times\frSi)\ot\kbar,\Ql).
\end{equation*}
For $0<\un\k\le \un\k' \in\frK_{r}$, the open inclusion ${}^{\le\un\k}\Sht\incl {}^{\le\un\k'}\Sht$ induces a map
\begin{equation*}
\io_{\un\k,\un\k'}: K_{\le \un\k}\to K_{\le \un\k'}.
\end{equation*}

\sss{Ind-perverse sheaves}
The perverse sheaves $\{\pH^{i}K_{\le\un\k}\}_{\un\k\in \frK_{r}}$ form an inductive system indexed by the directed set $\frK_{r}$. Consider the inductive limit
\begin{equation*}
\pH^{i}K:=\varinjlim_{\un\k}\pH^{i}K_{\le\un\k}\in \ind \Perv((X^{r}\times\frSi)\ot\kbar,\Ql).
\end{equation*}
Here the right side is the category of ind-objects in the abelian category $\Perv((X^{r}\times\frSi)\ot\kbar,\Ql)$ of perverse constructible sheaves on $(X^{r}\times\frSi)\ot\kbar$, which is again an abelian category. Note that the notation $\pH^{i}K$ comes as a whole, as we are not defining $K$ as the inductive limit of $K_{\le\un\k}$, but only defining the ind-perverse sheaves $\pH^{i}K$. 

\begin{defn} Let $\ph: P\to P'$ be a morphism in $\ind \Perv((X^{r}\times\frSi)\ot\kbar,\Ql)$.
\begin{enumerate}
\item We say $\ph$ is an {\em mc-isomorphism} (mc for modulo constructibles), if the kernel and cokernel of $\ph$ are in the essential image of the natural embedding $\Perv((X^{r}\times\frSi)\ot\kbar,\Ql)\incl \ind\Perv((X^{r}\times\frSi)\ot\kbar,\Ql)$. 
\item We say $\ph$ is {\em mc-zero} if its image is in the essential image of the natural embedding $\Perv((X^{r}\times\frSi)\ot\kbar,\Ql)\incl \ind\Perv((X^{r}\times\frSi)\ot\kbar,\Ql)$. 
\end{enumerate}
\end{defn}

Likewise we have the notion of an mc-commutative square of ind-perverse sheaves, i.e.,  the appropriate difference of the compositions is mc-zero.  Concatenation of mc-commutative squares is still mc-commutative.

\begin{lemma}\label{l:pH tran}
Let $0<\un\k\le \un\k' \in\frK_{r}$. Then the map $\io_{\un\k,\un\k'}$ on the perverse cohomology sheaves
\begin{equation*}
\pH^{i}\io_{\un\k,\un\k'}: \pH^{i}K_{\le \un\k}\to \pH^{i}K_{\le \un\k'}
\end{equation*} 
is injective for $i=0$, surjective for $i=1$ and an isomorphism for $i\ne0,1$.

In particular, $\pH^{i}K$ is eventually stable when $i\ne0$ (i.e.,  the natural map $\pH^{i}K_{\le\un\k}\to \pH^{i}K$ is an isomorphism for sufficiently large $\un\k$).
\end{lemma}
\begin{proof} Let ${}^{(\un\k,\un\k']}\Sht={}^{\le\un\k'}\Sht-{}^{\le\un\k}\Sht$, which is a union of horocycles ${}^{\un\k''}\Sht$ for $\un\k''\le \un\k'$ but $\un\k''\nleq\un\k$. The horocycles form a stratification of ${}^{\le\un\k'}\Sht-{}^{\le\un\k}\Sht$. Let $\pi_{(\un\k,\un\k']}:{}^{(\un\k,\un\k']}\Sht\to(X^{r}\times\frSi)\ot\kbar $ be the projection. Then $K_{(\un\k,\un\k']}:=\bR\pi_{(\un\k,\un\k'],!}\Ql[2r](r)$ is the cone of $\io_{\un\k,\un\k'}$, and it is a successive extension of $P_{\un\k''}$ (see \eqref{defn Pk}), viewed as a complex on $(X^{r}\times\frSi)\ot\kbar$. By Corollary \ref{c:Pk perv}, $P_{\un\k''}$ is a perverse sheaf, therefore so is $K_{(\un\k,\un\k']}$. The long exact sequence for the perverse cohomology sheaves attached to the triangle $K_{\le \un\k}\to K_{\le \un\k'}\to K_{(\un\k,\un\k']}\to K_{\le \un\k}[1]$ then gives the desired statements.
\end{proof} 

\sss{Hecke symmetry on ind-perverse sheaves} A variant of the construction in \S\ref{sss:Hk coho} gives an $\sH^{\Sig}_{G}$-action on $\pH^{i}K$ for any $i\in \ZZ$. Namely, for each effective divisor $D$ on $X-\Sig$, the fundamental cycle of the Hecke correspondence $\Sht^{r}_{G}(\Sii;h_{D})$ (as a cohomological correspondence between constant sheaves on truncated $\Sht^{r}_{G}(\Sii)$) induces a map $K_{\le\un\k}\to K_{\le\un\k'}$ for $\k'-\k\ge d$. Passing to perverse cohomology sheaves and passing to inductive limits, we get a map in $\ind\Perv((X^{r}\times\frSi)\ot\kbar,\Ql)$
\begin{equation*}
\pH^{i}(h_{D}): \pH^{i}K\to \pH^{i}K.
\end{equation*}
The same argument as \cite[Prop. 7.1]{YZ}, using the dimension calculation in Lemma \ref{l:fiber ShthD}(3), shows that the assignment $h_{D}\mapsto \pH^{i}(h_{D})$, extended linearly, gives an action of $\sH^{\Sig}_{G}$ on $\pH^{i}K$.

\sss{The constant term map}

Recall the closed substack ${}^{\sh}\Sht$ of $\Sht$ and its open complement ${}^{\fl}\Sht$ from \S\ref{sss:ft}. Let $\pi_{\fl}:{}^{\fl}\Sht \to (X^{r}\times\frSi)\ot\kbar$ and $K_{\fl}=\bR\pi_{\fl,!}\Ql[2r](r)\in D^{b}((X^{r}\times\frSi)\ot\kbar,\Ql)$. 

We have a stratification of ${}^{\sh}\Sht$ by locally closed substacks ${}^{\un\k}\Sht$. Therefore we may similarly define $\pH^{i}K_{\sh}$ as the inductive limit of the perverse sheaves $\pH^{i}K_{\sh, \le\un\k}$ as $\un\k$ runs over $\frK_{r}$, where $K_{\sh,\le\un\k}$ is the direct image complex of ${}^{\sh}\Sht\cap {}^{\le\un\k}\Sht\to (X^{r}\times\frSi)\ot\kbar$. 

\begin{lemma}\label{l:const term}
\begin{enumerate}
\item The restriction map associated to the closed inclusion ${}^{\sh}\Sht\incl\Sht$ induces an mc-isomorphism of ind-perverse sheaves
\begin{equation*}
\pH^{0}K\to \pH^{0}K_{\sh}.
\end{equation*}
\item We have $\pH^{i}K_{\sh}=0$ for all $i\ne0$. Moreover, there is a {\em canonical} isomorphism of perverse sheaves on $(X^{r}\times\frSi)\ot\kbar$
\begin{equation*}
\pH^{0}K_{\sh}\cong \oplus_{\un\k\in\frK^{\sh}_{r}} P_{\un\k}.
\end{equation*}
\end{enumerate}
\end{lemma}
\begin{proof}
(1) By definition, we have ${}^{\fl}\Sht\subset {}^{\le\un\k}\Sht$  for $\un\k$ large enough, with complement $\cup_{\un\k'\in\frK^{\sh}_{r}, \un\k'\le\un\k}{}^{\un\k'}\Sht$. This gives a distinguished triangle $K_{\fl}\to K_{\le\un\k}\to K_{\sh, \le\un\k}\to$. The long exact sequence of perverse cohomology sheaves gives
\begin{equation*}
\pH^{0}K_{\fl}\to \pH^{0}K_{\le\un\k}\to \pH^{0}K_{\sh,\le\un\k}\to \pH^{1}K_{\fl}.
\end{equation*}
Taking inductive limit we get an exact sequence
\begin{equation*}
\pH^{0}K_{\fl}\to \pH^{0}K\to \pH^{0}K_{\sh}\to \pH^{1}K_{\fl}.
\end{equation*}
By Lemma \ref{l:ft}, ${}^{\fl}\Sht$ is a DM stack of finite type over $\kbar$, hence $K_{\fl}$ is constructible, and the middle map above is an mc-isomorphism. 

To show (2), it suffices to give a canonical isomorphism (again $\un\k$ is large enough so that ${}^{\fl}\Sht\subset {}^{\le\un\k}\Sht$)
\begin{equation*}
K_{\sh,\le\un\k}\cong \oplus_{\un\k'\in\frK^{\sh}_{r}, \un\k'\le\un\k} P_{\un\k'}.
\end{equation*}
compatible with the transition maps when $\un\k$ grows. Since $K_{\sh, \le\un\k}$ is a successive extension of $P_{\un\k'}$ for $\un\k'\in\frK^{\sh}_{r}$ and $\un\k'\le\un\k$, we have a canonical decomposition according support
\begin{equation*}
K_{\sh,\le\un\k}\cong \oplus_{\s\in [\frK_{r}]} (K_{\sh,\le\un\k})_{\s}
\end{equation*}
where we recall from Lemma \ref{l:supp} that the support of $P_{\un\k}$ is determined by the image of $\un\k$ in $[\frK_{r}]$, and different classes in $[\frK_{r}]$ give different supports. Each $(K_{\sh,\le\un\k})_{\s}$ is then a successive extension of those $P_{\un\k'}$ where $\un\k'\in\frK^{\sh}_{r}\cap \s$ and $\un\k'\le\un\k$. Hence $(K_{\sh,\le\un\k})_{\s}$ is a local system on $X(\s)$ shifted in degree $-\dim X(\s)=-\#I(\s)$. Let $\y_{\s}$ be a geometric generic point of $X(\s)$. It suffices to give a canonical decomposition of the stalks at $\y_{\s}$:
\begin{equation}\label{decomp ys}
(K_{\sh,\le\un\k})_{\s}|_{\y_{\s}}\cong \oplus_{\un\k'\in\frK^{\sh}_{r}\cap \s,\un\k'\le\un\k}P_{\un\k'}|_{\y_{\s}}.
\end{equation}

Now $K_{\sh,\le\un\k}|_{\y_{\s}}\cong \cohoc{2r-\#I(\s)}{{}^{\sh}\Sht_{\y_{\s}}\cap {}^{\le\un\k}\Sht_{\y_{\s}},\Ql}(r)$, and ${}^{\sh}\Sht_{\y_{\s}}\cap {}^{\le\un\k}\Sht_{\y_{\s}}=\cup_{\un\k'\le\un\k}{}^{\un\k'}\Sht_{\y_{\s}}$. If ${}^{\un\k'}\Sht_{\y_{\s}}\ne\vn$, we must have $X(\un\k')\supset X(\s)$, hence $\dim {}^{\un\k'}\Sht_{\y_{\s}} =r-\#I(\un\k')/2\le r-\#I(\s)/2$ with equality if and only if $\un\k'\in\s$. Hence $\dim {}^{\sh}\Sht_{\y_{\s}}\cap {}^{\le\un\k}\Sht_{\y_{\s}}\le r-\#I(\s)/2$, with top-dimensional components given by ${}^{\un\k'}\Sht_{\y_{\s}}$ for those $\un\k'\in\frK^{\sh}_{r}\cap \s$ and $\un\k'\le\un\k$.  This implies a canonical isomorphism
\begin{equation*}
\cohoc{2r-\#I(\s)}{{}^{\sh}\Sht_{\y_{\s}}\cap {}^{\le\un\k}\Sht_{\y_{\s}},\Ql}(r)\cong \oplus_{\un\k'\in\frK^{\sh}_{r}\cap \s,\un\k'\le\un\k}\cohoc{2r-\#I(\s)}{{}^{\un\k'}\Sht_{\y_{\s}},\Ql}(r),
\end{equation*}
which is exactly \eqref{decomp ys}.
\end{proof}
Combining the two maps in the above lemma, we get a canonical map of ind-perverse sheaves which is an mc-isomorphism
\begin{equation}\label{const term}
\g: \pH^{0}K\to \oplus_{\un\k\in\frK^{\sh}_{r}} P_{\un\k}.
\end{equation}
This can be called the {\em cohomological constant term operator}. 
\begin{remark}
Compared to the treatment in \cite[\S7.3.1]{YZ}, we do not need the generic fibers of the horocycles to be closed in $\Sht$. In fact the horocycle ${}^{\un\k}\Sht$ is  not necessarily closed when restricted to the generic point of $X(\un\k)$: for example this fails when $X(\un\k)$ is a point. 
\end{remark}

\sss{Constant term intertwines with Satake} 
Recall from Corollary \ref{c:Pk perv} that whenever $\un\k\in\frK^{\sh}_{r}$, we have an isomorphism
\begin{equation*}
P_{\un\k}\cong \bR\pi^{N(\un\k)}_{1,!}\Ql[-\#I(\un\k)](-\#I(\un\k)/2)
\end{equation*}
The map $\pi^{N(\un\k)}_{1}: \Sht^{N(\un\k)}_{1}\ot \kbar\to X(\un\k)$ is a $\Pic^{0}_{X}(k)$-torsor.  

Now for any $\un\k\in\frK_{r}$ (without assuming $\un\k>0$), the stack $\Sht^{N(\un\k)}_{1}$ is always defined, and $\pi^{N(\un\k)}_{1}:\Sht^{N(\un\k)}_{1}\ot\kbar\to X(\un\k)$ is a $\Pic_{X}(k)$-torsor. Moreover, the union
\begin{equation*}
\coprod_{\un\k'\in\un\k+\ZZ}\Sht^{N(\un\k')}_{1}\ot\kbar\to X(\un\k)
\end{equation*}
is a $\Pic_{X}(k)$-torsor, extending the $\Pic^{0}_{X}(k)$-torsor structure on each component of the LHS. Here we write $\un\k+\ZZ$ for $\ZZ$-orbit of $\un\k$ in $\frK_{r}$, and $\ZZ$ acts by translating the degree parts of $\un\k\in\frK_{r}$ simultaneously (note that $X(\un\k)$ is unchanged under the $\ZZ$-action). The $\Pic_{X}(k)$-action then gives an action on the ind-perverse sheaf
\begin{equation*}
\oplus_{\un\k'\in\un\k+\ZZ}\bR\pi^{N(\un\k)}_{1,!}\Ql[-\#I(\un\k)](-\#I(\un\k)/2).
\end{equation*}
Summing over all $\ZZ$-orbits of $\frK_{r}$ we get a canonical $\Pic_{X}(k)$-action on
\begin{equation*}
\oplus_{\un\k\in\frK_{r}}\bR\pi^{N(\un\k)}_{1,!}\Ql[-\#I(\un\k)](-\#I(\un\k)/2).
\end{equation*}
For any $u\in \Pic_{X}(k)$, restricting the source to $\oplus_{\un\k\in\frK^{\sh}_{r}} P_{\un\k}$ and projecting the target to $\oplus_{\un\k\in\frK^{\sh}_{r}} P_{\un\k}$, the $u$-action gives a map
\begin{equation*}
\a(u): \oplus_{\un\k\in\frK^{\sh}_{r}} P_{\un\k}\to \oplus_{\un\k\in\frK^{\sh}_{r}} P_{\un\k}.
\end{equation*}
However, this no longer gives an action of $\Pic_{X}(k)$. Instead, it is an mc-action: for $u,v\in \Pic_{X}(k)$, the endomorphism $a(uv)-a(u)a(v)$ of $\oplus_{\un\k\in\frK^{\sh}_{r}} P_{\un\k}$ is zero on $P_{\un\k}$ for $\un\k$ large enough, hence a mc-zero map. This mc-action extends to an mc-action of $\Ql[\Pic_{X}(k)]$ on $\oplus_{\un\k\in\frK^{\sh}_{r}} P_{\un\k}$, which we also denote by $\a$.

Recall the ring homomorphism
\begin{equation*}
a_{\Eis}: \sH^{\Sig}_{G}\xr{\Sat}\sH_{A}^{\Sig}=\QQ[\Div(X-\Sig)]\to \QQ[\Pic_{X}(k)].
\end{equation*}

\begin{lemma}\label{l:const Sat}
For any $f\in \sH^{\Sig}_{G}$, we have an mc-commutative diagram
\begin{equation*}
\xymatrix{  \pH^{0}K\ar[rr]^-{\pH^{0}(f)}\ar[d]^{\g} && \pH^{0}K \ar[d]^{\g}         \\
\oplus_{\un\k\in\frK^{\sh}_{r}}P_{\un\k}\ar[rr]^-{\a(a_{\Eis}(f))} && \oplus_{\un\k\in\frK^{\sh}_{r}}P_{\un\k}
}
\end{equation*}
In particular, if $f\in\cI_{\Eis}$, then the action $\pH^{0}(f):\pH^{0}K\to\pH^{0}K$ is mc-zero.
\end{lemma}
\begin{proof}
Since $\{h_{y}\}_{y\in|X-\Sig|}$ generate $\sH^{\Sig}_{G}$ as an algebra, it suffices to check the lemma for $f=h_{y}$ (we are also using the fact that $u\mapsto \a(u)$ is an mc-action of $\Ql[\Pic_{X}(k)]$ on $\oplus_{\un\k\in\frK^{\sh}_{r}}P_{\un\k}$). Let $d_{y}=[k(y):k]$.  We will show that $\g\circ \pH^{0}(f)$ and $\a(a_{\Eis}(f))\circ \gamma: \pH^{0}K\to\oplus_{\un\k\in\frK^{\sh}_{r}}P_{\un\k}$ agree on the factors $P_{\un\k}$ whenever $\un\k> \max\{2g-2,0\}+d_{y}$. Since we are checking whether two maps $\pH^{0}K\to P_{\un\k}$ agree, and $P_{\un\k}$ is a perverse sheaf all of whose simple constituents have full support on $X(\un\k)$, it suffices to check at a geometric generic point $\y$ of $X(\un\k)$. 

Since $a_{\Eis}(h_{y})=\one_{\cO(y)}+q^{d_{y}}\one_{\cO(-y)}$, we see that $a_{\Eis}(h_{y})P_{\un\k'}$ has $P_{\un\k}$-component only when $\un\k'> \max\{2g-2,0\}$ and $\un\k'\in\un\k+\ZZ$. In particular,  $\un\k'\in \frK^{\sh}_{r}$. Therefore, we only need to check that the following diagram is commutative
\begin{equation}\label{comm y}
\xymatrix{    \cohoc{2r-\#I(\un\k)}{\Sht_{\y}}\ar[rr]^{h_{y}}\ar[d]^{\g_{\y}} && \cohoc{2r-\#I(\un\k)}{\Sht_{\y}} \ar[d]^{\g_{\y,\un\k}}  \\
\oplus_{\un\k'\in \frK^{\sh}_{r}, \un\k'\in\un\k+\ZZ}\cohoc{0}{\Sht^{N(\un\k')}_{1,\y}}\ar[rr]^-{(a_{\Eis}(h_{y}))_{\un\k}} && \cohoc{0}{\Sht^{N(\un\k)}_{1,\y}}}
\end{equation}
Here the $\un\k'$ component of $\g_{\y}$ is the composition (where the first one is induced by the closed embedding of the closure of ${}^{\un\k'}\Sht_{\y}$)
\begin{equation*}
\g_{\y, \k'}: \cohoc{2r-\#I(\un\k)}{\Sht_{\y}}\to \cohoc{2r-\#I(\un\k)}{\ov{{}^{\un\k'}\Sht_{\y}}}\cong \cohoc{2r-\#I(\un\k)}{{}^{\un\k'}\Sht_{\y}}\cong \cohoc{0}{\Sht^{N(\un\k')}_{1,\y}}.
\end{equation*}
The proof of \eqref{comm y} is similar to that of \cite[Lemma 7.8]{YZ}. The key point is: if we restrict the Hecke correspondence $\Sht(h_{y})_{\y}$
\begin{equation*}
\xymatrix{\Sht_{\y} & \Sht(h_{y})_{\y}\ar[r]^-{\orr{p}_{\y}}\ar[l]_-{\oll{p}_{\y}} & \Sht_{\y}}
\end{equation*}
over the horocycle ${}^{\un\k}\Sht_{\y}$ via $\oll{p}_{\y}$, it decomposes into two pieces, one mapping isomorphically to ${}^{\un\k-d_{y}}\Sht_{\y}$ via $\orr{p}_{\y}$ and the other one is a finite \'etale cover of ${}^{\un\k+d_{y}}\Sht_{\y}$ of degree $q^{d_{y}}$ via $\orr{p}_{\y}$. We omit details.
\end{proof}

\sss{Key finiteness results}\label{sss:finiteness}
For $i\in\ZZ$, let 
\begin{equation*}
V_{\le i}:=\varinjlim_{\un\k}\cohog{0}{(X^{r}\times\frSi)\ot\kbar, \ptau_{\le i}K_{\le\un\k}}. 
\end{equation*}
Then we have natural maps
\begin{equation*}
\cdots\to V_{\le -1}\to V_{\le 0} \to V_{\le 1}\to \cdots \to V.
\end{equation*}
which are not necessarily injective. Since the action of $f$ comes from a cohomological correspondence,  the same  cohomological correspondence also acts on each $V_{\le i}$ making the above maps equivariant under the action of $\sH^{\Sig}_{G}$. We also have an $\sH^{\Sig}_{G}$-module map
\begin{equation*}
V_{\le i}\to  \cohog{-i}{(X^{r}\times\frSi)\ot\kbar, \pH^{i}K}.
\end{equation*} 

\begin{lemma}\label{l:W fin}
\begin{enumerate}
\item The kernel and the cokernel of $V_{\le 0}\to V$ are finite-dimensional.
\item The kernel and the cokernel of $V_{\le 0}\to \cohog{0}{(X^{r}\times\frSi)\ot\kbar, \pH^{0}K}$ are finite-dimensional.
\end{enumerate}
\end{lemma}
\begin{proof}
(1)  Since $\pH^{i}K=0$ for $i$ large, $V_{\le i}\isom V$ for $i$ sufficiently large. Similarly, $V_{i}=0$ for $i$ sufficiently small. Therefore it suffices to show that $V_{\le i}/V_{\le i-1}$ (namely modulo the image of $V_{\le i-1}$) is finite-dimensional for $i\ne0$. 

The triangle $\ptau_{\le i-1}K_{\le\un\k}\to \ptau_{\le i}K_{\le\un\k}\to \pH^{i}K_{\le\un\k}[-i]\to 0$ induces an injective map
\begin{equation*}
\cohog{0}{\ptau_{\le i}K_{\le\un\k}}/\cohog{0}{\ptau_{\le i-1}K_{\le\un\k}}\incl \cohog{-i}{(X^{r}\times\frSi)\ot\kbar,\pH^{i}K_{\le\un\k}}.
\end{equation*}
Taking inductive limit over $\un\k$, we have an injection
\begin{equation}\label{quot W}
V_{\le i}/V_{\le i-1}\incl\varinjlim_{\un\k}\cohog{-i}{(X^{r}\times\frSi)\ot\kbar,\pH^{i}K_{\le\un\k}}=\cohog{-i}{(X^{r}\times\frSi)\ot\kbar,\pH^{i}K}.
\end{equation}
(we use that $\varinjlim_{\un\k}$ commutes with taking cokernel). By Lemma \ref{l:pH tran}, the right side stabilizes as $\{\pH^{i}K_{\le\un\k}\}$ stabilizes for $i\ne0$, hence is finite-dimensional. Therefore, for $i\ne0$, $V_{\le i}/V_{\le i-1}$ is finite-dimensional.  In particular, $V_{\le -1}$ is finite-dimensional.

(2) The injection \eqref{quot W} is still valid for $i=0$, and it can be extended to an exact sequence
\begin{equation*}
0\to V_{\le 0}/V_{\le -1}\to \cohog{0}{(X^{r}\times\frSi)\ot\kbar,\pH^{0}K}\to \varinjlim_{\un\k}\cohog{1}{(X^{r}\times\frSi)\ot\kbar, \ptau_{\le-1}K_{\un\k}}.
\end{equation*}
By Lemma \ref{l:pH tran}, $\ptau_{\le-1}K_{\un\k}$ is eventually stable (in fact a constant inductive system), hence the last term above is finite-dimensional. Since $V_{\le -1}$ is also finite-dimensional, $V_{\le 0}\to \cohog{0}{(X^{r}\times\frSi)\ot\kbar,\pH^{0}K}$ has finite-dimensional kernel and the cokernel. 
\end{proof}

\begin{cor}[of Lemma \ref{l:const Sat} and \ref{l:W fin}]\label{c: IEis fin}
If $f\in \cI_{\Eis}$, then the image of the Hecke action $f: V\to V$ (defined in Proposition \ref{p:Hk coho action}) is finite-dimensional. 
\end{cor}
\begin{proof}
By Lemma \ref{l:W fin}(1), it suffices to show that the $f$-action on $V_{\le 0}$ has finite rank. By Lemma \ref{l:W fin}(2),  it suffices to show that $\pH^{0}(f): \pH^{0}K\to \pH^{0}K$ induces a finite-rank map after applying $\cohog{0}{(X^{r}\times\frSi)\ot\kbar,-}$. However, by Lemma \ref{l:const Sat},  $\pH^{0}(f)$ is mc-zero since $a_{\Eis}(f)=0$, and the conclusion follows.
\end{proof}

\begin{prop}\label{p:Hy fin} For any place $y\in |X|-\Sig$, $V$ is a finitely generated $\sH_{y}\ot\Ql$-module.
\end{prop}
\begin{proof}
By Lemma \ref{l:W fin}, it suffices to show that $\cohog{0}{(X\times\frSi)\ot\kbar,\pH^{0}K}$ is a finitely generated $\sH_{y}\ot\Ql$-module.

The ind-perverse sheaf $\pH^{0}K$ has an increasing filtration given by $\pH^{0}K_{\le\un\k}$ (by Lemma \ref{l:pH tran}) with associated graded $P_{\un\k}$. Let $F_{\le N}(\pH^{0}K)\subset \pH^{0}K$ be the sum of $\pH^{0}K_{\le\un\k}$ for $\un\k\in\frK^{\sh}_{r}$ and $\un\k\le Nd_{y}$. Then $\{F_{\le N}(\pH^{0}K)\}$ gives an increasing filtration on $\pH^{0}K$. The map $\g$ in \eqref{const term} induces 
\begin{equation*}
\Gr^{F}_{N}(\g): \Gr^{F}_{N}(\pH^{0}K)\to \oplus_{\un\k\in\frK^{\sh}_{r}, \un\k\le Nd_{y}, \un\k\nleq (N-1)d_{y}}P_{\un\k}
\end{equation*}
which is an isomorphism for large $N$, by Lemma \ref{l:const term}.

Now $h_{y}$ sends $F_{\le N}(\pH^{0}K)$ to $F_{\le N+1}(\pH^{0}K)$. By Lemma \ref{l:const Sat}, for $N$ large enough, the induced map
\begin{equation*}
\Gr^{F}_{N}(h_{y}): \Gr^{F}_{N}(\pH^{0}K)\to\Gr^{F}_{N+1}(\pH^{0}K)
\end{equation*}
is the same as the action of $\one_{\cO(y)}\in \Pic_{X}(k)$
\begin{equation}\label{one y action}
\one_{\cO(y)}: \oplus_{\un\k\in\frK^{\sh}_{r}, \un\k\le Nd_{y}, \un\k\nleq (N-1)d_{y}}P_{\un\k}\to \oplus_{\un\k\in\frK^{\sh}_{r}, \un\k\le (N+1)d_{y}, \un\k\nleq Nd_{y}}P_{\un\k}.
\end{equation}
Since $\one_{\cO(y)}$ maps $P_{\un\k}$ isomorphically to $P_{\un\k+d_{y}}$, \eqref{one y action} is an isomorphism. Therefore, $\Gr^{F}_{N}(h_{y})$ is an isomorphism for large $N$.

Next we apply $\cohog{0}{(X^{r}\times\frSi)\ot\kbar,-}$ to $F_{\le N}(\pH^{0}K)$ and $\pH^{0}K$, which we abbreviate as $\cohog{0}{F_{\le N}(\pH^{0}K)}$ and $\cohog{0}{\pH^{0}K}$. Note that each $F_{\le N}(\pH^{0}K)$ has a Weil structure, $\cohog{0}{F_{\le N}\pH^{0}K}$ is a Frobenius module and we can talk about its weight. We have an exact sequence
\begin{eqnarray}\label{longH1}
\cohog{0}{\Gr^{F}_{N}(\pH^{0}K)}\to  \cohog{1}{F_{\le N-1}(\pH^{0}K)}\to \cohog{1}{F_{\le N}(\pH^{0}K)}\to \cohog{1}{\Gr^{F}_{N}(\pH^{0}K)}
\end{eqnarray}
Since $\Gr^{F}_{N}(\pH^{0}K)$ is a sum of $P_{\un\k}$,  it is pure of weight 0 by Corollary \ref{c:Pk perv}. Therefore $\cohog{0}{\Gr^{F}_{N}(\pH^{0}K)}$ is pure of weight 0 and $\cohog{1}{\Gr^{F}_{N}(\pH^{0}K)}$ is pure of weight 1. For weight reasons, \eqref{longH1} gives a long exact sequence
\begin{eqnarray}
\label{longH0} \cohog{0}{F_{\le N-1}(\pH^{0}K)}\to \cohog{0}{F_{\le N}(\pH^{0}K)}\to \cohog{0}{\Gr^{F}_{N}(\pH^{0}K)}\\
\label{longW0H1} \to W_{\le 0}\cohog{1}{F_{\le N-1}(\pH^{0}K)}\to W_{\le 0}\cohog{1}{F_{\le N}(\pH^{0}K)}\to 0
\end{eqnarray}
where $W_{\le0}(-)$ means the sub Frobenius-module of weight $\le0$.  The surjectivity of \eqref{longW0H1} implies $W_{\le 0}\cohog{1}{F_{\le N}(\pH^{0}K)}$ is eventually stable for $N$ large, and hence the last arrow in \eqref{longH0} is surjective for  $N$ large. As $\Gr^{F}_{N}(h_{y})$ is an isomorphism for large $N$,  it induces an isomorphism $\cohog{0}{\Gr^{F}_{N}(\pH^{0}K)}\isom \cohog{0}{\Gr^{F}_{N+1}(\pH^{0}K)}$ for large $N$. This implies that for large $N$, the image of $\cohog{0}{F_{\le N}(\pH^{0}K)}$ in $\cohog{0}{\pH^{0}K}$ generates it as an $\sH_{y}\ot\Ql$-module.
\end{proof}

Let $\ov\sH^{\Sig}_{\ell}$ be the image of the ring homomorphism
\begin{equation*}
\sH^{\Sig}_{G}\ot\Ql\to \End_{\Ql}(V)\times\Ql[\Pic_{X}(k)]^{\io_{\Pic}}
\end{equation*}
given by the product of the action map on $V$ and $a^{\Sig}_{\Eis}$.

\begin{cor}[of Prop. \ref{p:Hy fin}]\label{c:H fin}
\begin{enumerate}
\item $\ov\sH^{\Sig}_{\ell}$ is a finitely generated $\Ql$-algebra of Krull dimension one.
\item $V$ is finitely generated as a $\ov\sH^{\Sig}_{\ell}$-module.
\end{enumerate}
\end{cor}
\begin{proof}
(2) is an obvious consequence of Prop. \ref{p:Hy fin}. The proof of part (1) is the same as \cite[Lemma 7.13(2)]{YZ}.
\end{proof}

\begin{theorem}[Cohomological spectral decomposition]\label{th:spec decomp}
\begin{enumerate}
\item There is a decomposition of the reduced scheme of $\Spec \ov\sH^{\Sig}_{\ell}$ into a disjoint union
\begin{equation*}
\Spec (\ov\sH^{\Sig}_{\ell})^{\red}=Z_{\Eis,\Ql}\coprod Z^{\Sig}_{0,\ell}
\end{equation*}
where $Z_{\Eis,\Ql}=\Spec \Ql[\Pic_{X}(k)]^{\io_{\Pic}}$ and $Z^{\Sig}_{0,\ell}$ consists of a finite set of closed points. 
\item There is a unique decomposition
\begin{equation*}
V=V_{0}\oplus V_{\Eis}
\end{equation*}
into $\sH^{\Sig}_{G}\ot\Ql$-submodules, such that $\Supp(V_{\Eis}) \subset Z_{\Eis,\Ql}$, and $\Supp(V_{0}) = Z^{\Sig}_{0,\ell}$.
\item The subspace $V_{0}$ is finite dimensional over $\Ql$.
\end{enumerate}
\end{theorem}
\begin{proof}
(1) By Lemma \ref{l:aEis surj}, $a^{\Sig}_{\Eis}$ induces a closed embedding $Z_{\Eis,\Ql}\incl \Spec\ov\sH^{\Sig}_{\ell}$. We are going to show that the complement of  $Z_{\Eis,\Ql}$ in $\Spec\ov\sH^{\Sig}_{\ell}$ is a finite set of closed points.

Let $\ov\cI_{\Eis}$ be the image of $\cI_{\Eis}$ in $\ov\sH^{\Sig}_{\ell}$, then by  Corollary \ref{c:H fin}, $\ov\sH^{\Sig}_{\ell}$ is noetherian and hence $\ov\cI_{\Eis}$ is finitely generated, say by $f_{1},\dotsc, f_{N}$. By Corollary \ref{c: IEis fin}, each $f_{i}\cdot V$ is finite-dimensional, therefore so is $\ov\cI_{\Eis} \cdot V=f_{1}\cdot V+\cdots + f_{N}\cdot V$. Now let $Z'_{0}\subset \Spec(\sH^{\Sig}_{\ell})^{\red}$ be the support of the finite-dimensional $\ov\sH^{\Sig}_{\ell}$-module $\ov\cI_{\Eis} \cdot V$. Hence $Z'_{0}$ is a finite set of closed points. The same argument as that of \cite[Theorem 7.14]{YZ} shows that $\Spec(\sH^{\Sig}_{\ell})^{\red}$ is the union of $Z_{\Eis,\Ql}$ and $Z'_{0}$. Finally we let $Z^{\Sig}_{0,\ell}$ be the complement of $Z_{\Eis,\Ql}$ in $\Spec (\ov\sH^{\Sig}_{\ell})^{\red}$.

The argument for (2) and (3) is the same as that of \cite[Theorem 7.14]{YZ}.
\end{proof}

\sss{The base-change situation}\label{sss:var coho decomp}
Consider the situation as in \S\ref{sss:bc}.  We argue that the analogue of Theorem \ref{th:spec decomp} holds for $\Sht'^{r}_{G}(\Sii)$ in place of $\Sht^{r}_{G}(\Sii)$. Let
\begin{equation*}
V'=\cohoc{2r}{\Sht'^{r}_{G}(\Sii)\ot\kbar, \Ql}(r).
\end{equation*}
Then $V'$ is also a $\sH^{\Sig}_{G}$-module, see the discussion in \S\ref{sss:bc Hk action}. The results in this subsection for the $\sH^{\Sig}_{G}$-module $V$ have obvious analogues for $V'$, because most of these results are consequences of finiteness results on $\pH^{i}K$ and similar results formally hold for its pullback to $X'^{r}\times\frSi'$. There is one place in the proof of Prop. \ref{p:Hy fin} where we used purity argument for the cohomology $\cohog{*}{(X^{r}\times \frSi)\ot\kbar,P_{\un\k}}$, which continues to hold for $\cohog{*}{(X'^{r}\times \frSi')\ot\kbar,\nu'^{r,*}P_{\un\k}}$. Therefore all results in this subsection hold for $V'$ in place of $V$. In particular, Theorem \ref{th:intro spec decomp} holds.

\section{The Heegner--Drinfeld cycles}

In this section we define Heegner--Drinfeld cycles in the ramified case. All the notation appearing on the geometric side of our main Theorem \ref{th:main} will be explained in this section.

\subsection{$T$-Shtukas}

\sss{The double cover}\label{sss:double cover} Let $X'$ be another smooth, projective and geometrically connected curve over $k$ and $\nu:X'\to X$ be a finite morphism of degree $2$. Let $R'\subset X'$ be the (reduced) ramification locus of $\nu$, and let $R\subset X$ be its image under $\nu$. Then $\nu$ induces an isomorphism $R'\isom R$.  Let $\s:X'\to X'$ be the nontrivial involution over $X$.

We always assume that the conditions \eqref{Sf split} and \eqref{Si inert} hold. In particular, they imply that 
\begin{equation*}
 R\cap\Sig=\vn.
\end{equation*}
Let
\begin{equation*}
\Si'=\nu^{-1}(\Si)\subset |X'|.
\end{equation*}
Then $\nu: \Si'\to \Si$ is a bijection. For $x\in\Si$ we denote its preimage in $\Si'$ by $x'$. Set
\begin{equation*}
\frSi'=\prod_{x'\in\Si'}\Spec k(x').
\end{equation*}

An  $S$-point of $\frSi'$ is $\{x'^{(1)}\}_{x'\in \Si'}$, where $x'^{(1)}: S\to \Spec k(x')\incl X'$. We introduce the notation $x'^{(i)}$ for all $i\in\ZZ$ as before. 

\sss{Hecke stack for $T$-bundles} Let
\begin{equation*}
\Bun_{T}=\Pic_{X'}/\Pic_{X}.
\end{equation*}
As  a special case of \cite[Definition 5.1]{YZ}, for  $\un\mu\in\{\pm1\}^{r}$,  we have the Hecke stack $\Hk^{\un\mu}_{1,X'}$ classifying a chain of $r+1$ line bundles on $X'$
\begin{equation*}
\xymatrix{\cL_{0}\ar@{-->}[r]^-{f'_{1}} & \cL_{1}\ar@{-->}[r]^-{f'_{2}} &\cdots\ar@{-->}[r]^-{f'_{r}} & \cL_{r}}
\end{equation*}
with modification type of $f'_{i}$ given by $\mu_{i}$. Then $\Hk^{\un\mu}_{1,X'}\cong \Pic_{X'}\times X'^{r}$ where the projection to $\Pic_{X'}$ records $\cL_{0}$, and the projection to $X'^{r}$ records the locus of modification of $f_{i}:\cL_{i-1}\dashrightarrow\cL_{i}$. We define
\begin{equation*}
\Hk^{\un\mu}_{T}:=\Hk^{\un\mu}_{1,X'}/\Pic_{X}
\end{equation*}
together with maps recording $\cL_{i}$
\begin{equation*}
p^{\un\mu}_{T,i}: \Hk^{\un\mu}_{T}\to \Bun_{T} ,\quad i=0,\dotsc, r.
\end{equation*}

\sss{$T$-Shtukas} 
For $x'\in\Si'$ and $i\in\ZZ$, we have a map
\begin{equation*}
\bx'^{(i)}: \frSi'\to \Spec k(x')\xr{\Fr^{i-1}}\Spec k(x')\incl X', \quad 1\le i\le d_{x'}=2d_{x}. 
\end{equation*}
where the first map is the projection to the $x'$-factor, and the last one is the natural embedding. Again we denote the graph of  $\bx'^{(i)}$ (as a divisor on $X'\times\frSi'$) by the same notation $\bx'^{(i)}$.

Let $\sD'_{\infty}$ be the group of $\ZZ$-valued divisors on $X'\times\frSi'$ supported on $\Si'\times\frSi'$, which is the union of the graphs of $\bx'^{(i)}$ for $x'\in \Si', 1\le i\le d_{x'}$. For any $\Di'\in\sD_{\infty}'$ as above,  we have morphisms
\begin{eqnarray*}
\wt\AL(\Di'): \Pic_{X'}\times \frSi' &\to& \Pic_{X'},\\
\AL(\Di'): \Bun_{T}\times \frSi' &\to& \Bun_{T}\\
(\cL, \{x'^{(1)}\}_{x'\in \Si'})& \mapsto & \cL(\sum_{x'\in \Si', 1\le i\le d_{x'}}c_{x'}^{(i)}\Gamma_{x'^{(i)}}).
\end{eqnarray*}

Suppose $\un\mu\in\{\pm1\}^{r}$ and $\Di'\in\sD'_{\infty}$ satisfy
\begin{equation}\label{comp mu Di'}
\sum_{i=1}^{r}\mu_{i}=\deg \Di'=\sum_{x'\in\Si',1\le i\le d_{x'}}c^{(i)}_{x'}.
\end{equation}
We then apply the definition of $\Sht^{\un\mu}_{n}(\Sig;\Di)$ to the case $n=1$, the curve being $X'$ and $\Sig$ and $\Si$ are both replaced by $\Si'$.  Denote the resulting moduli stack by $\Sht^{\un\mu}_{1,X'}(\Di')$.

The groupoid $\Pic_{X}(k)$ acts on $\Sht^{\un\mu}_{1,X'}(\Di')$ by tensoring all the line bundles in the data with the pullback of $\cK\in \Pic_{X}(k)$ to $X'$. We define
\begin{equation*}
\Sht^{\un\mu}_{T}(\Di')=\Sht^{\un\mu}_{1,X'}(\Di')/\Pic_{X}(k).
\end{equation*}
We have a morphism
\begin{equation*}
\Pi^{\un\mu}_{T,\Di'}: \Sht^{\un\mu}_{T}(\Di')\to X'^{r}\times\frSi'.
\end{equation*}

From the definition we have a Cartesian diagram
\begin{equation}\label{ShtT diag 1}
\xymatrix{       \Sht^{\un\mu}_{T}(\Di')\ar[r]\ar[d]^{\om_{T,0}} &  \Hk^{\un\mu}_{T}\times\frSi' \ar[d]^{(p^{\un\mu}_{T,0}, \AL(-\Di')\circ (p^{\un\mu}_{T,r}\times\id_{\frSi'}))}  \\
\Bun_{T} \ar[r]^-{(\id,\Fr)}& \Bun_{T}\times\Bun_{T}}
\end{equation}

From the diagram we get the following statement.
\begin{lemma}\label{l:indep Di'}
The moduli stack $\Sht^{\un\mu}_{T}(\Di')$ depends only on the image of $\Di'$ in $\sD_{\infty}'/\nu^{*}\sD_{\infty}$.
\end{lemma}

The following alternative description of $\Sht^{\un\mu}_{T}(\Di')$ follows easily from the definitions.

\begin{lemma}\label{l:ShtT Lang} We have a Cartesian diagram
\begin{equation*}
\xymatrix{\Sht^{\un\mu}_{T}(\Di')\ar[r]^(.6){\om_{T,0}}\ar[d]_{\Pi^{\un\mu}_{T,\Di'}} & \Bun_{T}\ar[d]^{\l}\\
X'^{r}\times\frSi'\ar[r]^{\a^{\un\mu}_{\Di'}} & \Bun_{T}}
\end{equation*}
where $\l:\cL\mapsto\cL^{-1}\otimes\leftexp{\tau}{\cL}$ is the Lang map for $\Bun_{T}$; $\a^{\un\mu}_{\Di'}$ sends $(x'_{1},\dotsc, x'_{r};\{x'^{(1)}\}_{x'\in\Si'})$  to the image of the line bundle 
$$\cO_{X'}\left(\sum_{i=1}^{r}\mu_i\Gamma_{x'_{i}}-\sum_{x'\in\Si', 1\le i\le d_{x'}}c^{(i)}_{x'}\Gamma_{x'^{(i)}}\right)$$ 
in $\Bun_{T}$.
\end{lemma}

\begin{cor}[of Lemma \ref{l:ShtT Lang}]\label{c:ShtT proper} The morphism $\Pi^{\un\mu}_{T,\Di'}$ is a torsor under the (finite discrete) groupoid $\Bun_{T}(k)$. In particular, $\Sht^{\un\mu}_{T}(\Di')$ is a smooth and proper DM stack over $k$ of dimension $r$.
\end{cor}

\sss{Specific choice of $\Di'$}
For each $\mi=(\mu_{x})_{x\in \Si}\in\{\pm1\}^{\Si}$, define the following element in $\sD'_{\infty}$
\begin{equation*}
\mi\cdot\Si':=\sum_{x\in\Si} \mu_{x}\bx'^{(1)}\in\sD'_{\infty}.
\end{equation*}


\begin{defn}\label{d:ShtTmu} Fix $r$ satisfying the parity condition \eqref{parity}. Let $\un\mu\in\{\pm1\}^{r}, \mi\in\{\pm1\}^{\Si}$. For any $\Di'\in \sD'_{\infty}$ satisfying $\Di'\equiv \mi\cdot\Si'\mod \nu^{*}\sD_{\infty}$ and \eqref{comp mu Di'}, define
\begin{equation*}
\Sht^{\un\mu}_{T}(\mi\cdot\Si'):=\Sht^{\un\mu}_{T}(\Di').
\end{equation*}
The notation is justified because the right side above depends only on $\mu_\infty$ by Lemma \ref{l:indep Di'}. We denote the projection $\Pi^{\un\mu}_{T,\Di'}$ for such $\Di'$ by
\begin{equation*}
\Pi^{\un\mu}_{T,\mi}: \Sht^{\un\mu}_{T}(\mi\cdot\Si')\to X'^{r}\times\frSi'.
\end{equation*}
\end{defn}

\begin{remark}
Whenever $r$ satisfies the parity condition \eqref{parity}, for any $(\un\mu,\mi)\in \{\pm1\}^{r}\times\{\pm1\}^{\Si}$, the divisor $\Di'\in \sD_{\infty}'$ satisfying the conditions in Definition \ref{d:ShtTmu} always exists. Therefore, $\Sht^{\un\mu}_{T}(\mi\cdot\Si')$ is always defined (and non-empty).
\end{remark}

The following lemma is a direct consequence of the diagram  \eqref{ShtT diag 1}.
\begin{lemma}
The following diagram is Cartesian
\begin{equation}\label{ShtT diag 2}
\xymatrix{  \Sht^{\un\mu}_{T}(\mi\cdot\Si')\ar[d]      \ar[r]   & \Hk^{\un\mu}_{T}\times \frSi'\ar[d]^{(p^{\un\mu}_{T,0}\times\id_{\frSi'}, \AL^{\sh}_{T,\mi}\circ (p^{\un\mu}_{T,r}\times\id_{\frSi'}))}\\
\Bun_{T}\times\frSi'\ar[r]^-{(\id,\Fr)} & (\Bun_{T}\times\frSi')\times(\Bun_{T}\times\frSi')}
\end{equation}
where $\AL^{\sh}_{T,\mi}$ is the map
\begin{equation}\label{AL sh T}
\AL^{\sh}_{T,\mi}=(\AL(-\mi\cdot\Si'), \Fr_{\frSi'}): \Bun_{T}\times\frSi'\to \Bun_{T}\times\frSi'.
\end{equation}
\end{lemma}

\sss{Relation to $T$-Shtukas in \cite{YZ}} For $(\un\mu,\mi)\in \{\pm1\}^{r}\times\{\pm1\}^{\Si}$, let $\wt\mu=(\un\mu,-\mi)$. Then $\Sht^{\wt\mu}_{T}$ is defined as in \cite[\S5.4]{YZ} (the {\em loc. cit.} also applies to a ramified cover $X'/X$), with a map $\pi^{\wt\mu}_{T}: \Sht^{\wt\mu}_{T}\to X'^{r}\times X'^{\Si}$. Let $\frSi'\incl X'^{\Si}$ be the product of the natural embeddings $\Spec k(x')\incl X'$ for each $x\in\Si$. From the definitions, we see that $\Sht^{\un\mu}_{T}(\mi\cdot\Si')$ fits into a Cartesian diagram
\begin{equation*}
\xymatrix{     \Sht^{\un\mu}_{T}(\mi\cdot\Si')\ar@{^{(}->}[r]\ar[d]^{\Pi^{\un\mu}_{T,\mi}}  &   \Sht^{\wt\mu}_{T}\ar[d]^{\pi^{\wt\mu}_{T}} \\
X'^{r}\times \frSi'\ar@{^{(}->}[r] & X'^{r}\times X'^{\Si}}
\end{equation*}


\subsection{The Heegner--Drinfeld cycles}
In this subsection we will define a map from $\Sht^{\un\mu}_{T}(\mi\cdot\Si')$ to $\Sht^{r}_{G}(\Sii)$ depending on an auxiliary choice. 

Recall that the condition \eqref{Sf split} is assumed. Let $\Sf'=\nu^{-1}(\Sf)$. Let $\Sect(\Sf'/\Sf)$ be the set of sections of the two-to-one map $\Sf'\to \Sf$. Then $\Sect(\Sf'/\Sf)$ is a torsor under $\{\pm1\}^{\Sf}$. The auxiliary choice we need is an element $\mf\in \Sect(\Sf'/\Sf)$. 

\sss{The map $\th^{\mu_{\Sig}}_{\Bun}$}\label{sss:th Bun}  Let $\mu_{\Sig}=(\mf,\mi)\in \Sect(\Sf'/\Sf)\times \{\pm1\}^{\Si}$.  We define a map
\begin{equation*}
\wt\th^{\mu_{\Sig}}_{\Bun}: \Pic_{X'}\times\frSi'\to \Bun_{2}(\Sig).
\end{equation*}
To an $S$-point $(\cL, \{x'^{(1)}\}_{x'\in\Si'})$ of $\Pic_{X'}\times\frSi'$, we assign the following $S$-point of $\Bun_{2}(\Sig)$
\begin{equation*}
\cE^{\da}=(\cE, \{\cE(-\ha x)\}_{x\in \Sig})
\end{equation*}
where
\begin{itemize}
\item $\cE=\nu_{S,*}\cL$, where $\nu_{S}=\nu\times \id_{S}: X'\times S\to X\times S$.
\item For $x\in \Sf$, denote the value of $\mf$ at $x$ by $\mu_{x}\in \nu^{-1}(x)$. Then $\cE(-\ha x)=\nu_{S,*}(\cL(-\mu_{x}))$.
\item For $x\in \Si$,  
\begin{equation*}
\cE(-\ha x)=\begin{cases}\nu_{S,*}(\cL(-\Gamma_{x'^{(1)}}-\Gamma_{x'^{(2)}}-\cdots-\Gamma_{x'^{(d_{x})}})) & \mu_{x}=1;\\
\nu_{S,*}(\cL(-\Gamma_{x'^{(d_{x}+1)}}-\Gamma_{x'^{(d_{x}+2)}}-\cdots-\Gamma_{x'^{(2d_{x})}})) & \mu_{x}=-1.
\end{cases}
\end{equation*}
\end{itemize}
Note here that for $x\in \Si$, the divisors $\Gamma_{x'^{(1)}}+\Gamma_{x'^{(2)}}+\cdots+\Gamma_{x'^{(d_{x})}}$ and $\Gamma_{x'^{(d_{x}+1)}}+\Gamma_{x'^{(d_{x}+2)}}+\cdots+\Gamma_{x'^{(2d_{x})}}$ in the above formulae are ``half'' of the divisor $\{x'\}\times S\subset X'\times S$.

Dividing by $\Pic_{X}$ we get a morphism 
\begin{equation*}
\th^{\mu_{\Sig}}_{\Bun}: \Bun_{T}\times\frSi'\to \Bun_{G}(\Sig).
\end{equation*}

The next lemma is a direct calculation.
\begin{lemma}\label{l:comp Di Di'} Let $\mu_{\Sig}=(\mf,\mi)$. The following diagram is commutative
\begin{equation*}
\xymatrix{   \Bun_{T} \times\frSi' \ar[d]_{(\th^{\mu_{\Sig}}_{\Bun}, \nu_{\infty})}     \ar[rr]^-{\AL^{\sh}_{T,\mi}} && \Bun_{T}\times\frSi' \ar[d]^{\th^{\mu_{\Sig}}_{\Bun}}\\
\Bun_{G}(\Sig)\times\frSi \ar[rr]^-{\AL_{G,\infty}} && \Bun_{G}(\Sig)}
\end{equation*}
where $\nu_{\infty}:\frSi'\to\frSi $ is the map induced from $\nu$.
\end{lemma}

\sss{Heegner--Drinfeld cycle} \label{sss HD}
We define
\begin{equation*}
\frT_{r,\Sig}:=\{\pm1\}^{r}\times \Sect(\Sf'/\Sf)\times \{\pm1\}^{\Si}.
\end{equation*}
For $\mu=(\un\mu, \mf,\mi)\in \frT_{r,\Sig}$, we have a map
\begin{equation*}
\th^{\mu}_{\Hk}: \Hk^{\un\mu}_{T}\times\frSi'\to \Hk^{r}_{G}(\Sig)
\end{equation*}
by applying $\th^{\mu_{\Sig}}_{\Bun}$ (where $\mu_{\Sig}=(\mf,\mi)$) to each member of the chain $\{\cL_{i}\}_{0\le i\le r}$ classified by $\Hk^{\un\mu}_{T}$. By construction we have $p_{i}\circ \th^{\mu}_{\Hk}=\th^{\mu_{\Sig}}_{\Bun}\circ(p^{\un\mu}_{T,i}\times\id_{\frSi'}): \Hk^{\un\mu}_{T}\times\frSi'\to \Bun_{G}(\Sig)$ for $1\le i\le r$.

Now compare the Cartesian diagrams \eqref{ShtT diag 2} and \eqref{Cart ShtG}. Each corner of the diagram \eqref{ShtT diag 2} except the upper left corner maps to the corresponding corner of \eqref{Cart ShtG} by $\th_{\Bun}$ and $\th^{\mu}_{\Hk}$; Lemma \ref{l:comp Di Di'} says that the corresponding maps in the two diagrams are intertwined. Therefore we get a morphism between the upper left corners since both diagrams are Cartesian
\begin{equation*}
\th^{\mu}: \Sht^{\un\mu}_{T}(\mi\cdot\Si')\to \Sht^{r}_{G}(\Sii).
\end{equation*}
We have a commutative diagram 
\begin{equation*}
\xymatrix{\Sht^{\un\mu}_{T}(\mi\cdot\Si')\ar[r]^{\th^{\mu}}\ar[d]^{\Pi^{\un\mu}_{T,\mi}} &  \Sht^{r}_{G}(\Sii)\ar[d]^{\Pi^{r}_{G}}\\
X'^{r}\times\frSi'\ar[r]^{(\nu^{r},\nu_{\infty})} & X^{r}\times\frSi}
\end{equation*}
which induces a morphism
\begin{equation*}
\th'^{\mu}: \Sht^{\un\mu}_{T}(\mi\cdot\Si')\to \Sht'^{r}_{G}(\Sii):=\Sht^{r}_{G}(\Sii)\times_{X^{r}\times\frSi}(X'^{r}\times\frSi').
\end{equation*}

Since $\Sht^{\un\mu}_{T}(\mi\cdot\Si')$ is proper over $k$ of dimension $r$ by Corollary \ref{c:ShtT proper}, its image in $\Sht'^{r}_{G}(\Sii)$ defines an element in the Chow group of proper cycles.

\begin{defn} The {\em Heegner--Drinfeld cycle of type $\mu=(\un\mu,\mf,\mi)\in\frT_{r,\Sig}$} is the class
\begin{equation*}
\cZ^{\mu}:=\th'^{\mu}_{*}[\Sht^{\un\mu}_{T}(\mi\cdot\Si')]\in \Ch_{c,r}(\Sht'^{r}_{G}(\Sii))_{\QQ}.
\end{equation*}
\end{defn}

\begin{defn} Let $\mu,\mu'\in\frT_{r,\Sig}$.
Define a linear functional $\II^{\mu,\mu'}$ on $\sH^{\Sig}_{G}$ by 
\begin{equation*}
\II^{\mu,\mu'}(f)=\left(\prod_{x'\in \Si'}d_{x'}\right)^{-1}\langle\cZ^{\mu}, f*\cZ^{\mu'}\rangle_{\Sht'^{r}_{G}(\Sii)}\in\QQ. \quad f\in \sH^{\Sig}_{G}.
\end{equation*}
Here we are using the $\sH^{\Sig}_{G}$-action on $\Ch_{c,r}(\Sht'^{r}_{G}(\Sii))_{\QQ}$ defined in \S\ref{sss:bc Hk action}.
\end{defn}

\subsection{Symmetry among Heegner--Drinfeld cycles}

Let $\mu=(\un\mu, \mf,\mi)\in \frT_{r,\Sig}$. We study how $\cZ^{\mu}$ changes when we vary $\mu$.

\sss{Changing  $\un\mu$}\label{sss:change unmu}
As in \cite[\S5.4.6]{YZ}, for two choices $\un\mu,\un\mu'\in\{\pm1\}^{r}$, there is a canonical isomorphism $\io_{\un\mu,\un\mu'}: \Sht^{\un\mu}_{T}(\mi\cdot\Si')\cong \Sht^{\un\mu'}_{T}(\mi\cdot\Si')$ preserving the $T$-bundle $\cL_{i}$ and the projection to $\frSi'$. However, $\io_{\un\mu,\un\mu'}$ does not preserve the projections $\Pi^{\un\mu}_{T,\mi}$ and $\Pi^{\un\mu'}_{T,\mi}$. Instead, we have a commutative diagram
\begin{equation*}
\xymatrix{ \Sht^{\un\mu}_{T}(\mi\cdot\Si')\ar[rr]^-{\io(\un\mu,\un\mu')}\ar[d]^{\Pi^{\un\mu}_{T,\mi}} && \Sht^{\un\mu'}_{T}(\mi\cdot\Si')\ar[d]^{\Pi^{\un\mu'}_{T,\mi}}\\
X'^{r}\times\frSi'\ar[rr]^-{\s(\un\mu,\un\mu')\times\id} && X'^{r}\times\frSi'
}
\end{equation*}
where the involution $\s(\un\mu,\un\mu'): X'^{r}\to X'^{r}$ sends a point $(x'_{1},\dotsc,x'_{r})$ to the point $(x''_{1},\dotsc,x''_{r})$, where, for $1\le i\le r$,
\begin{eqnarray*}
x''_{i}=\begin{cases}x'_{i} & \textup{ if }\mu_{i}=\mu'_{i}; \\ \s(x'_{i}) &\textup{ if }\mu_{i}\ne\mu'_{i}. \end{cases}
\end{eqnarray*}
Letting $\mu'=(\un\mu', \mf, \mi)$, it is easy to check that $\io(\un\mu,\un\mu')$ intertwines the map $\th^{\mu}$ and $\th^{\mu'}$.

\sss{Changing $\mf$}\label{sss:change mf}
Let $\mf'=\{\mu'_{x}\}_{x\in \Sf}\in \Sect(\Sf'/\Sf)$ be another element. Consider the following divisor on $X'$
\begin{equation*}
D(\mf,\mf')=\sum_{x\in \Sf, \mu_{x}\ne\mu'_{x}}\mu_{x}.
\end{equation*}
We have an automorphism
\begin{equation*}
\io(\mf,\mf'): \Sht^{\un\mu}_{T}(\mi\cdot\Si')\to\Sht^{\un\mu}_{T}(\mi\cdot\Si') 
\end{equation*}
sending $(\cL_{i}; x_{i}; \{x'^{(1)}\})$ to $(\cL_{i}(-D(\mf,\mf')); x_{i}; \{x'^{(1)}\})$. Letting $\mu'=(\un\mu, \mf', \mi)$, direct calculation shows that the following diagram is commutative
\begin{equation*}
\xymatrix{   \Sht^{\un\mu}_{T}(\mi\cdot\Si')\ar[rr]^-{\io(\mf,\mf')}\ar[d]^{\th^{\mu}} && \Sht^{\un\mu}_{T}(\mi\cdot\Si')\ar[d]^{\th^{\mu'}}        \\
\Sht^{r}_{G}(\Sii) \ar[rr]^{\AL_{\Sht}(\mf, \mf')} &&  \Sht^{r}_{G}(\Sii)}
\end{equation*}
where $\AL_{\Sht}(\mf, \mf')$ is the composition of $\AL_{\Sht, x}$ (see \S\ref{sss:AL Sht}) for $x\in\Sf$ such that $\mu_{x}\ne\mu'_{x}$.

\sss{Changing $\mi$}\label{sss:change mi} 
Let $\mi'\in \{\pm1\}^{\Si}$ be another element. Consider the following divisor on $X'\times\frSi'$
\begin{equation*}
D(\mi,\mi')=\sum_{\mu_{x}=1,  \mu'_{x}=-1}(\bx'^{(1)}+\cdots+\bx'^{(d_{x})})+\sum_{\mu_{x}=-1,  \mu'_{x}=1}(\bx'^{(d_{x}+1)}+\cdots+\bx'^{(2d_{x})}).
\end{equation*}
where both sums are over $x\in \Si$. Define an isomorphism
\begin{equation*}
\io(\mi,\mi'): \Sht^{\un\mu}_{T}(\mi\cdot\Si')\to \Sht^{\un\mu}_{T}(\mi'\cdot\Si')
\end{equation*}
sending $(\cL_{i}; x_{i}; \{x'^{(1)}\})$ to $(\cL_{i}(-D(\mi,\mi')); x_{i}; \{x'^{(1)}\})$.  Letting $\mu'=(\un\mu, \mf, \mi')$, direct calculation shows that the following diagram is commutative
\begin{equation*}
\xymatrix{   \Sht^{\un\mu}_{T}(\mi\cdot\Si')\ar[rr]^-{\io(\mi,\mi')}\ar[d]^{\th^{\mu}} && \Sht^{\un\mu}_{T}(\mi'\cdot\Si')\ar[d]^{\th^{\mu'}}        \\
\Sht^{r}_{G}(\Sii) \ar[rr]^{\AL_{\Sht}(\mi, \mi')} &&  \Sht^{r}_{G}(\Sii)}
\end{equation*}
where $\AL_{\Sht}(\mi, \mi')$ is the composition of $\AL_{\Sht, x}$ for $x\in\Si$ such that $\mu_{x}\ne\mu'_{x}$.

\sss{The action of $\frA_{r,\Sig}$} We observe that $\frT_{r,\Sig}$ is a torsor under the group $\frA_{r,\Sig}:=(\ZZ/2\ZZ)^{\{1,2,\dotsc, r\}\sqcup\Sig}$.  We denote the action of $a\in \frA_{r,\Sig}$ on $\frT_{r,\Sig}$ by $a\cdot(-)$.

We also have an action of $\frA_{r,\Sig}$ on $\Sht'^{r}_{G}(\Sii)$ defined as follows. The factor of $\ZZ/2\ZZ$ indexed by $1\le i\le r$ acts on the $i$th factor of $X'$ by Galois involution over $X$. For $x\in \Sig$, the nontrivial element in the factor of $\ZZ/2\ZZ$ indexed by $x$ acts by the involution $\AL_{\Sht, x}$ defined in \S\ref{sss:AL Sht} on the $\Sht^{r}_{G}(\Sii)$-factor and identity on $X'^{r}\times\frSi'$. We denote this action by
\begin{equation*}
\frA_{r,\Sig}\ni a\mapsto \AL_{\Sht',a}.
\end{equation*}

The following lemma summarizes the calculations in \S\ref{sss:change unmu}, \S\ref{sss:change mf} and \S\ref{sss:change mi}.
\begin{lemma}\label{l:cycle mu mu'} For any $\mu\in\frT_{r,\Sig}$ and $a\in \frA_{r,\Sig}$, the following diagram is commutative
\begin{equation*}
\xymatrix{   \Sht^{\un\mu}_{T}(\mi\cdot\Si')\ar[d]^{\th'^{\mu}}\ar[rr]^-{\io(\mu,a\cdot\mu)}  & & \Sht^{a\cdot\un\mu}_{T}((a\cdot \mi)\cdot\Si')\ar[d]^{\th'^{a\cdot \mu}}     \\
\Sht'^{r}_{G}(\Sii)\ar[rr]^-{\AL_{\Sht',a}} & &\Sht'^{r}_{G}(\Sii)}
\end{equation*}
Here the upper horizontal arrow is the composition of $\io(\un\mu,\un\mu'), \io(\mf,\mf')$ and $\io(\mi,\mi')$ defined in  \S\ref{sss:change unmu}, \S\ref{sss:change mf} and \S\ref{sss:change mi}. In particular, we have
\begin{equation*}
\cZ^{\mu}=\AL_{\Sht',a}^{*}(\cZ^{a\cdot \mu}),\quad \forall \mu\in\frT_{r,\Sig}, a\in \frA_{r,\Sig}.
\end{equation*}
\end{lemma}

Let $\mu=(\un\mu,\mf,\mi),\mu'=(\un\mu',\mf',\mi')\in  \frT_{r,\Sig}$. Let
\begin{eqnarray}
\notag\D(\mu,\mu')&:=&\{1\le i\le r|\mu_{i}\ne \mu_{i}'\};\\
\label{Sig +}\Sig_{-}(\mu,\mu')&:=&\{x\in \Sig|\mu_{x}\ne\mu'_{x}\}\subset \Sig;\\
\label{Sig -}\Sig_{+}(\mu,\mu')&:=&\{x\in \Sig|\mu_{x}=\mu'_{x}\}=\Sig-\Sig_{-}(\mu,\mu').
\end{eqnarray}

\begin{cor}[of Lemma \ref{l:cycle mu mu'}]\label{c:I dep a} Let $\mu,\mu'\in  \frT_{r,\Sig}$. Then $\II^{\mu,\mu'}$ depends only on the sets $\D(\mu,\mu')$ and $\Sig_{-}(\mu,\mu')$. 
\end{cor}
\begin{proof}  Let $a(\mu,\mu')\in\frA_{r,\Sig}$ be the unique element such that $a(\mu,\mu')\cdot \mu=\mu'$. Then $\D(\mu,\mu')$ and $\Sig_{-}(\mu,\mu')$ determines $a(\mu,\mu')$ and vice versa. Therefore we only need to show that $\II^{\mu,\mu'}$ depends only on $a(\mu,\mu')$. 

Suppose $\mu,\mu'$ and $\wh\mu,\wh\mu'$ satisfy $a(\mu,\mu')=a(\wh\mu,\wh\mu')$, we will show that $\II^{\mu,\mu'}=\II^{\wh\mu,\wh\mu'}$. Since $\frT_{r,\Sig}$ is a torsor under $\frA_{r,\Sig}$, there is a unique $b\in\frA_{r,\Sig}$ such that $\wh\mu=b\cdot \mu$, $\wh\mu'=b\cdot \mu'$.  Since $\AL_{\Sht', b}$ commutes with the action of any $f\in \sH^{\Sig}_{G}$, we have 
\begin{equation*}
\jiao{\cZ^{\wh\mu}, f*\cZ^{\wh\mu'}}=\jiao{\AL^{*}_{\Sht',b}(\cZ^{\wh\mu}), \AL^{*}_{\Sht',b}(f*\cZ^{\wh\mu'})}=\jiao{\AL^{*}_{\Sht',b}(\cZ^{\wh\mu}), f*\AL^{*}_{\Sht',b}(\cZ^{\wh\mu'})}.
\end{equation*}
By Lemma \ref{l:cycle mu mu'}, we have
\begin{equation*}
\AL^{*}_{\Sht',b}(\cZ^{\wh\mu})=\cZ^{\mu},\quad \AL^{*}_{\Sht', b}(\cZ^{\wh\mu'})=\cZ^{\mu'}.
\end{equation*}
Therefore we get
\begin{equation*}
\jiao{\cZ^{\wh\mu}, f*\cZ^{\wh\mu'}}=\jiao{\cZ^{\mu}, f*\cZ^{\mu'}}.
\end{equation*}
i.e., $\II^{\mu,\mu'}(f)=\II^{\wh\mu,\wh\mu'}(f)$ for all $f\in\sH^{\Sig}_{G}$.
\end{proof}

We will see later (in Theorem \ref{th:Ir}) that in fact $\II^{\mu,\mu'}$ only depends on $\Sig_{-}(\mu,\mu')$ and the {\em cardinality} of $\Delta(\mu,\mu')$.

\sss{Heegner--Drinfeld cycles over $\kbar$}
Fix a $\kbar$-point $\xi\in\frSi'(\kbar)$. Concretely this means a collection of field embeddings
\begin{equation*}
\xi=(\xi_{x'})_{x'\in \Si'}, \quad \xi_{x'}: k(x')\incl \kbar.
\end{equation*}
Then $\xi$ also determines a $\kbar$-point of $\frSi$ by the projection $\frSi'\to\frSi$, which we still denote by $\xi$. We denote
\begin{eqnarray*}
\Sht^{r}_{G}(\Sig;\xi)&:=&\Sht^{r}_{G}(\Sii)\times_{\frSi}\xi;\\
\Sht'^{r}_{G}(\Sig;\xi)&:=&\Sht'^{r}_{G}(\Sii)\times_{\frSi'}\xi\cong \Sht^{r}_{G}(\Sig;\xi)\times_{X^{r}}X'^{r};\\
\Sht^{\un\mu}_{T}(\mi\cdot\xi)&:=&\Sht^{\un\mu}_{T}(\mi\cdot\Si')\times_{\frSi'}\xi.
\end{eqnarray*}
Then we have maps
\begin{equation*}
\xymatrix{    &  \Sht^{\un\mu}_{T}(\mi\cdot\xi)\ar[dl]_{\th'^{\mu}_{\xi}}\ar[dr]^{\th^{\mu}_{\xi}}     \\
\Sht'^{r}_{G}(\Sig;\xi)\ar[rr] & & \Sht^{r}_{G}(\Sig;\xi)}
\end{equation*}

\begin{defn} The {\em Heegner--Drinfeld cycle of type $\mu=(\un\mu,\mf,\mi)\in\frT_{r,\Sig}$ over $\xi$} is the class
\begin{equation*}
\cZ^{\mu}(\xi):=\th'^{\mu}_{\xi,*}[\Sht^{\un\mu}_{T}(\mi\cdot\xi)]\in \Ch_{c,r}(\Sht'^{r}_{G}(\Sig;\xi))_{\QQ}.
\end{equation*}
\end{defn}

By definition, the pullback of $\cZ^{\mu}$ to $\Sht'^{r}_{G}(\Sii)\ot\kbar$ is the disjoint union of $\cZ^{\mu}(\xi)$ for various $\xi\in\frSi'(\kbar)$.

\begin{cor}[of Lemma \ref{l:cycle mu mu'}]\label{c:change xi} For $\mu=(\un\mu,\mf,\mi)\in\frT_{r,\Sig}$ and $a\in\frA_{r,\Sig}$, we have
\begin{equation*}
\cZ^{\mu}(\xi)=\AL^{*}_{\Sht', a}(\cZ^{a\cdot\mu}(\xi)).
\end{equation*}
\end{cor}

\begin{lemma} For any $\xi\in \frSi'(\kbar)$, any $\mu,\mu'\in\frT_{r,\Sig}$ and any $f\in \sH^{\Sig}_{G}$, we  have an identity
\begin{equation}\label{I and xi}
\II^{\mu,\mu'}(f)=\jiao{\cZ^{\mu}(\xi),f*\cZ^{\mu'}(\xi)}_{\Sht'^{r}_{G}(\Sig;\xi)}.
\end{equation}
In particular, by Corollary \ref{c:I dep a}, the right side depends only on the sets $\D(\mu,\mu')$ and $\Sig_{-}(\mu,\mu')$. 
\end{lemma}
\begin{proof}
Since $\Sht'^{r}_{G}(\Sii)\ot\kbar$ is the disjoint union of $\Sht'^{r}_{G}(\Sig;\xi)$ for $\prod_{x'\in\Si'}d_{x'}$ different choices of $\xi$, it suffices to show that the right side of \eqref{I and xi} is independent of the choice of $\xi$. To compare a general $\xi'$ to $\xi$, we may reduce to the case where $\xi'\in\frSi'(\kbar)$ is obtained by changing $\xi_{x'}$ to $\Fr(\xi_{x'})$ for a unique $x'\in \Si'$, and keeping the other coordinates. 

Consider the isomorphism
\begin{equation*}
\jmath_{x'}: \Sht^{\un\mu}_{T}(\mi\cdot \Si')\isom \Sht^{\un\mu}_{T}(\mi\cdot \Si')
\end{equation*}
sending $(\cL_{i}; x'_{i}; x'^{(1)}, \{y'^{(1)}\}_{y'\in\Si', y'\ne x'})$ to $(\cL_{i}(-\mu_{x}x'^{(1)}); x'_{i}; x'^{(2)}, \{y'^{(1)}\}_{y'\in\Si', y'\ne x'})$. Direct calculation shows that the following diagram is commutative
\begin{equation}\label{jShtTth}
\xymatrix{   \Sht^{\un\mu}_{T}(\mi\cdot \Si')\ar[d]^{\th'^{\mu}}\ar[rr]^-{\jmath_{x'}} & & \Sht^{\un\mu}_{T}(\mi\cdot \Si')   \ar[d]^{\th'^{\mu}}     \\
\Sht'^{r}_{G}(\Sii) \ar[rr]^-{\AL^{(1)}_{x'}} & & \Sht'^{r}_{G}(\Sii)
}
\end{equation}
where $\AL^{(1)}_{x'}$ sends $(\cE^{\da}_{i}; x'_{i}; x'^{(1)}, \{y'^{(1)}\}_{y'\in \Si', y'\ne x'})$ to $(\cE^{\da}_{i}(-\ha x^{(1)}); x'_{i}; x'^{(2)}, \{y'^{(1)}\}_{y'\in \Si', y'\ne x'})$ (here $x^{(1)}$ is the image of $x'^{(1)}$).  The diagram \eqref{jShtTth} implies that 
\begin{equation*}
(\AL^{(1)}_{x'})^{*}\cZ^{\mu}(\xi')=\cZ^{\mu}(\xi).
\end{equation*}
Therefore, using that $\AL^{(1)}_{x'}$ commutes with the $\sH^{\Sig}_{G}$-action, we have
\begin{eqnarray*}
\jiao{\cZ^{\mu}(\xi),f*\cZ^{\mu'}(\xi)}_{\Sht'^{r}_{G}(\Sig;\xi)}&=&\jiao{(\AL^{(1)}_{x'})^{*}(\cZ^{\mu}(\xi')),f*(\AL^{(1)}_{x'})^{*}\cZ^{\mu'}(\xi')}_{\Sht'^{r}_{G}(\Sig;\xi)}\\
&=&\jiao{(\AL^{(1)}_{x'})^{*}(\cZ^{\mu}(\xi')),(\AL^{(1)}_{x'})^{*}(f*\cZ^{\mu'}(\xi'))}_{\Sht'^{r}_{G}(\Sig;\xi)}\\
&=&\jiao{\cZ^{\mu}(\xi'),f*\cZ^{\mu'}(\xi')}_{\Sht'^{r}_{G}(\Sig;\xi')}.
\end{eqnarray*} 
\end{proof}

\section{The moduli stack $\cM_{d}$ and intersection numbers}\label{s:M}

The goal of this subsection is to give a Lefschetz-type formula for the intersection number $\II^{\mu,\mu'}_{r}(h_{D})$, see Theorem \ref{th:Ir}. This is parallel to \cite[\S6]{YZ} in the unramified case.

Recall that $\Sig'$ and $R'$ are the preimages of $\Sig$ and $R$ under $\nu$. We introduce the notation
\begin{eqnarray*}
U&=&X-\Sig-R;\\
U'&=&X'-\Sig'-R'.
\end{eqnarray*}
Our construction below will rely on variants of the Picard stack with an extra choice of a square root along the divisor $R$, which naturally appears in the geometric class field theory of $X$ with ramification along $R$. We refer to our Appendix \ref{A:Pic} for the definitions and properties of such variants of the Picard stack.

\subsection{Definition of $\cM_{d}$ and statement of the formula}
Let $d$ be an integer. We shall define an analog of the moduli stacks $\cM_{d}$ and $\cA_{d}$ in \cite[\S6.1]{YZ}, for the possibly ramified double cover $\nu:X'\to X$.

\sss{The stack $\cM_{d}$}\label{sss:dot O}  For any divisor $D$ of $X$ disjoint from $R$, $\cO_{X}(D)$ has a canonical lift  $\cO_{X}(D)^{\na}=(\cO_{X}(D), \cO_{R}, 1)\in \Pic^{\sqR}_{X}(k)$, and a canonical lift $\dot\cO_{X}(D)=(\cO_{X}(D), \cO_{R},1, 1)\in \Pic^{\sqR;\sqR}_{X}(k)$.

Suppose we are given a decomposition
\begin{equation*}
\Sig=\Sig_{+}\sqcup\Sig_{-}.
\end{equation*}
Let 
\begin{eqnarray*}
\r=\deg R=\deg R';\quad N=\deg \Sig; \quad N_{\pm}=\deg \Sig_{\pm}.
\end{eqnarray*}

\begin{defn}\label{defn Md} Let $\cM_{d}=\cM_{d}(\Sig_{\pm})$ be the moduli stack whose $S$-points consist of tuples $(\cI,\cJ, \a,\b, \j)$ where
\begin{itemize}
\item $\cI$ is a line bundle on $X'\times S$ with fiber-wise degree $d+\r-N_{-}$, and $\a$ is a section of $\cI$.
\item $\cJ$ is a line bundle on $X'\times S$ with fiber-wise degree $d+\r-N_{+}$, and $\b$ is a section of $\cJ$.
\item $\j$ is an isomorphism $\Nm^{\sqR}_{X'/X}(\cI)\ot\cO_{X}(\Sig_{-})^{\na}\isom \Nm^{\sqR}_{X'/X}(\cJ)\ot\cO_{X}(\Sig_{+})^{\na}$, as $S$-points of $\Pic^{\sqR,d+\r}_{X}$. Concretely, $\j$ is a collection of isomorphisms
\begin{eqnarray}\label{j collection}
\j_{\Nm}: \Nm_{X'/X}(\cI)\ot\cO_{X}(\Sig_{-})\isom \Nm_{X'/X}(\cJ)\ot\cO_{X}(\Sig_{+}),\\
\notag\j_{x}: \cI|_{x'\times S}\isom \cJ|_{x'\times S}, \quad \forall x\in R
\end{eqnarray}
such that the following diagram is commutative for all $x\in R$
\begin{equation}\label{ep x comm}
\xymatrix{\cI^{\ot2}|_{x'\times S}\ar[d]_{\wr}\ar[rr]^-{\j^{\ot 2}_{x}}_{\sim} && \cJ^{\ot2}|_{x'\times S}\ar[d]_{\wr}\\
\Nm_{X'/X}(\cI)|_{x\times S}\ar[rr]^-{\j_{\Nm}|_{x\times S}}_{\sim} && \Nm_{X'/X}(\cJ)|_{x\times S}}
\end{equation}
Here the vertical maps are the tautological isomorphisms.
\end{itemize}
These data are required to satisfy the following conditions
\begin{enumerate} 
\item $\a|_{\nu^{-1}(\Sig_{+})\times S}$ is nowhere vanishing.
\item $\b|_{\nu^{-1}(\Sig_{-})\times S}$ is nowhere vanishing. 
\item  For each $x\in R$, we have
\begin{equation*}
\j_{x}(\a|_{x'\times S})=\b|_{x'\times S}.
\end{equation*}
Moreover, $\Nm(\a)-\Nm(\b)$ vanishes only to the first order along $R\times S$.
\item This condition is non-void only when $\Sig=\vn$ and $R=\vn$: for each geometric point $s\in S$, the restriction $(\Nm(\a)-\Nm(\b))|_{X\times s}$ is not identically zero. 
\end{enumerate}
\end{defn}

From the definition we have an open embedding
\begin{equation}\label{Md emb}
\io_{d}: \cM_{d}\incl \hX'_{d+\r-N_{-}}\times_{\Pic^{\sqR;\sqR, d+\r}_{X}}\hX'_{d+\r-N_{+}}
\end{equation}
where the fiber product is taken over
\begin{equation*}
\nu_{\a}: \hX'_{d+\r-N_{-}}\xr{\wh\nu^{\sqR}}\hX^{\sqR}_{d+\r-N_{-}}\xr{\wh\AJ^{\sqR; \sqR}_{d+\r-N_{-}}} \Pic^{\sqR;\sqR, d+\r-N_{-}}_{X}\xr{\ot\dot\cO_{X}(\Sig_{-})}\Pic^{\sqR;\sqR,d+\r}_{X}
\end{equation*}
and
\begin{equation*}
\nu_{\b}: \hX'_{d+\r-N_{+}}\xr{\wh\nu^{\sqR}}\hX^{\sqR}_{d+\r-N_{+}}\xr{\wh\AJ^{\sqR; \sqR}_{d+\r-N_{+}}}\Pic^{\sqR;\sqR, d+\r-N_{+}}_{X}\xr{\ot\dot\cO_{X}(\Sig_{+})}\Pic^{\sqR;\sqR,d+\r}_{X}
\end{equation*}
Here the Abel-Jacobi maps $\wh\AJ^{\sqR; \sqR}_{d+\r-N_{\pm}}$ are defined in \S\ref{sss:AJ}.

\begin{remark} When $\Sig=\vn$ and $R=\vn$, there is a slight difference between the current definition of $\cM_{d}$ and the one in \cite{YZ}. In \cite{YZ}, we only require that $\a|_{X'\times s}$ and $\b|_{X'\times s}$ are not both zero for any geometric point  $s\in S$; here we impose a stronger open condition that $\Nm(\a)-\Nm(\b)$ is nonzero on $X\times s$ for any geometric point $s\in S$. Therefore the current version of $\cM_{d}$ is the one denoted by $\cM_{d}^{\hs}$ in \cite{YZ}. A similar remark applies to the space $\cA_{d}$ to be defined below.
\end{remark}

\sss{The base $\cA_{d}$}\label{sss:base}

\begin{defn}\label{defn Ad} Let $\cA_{d}=\cA_{d}(\Sig_{\pm})$ be the moduli stack whose $S$-points consist of tuples $$(\D, \Th_{R}, \iota, a, b, \vth_{R})$$ where
\begin{itemize}
\item $(\D, \Th_{R}, \iota)\in \Pic^{\sqR, d+\r}_{X}(S)$. Namely, $\D$ is a line bundle on $X\times S$ of fiber-wise degree $d+\r$, $\Th_{R}$ a line bundle over $R\times S$ and $\iota$ an isomorphism $\Th_{R}^{\otimes 2}\cong \D|_{R\times S}$.
\item $a$ and $b$ are sections of $\D$.
\item $\vth_{R}$ is a section of $\Th_{R}$.
\end{itemize}
These data are required to satisfy the following conditions.
\begin{enumerate}
\item\label{a nonvan} $a|_{\Sig_{-}\times S}=0$, and $a|_{\Sig_{+}\times S}$ is nowhere vanishing.
\item\label{b nonvan} $b|_{\Sig_{+}\times S}=0$, and $b|_{\Sig_{-} \times S}$ is nowhere vanishing.
\item\label{R van} $a|_{R\times S}=\iota(\vth^{\ot 2}_{R})=b|_{R\times S}$. Moreover, $a-b$ vanishes only to the first order along $R\times S$.
\item\label{a-b nonvan} This condition is only non-void when $\Sig=\vn$ and $R=\vn$: for every geometric point $s$ of  $S$, $(a-b)|_{X\times s}\ne0$.
\end{enumerate}
\end{defn}

The assignment $(\D, \Th_{R}, \iota, a, b, \vth_{R})\mapsto (\D(-\Sig_{-}),\Th_{R}, \iota, a,\vth_{R})$ gives a map
\begin{equation*}
\cA_{d}\to \hX^{\sqR}_{d+\r-N_{-}}.
\end{equation*}
Similarly, the assignment $(\D, \Th_{R}, \iota, a, b, \vth_{R})\mapsto (\D(-\Sig_{+}),\Th_{R}, \iota, b,\vth_{R})$ gives a map
\begin{equation*}
\cA_{d}\to \hX^{\sqR}_{d+\r-N_{+}}.
\end{equation*}

Combining these maps, we get  an open embedding 
\begin{equation}\label{emb Ad}
\om_{d}: \cA_{d}\incl \hX^{\sqR}_{d+\r-N_{-}}\times_{\Pic^{\sqR;\sqR, d+\r}_{X}}\hX^{\sqR}_{d+\r-N_{+}}
\end{equation}
where the the fiber product is formed using the Abel-Jacobi maps 
\begin{eqnarray*}
\nu_{a}: \hX^{\sqR}_{d+\r-N_{-}}\xr{\wh\AJ^{\sqR;\sqR}_{d+\r-N_{-}}}\Pic^{\sqR;\sqR, d+\r-N_{-}}_{X}
\xr{\ot\dot\cO_{X}(\Sig_{-})}\Pic^{\sqR;\sqR, d+\r}_{X}.\\
\nu_{b}: \hX^{\sqR}_{d+\r-N_{+}}\xr{\wh\AJ^{\sqR;\sqR}_{d+\r-N_{+}}}\Pic^{\sqR;\sqR, d+\r-N_{+}}_{X}
\xr{\ot\dot\cO_{X}(\Sig_{+})}\Pic^{\sqR;\sqR, d+\r}_{X}.
\end{eqnarray*}

\sss{The base $\cA^{\flat}_{d}$} 
Later we will need to use another base space $\cA^{\flat}_{d}$. 
\begin{defn}\label{defn Ad flat} Let $\cA^{\flat}_{d}=\cA^{\fl}_{d}(\Sig_{\pm})$ be the moduli stack whose $S$-points consist of tuples $(\D,  a, b)$ where
\begin{itemize}
\item $\D$ is a line bundle  on $X\times S$ of fiber-wise degree $d+\r$,
\item $a$ and $b$ are sections of $\D$,
\end{itemize}
such that the same conditions \eqref{a nonvan}-\eqref{a-b nonvan} hold as in Definition \ref{defn Ad}.
\end{defn}

Similar to the case of $\cA_{d}$, we have an open embedding
\begin{equation}\label{emb Ad fl}
\om_{d}^{\fl}: \cA^{\flat}_{d}\incl \hX_{d+\r-N_{-}}\times_{\Pic^{d+\r}_{X}}\hX_{d+\r-N_{+}}
\end{equation}
By \cite[\S3.2.3]{YZ}, $\cA^{\flat}_{d}$ is a scheme over $k$.  Later it will be technically more convenient to apply the Lefschetz trace formula to the base scheme $\cA^{\flat}_{d}$ instead of the stack $\cA_{d}$.

There is a forgetful map 
\begin{equation*}
\Om: \cA_{d}\to\cA^{\fl}_{d}
\end{equation*}
which corresponds to the forgetful maps $\hX^{\sqR}_{d+\r-N_{\pm}}\to \hX_{d+\r-N_{\pm}}$ under the embeddings \eqref{emb Ad} and \eqref{emb Ad fl}.

We  have a morphism
\begin{equation*}
\d: \cA^{\fl}_{d}\to U_{d}
\end{equation*}
sending $(\D, a,b)$ to the divisor of $a-b$ as a nonzero section of $\D(-R)$, the latter having degree $d$. The conditions \eqref{a nonvan}, \eqref{b nonvan} and \eqref{R van} in Definition \ref{defn Ad} imply that the divisor of $a-b$ does not meet $\Sig$ or $R$.

For $D$ be an effective divisor on $U$ of degree $d$, let
\begin{equation}\label{def AD}
\cA^{\fl}_{D}=\d^{-1}(D)\subset \cA^{\fl}_{d}.
\end{equation}

\sss{Geometric properties of $\cM_{d}$} We have a morphism
\begin{equation*}
f_{d}: \cM_{d}\to \cA_{d}
\end{equation*}
defined by applying $\wh\nu^{\sqR}$ to both $\hX'_{d+\r-N_{-}}$ and $\hX'_{d+\r-N_{+}}$. In other words, we have a commutative diagram
\begin{equation}\label{MAXX}
\xymatrix{ \cM_{d}\ar@{^{(}->}[rr]\ar[d]^{f_{d}} && \hX'_{d+\r-N_{-}}\times_{\nu_{\a},\Pic^{\sqR;\sqR, d+\r}_{X},\nu_{\b}}\hX'_{d+\r-N_{+}}\ar[d]^{\wh\nu^{\sqR}_{d+\r-N-}\times\wh\nu^{\sqR}_{d+\r-N_{+}}}\\
\cA_{d}\ar@{^{(}->}[rr] &&
\hX^{\sqR}_{d+\r-N_{-}}\times_{\nu_{a},\Pic^{\sqR;\sqR, d+\r}_{X}, \nu_{b}}\hX^{\sqR}_{d+\r-N_{+}}
}
\end{equation}
We denote by $f^{\fl}_{d}$ the composition
\begin{equation*}
f^{\fl}_{d}: \cM_{d}\xr{f_{d}} \cA_{d}\xr{\Om}\cA^{\fl}_{d}.
\end{equation*}

The following is a generalization of \cite[Prop 6.1]{YZ} to the ramified situation.

\begin{prop}\label{p:M}
\begin{enumerate}
\item\label{M smooth} When $d\ge 2g'-1+N=4g-3+\r+N$, the stack $\cM_{d}$ is a smooth DM stack pure of dimension $m=2d+\r-N-g+1$.
\item\label{MA Cart} The diagram \eqref{MAXX} is Cartesian.
\item\label{f proper} The morphisms $f_{d}$ and $f^{\fl}_{d}$ are proper.
\item\label{f small} When $d\ge 3g-2+N$, the morphism $f_{d}$ is small: it is generically finite and for any $n>0$, $\{a\in\cA_{d}|\dim f_{d}^{-1}(a)\ge n\}$ has codimension $\ge 2n+1$ in $\cA_{d}$.
\item\label{M ffp} The stack $\cM_{d}$ admits a finite flat presentation in the sense of \cite[Definition A.1]{YZ}.
\end{enumerate}
\end{prop}
\begin{proof}
\eqref{M smooth} To show that $\cM_{d}$ is smooth DM, it suffices to show that both of following stacks
\begin{eqnarray}\label{hXX'}
\hX'_{d+\r-N_{-}}\times_{\nu_{\a},\Pic^{\sqR;\sqR, d+\r}_{X},\nu_{\b}} X'_{d+\r-N_{+}}\\
\label{XhX'}
X'_{d+\r-N_{-}}\times_{\nu_{\a},\Pic^{\sqR;\sqR, d+\r}_{X},\nu_{\b}} \hX'_{d+\r-N_{+}}
\end{eqnarray}  
are smooth DM.

Let $Q^{R'}_{X'}$ be the moduli stack of pairs $(\cL',\vth_{R'})$ where $\cL'\in \Pic_{X'}$ and $\vth_{R'}$ is a section of $\cL'|_{R'}$. Then $Q^{R'}_{X'}\cong \Pic_{X'}\times_{\Pic_{X}^{\sqR}}\Pic_{X}^{\sqR;\sqR}$. In particular, the norm map $Q^{R'}_{X'}\to \Pic_{X}^{\sqR;\sqR}$ is smooth and relative DM.

For any geometric point $s$ and line bundle $\cL$ on $X'\times s$ of degree $n\ge 2g'+\r-1$, the restriction map $\cohog{0}{X'\times s,\cL}\to \cohog{0}{R'\times s, \cL|_{R'\times s}}$ is surjective with kernel dimension $n-g'+1-\r$. This implies $\hX'_{n}\to Q^{R'}_{X'}$ is a vector bundle of rank $n-g'+1-\r$, whenever $n\ge 2g'-1+\r$, in which case $\hX'_{n}$ itself is also smooth.

If $d\ge 2g'-1+N\ge 2g'-1+N_{+}$, then $d+\r-N_{+}\ge 2g'-1+\r$, the map $\nu_{\b}: \hX'_{d+\r-N_{+}}\to Q^{R'}_{X'}\to \Pic_{X}^{\sqR,\sqR}$ is then smooth and relative DM by the above discussion, therefore the fiber product \eqref{XhX'} is smooth over its first factor $X'_{d+\r-N_{-}}$. Since $X'_{d+\r-N_{-}}$ is a scheme smooth over $k$, the fiber product \eqref{XhX'} is smooth DM over $k$. The argument for \eqref{hXX'} is the same.

For the dimension, we have
\begin{eqnarray*}
\dim \cM_{d}&=&\dim \hX'_{d+\r-N_{-}}+\dim \hX'_{d+\r-N_{+}}-\dim \Pic_{X}^{\sqR,\sqR}\\
&=&(d+\r-N_{-})+(d+\r-N_{+})-(g-1+\r)\\
&=&2d+\r-N-g+1.
\end{eqnarray*}

\eqref{MA Cart} follows directly by comparing the four conditions in  Definition \ref{defn Md} and in Definition \ref{defn Ad}.

\eqref{f proper} Since $\Om$ is proper, it suffices to show that $f_{d}$ is proper. By \eqref{MA Cart}, it suffices to show that $\wh\nu^{\sqR}_{n}: \hX'_{n}\to \hX^{\sqR}_{n}$ is proper for any $n\ge0$. We consider the factorization of the usual norm map
\begin{equation*}
\wh\nu_{n}:\hX'_{n}\xrightarrow{\wh \nu^{\sqR}_{n}}  \hX^{\sqR}_{n} \xrightarrow{\wh\om^{\sqR}_{n}} \hX_{n}.
\end{equation*}
The same argument of \cite[Prop. 6.1(4)]{YZ} shows that $\wh\nu_{n}$ is proper. On the other hand, $\wh\om^{\sqR}_{n}$ is separated because it is obtained by base change from the separated map $[2]:[\Res^{R}_{k}\AA^1/\Res^{R}_{k}\Gm]\to[\Res^{R}_{k}\AA^1/\Res^{R}_{k}\Gm]$ (see the diagram \eqref{XR Cart}). Therefore, $\wh\nu^{\sqR}_{n}$ is proper. 

\eqref{f small} Over $\cA^{\dm}_{d}:=(X^{\sqR}_{d+\r-N_{-}}\times_{\Pic_{X}^{\sqR;\sqR,d+\r}}X^{\sqR}_{d+\r-N_{+}})\cap \cA_{d}$, $f_{d}$ is finite. The complement $\cA_{d}-\cA^{\dm}_{d}$ is the disjoint union of $\cA_{d}^{a=0}$ and $\cA^{b=0}_{d}$ corresponding to the locus $a=0$ or $b=0$. Note $\cA_{d}^{a=0}=\vn$ unless $\Sig_{+}=\vn$; $\cA_{d}^{b=0}=\vn$ unless $\Sig_{-}=\vn$.

We first analyze the fibers over $\cA_{d}^{b=0}$ when $\Sig_{-}=\vn$. The coarse moduli space of $\cA^{b=0}_{d}$ is $U_{d}$ (by taking $\div(a)-R$, note that $\Sig=\Sig_{+}$). Hence $\dim \cA^{b=0}_{d}=d$, and $\codim_{\cA_{d}}(\cA^{b=0}_{d})=d-g+1+\r-N$. The restriction of $f_{d}$ to $\cA^{b=0}_{d}$ is, up to passing to coarse moduli spaces, given by the norm map with respect to the double cover $U'\to U$
\begin{equation*}
U'_{d}\times_{\Pic_{X}^{\sqR,d+\r}}\Pic_{X'}^{d+\r-N_{+}}\to U_{d}.
\end{equation*}
From this we see that the fiber dimension of $f_{d}$ over $\cA^{b=0}_{d}$ is the same as that of the norm map $\Pic_{X'}\to \Pic^{\sqR}_{X}$, which is $g'-g$. 

Similar argument shows that when $\Sig_{+}=\vn$, $\codim_{\cA_{d}}(\cA^{a=0}_{d})=d-g+1+\r-N$ and the fiber dimension of $f_{d}$ over $\cA^{a=0}_{d}$ is still $g'-g$. In either case, since $d\ge3g-2+N$, we have
\begin{equation*}
 d-g+1+\r-N\ge 2g-1+\r=2(g'-g)+1
\end{equation*}
which checks the smallness of $f_{d}$.

\eqref{M ffp} We need to show that there is a finite flat map $Y\to \cM_{d}$ from an algebraic space $Y$ of finite type over $k$. As in \cite[proof of Prop. 6.1(1)]{YZ}, by introducing a rigidification at some closed point $y\in U'$, we may define a schematic map
\begin{equation*}
\cM_{d}\to J_{X'}^{d+\r}\times \Prym_{X'/X}
\end{equation*}
where $J_{X'}^{d+\r}$ is the Picard scheme of $X'$ of degree $d+\r$, and $\Prym_{X'/X}:=\ker(\Nm^{\sqR}_{X'/X}: \Pic^{0}_{X'}\to \Pic^{\sqR,0}_{X})$. Since $J_{X'}^{d+\r}$ is a scheme and $\Prym_{X'/X}$ is a global finite quotient of an abelian variety, $J_{X'}^{d+\r}\times \Prym_{X'/X}$ admits a finite flat presentation, therefore the same is true for $\cM_{d}$.
\end{proof}

\sss{The incidence correspondences}\label{sss:Hk Md} To state the formula for $\II^{\mu,\mu'}(h_{D})$, we need to introduce two self-correspondences of $\cM_{d}$. We define $\cH_{+}$ to be the substack of $\cM_{d}\times X'$ consisting of those $(\cI,\cJ, \a,\b,\j, x')$ such that $\b$ vanishes on $\Gamma_{x'}$. We have the natural projection
\begin{equation*}
\oll{\g}_{+}:\cH_{+}\to \cM_{d}
\end{equation*}
recording $(\cI,\cJ, \a,\b,\j)$. We also have another projection
\begin{equation*}
\orr{\g}_{+}: \cH_{+}\to \cM_{d}
\end{equation*}
sending $(\cI,\cJ, \a,\b,\j, x')$ to $(\cI, \cJ(\Gamma_{\s x'}-\Gamma_{x'}), \a,\b,\j)$. This makes sense since twisting by $\cO_{X'}(\Gamma_{\s x'}-\Gamma_{x'})$ does not affect the image under $\Nm^{\sqR}_{X'/X}$, and that $\b$ can be viewed as a section of $\cJ(\Gamma_{\s x'}-\Gamma_{x'})$ since it vanishes along $\Gamma_{x'}$. Via $(\oll{\g}_{+}, \orr{\g}_{+})$,  we view $\cH_{+}$ as a self-correspondence of $\cM_{d}$. We have a commutative diagram
\begin{equation}\label{corr cH}
\xymatrix{& \cH_{+} \ar[dl]_{\oll{\g}_{+}}\ar[dr]^{\orr{\g}_{+}}\\
\cM_{d}\ar[dr]_{f_{d}} & & \cM_{d}\ar[dl]^{f_{d}}\\
& \cA_{d} }
\end{equation}

Similarly, we define $\cH_{-}$ to be the substack of $\cM_{d}\times X'$ consisting of those $(\cI,\cJ, \a,\b,\j, x')$ such that $\a$ vanishes on $\Gamma_{x'}$. We view $\cH_{-}$ as a self-correspondence of $\cM_{d}$ over $\cA_{d}$
\begin{equation}\label{corr cH-}
\xymatrix{& \cH_{-} \ar[dl]_{\oll{\g}_{-}}\ar[dr]^{\orr{\g}_{-}}\\
\cM_{d}\ar[dr]_{f_{d}} & & \cM_{d}\ar[dl]^{f_{d}}\\
& \cA_{d} }
\end{equation}
where $\oll{\g}_{-}(\cI,\cJ, \a,\b,\j, x')=(\cI,\cJ, \a,\b,\j)$ and $\orr{\g}_{-}(\cI,\cJ, \a,\b,\j, x')=(\cI(\Gamma_{\s x'}-\Gamma_{x'}),\cJ, \a,\b,\j)$.

Let $\cA^{\dm}_{d}=(X^{\sqR}_{d+\r-N_{-}}\times_{\Pic^{\sqR;\sqR}_{X}}X^{\sqR}_{d+\r-N_{+}})\cap \cA_{d}$ be the locus where $a,b\ne0$ \footnote{The definition of $\cA^{\dm}_{d}$ is different from the one in \cite[]{YZ}.}. Let $\cM_{d}^{\dm}\subset \cM_{d}$ be the preimage of $\cA^{\dm}_{d}$. Let $\cH^{\dm}_{+}$ and $\cH^{\dm}_{-}$ be the restriction of $\cH_{+}$ and $\cH_{-}$ to $\cA^{\dm}_{d}$. 

Consider the incidence correspondence
\begin{equation}\label{corr I}
\xymatrix{& I'_{d+\r-N_{+}} \ar[dl]_{\oll{i}}\ar[dr]^{\orr{i}}\\
X'_{d+\r-N_{+}} && X'_{d+\r-N_{+}}}
\end{equation}
Here $I'_{d+\r-N_{+}}=\{(D,x')\in X'_{d+\r-N_{+}}\times X'|x'\in D\}$, $\oll{i}(D, x')=D$ and $\orr{i}(D,x')=D+\s (x')-x'$. 

By definition, over $\cM^{\dm}_{d}$, $\cH^{\dm}_{+}$ is obtained from the incidence correspondence $I'_{d+\r-N_{+}}$ by applying $X'_{d+\r-N_{-}}\times_{\Pic^{\sqR;\sqR}_{X}}(-)$ and then restricting to $\cA^{\dm}_{d}$. Similarly, $\cH^{\dm}_{-}$ is obtained from the incidence correspondence $I'_{d+\r-N_{-}}$  by applying $(-)\times_{\Pic^{\sqR;\sqR}_{X}}X'_{d+\r-N_{+}}$ and then restricting to $\cA^{\dm}_{d}$ (c.f.  \cite[Lemma 6.3]{YZ}).

From this description, we see that $\dim \cH^{\dm}_{\pm}=\dim \cM^{\dm}_{d}=2d+\r-N-g+1$. Let $\ov\cH^{\dm}_{\pm}$ be the closure of $\cH^{\dm}_{\pm}$ and let $[\ov\cH^{\dm}_{\pm}]$ denote its cycle class as an element in $\hBM{2(2d+\r-N-g+1)}{\cH_{\pm}}$. Then $[\ov\cH^{\dm}_{\pm}]$ is a cohomological correspondence between the constant sheaf on $\cM_{d}$ and itself, which then induces an endomorphism of $\bR f_{d,!}\Ql$
\begin{equation*}
f_{d,!}[\ov\cH^{\dm}_{\pm}]: \bR f_{d,!}\Ql\to \bR f_{d,!}\Ql.
\end{equation*}
Taking direct image under $\Om:\cA_{d}\to\cA^{\fl}_{d}$, we get an endomorphism
\begin{equation*}
f^{\fl}_{d,!}[\ov\cH^{\dm}_{\pm}]: \bR f^{\fl}_{d,!}\Ql\to \bR f^{\fl}_{d,!}\Ql.
\end{equation*}
For $a\in \cA^{\fl}_{d}(k)$, let $(f^{\fl}_{d,!}[\ov\cH^{\dm}_{\pm}])_{a}$ be the action of 
$f^{\fl}_{d,!}[\ov\cH^{\dm}_{\pm}]$ on the geometric stalk $(\bR f_{d,!}\Ql)_{a}$.

\sss{The formula} For the rest of the section, we fix a pair
\begin{equation*}
\mu=(\un\mu, \mf,\mi), \mu'=(\un\mu', \mf',\mi')\in \frT_{r,\Sig}.
\end{equation*}
We set
\begin{equation*}
\Sig_{+}:=\Sig_{+}(\mu,\mu'), \quad \Sig_{-}:=\Sig_{-}(\mu,\mu')
\end{equation*}
be defined as in \eqref{Sig +} and \eqref{Sig -}. Thus $\cM_{d}=\cM_{d}(\Sig_{\pm})$ is defined. We also let
\begin{equation}\label{defn r pm}
r_{+}=\{1\le i\le r|\mu_{i}=\mu_{i}'\}; \quad r_{-}=\{1\le i\le r|\mu_{i}\ne\mu_{i}'\}.
\end{equation}
The following is the main theorem of this section, parallel to \cite[Theorem 6.5]{YZ}.

\begin{theorem}\label{th:Ir} Suppose $D$ is an effective divisor on $U$ of degree $d\ge \max\{2g'-1+N, 2g\}$. Under the above notation, we have
\begin{equation} \label{Ir hD}
\II^{\mu,\mu'}(h_{D})=\sum_{a\in \cA^{\fl}_{D}(k)}\Tr\left((f^{\fl}_{d,!}[\ov\cH^{\dm}_{+}])^{r_{+}}_{a}\circ (f^{\fl}_{d,!}[\ov\cH^{\dm}_{-}])^{r_{-}}_{a}\circ \Fr_{a}, (\bR f^{\fl}_{d,!}\Ql)_{a}\right)
\end{equation}
where $\Fr_{a}$ is the geometric Frobenius at $a$.
\end{theorem}

\sss{Outline of the proof} The rest of the section is devoted to the proof of Theorem \ref{th:Ir}. The proof consists of three steps
\begin{enumerate}
\item[I.]  Introduce a moduli stack $\cM_{d}(\mu_{\Sig},\mu'_{\Sig})$ and a Hecke correspondence $\Hk^{\mu,\mu'}_{\cM,d}$ for $\cM_{d}(\mu_{\Sig},\mu'_{\Sig})$. 

This step is done in \S\ref{ss:pre M}. We also introduce certain auxiliary spaces which form the ``master diagram'' \eqref{Tian}. Later we will apply the octahedron lemma \cite[Theorem A.10]{YZ} to this diagram. 

\item[II.] Relate $\cM_{d}(\mu_{\Sig},\mu'_{\Sig})$ and $\cM_{d}$; relate $\Hk^{\mu,\mu'}_{\cM,d}$ and a composition of $\cH_{\pm}$. 

This is done in \S\ref{ss:compare M}. This step is significantly more complicated than the unramified case treated in \cite{YZ}. It amounts to showing that $\cM_{d}$ is a descent of  $\cM_{d}(\mu_{\Sig},\mu'_{\Sig})$ from $\frSi'$ to $\Spec k$. 

\item[III.] Show that $\II^{\mu,\mu'}(h_{D})$ can be expressed as the intersection number of a cycle class supported on $\Hk^{\mu,\mu'}_{\cM,d}$ and the graph of Frobenius of $\cM_{d}(\mu_{\Sig},\mu'_{\Sig})$, and rewrite this intersection number into a trace as in the right hand side of \eqref{Ir hD}. 

This step is done in \S\ref{ss:proof Ir}. The argument is quite similar to the proof of \cite[Theorem 6.6]{YZ}, together with a standard application of a version of the Lefschetz trace formula reviewed in \cite[Prop A.12]{YZ}. 
\end{enumerate}


\subsection{Auxiliary moduli stacks}\label{ss:pre M}

\sss{The stack $H_{d}(\Sig)$}\label{sss:Hd}

\begin{defn}
\begin{enumerate}
\item Let $\wt H_{d}(\Sig)$ be the moduli stack whose $S$-points consist of triples $(\cE^{\da}, \cE'^{\da}, \ph)$ where
\begin{itemize}
\item $\cE^{\da}=(\cE; \{\cE(-\ha x)\})$ and $\cE'^{\da}=(\cE'; \{\cE'(-\ha x)\})$ are $S$-points of $\Bun_{2}(\Sig)$ such that $\deg(\cE'|_{X\times s})-\deg(\cE|_{X\times s})=d$ for all geometric points $s\in S$.
\item $\ph: \cE\to \cE'$ is a map of coherent sheaves which is injective when restricted to $X\times s$ for all geometric points $s\in S$, and mapping $\cE(-\ha x)$ to $\cE'(-\ha x)$ for all $x\in \Sig$.
\item The restriction $\ph|_{(\Sig\sqcup R)\times S}$ is an isomorphism. 
\end{itemize}
\item We define
\begin{equation*}
H_{d}(\Sig)=\wt H_{d}(\Sig)/\Pic_{X}
\end{equation*}
where $\Pic_{X}$ acts by tensoring on $\cE^{\da}$ and $\cE'^{\da}$ simultaneously. 
\end{enumerate}
\end{defn}

We have a map
\begin{equation*}
\olr{p_{H}}=(\oll{p_{H}},\orr{p_{H}}): H_{d}(\Sig)\to\Bun_{G}(\Sig)^{2}
\end{equation*}
recording $\cE^{\da}$ and $\cE'^{\da}$. We also have a map
\begin{equation}\label{Hd Ud}
s: H_{d}(\Sig)\to U_{d}
\end{equation}
recording the vanishing divisor of $\det(\ph)$ as a section of $\det(\cE)^{-1}\ot\det(\cE')$.

We also have an Atkin--Lehner operator
\begin{equation}\label{ALH inf}
\AL_{H,\infty}: H_{d}(\Sig)\times\frSi\to H_{d}(\Sig)
\end{equation}
defined by applying $\AL_{G,\infty}$ (see \eqref{ALG infty}) to both $\cE$ and $\cE'$, and keeping $\ph$.

\sss{The Hecke correspondence for $H_{d}(\Sig)$}

\begin{defn} Let $\un\mu\in\{\pm1\}^{r}$.
\begin{enumerate}
\item Let $\wt\Hk^{\un\mu}_{H,d}(\Sig)$  \footnote{In \cite{YZ}, the analogue of $\wt\Hk^{\un\mu}_{H,d}(\Sig)$ was denoted by $\wt\Hk^{\un\mu}_{G,d}$.} be the moduli stack of $(\{\cE_{i}^{\da}\}_{0\le i\le r},\{\cE'^{\da}_{i}\}_{0\le i\le r}, \{x_{i}\}_{1\le i\le r})$ together with a diagram
\begin{equation}\label{diag EE'}
\xymatrix{\cE_{0}\ar@{-->}[r]^{f_{1}}\ar[d]^{\ph_{0}} & \cE_{1}\ar@{-->}[r]^{f_{1}}\ar[d]^{\ph_{1}}  & \cdots \ar@{-->}[r]^{f_{r}} &  \cE_{r}\ar[d]^{\ph_{r}}\\
\cE_{0}'\ar@{-->}[r]^{f'_{1}} & \cE_{1}'\ar@{-->}[r]^{f'_{2}} & \cdots \ar@{-->}[r]^{f'_{r}}& \cE_{r}'
}
\end{equation}
where 
\begin{itemize}
\item Each $\cE_{i}$ and $\cE'_{i}$ are underlying rank two vector bundles of points $\cE^{\da}_{i}, \cE'^{\da}_{i}$ of $\Bun_{2}(\Sig)$. 
\item The upper and lower rows form objects in $\Hk^{\un\mu}_{2}(\Sig)$ with modifications at $\{x_{i}\}_{1\le i\le r}\in X^{r}$.
\item The vertical maps $\ph_{i}$ are such that $(\cE_{i}^{\da}, \cE'^{\da}_{i}, \ph_{i})\in \wt H_{d}(\Sig)$.
\end{itemize}
\item Let
\begin{equation*}
\Hk^{r}_{H,d}(\Sig):=\wt\Hk^{\un\mu}_{H,d}(\Sig)/\Pic_{X}
\end{equation*}
where $\Pic_{X}$ acts on $\wt\Hk^{\un\mu}_{H,d}(\Sig)$ by simultaneously tensoring on all $\cE^{\da}_{i}$ and $\cE'^{\da}_{i}$. 
\end{enumerate}
\end{defn}

The notation for $\Hk^{r}_{H,d}(\Sig)$ is justified because one can check, as in the case of $\Hk^{\un\mu}_{G}(\Sig)$, that $\wt\Hk^{\un\mu}_{H,d}(\Sig)/\Pic_{X}$ is canonically independent of $\un\mu$.

We have projections
\begin{equation*}
p_{H,i}: \Hk^{r}_{H,d}(\Sig)\to H_{d}(\Sig), \quad i=0,\dotsc, r.
\end{equation*}
recording the $i$-th column of the diagram \eqref{diag EE'}. We also have projections recording the upper and lower rows of the diagram \eqref{diag EE'}
\begin{equation*}
\olr{q}=(\oll{q}, \orr{q}): \Hk^{r}_{H,d}(\Sig)\to \Hk^{r}_{G}(\Sig)^{2}
\end{equation*}

Let 
\begin{eqnarray*}
\Hk'^{r}_{H,d}(\Sig):= \Hk^{r}_{H,d}(\Sig)\times_{X^{r}}X'^{r},\\
\Hk'^{r}_{G}(\Sig):=\Hk^{r}_{G}(\Sig)\times_{X^{r}}X'^{r}.
\end{eqnarray*}
The maps $p_{H,i}$ and $\olr{q}$ induce maps
\begin{eqnarray*}
p'_{H,i}&:& \Hk'^{r}_{H,d}(\Sig)\to \Hk^{r}_{H,d}(\Sig)\xr{p_{H,i}} H_{d}(\Sig), \quad i=0,\dotsc, r.\\
\olr{q}'=(\oll{q}', \orr{q}')&:& \Hk'^{r}_{H,d}(\Sig)\to \Hk'^{r}_{G}(\Sig)^{2}.
\end{eqnarray*}

\sss{The master diagram} Recall $\mu=(\un\mu, \mu_{\Sig}),\mu'=(\un\mu', \mu'_{\Sig})\in\frT_{r,\Sig}$. We consider the following diagram in which each square is commutative


\begin{equation}\label{Tian}
\xymatrix{(\Hk^{\un\mu}_{T}\times\Hk^{\un\mu'}_{T})_{\frSi'}\ar[d]^{(p^{\un\mu}_{T,0}\times p^{\un\mu'}_{T,0}\times\id_{\frSi'}, \a_{T})}\ar[r]^-{\th^{\mu,\mu'}_{\Hk}\times\id_{\frSi'}} & \Hk'^{r}_{G}(\Sig)^{2}_{\frSi'} \ar[d]^{(p'^{2}_{G,0}, \a_{G})} & \Hk'^{r}_{H,d}(\Sig)_{\frSi'}\ar[d]_{(p'_{H,0}, \a_{H})}\ar[l]_-{\olr{q}'\times\id_{\frSi'}}\\
(\Bun_{T}^{2})_{\frSi'}\times(\Bun_{T}^{2})_{\frSi'}\ar[r]^-{\th^{\mu,\mu'}_{\Bun}}_-{\times\th^{\mu,\mu'}_{\Bun}} & \Bun_{G}(\Sig)^{2}\times\Bun_{G}(\Sig)^{2} & H_{d}(\Sig)\times H_{d}(\Sig)\ar[l]_-{\olr{p_{H}}}^-{\times \olr{p_{H}}}\\
(\Bun^{2}_{T})_{\frSi'}\ar[u]^{(\id,\Fr)}\ar[r]^-{\th^{\mu,\mu'}_{\Bun}} & \Bun_{G}(\Sig)^{2}\ar[u]^{(\id,\Fr)} & H_{d}(\Sig)\ar[l]_-{\olr{p_{H}}}\ar[u]^{(\id,\Fr)}
}
\end{equation}
Here we use subscript $\frSi'$ to denote the product with $\frSi'$ over $k$.
The map $\th^{\mu,\mu'}_{\Bun}:\Bun^{2}_{T}\times\frSi'\to \Bun_{G}(\Sig)^{2}$ is given by $\th^{\mu_{\Sig}}_{\Bun}\times\th^{\mu'_{\Sig}}_{\Bun}$, using a common copy of $\frSi'$; $\th^{\mu,\mu'}_{\Hk}: \Hk^{\mu}_{T}\times\Hk^{\mu'}_{T}\times\frSi'\to \Hk'^{r}_{G}(\Sig)^{2}$ is similarly defined using $\th^{\mu}_{\Hk}$ and $\th^{\mu'}_{\Hk}$.

Let us explain the three maps $\a_{T}, \a_{G}$ and $\a_{H}$ that appear as the second components of the vertical maps connecting the first and the second rows. 
\begin{itemize}
\item The map $\a_{T}$ is the composition
\begin{equation*}
\Hk^{\un\mu}_{T}\times\Hk^{\un\mu'}_{T}\times\frSi'\xr{p^{\un\mu}_{T,r}\times p^{\un\mu'}_{T,r}\times\id_{\frSi'}}\Bun_{T}^{2}\times\frSi'\xr{\AL_{T,\mi,\mi'}}\Bun_{T}^{2}\times\frSi'
\end{equation*}
where $\AL_{T,\mi,\mi'}$ is defined as
\begin{equation}\label{ALT mu mu'}
\AL_{T,\mi,\mi'}(\cL_{1},\cL_{2},\{x'^{(1)}\})=\left(\cL_{1}(-\sum_{x\in\Si}\mu_{x}x'^{(1)}), \cL_{2}(-\sum_{x\in\Si}\mu'_{x}x'^{(1)}), \{x'^{(2)}\}\right).
\end{equation}
Hence on the $\frSi'$-factor, $\a_{T}$ is the Frobenius morphism.
\item The map $\a_{G}$ is the composition
\begin{equation*}
\Hk'^{r}_{G}(\Sig)^{2}\times\frSi'\xr{p'^{2}_{G,r}\times\nu_{\infty}}\Bun_{G}(\Sig)^{2}\times\frSi\xr{\AL^{(2)}_{G,\infty}}\Bun_{G}(\Sig)^{2}
\end{equation*}
where $\AL^{(2)}_{G,\infty}$ is $\AL_{G,\infty}$ on both copies of $\Bun_{G}(\Sig)$ using  a common copy of $\frSi$.
\item The map $\a_{H}$ is the composition
\begin{equation*}
\Hk'^{r}_{H,d}(\Sig)\times\frSi'\xr{p'_{H,r}\times\nu_{\infty}}H_{d}(\Sig)\times\frSi\xr{\AL_{H,\infty}}H_{d}(\Sig).
\end{equation*}
\end{itemize}

\sss{} We define $\Sht'^{r}_{H,d}(\Sii)$ to be the fiber product of the third column of \eqref{Tian}, i.e., the following diagram is Cartesian
\begin{equation}\label{defn ShtH}
\xymatrix{\Sht'^{r}_{H,d}(\Sii)\ar[r]\ar[d] & \Hk'^{r}_{H,d}(\Sig)\times\frSi'\ar[d]^{(p'_{H,0}, \a_{H})}\\
H_{d}(\Sig)\ar[r]^-{(\id,\Fr)} & H_{d}(\Sig)\times H_{d}(\Sig)
}
\end{equation}
Then the fiber product of the three columns are
\begin{equation}\label{hor maps}
\xymatrix{\Sht^{\un\mu}_{T}(\mi\cdot\Si')\times_{\frSi'}\Sht^{\un\mu'}_{T}(\mi'\cdot\Si')\ar[r]^-{\th'^{\mu}\times\th'^{\mu'}} & \Sht'^{r}_{G}(\Sii)\times_{\frSi'}\Sht'^{r}_{G}(\Sii) & \Sht'^{r}_{H,d}(\Sii)\ar[l]}
\end{equation}

Recall the map $s: H_{d}(\Sig)\to U_{d}$ from \eqref{Hd Ud}. The Hecke correspondence $\Hk'^{r}_{H,d}(\Sig)$ preserves the map $s$ while the Frobenius map on $H_{d}(\Sig)$ covers the Frobenius map of $U_{d}$. Therefore, from the definition of $\Sht'^{r}_{H,d}(\Sii)$, we get canonical decomposition of it indexed by $k$-points of $U_{d}$, i.e., effective divisors of degree $d$ on $U$. As in \cite[Lemma 6.12]{YZ}, one shows that the piece indexed by $D\in U_{d}(k)$ is exactly the Hecke correspondence $\Sht'^{r}_{G}(\Sii;h_{D})$ for $\Sht'^{r}_{G}(\Sii)$. In other words, we have a decomposition
\begin{equation}\label{decomp ShtH}
\Sht'^{r}_{H,d}(\Sii)=\coprod_{D\in U_{d}(k)}\Sht'^{r}_{G}(\Sii;h_{D}).
\end{equation}

\sss{The stack $\cM_{d}(\mu_{\Sig},\mu'_{\Sig})$ and its Hecke correspondence} Now we consider the fiber product of the three rows of the master diagram \eqref{Tian}.

\begin{defn}\label{defn:M} Let $\cM_{d}(\mu_{\Sig},\mu'_{\Sig})$  be the fiber product of the bottom row of \eqref{Tian}, i.e., we have the following Cartesian diagram
\begin{equation}\label{defn Md mu}
\xymatrix{\cM_{d}(\mu_{\Sig},\mu'_{\Sig})\ar[r]\ar[d] & H_{d}(\Sig)\ar[d]^{\olr{p_{H}}}\\
\Bun_{T}^{2}\times\frSi'\ar[r]^{\th^{\mu,\mu'}_{\Bun}} & \Bun_{G}(\Sig)^{2}}
\end{equation}
\end{defn}
Our notation suggests that $\cM_{d}(\mu_{\Sig},\mu'_{\Sig})$ depends only on $\mu_{\Sig}$ and $\mu'_{\Sig}$. This is indeed the case, because $\th^{\mu,\mu'}_{\Bun}$ depends only on $\mu_{\Sig}$ and $\mu'_{\Sig}$.

From the definition of $\cM_{d}(\mu_{\Sig},\mu'_{\Sig})$, the Atkin--Lehner automorphisms $\AL_{G,\infty}$ (see \eqref{ALG infty}), $\AL_{H,\infty}$ (see \eqref{ALH inf}) and $\AL_{T,\mi,\mi'}$ (see \eqref{ALT mu mu'}) together with Lemma \ref{l:comp Di Di'} induce an Atkin--Lehner automorphism for $\cM_{d}(\mu_{\Sig},\mu'_{\Sig})$
\begin{equation*}
\AL_{\cM,\infty}:\cM_{d}(\mu_{\Sig},\mu'_{\Sig})\to \cM_{d}(\mu_{\Sig},\mu'_{\Sig}). 
\end{equation*}

\begin{defn}
Let $\Hk^{\mu,\mu'}_{\cM,d}$ be the fiber product of the top row of \eqref{Tian}. Equivalently, we have the following Cartesian diagram
\begin{equation}\label{defn HkM}
\xymatrix{\Hk^{\mu,\mu'}_{\cM,d}\ar[rr]\ar[d] && \Hk'^{r}_{H,d}(\Sig)\ar[d]^{\olr{q}}\\
\Hk^{\un\mu}_{T}\times\Hk^{\un\mu'}_{T}\times\frSi'\ar[rr]^-{\th^{\mu,\mu'}_{\Hk}} && \Hk'^{r}_{G}(\Sig)^{2}}
\end{equation}  
\end{defn}

Comparing the diagrams \eqref{defn Md mu} and \eqref{defn HkM}, we get projections
\begin{equation*}
p_{\cM,i}: \Hk^{\mu,\mu'}_{\cM,d}\to \cM_{d}(\mu_{\Sig},\mu'_{\Sig}), \quad i=0,\dotsc,r
\end{equation*}
as the fiber product of $p^{\mu}_{T,i}\times p^{\mu'}_{T,i}\times \id_{\frSi'}$ and $p'_{H,i}$ over $p'^{2}_{G,i}$. We also let
\begin{equation*}
\a_{\cM}=\AL_{\cM,\infty}\circ p_{\cM,r}: \Hk^{\mu,\mu'}_{\cM,d}\to \cM_{d}(\mu_{\Sig},\mu'_{\Sig}).
\end{equation*}

The fiber products of the three rows of \eqref{Tian} now read
\begin{equation}\label{3 rows}
\xymatrix{\Hk^{\mu,\mu'}_{\cM,d}\ar[d]^{(p_{\cM,0}, \a_{\cM})}\\
\cM_{d}(\mu_{\Sig},\mu'_{\Sig})\times\cM_{d}(\mu_{\Sig},\mu'_{\Sig})\\
\cM_{d}(\mu_{\Sig},\mu'_{\Sig})\ar[u]^{(\id,\Fr)}}
\end{equation}

\sss{The stack $\Sht^{\mu,\mu'}_{\cM,d}$} 
\begin{defn} Let $\Sht^{\mu,\mu'}_{\cM,d}$ be the fiber product of the maps in \eqref{3 rows}, i.e., we have a Cartesian diagram
\begin{equation}\label{defn ShtM}
\xymatrix{\Sht^{\mu,\mu'}_{\cM,d} \ar[r] \ar[d] & \Hk^{\mu,\mu'}_{\cM,d}\ar[d]^{(p_{\cM,0}, \a_{\cM})}\\
\cM_{d}(\mu_{\Sig},\mu'_{\Sig})\ar[r]^-{(\id,\Fr)} & \cM_{d}(\mu_{\Sig},\mu'_{\Sig})\times \cM_{d}(\mu_{\Sig},\mu'_{\Sig})}
\end{equation}
\end{defn}

By the diagram \eqref{Tian}, $\Sht^{\mu,\mu'}_{\cM,d}$ is also the fiber product of the maps in \eqref{hor maps}, i.e., the following diagram is also Cartesian
\begin{equation}\label{var ShtM}
\xymatrix{\Sht^{\mu,\mu'}_{\cM,d}\ar[d]\ar[r] & \Sht'^{r}_{H,d}(\Sii)\ar[d]\\
\Sht^{\un\mu}_{T}(\mi\cdot\Si')\times_{\frSi'}\Sht^{\un\mu'}_{T}(\mi'\cdot\Si')\ar[r]^-{\th'^{\mu}\times\th'^{\mu'}} & \Sht'^{r}_{G}(\Sii)\times_{\frSi'}\Sht'^{r}_{G}(\Sii)}
\end{equation}

According to the decomposition \eqref{decomp ShtH}, we get a corresponding decomposition of $\Sht^{\mu,\mu'}_{\cM,d}$
\begin{equation}\label{decomp ShtM D}
\Sht^{\mu,\mu'}_{\cM,d}=\coprod_{D\in U_{d}(k)}\Sht^{\mu,\mu'}_{\cM,D}
\end{equation}
where $\Sht^{\mu,\mu'}_{\cM,D}$ is the preimage of $\Sht'^{r}_{G}(\Sii;h_{D})\subset \Sht'^{r}_{H,d}(\Sii)$ under the upper horizontal map in \eqref{var ShtM}. We have a Cartesian diagram
\begin{equation}\label{ShtMD}
\xymatrix{\Sht^{\mu,\mu'}_{\cM,D}\ar[d]\ar[r] & \Sht'^{r}_{G}(\Sii;h_{D})\ar[d]^{(\oll{p}', \orr{p}')}\\
\Sht^{\un\mu}_{T}(\mi\cdot\Si')\times_{\frSi'}\Sht^{\un\mu'}_{T}(\mi'\cdot\Si')\ar[r]^-{\th'^{\mu}\times\th'^{\mu'}} & \Sht'^{r}_{G}(\Sii)\times_{\frSi'}\Sht'^{r}_{G}(\Sii)}
\end{equation}
Here the maps $\oll{p}', \orr{p}': \Sht'^{r}_{G}(\Sii;h_{D})\to \Sht'^{r}_{G}(\Sii)$ are the base changes of the maps $\oll{p}$ and $\orr{p}$ in \eqref{ShthDpp}.

\subsection{Relation between $\cM_{d}$ and $\cM_{d}(\mu_{\Sig},\mu'_{\Sig})$}\label{ss:compare M}
In this subsection, we relate $\cM_{d}(\mu_{\Sig},\mu'_{\Sig})$ to the moduli stack $\cM_{d}$ which was defined earlier. For this, we first give an alternative description of $\cM_{d}(\mu_{\Sig},\mu'_{\Sig})$ in the style of the definition of $\cM_{d}$ in \cite[\S6.1.1]{YZ}.

\sss{Some preparation} Let $S$ be any scheme, and $\cL$ and $\cL'$ two line bundles over $X'\times S$. We denote by $H_{R'}(\cL,\cL')$ be the set of pairs $(\a,\b)$ where
\begin{eqnarray}
\label{al}\a&:& \cL\to\cL'(R'):=\cL'\otimes_{\cO_{X'}}\cO_{X'}(R')\\
\label{be}\b&:& \s^{*}\cL\to \cL'(R')
\end{eqnarray}
such that their restrictions to $R'\times S$ satisfy
\begin{equation}\label{a+b}
\a|_{R'\times S}=\b|_{R'\times S}.
\end{equation}
Note that $\cL$ and $\s^{*}\cL$ are the same when restricted to $R'\times S$, hence the above equality makes sense.

Recall $\nu_{S}=\nu\times\id_{S}: X'\times S\to X\times S$. 

\begin{lemma}\label{l:ab}
There is a canonical bijection
\begin{equation*}
\Hom_{X\times S}(\nu_{S,*}\cL, \nu_{S,*}\cL')\isom H_{R'}(\cL,\cL')
\end{equation*}
such that, if $\ph: \nu_{S,*}\cL\to \nu_{S,*}\cL'$ corresponds to $(\a,\b)$ under this bijection, we have
\begin{equation}\label{det ph a-b}
\det(\ph)=\Nm(\a)-\Nm(\b)
\end{equation}
as sections of $\det(\nu_{S,*}\cL)^{-1}\otimes\det(\nu_{S,*}\cL')\cong \Nm_{X'/X}(\cL)^{-1}\ot\Nm_{X'/X}(\cL')$.
\end{lemma}
\begin{proof}
By adjunction a map $\ph: \nu_{S,*}\cL\to\nu_{S,*}\cL'$ is equivalent to a map $\nu^{*}_{S}\nu_{S,*}\cL\to\cL'$. Note that $\nu^{*}_{S}\nu_{S,*}\cL\cong\cO_{X'}\otimes_{\cO_{X}}\cL\cong(\cO_{X'}\otimes_{\cO_{X}}\cO_{X'})\otimes_{\cO_{X'}}\cL$, whose $\cO_{X'}$-module structure is given by the first factor of $\cO_{X'}$. 

We have an injective map $\jmath: \cO_{X'}\otimes_{\cO_{X}}\cO_{X'}\to \cO_{X'}\oplus \cO_{X'}$ sending $a\otimes b\mapsto ab+a\s(b)$.  By a local calculation at points in $R'$ we see that the image of $\jmath$ is $\cO_{X'}\oplus_{R'}\cO_{X'}:=\ker(\cO_{X'}\oplus\cO_{X'}\xr{(i^{*},-i^{*})}\cO_{R'})$ (the difference of two restriction maps $i^{*}:\cO_{X'}\to \cO_{R'}$). Therefore $\nu^{*}_{S}\nu_{S,*}\cL\cong(\cO_{X'}\oplus_{R'}\cO_{X'})\otimes_{\cO_{X'}}\cL=\cL\oplus_{R'}\s^{*}\cL=\ker(\cL\oplus\s^{*}\cL\xr{(i^{*},-i^{*})}\cL_{R'\times S})$. Hence the map $\ph$ is equivalent to a map
\begin{equation*}
\psi: \cL\oplus_{R'}\s^{*}\cL\to \cL'.
\end{equation*}
Since $\cL(-R')\oplus\s^{*}\cL(-R')\subset\cL\oplus_{R'}\s^{*}\cL$, the map $\psi$ restricts to a map
\begin{equation*}
\cL(-R')\oplus\s^{*}\cL(-R')\to \cL'
\end{equation*}
or
\begin{equation*}
\cL\oplus\s^{*}\cL\to \cL'(R').
\end{equation*}
We then define the two components of above map to be $\a$ and $-\b$. The condition \eqref{a+b} is equivalent to that the map $\a\oplus(-\b):\cL\oplus\s^{*}\cL\to \cL'(R')$, when restricted to $\cL\oplus_{R'}\s^{*}\cL$, lands in $\cL'$.

If $\ph$ corresponds to $(\a,\b)$, we may pullback $\ph$ to $X'$ so it becomes the map $\cL\oplus_{R'}\s^{*}\cL\to \cL'\oplus_{R'}\s^{*}\cL'$ given by the matrix
\begin{equation*}
\mat{\a}{-\b}{-\s^{*}\b}{\s^{*}\a}.
\end{equation*}
Therefore $\det(\ph)=\Nm(\a)-\Nm(\b)$.
\end{proof}

\sss{Alternative description of $\cM_{d}(\mu_{\Sig},\mu'_{\Sig})$} We define $\wt\cM_{d}(\mu_{\Sig},\mu'_{\Sig})$ by the Cartesian diagram
\begin{equation*}
\xymatrix{\wt\cM_{d}(\mu_{\Sig},\mu'_{\Sig})\ar[r]\ar[d] & \wt H_{d}(\Sig)\ar[d]^{\olr{\wt p_{H}}}\\
\Pic_{X'}\times\Pic_{X'}\times\frSi'\ar[r]^{\wt\th^{\mu,\mu'}_{\Bun}} & \Bun_{2}(\Sig)\times\Bun_{2}(\Sig)}
\end{equation*}
Here $\wt\th^{\mu,\mu'}_{\Bun}$ is given by $\wt\th^{\mu_{\Sig}}_{\Bun}\times\wt\th^{\mu'_{\Sig}}_{\Bun}$, using a common copy of $\frSi'$, and $\olr{\wt p_{H}}$ sends $(\cE^{\da},\cE'^{\da}, \ph)\in \wt H_{d}(\Sig)(S)$ to $(\cE^{\da},\cE'^{\da})\in (\Bun_{2}(\Sig)(S))^{2}$. Comparing with the Definition \ref{defn:M},  we have
\begin{equation*}
\cM_{d}(\mu_{\Sig},\mu'_{\Sig})\cong \wt\cM_{d}(\mu_{\Sig},\mu'_{\Sig})/\Pic_{X}.
\end{equation*}

For $x'\in\Si'$ and $x'^{(1)}: S\to \Spec k(x')\xr{x'}X'$, recall we inductively defined $x'^{(j)}$ using $x'^{(j)}=x'^{(j-1)}\circ \Fr_{S}$ for $j\ge 2$. We have a morphism
\begin{equation*}
\frD_{+}:\frSi'\to X'_{N_{+}}
\end{equation*}
which sends $\{x'^{(1)}\}_{x'\in\Si'}\in \frSi'(S)$ to the following divisor of $X'\times S$ of degree $N_{+}$
\begin{eqnarray*}
\frD_{+}(\{x'^{(1)}\})&:=&\sum_{x\in\Sf\cap \Sig_{+}}\mu'_{x}\times S+\sum_{x\in\Si\cap\Sig_{+}}
\begin{cases}(\Gamma_{x'^{(1)}}+\Gamma_{x'^{(2)}}+\cdots+\Gamma_{x'^{(d_{x})}}), & \textup{ if }\mu'_{x}=1\\
(\Gamma_{x'^{(d_{x}+1)}}+\Gamma_{x'^{(d_{x}+2)}}+\cdots+\Gamma_{x'^{(2d_{x})}}), & \textup{ if }\mu'_{x}=-1.
\end{cases}
\end{eqnarray*}

Similarly, we define
\begin{equation*}
\frD_{-}:\frSi'\to X'_{N_{-}}
\end{equation*}
by sending $\{x'^{(1)}\}_{x'\in\Si'}\in \frSi'(S)$ to the following divisor of $X'\times S$ of degree $N_{-}$ 
\begin{eqnarray*}
\frD_{-}(\{x'^{(1)}\})&:=&\sum_{x\in\Sf\cap \Sig_{-}}\mu'_{x}\times S+\sum_{x\in\Si\cap\Sig_{-}}
\begin{cases}(\Gamma_{x'^{(1)}}+\Gamma_{x'^{(2)}}+\cdots+\Gamma_{x'^{(d_{x})}}), & \textup{ if }\mu'_{x}=1\\
(\Gamma_{x'^{(d_{x}+1)}}+\Gamma_{x'^{(d_{x}+2)}}+\cdots+\Gamma_{x'^{(2d_{x})}}), & \textup{ if }\mu'_{x}=-1.
\end{cases}
\end{eqnarray*}

Now we can state the alternative description of $\cM_{d}(\mu_{\Sig},\mu'_{\Sig})$.

\begin{lemma}\label{l:beta van}
For a scheme $S$, $\wt\cM_{d}(\mu_{\Sig},\mu'_{\Sig})(S)$ is canonically equivalent to the groupoid of tuples $(\cL,\cL',\a,\b, \{x'^{(1)}\}_{x'\in \Si'})$ where
\begin{itemize}
\item $\cL$ and $\cL'$ are line bundles on $X'\times S$ such that $\deg(\cL'|_{X'\times s})-\deg(\cL|_{X'\times s})=d$ for all geometric points $s\in S$;
\item $\a:\cL\to \cL'(R')$, $\b:\s^{*}\cL\to\cL'(R')$.
\end{itemize}
These data are required to satisfy the following conditions.
\begin{enumerate}
\item $\a|_{\frD_{-}(\{x'^{(1)}\})}=0$, and $\a|_{\nu^{-1}(\Sig_{+})\times S}$ is an isomorphism.
\item $\b|_{\frD_{+}(\{x'^{(1)}\})}=0$, and $\b|_{\nu^{-1}(\Sig_{-})\times S}$ is an isomorphism.
\item $\a|_{R'\times S}=\b|_{R'\times S}$. Moreover, $\Nm(\a)-\Nm(\b)$, viewed as a section of $\Nm_{X'/X}(\cL)^{-1}\ot\Nm_{X'/X}(\cL')$, is nowhere vanishing along $R\times S$.
\item This is non-void only when $\Sig=\vn$ and $R=\vn$: for every geometric point $s$ of $S$, $\Nm(\a)-\Nm(\b)$ is not identically zero on $X\times s$.
\end{enumerate}
\end{lemma}
\begin{proof} By definition, $S$-points of $\wt\cM_{d}(\mu_{\Sig},\mu'_{\Sig})$ consist of tuples $(\cL,\cL',\ph, \{x'^{(1)}\}_{x'\in \Si'})$ where
\begin{itemize}
\item $\cL$ and $\cL'$ are line bundles on $X'\times S$ such that $\deg(\cL'|_{X\times s})-\deg(\cL|_{X\times s})=d$ for all geometric points $s\in S$.
\item $\ph: \nu_{S,*}\cL\to \nu_{S,*}\cL'$ is an injective map when restricted to $X\times s$ for every geometric point $s\in S$. Moreover, $\ph$ is an isomorphism along $(\Sig\sqcup R)\times S$. 
\item For each $x'\in \Si'$,  $x'^{(1)}$ is a map $S\to \Spec k(x')\xr{x'}X'$.
\end{itemize}
These data are required to satisfy the following condition. We have two $S$-points of $\Bun_{2}(\Sig)$:
\begin{eqnarray*}
\cE^{\da}=\wt\th^{\mu_{\Sig}}_{\Bun}(\cL, \{x'^{(1)}\}_{x'\in \Si'}),\\
\cE'^{\da}=\wt\th^{\mu'_{\Sig}}_{\Bun}(\cL', \{x'^{(1)}\}_{x'\in \Si'}).
\end{eqnarray*}
Then $\ph:\cE=\nu_{S,*}\cL\to\cE'=\nu_{S,*}\cL'$ should respect the level structures of  $\cE^{\da}$ and $\cE'^{\da}$.  

By Lemma \ref{l:ab}, the map $\ph: \nu_{S,*}\cL\to \nu_{S,*}\cL'$ becomes a pair $\a:\cL\to \cL'(R')$ and $\b: \s^{*}\cL\to \cL'(R')$ satisfying $\a|_{R'\times S}=\b|_{R'\times S}$. Since $\ph|_{R\times S}$ is an isomorphism, the formula \eqref{det ph a-b} implies that $\Nm(\a)-\Nm(\b)$ is nowhere vanishing along $R\times S$, hence  condition (3) in the statement of the lemma is verified. Condition (4) also follows from \eqref{det ph a-b} and the condition on $\ph$ above.

Since $\ph$ respects the Iwahori level structures of $\nu_{S,*}\cL$ and $\nu_{S,*}\cL'$, it sends $\nu_{S,*}(\cL(-\mu_{x}))$ to $\nu_{S,*}(\cL'(-\mu'_{x}))$ for all $x\in \Sf$ (recall $\mu_{x}$ is the value of $\mf$ at $x$). A local calculation shows that $\a$ should vanish along $\mu'_{x}\times S$ for those $x\in\Sf$ such that $\mu_{x}\ne\mu'_{x}$, and $\b$ should vanish along $\mu'_{x}\times S$ for those $x\in\Sf$ such that $\mu_{x}=\mu'_{x}$. A similar local calculation at $x\in\Si$ implies the vanishing of $\a$ and $\b$ along the corresponding parts of $\frD_{-}$ and $\frD_{+}$. For example, if $\mu_{x}=\mu'_{x}=1$, then $\ph$ should send $\nu_{S,*}(\cL(-\Gamma_{x'^{(1)}}-\cdots-\Gamma_{x'^{(d_{x})}}))$ to $\nu_{S,*}(\cL'(-\Gamma_{x'^{(1)}}-\cdots-\Gamma_{x'^{(d_{x})}}))$, which implies that $\b$ vanishes along $\Gamma_{x'^{(1)}}+\Gamma_{x'^{(2)}}+\cdots+\Gamma_{x'^{(d_{x})}}$.  These verify the vanishing parts of the conditions (1)(2).

Finally, since $\ph|_{\Sig\times S}$ is an isomorphism, $\det(\ph)=\Nm(\a)-\Nm(\b)$ is nowhere vanishing on $\Sig\times S$. Since $\Nm(\a)|_{\Sig_{-}\times S}=0$ and $\Nm(\b)|_{\Sig_{+}\times S}=0$ by the vanishing parts of (1)(2), $\Nm(\a)|_{\Sig_{+}\times S}$ and $\Nm(\b)|_{\Sig_{-}\times S}$ are nowhere vanishing. These verify the nonvanishing parts of the conditions (1)(2). We have verified all the desired conditions for $(\cL,\cL',\a,\b,\{x'^{(1)}\}_{x'\in\Si'})$.
\end{proof}

Using the description of $\cM_{d}(\mu_{\Sig},\mu'_{\Sig})$ given in Lemma \ref{l:beta van}, we can describe its Atkin--Lehner automorphism $\AL_{\cM,\infty}$ as follows.

\begin{lemma}\label{l:pre ALM} Let $(\cL,\cL', \a,\b, \{x'^{(1)}\}_{x'\in\Si'})$ be an $S$-point of $\wt\cM_{d}(\mu_{\Sig},\mu'_{\Sig})$ as described in Lemma \ref{l:beta van}, and we use the same notation to denote its image in $\cM_{d}(\mu_{\Sig},\mu'_{\Sig})$. Then
\begin{eqnarray*}
&&\AL_{\cM,\infty}(\cL,\cL', \a,\b, \{x'^{(1)}\}_{x'\in\Si'})\\
&=&\left(\cL(-\sum_{x\in \Si}\mu_{x}\Gamma_{x'^{(1)}}),\cL'(-\sum_{x\in \Si\cap\Sig_{+}}\mu_{x}\Gamma_{x'^{(1)}}-\sum_{x\in \Si\cap\Sig_{-}}\mu_{x}\Gamma_{x'^{(d_{x}+1)}}), \a',\b', \{x'^{(2)}\}_{x'\in\Si'}\right)
\end{eqnarray*}
Here, $\a'$ is induced from $\a$ using the fact that $\a|_{\frD_{-}}=0$; $\b'$ is induced from $\b$ using the fact that $\b|_{\frD_{+}}=0$. 
\end{lemma}

The proof is by tracking the definitions and we omit it. 

The next result clarifies the relation between $\cM_{d}$ and $\cM_{d}(\mu_{\Sig},\mu'_{\Sig})$.
\begin{prop}\label{p:compare M} There is a canonical isomorphism over $\frSi'$ 
\begin{equation}\label{MMe}
\Xi_{\cM}: \cM_{d}\times\frSi'\isom \cM_{d}(\mu_{\Sig},\mu'_{\Sig})
\end{equation}
such that:
\begin{enumerate}
\item The automorphism $\id\times \Fr_{\frSi'}$ on the left corresponds to the automorphism $\AL_{\cM, \infty}$ on the right.
\item The following diagram is commutative
\begin{equation*}
\xymatrix{\cM_{d}\times\frSi'\ar[rr]^{\Fr\times\id}\ar[d]_{\wr}^{\Xi_{\cM}} && \cM_{d}\times\frSi'\ar[d]_{\wr}^{\Xi_{\cM}}\\
\cM_{d}(\mu_{\Sig},\mu'_{\Sig})\ar[rr]^{\AL_{\cM,\infty}^{-1}\circ\Fr} && \cM_{d}(\mu_{\Sig},\mu'_{\Sig})
}
\end{equation*}
\end{enumerate}
\end{prop}
\begin{proof} 
We first define a map
\begin{equation*}
\i_{d}: \cM_{d}(\mu_{\Sig},\mu'_{\Sig})\to \cM_{d}\times\frSi'\subset (\hX'_{d+\r-N_{-}}\times_{\Pic^{\sqR;\sqR, d+\r}_{X}}\hX'_{d+\r-N_{+}})\times \frSi'.
\end{equation*}
Using the description of points of $\wt\cM_{d}(\mu_{\Sig},\mu'_{\Sig})$ in Lemma \ref{l:beta van}, we have a morphism
\begin{equation*}
\i_{\a}: \cM_{d}(\mu_{\Sig},\mu'_{\Sig})\to \hX'_{d+\r-N_{-}}
\end{equation*}
sending  $(\cL,\cL',\a,\b,  \{x'^{(1)}\}_{x'\in \Si'})$ to the line bundle $\cL^{-1}\otimes\cL'(R'-\frD_{-}(\{x'^{(1)}\}))$ and its section given by $\a$. Similarly we have a morphism
\begin{equation*}
\i_{\b}: \cM_{d}(\mu_{\Sig},\mu'_{\Sig})\to \hX'_{d+\r-N_{+}}
\end{equation*}
sending $(\cL,\cL',\a,\b, \{x'^{(1)}\}_{x'\in \Si'})$ to the line bundle $\s^{*}\cL^{-1}\otimes\cL'(R'-\frD_{+}(\{x'^{(1)}\}))$ and its section given by $\b$. We have a canonical isomorphism $\nu_{\a}\circ\i_{\a}\cong \nu_{\b}\circ \i_{\b}$ using $\a|_{R'}=\b|_{R'}$. The map $\i_{d}$ is given by $(\i_{\a},\i_{\b})$ and the natural projection to $\frSi'$. It is easy to see that the image of $\i_{d}$ lies in the open substack $\cM_{d}\times\frSi'$.

Next we construct the desired map $\Xi_{\cM}$ as in \eqref{MMe}. Start with a point $(\cI,\cJ, \a,\b,\j)\in \cM_{d}(S)$, and $\{x'^{(1)}\}_{x'\in\frSi'}\in \frSi'(S)$. Let $D_{\pm}=\frD_{\pm}(\{x'^{(1)}\})$ (a divisor of degree $N_{\pm}$ on $X'\times S$ with image $\Sig_{\pm}\times S$ in $X\times S$), and $\cI'=\cI(D_{-})$ and $\cJ':=\cJ(D_{+})$. The isomorphism $\j$ then gives an $\Nm^{\sqR}_{X'/X}(\cI')\cong \Nm^{\sqR}_{X'/X}(\cJ')\in \Pic^{\sqR,d+\r}_{X}(S)$, or a trivialization of $\Nm^{\sqR}_{X'/X}(\cI'^{\ot-1}\otimes\cJ')$ as an $S$-point of $\Pic^{\sqR,d+\r}_{X}$. The exact sequence \eqref{Pic exact} then implies, upon localizing $S$ in the \'etale topology, there exists a line bundle $\cL\in\Pic_{X'}(S)$ together with an isomorphism $\tau: \cL^{-1}\otimes\s^{*}\cL\cong \cI'^{}\otimes\cJ'^{-1}$, and such a pair $(\cL,\tau)$ is unique up to tensoring with $\Pic_{X}(S)$ (upon further localizing $S$). Let $\cL'=\cL\otimes\cI'(-R')$, then $\a$ can be viewed as a section of $\cL^{-1}\otimes \cL'(R')$, or a map $\cL\to\cL'(R')$ which vanishes along $D_{-}$. Since $\cJ'\cong \cI'\otimes\cL\otimes\s^{*}\cL^{-1}\cong \s^{*}\cL^{-1}\otimes\cL'(R')$, $\b$ can be viewed as a section of $\s^{*}\cL^{-1}\otimes\cL'(R')$, or a map $\s^{*}\cL\to\cL'(R')$ which vanishes along $D_{+}$. Moreover, the equality $\a|_{R'\times S}=\b|_{R'\times S}$ is built into the definition of $\cM_{d}$. This way we get an $S$-point $(\cL,\cL',\a,\b, \{x'^{(1)}\})$ of $\cM_{d}(\mu_{\Sig},\mu'_{\Sig})$ using the description of $\wt\cM_{d}(\mu_{\Sig},\mu'_{\Sig})$ given in Lemma \ref{l:beta van}.

It is easy to see that $\Xi_{\cM}$ is inverse to $\i_{d}$. Therefore $\Xi_{\cM}$ is an isomorphism. This finishes the construction of the isomorphism $\Xi_{\cM}$.

Now property (1) follows  from Lemma \ref{l:pre ALM} by a direct calculation. 

To check property (2),  observe that the total Frobenius morphisms $\Fr\times\Fr$ on $\cM_{d}\times\frSi'$ and $\Fr$ on $\cM_{d}(\mu_{\Sig},\mu'_{\Sig})$ correspond to each other under $\Xi_{\cM}$. On the other hand, by (1),  $\id\times\Fr$ on  $\cM_{d}\times\frSi'$ corresponds to $\AL_{\cM,\infty}$ on $\cM_{d}(\mu_{\Sig},\mu'_{\Sig})$. Therefore, $\Fr\times\id=(\id\times\Fr^{-1})\circ(\Fr\times\Fr)$ on $\cM_{d}\times\frSi'$ corresponds to $\AL_{\cM,\infty}^{-1}\circ\Fr$ on $\cM_{d}(\mu_{\Sig},\mu'_{\Sig})$.
\end{proof}

\sss{Comparison of Hecke correspondences for $\cM_{d}(\mu_{\Sig},\mu'_{\Sig})$ and for $\cM_{d}$}

We have already defined two self-correspondences $\cH_{+}$ and $\cH_{-}$ of $\cM_{d}$ in \S\ref{sss:Hk Md}. For $\un\l=(\l_{1},\dotsc,\l_{r})\in\{\pm1\}^{r}$, let
\begin{equation*}
\cH_{\l_{i}}=\begin{cases}\cH_{+}, & \l_{i}=1;\\ \cH_{-}, & \l_{i}=-1.\end{cases}
\end{equation*} 
Let $\oll{\g}_{i}, \orr{\g}_{i}: \cH_{\l_{i}}\to \cM_{d}$ be the two projections. Then define
$\cH_{\un\l}$ to be the composition of $\cH_{\l_{i}}$ as follows
\begin{equation*}
\cH_{\un\l}:=\cH_{\l_{1}}\times_{\orr{\g}_{1}, \cM_{d},\oll{\g}_{2}} \cH_{\l_{2}}\times_{\orr{\g}_{2}, \cM_{d}, \oll{\g}_{3}}\cdots\times_{\orr{\g}_{r-1}, \cM_{d},\oll{\g}_{r}}\cH_{\l_{r}}.
\end{equation*}
We apply this construction to $\un\l=\un\mu\un\mu'=(\mu_{1}\mu'_{1},\dotsc, \mu_{r}\mu_{r}')$. Then we have $(r+1)$ projections
\begin{equation*}
\g_{i}: \cH_{\un\mu\un\mu'}\to \cM_{d}, \quad i=0,1,\dotsc, r.
\end{equation*}

\begin{prop}\label{p:HkM} There is a canonical isomorphism over $\frSi'$
\begin{equation}\label{HkM isom}
\Xi_{\cH}: \cH_{\un\mu\un\mu'}\times\frSi'\isom \Hk^{\mu,\mu'}_{\cM,d}
\end{equation}
such that the following diagram is commutative for $i=0,1,\dotsc, r$
\begin{equation*}
\xymatrix{\cH_{\un\mu\un\mu'}\times\frSi'\ar[r]^-{\Xi_{\cH}}_-{\sim}\ar[d]_{\g_{i}\times\id_{\frSi'}} & \Hk^{\mu,\mu'}_{\cM,d}\ar[d]^{p_{\cM,i}}\\
\cM_{d}\times\frSi'\ar[r]^-{\Xi_{\cM}}_-{\sim} & \cM_{d}(\mu_{\Sig},\mu'_{\Sig})}
\end{equation*}
\end{prop}
\begin{proof} 
By the iterative nature of $\Hk^{\mu,\mu'}_{\cM,d}$, it suffices to prove the case $r=1$ (at this point we may drop the assumption $r\equiv \#\Si\mod2$ because everything makes sense without this condition, before passing to Shtukas). We distinguish two cases.

Case 1. $\mu_{1}=\mu'_{1}$. We treat only the case $\mu_{1}=\mu'_{1}=1$ and the other case is similar.  In this case,  $\Hk^{\mu,\mu'}_{\cM,d}(S)$ classifies the following data up to the action of $\Pic_{X}$:
\begin{itemize}
\item A map $x'_{1}:S\to X'$ with graph $\Gamma_{x'_{1}}$.
\item For each $x'\in\Si'$, an $S$-point $x'^{(1)}: S\to \Spec k(x')\xr{x'}X'$.
\item Line bundles $\cL_{0}$ and $\cL'_{0}$ on $X'\times S$ such that $\deg(\cL'_{0}|_{X\times s})-\deg(\cL_{0}|_{X\times s})=d$ for all geometric points $s\in S$. Let 
\begin{equation*}
\cL_{1}=\cL_{0}(\Gamma_{x'_{1}}), \quad \cL'_{1}=\cL'_{0}(\Gamma_{x'_{1}}).
\end{equation*}
\item A map $\ph_{1}: \nu_{S,*}\cL_{1}\to \nu_{S,*}\cL'_{1}$ that restricts to a map $\ph_{0}:\nu_{S,*}\cL_{0}\to \nu_{S,*}\cL'_{0}$. Moreover, for $i=0$ and $1$, we require the tuple $(\cL_{i},\cL'_{i},\ph_{i},\{x'^{(1)}\})$ to give a point of $\cM_{d}(\mu_{\Sig},\mu'_{\Sig})$. In other words, $\ph_{i}$ preserves the level structures of $\nu_{S,*}\cL_{i}$ and $\nu_{S,*}\cL'_{i}$ given in \S\ref{sss:th Bun}; $\ph_{i}$ is injective when restricted to $X\times s$ for every geometric point $s\in S$; and $\ph_{i}|_{(\Sig\cup R)\times S}$ is an isomorphism. 
\end{itemize} 
Using Lemma \ref{l:beta van}, we may replace the data $\ph_{i}$ above by a pair of maps $(\a_{i},\b_{i})$ where $\a_{i}: \cL_{i}\to \cL'_{i}(R')$, $\b_{i}: \s^{*}\cL_{i}\to \cL'_{i}(R')$ satisfying certain conditions. Let $D_{\pm}=\frD_{\pm}(\{x'^{(1)}\})$, then $\a_{i}|_{D_{-}}=0$ and $\b_{i}|_{D+}=0$. Denote by
\begin{eqnarray*}
\a^{\na}_{i}: \cL_{i}\to \cL'_{i}(R'-D_{-})\\
\b^{\na}_{i}: \s^{*}\cL_{i}\to \cL'_{i}(R'-D_{+})
\end{eqnarray*}
the maps induced by $\a_{i}$ and $\b_{i}$.

The relation between $\ph_{0}$ and $\ph_{1}$ implies that the following two diagrams are commutative
\begin{equation}\label{Hmu1}
\xymatrix{ \cL_{0}\ar[d]^{\a^{\na}_{0}}\ar@{^{(}->}[r] & \cL_{1}\ar@{=}[r]\ar[d]^{\a^{\na}_{1}} & \cL_{0}(\Gamma_{x'_{1}}) \\
\cL'_{0}(R'-D_{-})\ar@{^{(}->}[r] &\cL'_{1}(R'-D_{-})\ar@{=}[r] & \cL'_{0}(R'-D_{-}+\Gamma_{x'_{1}})
}
\end{equation}
\begin{equation}\label{Hmu2}
\xymatrix{\s^{*}\cL_{0}\ar[d]^{\b^{\na}_{0}}\ar@{^{(}->}[r] & \s^{*}\cL_{1}\ar@{=}[r]\ar[d]^{\b^{\na}_{1}} & (\s^{*}\cL_{0})(\Gamma_{\s(x'_{1})}) \\
\cL'_{0}(R'-D_{+})\ar@{^{(}->}[r] &\cL'_{1}(R'-D_{+})\ar@{=}[r] & \cL'_{0}(R'-D_{+}+\Gamma_{x'_{1}})
}
\end{equation}
The diagram \eqref{Hmu1} simply says that $\a^{\na}_{1}$ is determined by $\a^{\na}_{0}$ (no condition on $\a^{\na}_{0}$, hence no condition on $\a_{0}$). The diagram \eqref{Hmu2} imposes a nontrivial condition on $\b^{\na}_{0}$, as claimed below.

\begin{claim} $\b^{\na}_{0}$ vanishes along $\Gamma_{\s(x'_{1})}$. 
\end{claim}
\begin{proof}[Proof of Claim]
The argument for this claim is more complicated than the argument in \cite[Lemma 6.3]{YZ} because of the ramification of $\nu$. To prove the Claim, it suffices to argue for the similar statement for the restriction of $\b^{\na}_{0}$ to $(X'-R')\times S$ and to the formal completions $\Spec \cO_{x'}\htimes S$ for each $x'\in R'$.

Computing the divisors of the maps in the first square of \eqref{Hmu2}, we get
\begin{equation}\label{div b eqn}
\div(\b_{0}^{\na})+\Gamma_{x'_{1}}=\div(\b^{\na}_{1})+\Gamma_{\s(x'_{1})}.
\end{equation}
Restricting both sides to $(X'-R')\times S$, and observing that  $\Gamma_{x'_{1}}$ and $\Gamma_{\s(x'_{1})}$ are disjoint when restricted to $(X'-R')\times S$, we see that $\Gamma_{\s(x'_{1})}\cap ((X'-R')\times S)$ is contained in $\div(\b^{\na}_{0})\cap ((X'-R')\times S)$. 

Now we consider the restriction of the diagram \eqref{Hmu2} to the formal completion $\Spec \cO_{x'}\htimes S$ at any $x'\in R'$. Since $D_{\pm}$ is disjoint from $R'$, after restricting to $\Spec \cO_{x'}\htimes S$ we may identify $\b_{i}$ and $\b^{\na}_{i}$. We may assume $S$ is affine, and by extending $k$ we may assume $k(x')=k$. Choose a uniformizer $\vp$ at $x'$ such that $\s(\vp)=-\vp$, then $\Spec \cO_{x'}\htimes S=\Spec \cO_{S}[[\vp]]$. After trivializing $\cL_{i}, \cL'_{i}(R')$ near $x'\times S$, we may assume $f_{1}=f'_{1}=\vp-a$ for some $a\in \cO_{S}$, $\a_{0}=\a_{1}\in \cO_{S}[[\vp]]$, The diagram \eqref{Hmu2} implies the equation in $\cO_{S}[[\vp]]$
\begin{equation*}
f'_{1}\cdot \b_{0}=\s^{*}f_{1}\cdot \b_{1},
\end{equation*}
where $\b_{0},\b_{1}\in \cO_{S}[[\vp]]$.  This equation is the same as
\begin{equation}\label{beta eqn}
(\vp-a)\b_{0}(\vp)=(-\vp-a)\b_{1}(\vp).
\end{equation}
Recall we also have the condition $\b_{i}|_{R'\times S}=\a_{i}|_{R'\times S}$ for $i=0,1$, which implies that $\b_{0}(0)=\a_{0}(0)=\a_{1}(0)=\b_{1}(0)$, or $\b_{1}(\vp)=\vp\g(\vp)+\b_{0}(\vp)$ for some $\g\in\cO_{S}[[\vp]]$. Combining this with \eqref{beta eqn} we get
\begin{equation*}
2\vp\b_{0}(\vp)=(-\vp-a)\vp\g(\vp).
\end{equation*}
Since $\vp$ is not a zero divisor, we conclude that $\b_{0}(\vp)=-(\vp+a)\g(\vp)/2$, hence $\vp+a$ divides $\b_{0}(\vp)$. This implies that $\Gamma_{\s(x'_{1})}\cap (\Spec \cO_{x'}\htimes S)$ is contained in $\div(\b_{0})\cap (\Spec \cO_{x'}\htimes S)=\div(\b^{\na}_{0})\cap (\Spec \cO_{x'}\htimes S)$. The proof of the claim is complete.
\end{proof}

On the other hand, the condition that $\b^{\na}_{0}$ vanishes along $\Gamma_{\s(x'_{1})}$ is sufficient for the existence of $\b_{1}$ making \eqref{Hmu2} commutative. Therefore, in this case, $\Hk^{\mu,\mu'}_{\cM}$ is the incidence correspondence for the divisor of $\b^{\na}$ in $\cM_{d}(\mu_{\Sig},\mu'_{\Sig})$ under the description of Lemma \ref{l:beta van}.  This gives the isomorphism $\Xi_{\cH}: \cH_{\un\mu\un\mu'}\times\frSi'\cong \Hk^{\mu,\mu'}_{\cM}$.

Case 2. $\mu_{1}\ne \mu'_{1}$. Let us consider only the case $\mu_{1}=1, \mu'_{1}=-1$. We only indicate the modifications from the previous case. In this case, $\cL_{1}=\cL_{0}(\Gamma_{x'_{1}})$ but $\cL'_{1}=\cL'_{0}(-\Gamma_{x'_{1}})$. We may change $\cL'_{1}$ to $\cL'_{0}(\Gamma_{\s(x'_{1})})$ (which has the same image as $\cL'_{0}(-\Gamma_{x'_{1}})$ in $\Bun_{T}$) so that $\deg\cL'_{1}-\deg\cL_{1}=d$ still holds. The diagrams \eqref{Hmu1} and \eqref{Hmu2} now become
\begin{equation}\label{Hmu3}
\xymatrix{ \cL_{0}\ar[d]^{\a^{\na}_{0}}\ar@{^{(}->}[r] & \cL_{1}\ar@{=}[r]\ar[d]^{\a^{\na}_{1}} & \cL_{0}(\Gamma_{x'_{1}}) \\
\cL'_{0}(R'-D_{-})\ar@{^{(}->}[r] &\cL'_{1}(R'-D_{-})\ar@{=}[r] & \cL'_{0}(R'-D_{-}+\Gamma_{\s(x'_{1})})
}
\end{equation}
\begin{equation}\label{Hmu4}
\xymatrix{\s^{*}\cL_{0}\ar[d]^{\b^{\na}_{0}}\ar@{^{(}->}[r] & \s^{*}\cL_{1}\ar@{=}[r]\ar[d]^{\b^{\na}_{1}} & (\s^{*}\cL_{0})(\Gamma_{\s(x'_{1})}) \\
\cL'_{0}(R'-D_{+})\ar@{^{(}->}[r] &\cL'_{1}(R'-D_{+})\ar@{=}[r] & \cL'_{0}(R'-D_{+}+\Gamma_{\s(x'_{1})})
}
\end{equation}
Now \eqref{Hmu4} imposes no condition on $\b_{0}$, but \eqref{Hmu3} gives
\begin{equation*}
\div(\a^{\na}_{0})+\Gamma_{\s(x'_{1})}=\div(\a^{\na}_{1})+\Gamma_{x'_{1}}.
\end{equation*}
An analog of the Claim in Case 1 says that $\a^{\na}_{0}$ must vanish along $\Gamma_{x'_{1}}$. Therefore, in this case, $\Hk^{\mu,\mu'}_{\cM}$ is the incidence correspondence for the divisor of $\a^{\na}$ in $\cM_{d}(\mu_{\Sig},\mu'_{\Sig})$ under the description of Lemma \ref{l:beta van}.  This gives the isomorphism $\Xi_{\cH}$.
\end{proof}

\subsection{Proof of Theorem \ref{th:Ir}}\label{ss:proof Ir}

\sss{Geometric facts}

We first collect some geometric facts about the stacks involved in the constructions in \S\ref{ss:pre M}.

\begin{prop}\label{p:geom facts}
\begin{enumerate}
\item\label{geom facts BunG} The stack $\Bun_{G}(\Sig)$ is smooth of pure dimension $3(g-1)+N$.
\item\label{geom facts HkG} The stack $\Hk^{r}_{G}(\Sig)$ is smooth of pure dimension $3(g-1)+N+2r$.
\item\label{geom facts BunT} The stack $\Bun_{T}$ is smooth, DM and proper over $k$ of pure dimension  $g'-g=g-1+\ha \r$.
\item\label{geom facts HkT} The stack $\Hk^{\un\mu}_{T}$ is smooth, DM and proper over $k$ of pure dimension  $g-1+\ha \r+r$.
\item\label{geom facts H} 
The morphisms $\oll{p_{H}},\orr{p_{H}}:H_{d}(\Sig)\to\Bun_{G}(\Sig)$ are representable and smooth of pure relative dimension $2d$. In particular, $H_{d}(\Sig)$ is a smooth algebraic stack over $k$ of pure dimension $2d+3(g-1)+N$.
\item\label{geom facts HkH} The stack $\Hk^{r}_{H,d}(\Sig)$ has dimension $2d+2r+3(g-1)+N$.
\item\label{geom facts M} For $d\ge 2g'-1+N$, $\cM_{d}(\mu_{\Sig},\mu'_{\Sig})$ is a smooth and separated DM stack pure of dimension $m=2d+\r-N-g+1$.
\item\label{geom facts ShtM} Let $D$ be an effective divisor on $U$. The stack $\Sht^{\mu,\mu'}_{\cM,D}$ is proper over $k$.
\end{enumerate}
\end{prop}
\begin{proof} 
\eqref{geom facts BunG}, \eqref{geom facts BunT} and \eqref{geom facts HkT} are standard facts. \eqref{geom facts HkG} follows from Prop. \ref{p:Hk smooth}(4).

\eqref{geom facts H} 
Recall the stack $H_{d}$ defined in \cite[\S6.3.2]{YZ}, with two maps $\oll{p}, \orr{p}$ to $\Bun_{G}$. We have an open embedding $H_{d}(\Sig)\incl \Bun_{G}(\Sig)\times_{\Bun_{G}, \oll{p}}H_{d}$ because once the $\Sig$-level structure of $\cE$ is chosen, it induces a unique  $\Sig$-level structure on $\cE'$ via $\ph$ (which is assumed to be an isomorphism near $\Sig$). Since $\oll{p}: H_{d}\to \Bun_{G}$ is smooth of relative dimension $2d$ by \cite[Lemma 6.8(1)]{YZ}, so is its base change $\oll{p_{H}}$. Similar argument works for $\orr{p_{H}}$. 

\eqref{geom facts HkH} As in \cite[\S6.3.4]{YZ}, we have a map $\Hk^{r}_{H,d}(\Sig)\to \Bun_{G}(\Sig)\times U_{d}\times X^{r}$ (the first factor records $\cE^{\da}_{0}$, second records the divisor of $\det(\ph_{0})$ and the third records $x_{i}$). The same argument as \cite[Lemma 6.10]{YZ} shows that all geometric fibers of this map have dimension $d+r$ (note that the horizontal maps are allowed to vanish at points in $\Sig$, but this does not complicate the argument because the vertical maps do not vanish at $\Sig$). Therefore $\dim \Hk^{r}_{H,d}(\Sig)=d+r+d+r+\dim \Bun_{G}(\Sig)=2d+2r+3(g-1)+N$.

\eqref{geom facts M} By Prop. \ref{p:compare M}, $\cM_{d}(\mu_{\Sig},\mu'_{\Sig})\cong \cM_{d}\times \frSi'$. Therefore, the required geometric properties of $\cM_{d}(\mu_{\Sig},\mu'_{\Sig})$ follow from those of $\cM_{d}$ proved in Prop. \ref{p:M}(1).

\eqref{geom facts ShtM} Consider the Cartesian diagram \eqref{ShtMD}. Since $\Sht'^{r}_{G}(\Sii)$ is separated over $\frSi'$ by Prop. \ref{p:ShtG} and $\oll{p}': \Sht'^{r}_{G}(\Sii;h_{D})\to \Sht'^{r}_{G}(\Sii)$ is proper by Lemma \ref{l:fiber ShthD}(1), the map $(\oll{p}', \orr{p}'): \Sht'^{r}_{G}(\Sii;h_{D})\to \Sht'^{r}_{G}(\Sii)\times_{\frSi'}\Sht'^{r}_{G}(\Sii)$ is proper. This implies $\Sht^{\mu,\mu'}_{\cM,D}\to \Sht^{\un\mu}_{T}(\mi\cdot\Si')\times_{\frSi'}\Sht^{\un\mu'}_{T}(\mi'\cdot\Si')$ is proper. Since $\Sht^{\un\mu}_{T}(\mi\cdot\Si')$ and $\Sht^{\un\mu'}_{T}(\mi'\cdot\Si')$ are proper over $k$ by Corollary \ref{c:ShtT proper}, so is $\Sht^{\mu,\mu'}_{\cM,D}$.
\end{proof}

\begin{prop} Suppose $D$ is an effective divisor on $U$ of degree $d\ge \max\{2g'-1+N, 2g\}$. Then the diagram \eqref{Tian} satisfies all the conditions for applying the Octahedron Lemma \cite[Theorem A.10]{YZ}.
\end{prop}
\begin{proof}
We refer to \cite[Theorem A.10]{YZ} for the statement of the conditions. 

Condition (1): we need to show the smoothness of all members in the diagram \eqref{Tian} except for $\Hk'^{r}_{H,d}(\Sig)$. This is done in Prop. \ref{p:geom facts}.

Condition (2): we need to check that $\cM_{d}(\mu_{\Sig},\mu'_{\Sig}), \cM_{d}(\mu_{\Sig},\mu'_{\Sig})^{2}, \Sht^{\un\mu}_{T}(\mi\cdot\Si')\times_{\frSi'}\Sht^{\un\mu'}_{T}(\mi'\cdot\Si')$ and $\Sht'^{r}_{G}(\Sii)\times_{\frSi'}\Sht'^{r}_{G}(\Sii)$ are smooth of the expected dimensions. These facts follow from Prop. \ref{p:geom facts}\eqref{geom facts M}, Corollary \ref{c:ShtT proper} and Prop. \ref{p:ShtG}.

Condition (3): we need to show that the diagrams \eqref{defn HkM} and \eqref{defn ShtH} satisfy either the conditions in \cite[\S A.2.8]{YZ}, or the conditions in \cite[\S A.2.10]{YZ}.

We first show that  \eqref{defn HkM}  satisfies the conditions in \cite[\S A.2.8]{YZ}. We claim that $\Hk^{\mu,\mu'}_{\cM,d}$ is a DM stack that admits a finite flat presentation. By Prop. \ref{p:compare M}, $\cM_{d}(\mu_{\Sig},\mu'_{\Sig})\cong \cM_{d}\times \frSi'$. By Prop. \ref{p:M}\eqref{M ffp}, $\cM_{d}$ is DM and admits a finite flat presentation, therefore the same is true for $\cM_{d}(\mu_{\Sig},\mu'_{\Sig})$. Since the map $p_{\cM,0}: \Hk^{\mu,\mu'}_{\cM,d}\to \cM_{d}(\mu_{\Sig},\mu'_{\Sig})$ is schematic, the same is true for $\cM_{d}(\mu_{\Sig},\mu'_{\Sig})$. It remains to check that $\th^{\mu,\mu'}_{\Hk}$ can be factored into a regular local immersion and a smooth relative DM map. It suffices to show the same thing for $\th^{\mu}_{\Hk}: \Hk^{\un\mu}_{T}\times\frSi'\to \Hk'^{r}_{G}(\Sig)$ (and the same result applies to $\mu'$ as well).  The argument is similar to that in \cite[Lemma 6.11(1)]{YZ}, and we only give a sketch here. We may enlarge the set $\Sig$ to $\wt\Sig\subset |X-R|$ such that $\deg\wt\Sig>\r/2$. By enlarging the base field $k$, we may assume that all points in $\nu^{-1}(\wt\Sig)$ are defined over $k$. Choose a section of $\nu^{-1}(\wt\Sig)\to \wt\Sig$ extending the existing section $\mu_{f}$, and call this section $\wt\Sig'$. Using $\wt\Sig'$ we have a map $\wt\th^{\mu}_{\Hk}: \Hk^{\un\mu}_{T}\to \Hk'^{r}_{G}(\wt\Sig)$. Since the projection $\Hk'^{r}_{G}(\wt\Sig)\to \Hk'^{r}_{G}(\Sig)$ is smooth and schematic, it suffices to show that $\wt\th^{\mu}_{\Hk}: \Hk^{\un\mu}_{T}=\Bun_{T}\times X'^{r}\to \Hk'^{r}_{G}(\wt\Sig)$ is a regular local embedding. To check this, we calculate the tangent map of $\wt\th^{\mu}_{\Hk}$ at a geometric point $b=(\cL, x'_{1},\dotsc, x'_{r})\in \Bun_{T}(K)\times X'^{r}(K)$. Or rather we calculate the relative tangent map with respect to the projections to $X'^{r}$. We base change to $K$ without changing notation. The relative tangent complex of $\Hk^{\un\mu}_{T}$ at $b$ is $\cohog{*}{X, \cO_{X'}/\cO_{X}}[1]$. The relative tangent complex of $\Hk'^{r}_{G}(\wt\Sig)$ at $\wt\th^{\mu}_{\Hk}(b)$ is $\cohog{*}{X, \Ad^{\un{x'}, \wt\Sig}(\nu_{*}\cL)}[1]$, where $\Ad^{\un{x'}, \wt\Sig}(\nu_{*}\cL)= \un\End^{\un{x'}, \wt\Sig}(\nu_{*}\cL)/\cO_{X}$, and $\End^{\un{x'}, \wt\Sig}(\nu_{*}\cL)$ is the endomorphism sheaf of the chain $\nu_{*}\cL\to \nu_{*}(\cL(x'_{1}))\to \cdots$ preserving the level structures at $\wt\Sig$. The tangent map of $\wt\th^{\mu}_{\Hk}$ is induced by a natural embedding $e: \nu_{*}\cO_{X'}/\cO_{X}\incl \Ad^{\un{x'}, \wt\Sig'}(\nu_{*}\cL)$. A calculation similar to Lemma \ref{l:ab} gives
\begin{equation*}
\End^{\un{x'}, \wt\Sig}(\nu_{*}\cL)\subset \nu_{*}(\cO_{X'}(R'))\oplus_{R'}\nu_{*}(\s^{*}\cL^{-1}\ot\cL(R'-\wt\Sig''))
\end{equation*}
where $\wt\Sig''=\s(\wt\Sig')$. Therefore we have
\begin{equation*}
\Ad^{\un{x'}, \wt\Sig}(\nu_{*}\cL)\subset (\nu_{*}(\cO_{X'}(R'))/\cO_{X})\oplus_{R'}\nu_{*}(\s^{*}\cL^{-1}\ot\cL(R'-\wt\Sig''))
\end{equation*} 
under which $e$ corresponds to the embedding of $\nu_{*}\cO_{X'}/\cO_{X}$ into the first factor. One checks the projection $\coker(e)\to \nu_{*}(\s^{*}\cL^{-1}\ot\cL(R'-\wt\Sig''))$ is injective, the latter having degree $\r/2-\deg\wt\Sig<0$, we have $\cohog{0}{X,\coker(e)}=0$, which implies that the tangent map of $\wt\th^{\mu}_{\Hk}$ is injective.

Next we show that \eqref{defn ShtH} satisfies the conditions in \cite[\S A.2.10]{YZ}. The argument is similar to that of \cite[Lemma 6.14(1)]{YZ}, using the smoothness of $H_{d}(\Sig)$ proved in Prop. \ref{p:geom facts}\eqref{geom facts H}.

Condition (4): we need to show that \eqref{defn ShtM} and \eqref{var ShtM} both satisfy the conditions in \cite[\S A.2.8]{YZ}. Again the argument is completely similar to the corresponding argument in the proof of \cite[Theorem 6.6]{YZ}. We omit details here.
\end{proof}

\sss{The cycle $\z$} Using the dimension calculations in Prop. \ref{p:geom facts}\eqref{geom facts HkH}\eqref{geom facts HkT} and \eqref{geom facts HkG},  we have
\begin{equation*}
\dim \Hk'^{r}_{H,d}(\Sig)+\dim(\Hk^{\un\mu}_{T}\times\Hk^{\un\mu'}_{T}\times\frSi')-2\dim \Hk'^{r}_{G}(\Sig)=m=2d+\r-N-g+1.
\end{equation*}
Therefore the Cartesian diagram \eqref{defn HkM} defines a cycle
\begin{equation}\label{defn zeta}
\z=(\th^{\mu,\mu'}_{\Hk})^{!}[\Hk'^{r}_{H,d}(\Sig)]\in \Ch_{m}(\Hk^{\mu,\mu'}_{\cM,d}).
\end{equation}

\begin{lemma}\label{l:z fund cycle} Assume $d\ge \max\{2g'-1+N, 2g+N\}$. Let $\z^{\sh}\in \Ch_{*}(\cH_{\un\mu\un\mu'}\times\frSi')$ be the pullback of $\z$ under the isomorphism $\Xi_{\cH}$. Then when restricted over $\cA^{\dm}_{d}$, $\z^{\sh}$ coincides with the fundamental class of $\cH_{\un\mu\un\mu'}\times\frSi'$.
\end{lemma}
\begin{proof}
We have a map $\Hk^{r}_{H,d}(\Sig)\to U_{d}\times X^{r}$ similar to the one defined in \cite[\S6.3.4]{YZ}. Let $(U_{d}\times X^{r})^{\circ}$ be the open subset consisting of $(D, x_{1},\dotsc, x_{r})$ such that each $x_{i}$ is disjoint from the support of $D$. Let $\Hk'^{r,\circ}_{H,d}(\Sig)$ be the preimage of $(U_{d}\times X^{r})^{\circ}$. Similarly, let $\Hk^{\mu,\mu',\circ}_{\cM,d}$ be the preimage of $(U_{d}\times X^{r})^{\circ}$ in $\Hk^{\mu,\mu'}_{\cM,d}$, which corresponds under $\Xi_{\cH}$ to an open subset of the form $\cH^{\circ}_{\un\mu\un\mu'}\times\frSi'$. 

We have a map $\Hk^{r}_{H,d}(\Sig)\to \Hk^{r}_{G}(\Sig)\times_{\Bun_{G}(\Sig)}H_{d}(\Sig)$ by considering the top row and left column of the diagram \eqref{diag EE'}.  When restricted to $(U_{d}\times X^{r})^{\circ}$, this map is an isomorphism. Therefore  $\Hk^{r,\circ}_{H,d}(\Sig)$, hence $\Hk'^{r,\circ}_{H,d}(\Sig)$ is smooth of dimension $3(g-1)+N+2r+2d$. Restricting the diagram \eqref{defn HkM} to $(U_{d}\times X^{r})^{\circ}$, $\Hk^{\mu,\mu',\circ}_{\cM,d}$ is the intersection of smooth stacks with the expected dimension $\dim \Hk'^{r,\circ}_{H,d}(\Sig)+\dim (\Hk^{\un\mu}_{T}\times \Hk^{\un\mu'}_{T}\times\frSi')- \dim \Hk'^{r}_{G}(\Sig)=m$, therefore, $\z$ is the fundamental class when restricted to $\Hk^{\mu,\mu',\circ}_{\cM,d}=\cH^{\circ}_{\un\mu\un\mu'}\times\frSi'$. 

It remains to show that $\dim (\cH^{\dm}_{\un\mu\un\mu'}-\cH^{\circ}_{\un\mu\un\mu'})<\dim \cH^{\dm}_{\un\mu\un\mu'}$. The map $\cH^{\dm}_{\un\mu\un\mu'}\to \cM_{d}^{\dm}\to \cA^{\dm}_{d}$ are finite surjective. On the other hand, as in \cite[\S6.4.3]{YZ}, the image of  $\cH^{\dm}_{\un\mu\un\mu'}-\cH^{\circ}_{\un\mu\un\mu'}$ in $\cA^{\dm}_{d}$ lies in the closed substack $\cC_{d}$ consisting of those $(\D, \Th_{R}, \io, a, b, \vth_{R})$ where $\div(a)$ and $\div(b)$ (both are divisors of degree $d+\r$ on $X$) have one point in common which lies in $U$. Therefore it suffices to show that $\dim \cC_{d}<\dim \cA_{d}=m$. Now $\cC_{d}$  is contained in the image of a map  $U\times (X^{\sqR}_{d+\r-N_{-}-1}\times_{\Pic^{\sqR;\sqR}_{X}} X^{\sqR}_{d+\r-N_{+}-1})\to X^{\sqR}_{d+\r-N_{-}}\times_{\Pic^{\sqR;\sqR}_{X}} X^{\sqR}_{d+\r-N_{+}}$. Using $d\ge 2g+N$ we may calculate the dimension of $X^{\sqR}_{d+\r-N_{-}-1}\times_{\Pic^{\sqR;\sqR}_{X}} X^{\sqR}_{d+\r-N_{+}-1}$ by Riemann-Roch, from which we conclude again that $\dim \cC_{d}\le m-1$. This completes the proof.
\end{proof}

\sss{} Consider the cycle
\begin{equation*}
(\id,\Fr_{\cM_{d}(\mu_{\Sig},\mu'_{\Sig})})^{!}\z\in \Ch_{0}(\Sht^{\mu,\mu'}_{\cM,d}).
\end{equation*}
This is well-defined because $\cM_{d}(\mu_{\Sig},\mu'_{\Sig})$ is smooth DM by Prop. \ref{p:geom facts}\eqref{geom facts M}, hence $(\id,\Fr)$ is a regular local immersion. Let
\begin{equation*}
((\id,\Fr_{\cM_{d}(\mu_{\Sig},\mu'_{\Sig})})^{!}\z)_{D}\in \Ch_{0}(\Sht^{\mu,\mu'}_{\cM,D})
\end{equation*}
be its $D$-component. Since $\Sht^{\mu,\mu'}_{\cM,D}$ is proper by Prop. \ref{p:geom facts}\eqref{geom facts ShtM}, it makes sense to take degrees of $0$-cycles on it. Hence we define
\begin{equation*}
\jiao{\z, \Gamma(\Fr_{\cM_{d}(\mu_{\Sig},\mu'_{\Sig})})}_{D}:=\deg((\id,\Fr_{\cM_{d}(\mu_{\Sig},\mu'_{\Sig})})^{!}\z)_{D}\in\QQ.
\end{equation*}

\begin{theorem}\label{th:pre Ir} Suppose  $D$ is an effective divisor on $U$ of degree $d\ge \max\{2g'-1+N, 2g\}$. We have 
\begin{equation}\label{pre Ir}
\left(\prod_{x'\in \Si'}d_{x'}\right)\cdot\II^{\mu,\mu'}(h_{D})=\jiao{\z, \Gamma(\Fr_{\cM_{d}(\mu_{\Sig},\mu'_{\Sig})})}_{D}.
\end{equation}
\end{theorem}
\begin{proof}
From the definition of Heegner--Drinfeld cycles, it is easy to see using the diagram \eqref{ShtMD} that
\begin{equation}\label{pre Sht Ir}
\left(\prod_{x'\in \Si'}d_{x'}\right)\cdot\II^{\mu,\mu'}(h_{D})=\deg\left((\th'^{\mu}\times\th'^{\mu'})^{!}[\Sht'^{r}_{G}(\Sii;h_{D})]\right).
\end{equation}

On the other hand,  applying the Octahedron Lemma \cite[Theorem A.10]{YZ} to \eqref{Tian}, we get that
\begin{eqnarray}\label{oct result}
\notag &&(\th'^{\mu}\times\th'^{\mu'})^{!}(\id, \Fr_{H_{d}(\Sig)})^{!}[\Hk'^{r}_{H,d}(\Sig)\times\frSi']\\
\notag &=&(\id,\Fr_{\cM_{d}(\mu_{\Sig},\mu'_{\Sig})})^{!}(\th^{\mu,\mu'}_{\Hk}\times \id_{\frSi'})^{!}[\Hk'^{r}_{H,d}(\Sig)\times\frSi']\\
&=&(\id,\Fr_{\cM_{d}(\mu_{\Sig},\mu'_{\Sig})})^{!}\z\in \Ch_{0}(\Sht^{\mu,\mu'}_{\cM,d}).
\end{eqnarray}
If we can show that 
\begin{equation}\label{ShtH ok}
(\id, \Fr_{H_{d}(\Sig)})^{!}[\Hk'^{r}_{H,d}(\Sig)\times\frSi']=[\Sht'^{r}_{H,d}(\Sii)]
\end{equation}
then extracting the $D$-components of \eqref{oct result} and \eqref{ShtH ok} identifies $[(\th'^{\mu}\times\th'^{\mu'})^{!}[\Sht'^{r}_{G}(\Sii;h_{D})]$ with the cycle $((\id,\Fr_{\cM_{d}(\mu_{\Sig},\mu'_{\Sig})})^{!}\z)_{D}$.  Taking degrees then identifies the right side of  \eqref{pre Sht Ir} with the right side of \eqref{pre Ir}, and we are done. Therefore it remains to show \eqref{ShtH ok}. The argument is similar to \cite[Lemma 6.14(2)]{YZ}. Let $\Sht'^{r,\circ}_{H,d}(\Sii)\subset\Sht'^{r}_{H,d}(\Sii)$ be the preimage of $(U_{d}\times X^{r})^{\circ}$. By \eqref{decomp ShtH}, $\Sht'^{r,\circ}_{H,d}(\Sii)$ is the disjoint union over $D\in U_{d}(k)$ of $(\Sht^{r}_{G}(\Sii; h_{D})|_{(X-D)^{r}}\times_{X^{r}}X'^{r}$. By Lemma \ref{l:fiber ShthD}(2), $\Sht^{r}_{G}(\Sii; h_{D})|_{(X-D)^{r}}$ is smooth of dimension $2r$, which is the expected dimension from the diagram \eqref{defn ShtH}. Therefore, the restriction of $(\id, \Fr_{H_{d}(\Sig)})^{!}[\Hk'^{r}_{H,d}(\Sig)\times\frSi']$  to $\Sht'^{r,\circ}_{H,d}(\Sii)$ is the fundamental class. By Lemma \ref{l:fiber ShthD}(3), $\Sht^{r}_{G}(\Sii; h_{D})$ has the same dimension as its restriction over $(X-D)^{r}$, hence $\dim \Sht'^{r,\circ}_{H,d}(\Sii)=\Sht'^{r}_{H,d}(\Sii)$, therefore \eqref{ShtH ok} holds as cycles on the whole of $\Sht'^{r}_{H,d}(\Sii)$. This finishes the proof.
\end{proof}

\sss{Proof of Theorem \ref{th:Ir}} Now we can deduce  Theorem \ref{th:Ir} from Theorem \ref{th:pre Ir}.

Consider the diagram \eqref{defn ShtM}.  Moving the Atkin--Lehner automorphism of $\cM_{d}(\mu_{\Sig},\mu'_{\Sig})$ from the vertical arrow to the horizontal arrow, we get another Cartesian diagram
\begin{equation}\label{ShtM sh}
\xymatrix{\Sht^{\mu,\mu'}_{\cM,d}\ar[d] \ar[rr] && \Hk^{\mu,\mu'}_{\cM}  \ar[d]^{(p_{\cM,0}, p_{\cM,r})}   \\
\cM_{d}(\mu_{\Sig},\mu'_{\Sig})\ar[rr]^-{(\id,\AL_{\cM,\infty}^{-1}\circ \Fr)} && \cM_{d}(\mu_{\Sig},\mu'_{\Sig})\times\cM_{d}(\mu_{\Sig},\mu'_{\Sig})}
\end{equation}
From this we get
\begin{equation}\label{first pullback z}
(\id, \Fr_{\cM_{d}(\mu_{\Sig},\mu'_{\Sig})})^{!}\z=(\id,\AL_{\cM,\infty}^{-1}\circ \Fr)^{!}\z\in \Ch_{0}(\Sht^{\mu,\mu'}_{\cM,d}).
\end{equation}

Define $\cS_{\un\mu\un\mu'}$ by the Cartesian diagram
\begin{equation}\label{ShtM ep}
\xymatrix{\cS_{\un\mu\un\mu'}\ar[r]\ar[d] & \cH_{\un\mu\un\mu'}\ar[d]^{(p_{\cH,0}, p_{\cH,r})}\\
\cM_{d}\ar[r]^-{(\id, \Fr)} & \cM_{d}\times\cM_{d}}
\end{equation}

Using the isomorphisms $\Xi_{\cM}$ and $\Xi_{\cH}$ established  in Prop. \ref{p:compare M} and \ref{p:HkM}, \eqref{ShtM sh} is isomorphic to the Cartesian diagram
\begin{equation}\label{ShtM ep'}
\xymatrix{\cS_{\un\mu\un\mu'}\times\frSi'\ar[rr]\ar[d] & &\cH_{\un\mu\un\mu'}\times\frSi'\ar[d]^{(p_{\cH,0}\times\id_{\frSi'}, p_{\cH,r}\times\id_{\frSi'})}\\
\cM_{d}\times\frSi'\ar[rr]^-{(\id, \Fr_{\cM_{d}}\times\id_{\frSi'})} & &(\cM_{d}\times\frSi')\times(\cM_{d}\times\frSi')}
\end{equation}
Here we are using Prop. \ref{p:compare M}(2) to identify $\AL_{\cM,\infty}^{-1}\circ \Fr$ on $\cM_{d}(\mu_{\Sig},\mu'_{\Sig})$ with $\Fr_{\cM_{d}}\times\id_{\frSi'}$ on $\cM_{d}\times\frSi'$. In particular, we get an isomorphism
\begin{equation*}
\Xi_{\cS}: \cS_{\un\mu\un\mu'}\times\frSi'\isom \Sht^{\mu,\mu'}_{\cM,d}.
\end{equation*}

Recall that $\z^{\sh}\in \Ch_{m}(\cH_{\un\mu\un\mu'}\times\frSi')$ is the transport of $\z$ under the isomorphism $\Xi_{\cH}$, then we have
\begin{equation}\label{two pullback z}
(\id, \AL^{-1}_{\cM,\infty}\circ \Fr)^{!}\z=(\id, \Fr_{\cM_{d}}\times\id_{\frSi'})^{!}\z^{\sh}\in \Ch_{0}(\cS_{\un\mu\un\mu'}\times\frSi').
\end{equation}
By Lemma \ref{l:z fund cycle},  $\z^{\sh}$ is the fundamental cycle of $\cH_{\un\mu\un\mu'}\times\frSi'$ when restricted to $\cA^{\dm}_{d}$. By Prop. \ref{p:M}\eqref{f small}, the complement of $\cM^{\dm}_{d}\times_{\cA^{\dm}_{d}}\cM_{d}^{\dm}$ in $\cM_{d}\times_{\cA_{d}}\cM_{d}$ has dimension strictly smaller than $\dim \cM_{d}$ (the condition $d\ge 2g'-1+N=4g-3+\r+N$ implies $d\ge3g-2+N$). Therefore, we may replace $\z^{\sh}$ with the fundamental cycle of the closure of $\cH_{\un\mu\un\mu'}|_{\cA^{\dm}_{d}}\times\frSi'$ and the intersection number on the right hand side of \eqref{two pullback z} does not change. We denote the latter by $\ov\cH^{\dm}_{\un\mu\un\mu'}\times\frSi'$. Combining \eqref{first pullback z} and \eqref{two pullback z} we get
\begin{eqnarray*}
&&(\id, \Fr_{\cM_{d}}\times\id_{\frSi'})^{!}\z^{\sh}\\
&=&(\id, \Fr_{\cM_{d}}\times\id_{\frSi'})^{!}[\ov\cH^{\dm}_{\un\mu\un\mu'}\times\frSi']\\
&=&((\id,\Fr_{\cM_{d}})^{!}[\ov\cH^{\dm}_{\un\mu\un\mu'}])\times[\frSi']\in\Ch_{0}(\cS_{\un\mu\un\mu'}\times\frSi').
\end{eqnarray*}
Taking the degree of the $D$-component, we get
\begin{equation*}
\jiao{\z, \Gamma(\Fr_{\cM_{d}(\mu_{\Sig},\mu'_{\Sig})})}_{D}=\deg(\frSi')\cdot\jiao{[\ov\cH^{\dm}_{\un\mu\un\mu'}],\Gamma(\Fr_{\cM_{d}})}_{D}.
\end{equation*}
Using Theorem \ref{th:pre Ir}, we get
\begin{eqnarray*}
\II^{\mu,\mu'}(h_{D})&=&\left(\prod_{x'\in \Si'}d_{x'}\right)^{-1}\jiao{\z, \Gamma(\Fr_{\cM_{d}(\mu_{\Sig},\mu'_{\Sig})})}_{D}\\
&=&\left(\prod_{x'\in \Si'}d_{x'}\right)^{-1}\deg(\frSi')\cdot\jiao{[\ov\cH^{\dm}_{\un\mu\un\mu'}],\Gamma(\Fr_{\cM_{d}})}_{D}\\
&=& \jiao{[\ov\cH^{\dm}_{\un\mu\un\mu'}],\Gamma(\Fr_{\cM_{d}})}_{D}.
\end{eqnarray*}

It remains to calculate $\jiao{[\ov\cH^{\dm}_{\un\mu\un\mu'}],\Gamma(\Fr_{\cM_{d}})}_{D}$.

Note that $\cH_{\un\mu\un\mu'}$ is a self-correspondence of $\cM_{d}$ over $\cA_{d}$. By the discussion in \cite[\S A.4.5]{YZ}, the map $\cS_{\un\mu\un\mu'}\to \cM_{d}\xr{f^{\fl}_{d}}\cA^{\fl}_{d}$ lands in the rational points  $\cA^{\fl}_{d}(k)$, hence we have a decomposition 
\begin{equation*}
\cS_{\un\mu\un\mu'}=\coprod_{a\in \cA^{\fl}_{d}(k)} \cS_{\un\mu\un\mu'}(a).
\end{equation*}
Under the isomorphism $\Xi_{\cM}$, this gives a refinement of the decomposition \eqref{decomp ShtM D}, namely
\begin{equation*}
\xymatrix{\Sht^{\mu,\mu'}_{\cM,D} && \coprod_{a\in \cA^{\fl}_{D}(k)}\cS_{\un\mu\un\mu'}(a)\times\frSi'\ar[ll]_-{\Xi_{\cS}}^-{\sim}.}
\end{equation*}

The fundamental cycle $[\ov\cH^{\dm}_{\un\mu\un\mu'}]$ gives a cohomological correspondence between the constant sheaf on $\cM_{d}$ and itself. It induces an endomorphism of the complex $\bR f_{d,!}\Ql$
\begin{equation*}
f_{d,!}[\ov\cH^{\dm}_{\un\mu\un\mu'}]: \bR f_{d,!}\Ql\to \bR f_{d,!}\Ql.
\end{equation*}
Taking direct image under $\Om$, we also get an endomorphism of $\bR f^{\fl}_{d,!}\Ql$
\begin{equation*}
f^{\fl}_{d,!}[\ov\cH^{\dm}_{\un\mu\un\mu'}]: \bR f^{\fl}_{d,!}\Ql\to \bR f^{\fl}_{d,!}\Ql.
\end{equation*}
Applying the Lefschetz trace formula \cite[Prop. A.12]{YZ} to the diagram \eqref{ShtM ep} (which is stated for $S$ being a scheme, so we apply it to the map $f^{\fl}_{d}$ rather than $f_{d}$), we get that 
\begin{equation}\label{pre Tr}
\jiao{[\ov\cH^{\dm}_{\un\mu\un\mu'}], \Gamma(\Fr_{\cM_{d}})}_{D}=\sum_{a\in \cA^{\fl}_{D}(k)}\Tr(f^{\fl}_{d,!}[\ov\cH^{\dm}_{\un\mu\un\mu'}]\circ \Fr_{a},(\bR f^{\fl}_{d,!}\Ql)_{a})
\end{equation}

Since $\cH_{\un\mu\un\mu'}$ is the composition of $r_{+}$ times $\cH_{+}$ and $r_{-}$ times $\cH_{-}$, the cohomological correspondence $[\ov\cH^{\dm}_{\un\mu\un\mu'}]$ is equal to the composition of $r_{+}$ times $[\ov\cH^{\dm}_{+}]$ and $r_{-}$ times $[\ov\cH^{\dm}_{-}]$ over $\cA^{\dm}_{d}$. By Prop. \ref{p:M}\eqref{f small}, the complement of $\cM^{\dm}_{d}\times_{\cA^{\dm}_{d}}\cM_{d}^{\dm}$ in $\cM_{d}\times_{\cA_{d}}\cM_{d}$ has dimension strictly smaller than $\dim \cM_{d}$, therefore $[\ov\cH^{\dm}_{\un\mu\un\mu'}]$ and the composition of $r_{+}$ times $[\ov\cH^{\dm}_{+}]$ and $r_{-}$ times $[\ov\cH^{\dm}_{-}]$ induce the same endomorphism on $f_{d,!}\Ql$ . This implies
\begin{equation*}
f_{d,!}[\ov\cH^{\dm}_{\un\mu\un\mu'}]=(f_{d,!}[\ov\cH^{\dm}_{+}])^{r_{+}}\circ (f_{d,!}[\ov\cH^{\dm}_{-}])^{r_{-}}\in\End(\bR f_{d,!}\Ql).
\end{equation*}
Taking direct image under $\Om$, we get
\begin{equation*}
f^{\fl}_{d,!}[\ov\cH^{\dm}_{\un\mu\un\mu'}]=(f^{\fl}_{d,!}[\ov\cH^{\dm}_{+}])^{r_{+}}\circ (f^{\fl}_{d,!}[\ov\cH^{\dm}_{-}])^{r_{-}}\in\End(\bR f^{\fl}_{d,!}\Ql).
\end{equation*}
This combined with \eqref{pre Tr} gives \eqref{Ir hD}. The proof of Theorem \ref{th:Ir} is now complete.

\section{The moduli stack $\cN_{d}$ and orbital integrals}\label{s:N}

In this section we introduce another moduli stack $\cN_{d}$, similar to $\cM_{d}$. The point-counting on $\cN_d$ is closely related to orbital integrals appearing in Jacquet's RTF we set up in \S\ref{s:RTF} for our specific test functions.

\subsection{Definition of $\cN_{d}$}\label{ss:defn N}

\sss{} 
Our moduli space $\cN_{d}$ depends on the ramification set $R$ with degree $\r$, a fixed finite set $\Sigma$ and a decomposition
\begin{equation*}
\Sig=\Sig_{+}\sqcup\Sig_{-},\quad N_\pm=\deg \Sig_{\pm}.
\end{equation*}
In our application, such a decomposition comes from a pair $\mu,\mu'\in \frT_{r,\Sig}$, for which we take $\Sig_{\pm}=\Sig_{\pm}(\mu,\mu')$ as in \eqref{Sig +} and \eqref{Sig -}. We are also assuming that $\Sigma\cap R=\vn$.

Let $d\ge0$ be an integer. Let $Q_{d}$ be the set of quadruples $\un{d}=(d_{11},d_{12},d_{21},d_{22})\in\ZZ^{4}_{\ge0}$ satisfying $d_{11}+d_{22}=d_{12}+d_{21}=d+\r$.

\begin{defn}\label{defn Nd} Let $\un d\in Q_{d}$. Let $\wt\cN_{\un{d}}=\wt\cN_{\un{d}}(\Sig_{\pm})$ be the stack whose $S$-points consist of 
$$(\cL^{\na}_{1}, \cL^{\na}_{2}, \cL'^{\na}_{1},\cL'^{\na}_{2}, \ph, \psi_{R})$$ where
\begin{itemize}
\item For $i=1,2$, $\cL^{\na}_{i}=(\cL_{i},\cK_{i,R},\io_{i})$ and $\cL'^{\na}_{i}=(\cL'_{i},\cK'_{i,R},\io'_{i})\in \Pic^{\sqR}_{X}(S)$, such that for any geometric point $s\in S$, $\deg (\cL'_{i}|_{X\times s})-\deg (\cL_{j}|_{X\times s})=d_{ij}$ for $i,j\in\{1,2\}$.
\item $\ph$ is an $\cO_{X\times S}$-linear map $\cL_{1}\op\cL_{2}\to \cL'_{1}\op\cL'_{2}$. We write it as a matrix
\begin{equation*}
\ph=\mat{\ph_{11}}{\ph_{12}}{\ph_{21}}{\ph_{22}}
\end{equation*}
where $\ph_{ij}: \cL_{j}\to \cL'_{i}$.
\item $\psi_{R}$ is an $\cO_{R\times S}$-linear map $\cK_{1,R}\op\cK_{2,R}\to \cK'_{1,R}\op\cK'_{2,R}$. Again we write $\psi_{R}$ as a matrix
\begin{equation*}
\psi_{R}=\mat{\psi_{11,R}}{\psi_{12,R}}{\psi_{21,R}}{\psi_{22,R}}
\end{equation*}
with $\psi_{ij,R}: \cK_{j,R}\to \cK'_{i,R}$. 
\end{itemize}
These data are required to satisfy the following conditions.
\begin{enumerate}
\setcounter{enumi}{-1}
\item The following diagram is commutative for $1\le i,j\le 2$
\begin{equation}\label{phiij}
\xymatrix{ \cK^{\ot2}_{j,R}\ar[r]^{\psi_{ij,R}^{\ot2}}\ar[d]^{\io_{j}} & \cK'^{\ot2}_{i,R}\ar[d]^{\io'_{i}}         \\
\cL_{j}|_{R\times S}\ar[r]^{\ph_{ij}|_{R\times S}} & \cL'_{i}|_{R\times S}}
\end{equation}
\item $\ph_{22}|_{\Sig_{-}\times S}=0$; $\ph_{11}|_{\Sig_{+}\times S}$ and $\ph_{22}|_{\Sig_{+}\times S}$ are nowhere vanishing. 
\item $\ph_{21}|_{\Sig_{+}\times S}=0$; $\ph_{12}|_{\Sig_{-}\times S}$ and $\ph_{21}|_{\Sig_{-}\times S}$ are nowhere vanishing. 
\item\label{det ph R} $\det(\psi_{R})=0$. Moreover, $\det(\ph)$ vanishes only to the first order along $R\times S$ (by \eqref{phiij} and $\det(\psi_{R})=0$, $\det(\ph)$ does vanish along $R\times S$).
\item\label{det nonvan} This condition is only non-void when $\Sig=\vn$ and $R=\vn$: $\det(\ph)$ is not identically zero on $X\times s$ for any geometric point $s$ of $S$. 
\item\label{ph degree} For each geometric point $s\in S$ the following conditions hold. If $d_{11}<d_{22}-N_{-}$, then $\ph_{11}|_{X\times s}\ne0$; if $d_{11}\ge d_{22}-N_{-}$, then $\ph_{22}|_{X\times s}\ne0$. If $d_{12}<d_{21}-N_{+}$ then $\ph_{12}|_{X\times s}\ne0$; if $d_{12}\ge d_{21}-N_{+}$ then $\ph_{21}|_{X\times s}\ne0$.
\end{enumerate}
\end{defn}

There is an action of $\Pic^{\sqR}_{X}$ on $\wt\cN_{\un d}$ by twisting each $\cL^{\na}_{i}$ and $\cL'^{\na}_{i}$ simultaneously ($i=1,2$). Let $\cN_{\un d}$ be the quotient
\begin{equation*}
\cN_{\un d}:=\wt\cN_{\un d}/\Pic^{\sqR}_{X}. 
\end{equation*}
Let $\cN_{d}$ be the disjoint union
\begin{equation*}
\cN_{d}=\coprod_{\un{d}\in Q_{d}}\cN_{\un d}.
\end{equation*}

\sss{}\label{sss:jij} Next we give an alternative description of $\cN_{d}$ in the style of \cite[\S3]{YZ}, which makes its similarity with $\cM_{d}$ more transparent.

Let $(\cL^{\na}_{1}, \cL^{\na}_{2}, \cL'^{\na}_{1}, \cL'^{\na}_{2},\ph,\psi_{R})\in\cN_{\un d}(S)$. For $i,j\in\{1,2\}$, define $\cL^{\na}_{ij}=\cL^{\na, \ot-1}_{j}\ot\cL'^{\na}_{i}=(\cL^{\ot-1}_{j}\ot\cL'_{i}, \cK^{\ot-1}_{j,R}\ot\cK'_{i,R}, \io^{-1}_{j}\ot\io'_{i})$. We have $\cL^{\na}_{ij}\in \Pic^{\sqR}_X(S)$.  By the diagram \eqref{phiij}, $(\cL^{\na}_{ij}, \ph_{ij}, \psi_{ij,R})$ defines a point in $\hX^{\sqR}_{d_{ij}}(S)$.

For $(i,j)=(1,1)$ or $(1,2)$, we thus have a morphism $\j_{ij}: \cN_{\un d}\to \hX^{\sqR}_{d_{ij}}$ sending the data $(\cL^{\na}_{1}, \cL^{\na}_{2}, \cL'^{\na}_{1}, \cL'^{\na}_{2},\ph,\psi_{R})\in\cN_{\un d}(S)$ to $(\cL^{\na}_{ij}, \ph_{ij}, \psi_{ij,R})\in \hX^{\sqR}_{d_{ij}}(S)$. 

The condition $\ph_{21}|_{\Sig_{+}\times S}=0$ allows us to view $\ph_{21}$ as a section of $\cL_{21}(-\Sig_{+})$, which has degree $d_{21}-N_{+}$ and extends to a point $\cL^{\na}_{21}(-\Sig_{+})\in\Pic^{\sqR}_{X}(S)$ using the original $\cK_{21,R}=\cK^{\ot-1}_{1}\ot\cK'_{2}$ and $\io^{-1}_{1}\ot\io'_{2}$ (because $\Sig_{+}\cap R=\vn$). We then define a morphism $\j_{21}: \cN_{\un d}\to \hX^{\sqR}_{d_{21}-N_{+}}$ sending $(\cL^{\na}_{1}, \cL^{\na}_{2}, \cL'^{\na}_{1}, \cL'^{\na}_{2},\ph,\psi_{R})$ to $(\cL^{\na}_{21}(-\Sig_{+}), \ph_{21}, \psi_{21,R})$.  Similarly we can define $\j_{22}: \cN_{\un d}\to \hX^{\sqR}_{d_{22}-N_{-}}$.  We have constructed a morphism
\begin{equation*}
\j_{\un d}=(\j_{ij})_{i,j\in\{1,2\}}: \cN_{\un d}\to \hX^{\sqR}_{d_{11}}\times\hX^{\sqR}_{d_{22}-N_{-}}\times\hX^{\sqR}_{d_{12}}\times\hX^{\sqR}_{d_{21}-N_{+}}.
\end{equation*}

In the above construction, we have canonical isomorphisms $\cL_{11}\ot\cL_{22}\cong\cL_{12}\ot\cL_{21}$ and $\cK_{11,R}\ot\cK_{22,R}\cong \cK_{12,R}\ot\cK_{21,R}$, which give a canonical isomorphism
\begin{equation}\label{4 term prod}
\cL^{\na}_{11}\ot\cL^{\na}_{22}\cong \cL^{\na}_{12}\ot\cL^{\na}_{21}\in \Pic^{\sqR,d+\r}_{X}(S).
\end{equation}
Moreover, the condition that $\det(\psi_{R})=0$ implies that $\psi_{11,R}\psi_{22,R}=\psi_{12,R}\psi_{21,R}$. Therefore, the isomorphism \eqref{4 term prod} extends to an isomorphism
\begin{equation*}
(\cL^{\na}_{11}\ot\cL^{\na}_{22},\psi_{11,R}\psi_{22,R})\cong (\cL^{\na}_{12}\ot\cL^{\na}_{21},\psi_{12,R}\psi_{21,R})\in \Pic^{\sqR;\sqR,d+\r}_{X}(S).
\end{equation*}
Therefore $\j_{\un d}$ lifts to a morphism
\begin{equation}\label{mor j d}
\j_{\un d}:\cN_{\un d}\to (\hX^{\sqR}_{d_{11}}\times \hX^{\sqR}_{d_{22}-N_{-}})\times_{\Pic^{\sqR;\sqR,d+\r}_{X}}(\hX^{\sqR}_{d_{12}}\times \hX^{\sqR}_{d_{21-N_{+}}}).
\end{equation}
Here the fiber product is formed using the following maps
\begin{eqnarray*}
&&\hX^{\sqR}_{d_{11}}\times \hX^{\sqR}_{d_{22}-N_{-}}\xr{(\wh\AJ^{\sqR,\sqR}_{d_{11}},\wh\AJ^{\sqR,\sqR}_{d_{22}-N_{-}})}\Pic^{\sqR;\sqR,d_{11}}_{X}\times \Pic^{\sqR;\sqR,d_{22}-N_{-}}_{X}\\
&&\xr{(\id, \otimes\dot\cO_{X}(\Sig_{-}))} \Pic^{\sqR;\sqR, d_{11}}_{X}\times \Pic^{\sqR;\sqR, d_{22}}_{X}\xr{\mult}\Pic^{\sqR;\sqR,d+\r}_{X}
\end{eqnarray*}
(where $\mult$ is the multiplication map for $\Pic^{\sqR;\sqR}_{X}$) and
\begin{eqnarray*}
&& \hX^{\sqR}_{d_{12}}\times \hX^{\sqR}_{d_{21}-N_{+}}\xr{(\wh\AJ^{\sqR;\sqR}_{d_{12}}, \wh\AJ^{\sqR;\sqR}_{d_{21}-N_{+}})}\Pic^{\sqR;\sqR, d_{12}}_{X}\times \Pic^{\sqR;\sqR, d_{21}-N_{+}}_{X}\\
&&\xr{(\id, \otimes\dot\cO_{X}(\Sig_{+}))} \Pic^{\sqR;\sqR, d_{12}}_{X}\times \Pic^{\sqR;\sqR, d_{21}}_{X}\xr{\mult}\Pic^{\sqR;\sqR, d+\r}_{X}.
\end{eqnarray*}

\sss{} We have a morphism to the base (cf. \S\ref{sss:base})
\begin{equation*}
g_{\un d}: \cN_{\un d}\to\cA_{d}=\cA_{d}(\Sig_{\pm})
\end{equation*}
sending $(\cL^{\na}_{1}, \cL^{\na}_{2}, \cL'^{\na}_{1},\cL'^{\na}_{2}, \ph, \psi_{R})$ to $(\D, \Th_{R}, \iota, a, b, \vth_{R})$ where $\D=\cL^{\ot-1}_{1}\ot\cL^{\ot-1}_{2}\ot\cL'_{1}\ot\cL'_{2}$, $\Th_{R}=\cK^{\ot-1}_{1,R}\ot\cK^{\ot-1}_{2,R}\ot\cK'_{1,R}\ot\cK'_{2,R}$, $\io_{R}$ is the obvious product of $\io_{1}\io_{2}$ and $\io'_{1}\io'_{2}$, $a=\ph_{11}\ph_{22}$, $b=\ph_{12}\ph_{21}$, $\vth_{R}=\psi_{11,R}\psi_{22,R}=\psi_{12,R}\psi_{21,R}$. 
We also have the composition
\begin{equation*}
g^{\fl}_{\un d}=\Om\circ g_{\un d} \colon \cN_{\un d}\xr{g_{\un d}}\cA_{d}\xr{\Om}\cA_{d}^{\fl}.
\end{equation*}

\begin{prop}\label{p:N} Let $\un{d}\in \Sig_{d}$. Then
\begin{enumerate}
\item The morphism $\j_{\un d}$ in \eqref{mor j d} is an open embedding, and $\cN_{\un d}$ is geometrically connected.
\item\label{N smooth} If $d\ge 4g-3+\r+N$, $\cN_{\un d}$ is a smooth DM stack of dimension $2d+\r-g-N+1=m$.
\item\label{N comm} The following diagram is commutative
\begin{equation}\label{N g comm}
\xymatrix{\cN_{\un d}\ar@{^{(}->}[r]^-{\j_{\un d}}\ar[d]^{g_{\un d}} & (\hX^{\sqR}_{d_{11}}\times \hX^{\sqR}_{d_{22}-N_{-}})\times_{\Pic^{\sqR;\sqR,d+\r}_{X}}(\hX^{\sqR}_{d_{12}}\times \hX^{\sqR}_{d_{21-N_{+}}})\ar[d]^{\wh\add^{\sqR}\times \wh\add^{\sqR}}\\
\cA_{d}\ar@{^{(}->}[r]^-{\om_{d}} & \hX^{\sqR}_{d+\r-N_{-}}\times_{\Pic^{\sqR;\sqR,d+\r}_{X}}\hX^{\sqR}_{d+\r-N_{+}}}
\end{equation}
\item\label{g proper} The morphisms $g_{\un d}$ and $g^{\fl}_{\un d}$ are proper.
\end{enumerate}
\end{prop}
\begin{proof}
The proofs of (1) and (3) are similar to their counterparts in \cite[Prop 3.1]{YZ}.

(2) We first show that $\cN_{\un d}$ is a DM stack. By conditions \eqref{det nonvan} and \eqref{ph degree} of Definition \ref{defn Nd}, at most one of $\ph_{ij}$ can be identically zero, so $\cN_{\un d}$ is covered by four open substacks $U_{ij}$, $i,j\in \{1,2\}$, in which only $\ph_{ij}$ is allowed to be zero (in fact two of these will be empty by condition \eqref{ph degree}). We will show that $U_{11}$ is a DM stack, and the argument for other $U_{ij}$ is similar. Since $U_{11}$ is open in
\begin{eqnarray*}
V_{11}=(\hX^{\sqR}_{d_{11}}\times X^{\sqR}_{d_{22}-N_{-}})\times_{\Pic^{\sqR;\sqR}_{X}}(X^{\sqR}_{d_{12}}\times X^{\sqR}_{d_{21-N_{+}}})
\end{eqnarray*}
it suffices to show $V_{11}$ is DM. The projection $V_{11}\to X^{\sqR}_{d_{22}-N_{-}}\times X^{\sqR}_{d_{12}}\times X^{\sqR}_{d_{21-N_{+}}}$ is schematic. By Lemma \ref{l:ev sm}(2), $X^{\sqR}_{n}$ is DM for any $n$, therefore $V_{11}$, hence $U_{11}$ is also DM.

We now prove the smoothness of $\cN_{\un d}$ in the case $d_{11}<d_{22}-N_{-}$ and  $d_{12}<d_{21}-N_{+}$; the other cases are similar. In this case the image of $\j_{\un d}$ lies in the open substack
\begin{equation*}
(X^{\sqR}_{d_{11}}\times \hX^{\sqR}_{d_{22}-N_{-}})\times_{\Pic^{\sqR;\sqR}_{X}}(X^{\sqR}_{d_{12}}\times \hX^{\sqR}_{d_{21-N_{+}}})
\end{equation*}
Since $d_{12}+(d_{21}-N)=d+\r-N\ge2(2g-1+\r)-1$ by assumption on $d$, and $d_{12}<d_{21}-N_{+}$,  we have $d_{21}-N_{+}\ge 2g-1+\r$. Similarly, we have $d_{22}-N_{-}\ge 2g-1+\r$. Therefore the Abel-Jacobi maps $\hX^{\sqR}_{d_{22}-N_{-}}\to \Pic^{\sqR;\sqR,d_{22}-N_{-}}_{X}$ and $\hX^{\sqR}_{d_{21}-N_{+}}\to \Pic^{\sqR;\sqR,d_{21}-N_{+}}_{X}$ are affine space bundles by Riemann-Roch, hence smooth. It therefore suffices to show the smoothness of
\begin{equation}\label{XRP}
\cQ:=(X^{\sqR}_{d_{11}}\times \Pic^{\sqR;\sqR,d_{22}-N_{-}}_{X})\times_{\Pic^{\sqR;\sqR}_{X}}(X^{\sqR}_{d_{12}}\times \Pic^{\sqR;\sqR,d_{21}-N_{+}}_{X}).
\end{equation}
We have the evaluation maps (by recording the square root line along $R$ and its section)
\begin{eqnarray*}
\ev^{\sqR}_{d_{ij}}: X^{\sqR}_{d_{ij}}\to [\Res^{R}_{k}\AA^1/\Res^{R}_{k}\Gm]\\
\ev^{\sqR}_{\Pic}: \Pic^{\sqR;\sqR}_{X}\to [\Res^{R}_{k}\AA^1/\Res^{R}_{k}\Gm]
\end{eqnarray*} 
which are both smooth, by Lemma \ref{l:ev sm}. To simplify notation, we write $$[\Res^{R}_{k}\AA^1/\Res^{R}_{k}\Gm]=[\AA^1/\Gm]_R.$$ Then the fiber product of these maps give a smooth map
\begin{equation*}
\ev^{\sqR}_{\cQ}: \cQ\to ([\AA^1/\Gm]_R \times [\AA^1/\Gm]_R )\times_{[\AA^1/\Gm]_R}([\AA^1/\Gm]_R \times [\AA^1/\Gm]_R).
\end{equation*}
Let  $C_{R}:=\Res^{R}_{k}\AA^{2}\times_{\Res^{R}_{k}\AA^1}\Res^{R}_{k}\AA^{2}$
with the two maps $\Res^{R}_{k}\AA^{2}\to \Res^{R}_{k}\AA^1$ both given by $(u,v)\mapsto uv$. Then the target of $\ev^{\sqR}_{\cQ}$ can be written as $[C_{R}/\Res^{R}_{k}\Gm^{3}]$ where the torus $\Gm^{3}$ is the subtorus of $\Gm^{4}$ consisting of $(u,v,s,t)$ such that $uv=st$.  Base change to $\kbar$, we have $C_{R,\kbar}\cong \prod_{x\in R(\kbar)}C_{x}$, where $C_{x}\subset \AA^{4}_{\kbar}$ is the cone defined by $uv-st=0$. Note that $C^{\circ}_{x}=C_{x}-\{(0,0,0,0)\}$ is smooth over $\kbar$.  The product $\prod_{x\in R(\kbar)}C^{\circ}_{x}$ defines a smooth open subset $C^{\circ}_{R}\subset C_{R}$. We claim that the image of $\ev^{\sqR}_{\cQ}$ lies in $[C^{\circ}_{R}/\Res^{R}_{k}\Gm^{3}]$. For otherwise, there would be a point $(\cL_{i}, \dotsc, \ph, \psi_{R})\in \cN_{\un d}(\kbar)$ and some $x\in R(\kbar)$ such that $\psi_{ij,R}$ (hence $\ph_{ij}$) vanishes at $x$ for all $i,j\in\{1,2\}$, implying that $\det(\ph)$ vanished twice at $x$ and contradicting the condition \eqref{det ph R}. Therefore the image of $\ev^{\sqR}_{\cQ}$ lies in the smooth locus of $[C_{R}/\Res^{R}_{k}\Gm^{3}]$, showing that $\cQ$ is itself smooth over $k$. This implies that $\cN_{\un d}$ is smooth over $k$. The dimension calculation is similar to Prop. \ref{p:M}\eqref{M smooth} for $\dim \CM_d$ and we omit it here.

(4) Since $\Om$ is proper, it suffices to show that $g_{\un d}$ is proper. As in the proof of \cite[Prop. 3.1(3)]{YZ}, it suffices to show that the restriction of $\wh\add^{\sqR}_{d_{1},d_{2}}$
\begin{equation}\label{add d1d2}
X^{\sqR}_{d_{1}}\times \hX^{\sqR}_{d_{2}}\to \hX^{\sqR}_{d_{1}+d_{2}}
\end{equation}
is proper for any $d_{1},d_{2}\ge0$. Since $\hX^{\sqR}_{n}\to \hX_{n}$ is finite (hence proper), the properness of \eqref{add d1d2} follows from the properness of $\wh\add_{d_{1},d_{2}}:X_{d_{1}}\times \hX_{d_{2}}\to \hX_{d_{1}+d_{2}}$, which was shown in the proof of \cite[Prop. 3.1(3)]{YZ}.
\end{proof}

\subsection{Relation with orbital integrals}
\sss{The rank one local system} Recall the double cover $\nu:X'\to X$ from \S\ref{sss:double cover}.  Let $\s:X'\to X'$ be the nontrivial involution over $X$. The direct image sheaf $\nu_{*}\Ql$ has a decomposition $\nu_{*}\Ql=\Ql\oplus L_{X'/X}$ into $\s$ eigenspaces of eigenvalue $1$ and $-1$. Then $L_{X'/X}|_{X-R}$ is a local system of rank one with geometric monodromy of order $2$ around each $\kbar$-point of the ramification locus $R$. 

Starting with $L=L_{X'/X}$, in \S\ref{sss:cons LPic} we construct a rank one local system $L^{\Pic}$ on $\Pic^{\sqR}_{X}$ whose corresponding trace function is the quadratic id\`ele class character $\y=\y_{F'/F}$ (Prop. \ref{p:LPic fun}). Via pullback along $\wh\AJ^{\sqR}_{d}: \hX^{\sqR}_{d}\to \Pic^{\sqR,d}_{X}$, it gives a rank one local system $\wh L_{d}$ on $\hX^{\sqR}_{d}$ for each $d\in\ZZ$ extending the local system $L_{d}$ on $X^{\sqR}_{d}$ defined in Lemma \ref{l:Ld}.

For $\un d\in Q_{d}$, we define a local system $L_{\un d}$ on $\cN_{\un d}$ by
\begin{equation*}
L_{\un d}=\j_{\un d}^{*}(\wh L_{d_{11}}\boxtimes\Ql\boxtimes \wh L_{d_{12}}\boxtimes \Ql).
\end{equation*}

\sss{} Recall that, for each $f\in \sH^{\Sig\cup R}_{G}$, we have defined by \eqref{def f Sig pm}
\begin{equation*}
f^{\Sig_{\pm}}=f\cdot \left(\bigotimes_{x\in R}h^{\bsq}_{x}\right)\otimes \left(\bigotimes_{x\in \Sig}\one_{\bJ_{x}}\right)\in C_c^\infty(G(\BA)).
\end{equation*}

Let $D$ be an effective divisor on $U=X-\Sig- R$ of degree $d$. In \cite[\S3.1]{YZ}  we have defined a spherical Hecke function $h_{D}\in\sH^{\Sig\cup R}_{G}$.  Therefore the element $h_{D}^{\Sig_{\pm}}\in C_c^\infty(G(\BA))$ is defined.

For $u\in \PP^{1}(F)-\{1\}$ and $h\in C_c^\infty(G(\BA))$, let
\begin{equation}\label{J(u,h)}
\JJ(u,h, s_{1},s_{2})=\sum_{\g\in A(F)\bs G(F)/A(F), \inv(\g)=u}\JJ(\g, h,s_{1},s_{2}).
\end{equation}
Note that when $u\notin \{0,1,\infty\}$, the RHS of \eqref{J(u,h)} has only one term; when $u=0$ or $\infty$, the RHS of \eqref{J(u,h)} has three terms (cf. \cite[3.3.2]{YZ}).

Recall the space $\cA^{\fl}_{D}$ defined in \eqref{def AD}. Then we have a map
\begin{equation*}
\inv_{D}: \cA^{\fl}_{D}(k)\to \PP^{1}(F)-\{1\}
\end{equation*}
sending $(\D,a,b)$ to the rational function $b/a\in \PP^{1}(F)$. As in \cite[3.3.2]{YZ}, the map $\inv_{D}$ is injective.

\begin{theorem}\label{th:J tr} Let $D$ be an effective divisor on $U=X-\Sig-R$ of degree $d$. Let $u\in \PP^{1}(F)-\{1\}$. 
\begin{enumerate}
\item If $u$ is not in the image of $\inv_{D}: \cA^{\fl}_{D}(k)\incl \PP^{1}(F)-\{1\}$, then $\JJ(u,h^{\Sig_{\pm}}_{D}, s_{1},s_{2})=0$.
\item If $u\notin\{0,1,\infty\}$ and  $u=\inv_{D}(a)$ for $a\in \cA^{\fl}_{D}(k)$ (which is then unique), then
\begin{equation}\label{J tr}
\JJ(u, h^{\Sig_{\pm}}_{D},s_{1},s_{2})=\sum_{\un d\in Q_{d}}q^{(2d_{12}-d-\r)s_{1}+(2d_{11}-d-\r)s_{2}}\Tr(\Fr_{a}, (\bR g^{\fl}_{\un d, !}L_{\un d})_{\ov a}).
\end{equation}
\item Assume $d\ge 4g-3+\r+N$. If $u=0$ or $\infty$, and  $u=\inv_{D}(a)$ for $a\in \cA^{\fl}_{D}(k)$ (which is then unique), then \eqref{J tr} still holds. 
\end{enumerate}
\end{theorem}
The proof of this theorem will occupy the rest of this subsection. From now on, we fix an effective divisor $D$ on $U$ of degree $d$.

\sss{The set $\frX_{D,\wt\g}$} Recall from \S\ref{sss:OsqR} the definition of $\OO^{\times}_{\sqR}$, which maps to $\OO^{\times}$ and hence acts on $\AA^{\times}$ by translation. Define a groupoid
\begin{equation*}
\Div^{\sqR}(X)=\AA^{\times}/\OO^{\times}_{\sqR}
\end{equation*}
There are natural maps
\begin{eqnarray*}
\AJ^{\sqR}(k)&:& \Div^{\sqR}(X) \to F^{\times}\bs\AA^{\times}/\OO^{\times}_{\sqR}=\Pic^{\sqR}_{X}(k),\\
\om&:& \Div^{\sqR}(X) \to \AA^{\times}/\OO^{\times}=\Div(X). 
\end{eqnarray*}
We denote an element in $\Div^{\sqR}(X)$ by $E^{\na}$, and denote its image in $\Div(X)$ by $E$. We denote the multiplication in $\Div^{\sqR}(X)$ by $+$. For $E^{\na}\in \Div^{\sqR}(X)$, the line bundle $\cO_{X}(-E)$, when restricted to $R$, carries a canonical square root which we denote by $\cO_{X}(-E^{\na})_{\sqR}$ (an invertible $\cO_{R}$-module). The character $\y=\y_{F'/F}$ on $\Pic^{\sqR}_{X}(k)$ can also be viewed as a character on $\Div^{\sqR}(X)$ by pullback.

Let $\wt\g\in \GL_2(F)$. Let $\wt\frX_{D,\wt\g}$ be the groupoid of $(E^{\na}_{1},E^{\na}_{2},E'^{\na}_{1},E'^{\na}_{2}, \psi_{R})$ where
\begin{itemize}
\item $E^{\na}_{i}, E'^{\na}_{i}\in \Div^{\sqR}(X)$, for $i=1,2$.
\item $\psi_{R}: \cO_{X}(-E_{1}^{\na})_{\sqR}\op\cO_{X}(-E^{\na}_{2})_{\sqR}\to  \cO_{X}(-E'^{\na}_{1})_{\sqR}\op\cO_{X}(-E'^{\na}_{2})_{\sqR}$ is an $\cO_{R}$-linear map. Write $\psi_{R}$ as a matrix $\smat{\psi_{11,R}}{\psi_{12,R}}{\psi_{21,R}}{\psi_{22,R}}$. 
\end{itemize}
These data are required to satisfy the following conditions.
\begin{enumerate}
\setcounter{enumi}{-1}
\item The rational map $\wt \g: \cO^{2}_{X}\dashrightarrow \cO^{2}_{X}$ given by the matrix $\wt \g$ induces an everywhere defined map
\begin{equation*}
\ph: \cO_{X}(-E_{1})\op\cO_{X}(-E_{2})\to  \cO_{X}(-E'_{1})\op\cO_{X}(-E'_{2}).
\end{equation*}
We write $\ph$ as a matrix $\smat{\ph_{11}}{\ph_{12}}{\ph_{21}}{\ph_{22}}$. Moreover, $\psi^{2}_{ij,R}=\ph_{ij}|_{R}$ for $1\le i,j\le 2$.
\item $\ph_{22}$ vanishes along $\Sig_{-}$.
\item $\ph_{21}$ vanishes along $\Sig_{+}$.
\item $\det(\ph)$ has divisor $D+R$.
\end{enumerate}

Define the groupoid
\begin{equation*}
\frX_{D, \wt\g}=\wt\frX_{D,\wt\g}/\Div^{\sqR}(X)
\end{equation*}
with the action of $\Div^{\sqR}(X)$ given by simultaneous translation on $E^{\na}_{i}$ and $E'^{\na}_{i}$. We may identify $\frX_{D, \wt\g}$ with the sub groupoid of $\wt\frX_{D,\wt\g}$ where $E'^{\na}_{2}$ is equal to the identity element in $\Div^{\sqR}(X)$.

\begin{lemma}\label{l:J count frX} We have
\begin{eqnarray}\label{J count frX}
&&\JJ(\g, h_{D}^{\Sig_{\pm}}, s_{1}, s_{2})\\
\notag&=&\sum_{\L=(E^{\na}_{1},\cdots, E'^{\na}_{2}, \psi_{R})\in \frX_{D,\wt\g}}\frac{1}{\#\Aut(\L)}q^{-\deg(E_{1}-E_{2}+E'_{1}-E'_{2})s_{1}}q^{-\deg(-E_{1}+E_{2}+E'_{1}-E'_{2})s_{2}}\y(E^{\na}_{1}-E^{\na}_{2}).
\end{eqnarray}
\end{lemma}
\begin{proof} Let $\wt A\subset \GL_{2}$ be the diagonal torus, and $Z\subset \GL_{2}$ be the center. Let
\begin{equation*}
\wt h^{\Sig_{\pm}}_{D}=\wt h_{D}\cdot \left(\bigotimes_{x\in R}\wt h^{\bsq}_{x}\right)\ot \left(\bigotimes_{x\in \Sig}\one_{\wt\bJ_{x}}\right).
\end{equation*}
Here $\wt h_{D}\in \sH_{\GL_{2}}$ is as defined in  \cite[proof of Prop 3.2]{YZ}, and $\wt\bJ_{x}\subset \GL_{2}(\cO_{x})$ is defined by the same formulae as $\bJ_{x}$ (see \eqref{def bJx}), with $G$ replaced by $\GL_{2}$. Then we have $h^{\Sig_{\pm}}_{D}=p_{*}\wt h^{\Sig_{\pm}}_{D}$ where $p_{*}:C_c^{\infty}(\GL_{2})\to C_c^{\infty}(G(\AA))$ is the tensor product of $p_{x,*}$. This allows us to convert the integral $\JJ(\g, h^{\Sig_{\pm}}_{D}, s_{1}, s_{2})$ into an integral on $\GL_{2}$, i.e., 
\begin{equation*}
\JJ(\g, h^{\Sig_{\pm}}_{D}, s_{1}, s_{2})=\int_{\Delta(Z(\AA))\bs(\wt A(\AA)\times\wt A(\AA))}\wt h^{\Sig_{\pm}}_{D}(t'^{-1}\wt \g t)|\a(t)\a(t')|^{s_{1}}|\a(t')/\a(t)|^{s_{2}}\y(\a(t))dtdt'.
\end{equation*}
Here $\a:\wt A\to \Gm$ is the positive root $\smat{t_{1}}{0}{0}{t_{2}}\mapsto t_{1}/t_{2}$, and the measure on $\AA^{\times}$ is such that $\vol(\OO^{\times})=1$.  We may identify $\D(Z)\bs \wt A\times \wt A$ with $\Gm^{3}$ such that $(\smat{t_{1}}{0}{0}{t_{2}},\smat{t'_{1}}{0}{0}{1})$ corresponds to $(t_{1}, t_{2}, t'_{1})\in \Gm^{3}$, and rewrite the above integral as
\begin{equation}\label{J GL2}
\JJ(\g, h^{\Sig_{\pm}}_{D}, s_{1}, s_{2})=\int_{(\AA^{\times})^{3}}\wt h^{\Sig_{\pm}}_{D}(\smat{t'^{-1}_{1}}{0}{0}{1}\wt \g \smat{t_{1}}{0}{0}{t_{2}})|t_{1}t_{2}^{-1}t'_{1}|^{s_{1}}|t_{2}t_{1}^{-1}t'_{1}|^{s_{2}}\y(t_{1}t_{2}^{-1})dt_{1}dt_{2}dt'_{1}.
\end{equation}

For $x\in |X|$, define a set $\Xi_{D,x}$ as follows:
\begin{itemize}
\item For $x\in R$, let $\Xi_{D,x}=\Xi_x$  defined in \S\ref{sss:hx at R};
\item For $x\in \Sig$, $\Xi_{D,x}=\wt \bJ_{x}$;
\item For $x\in |X|-R-\Sig$, $\Xi_{D,x}=\Mat_{2}(\cO_{x})_{v_{x}(\det)=n_{x}}$, where $n_{x}$ is the coefficient of $x$ in $D$.
\end{itemize}
Let $\Xi_{D}=\prod_{x\in |X|}\Xi_{D,x}$, then  there is a projection map
$\mu: \Xi_{D}\to \Mat_{2}(\OO)_{\div(\det)=D+R}$. We have
\begin{equation}\label{one XiD}
\wt h^{\Sig_{\pm}}_{D}=\mu_{*}\one_{\Xi_{D}}.
\end{equation}
In fact, this can be checked place by place. The assertion is trivial when $x\notin R$, and it follows from Lemma \ref{lem mu Xix} when $x\in R$.

By \eqref{one XiD}, we may rewrite \eqref{J GL2} as
\begin{equation}\label{J GL2'}
\JJ(\g, h^{\Sig_{\pm}}_{D}, s_{1},s_{2})=\int_{(\AA^{\times})^{3}}\#\mu^{-1}\left(\smat{t'^{-1}_{1}}{0}{0}{1}\wt\g \smat{t_{1}}{0}{0}{t_{2}}\right)|t_{1}t_{2}^{-1}t'_{1}|^{s_{1}}|t_{2}t_{1}^{-1}t'_{1}|^{s_{2}}\y(t_{1}t_{2}^{-1})dt_{1}dt_{2}dt'_{1}.
\end{equation}
Here $\mu^{-1}(g)=\vn$ if $g\notin \Mat_{2}(\OO)_{\div(\det)=D+R}$.

Note that the integrand in \eqref{J GL2'} is invariant under translating each of the variables by $\OO^{\times}_{\sqR}$, therefore we may turn $\JJ(\g, h^{\Sig_{\pm}}_{D}, s_{1},s_{2})$ into an integration over $\Div^{\sqR}(X)^{3}$.  To do this, we first write the integrand as a function on $\Div^{\sqR}(X)^{3}$. Denote the images of $t_{1},t_{2}, t'_{1}$ and $t'_{2}=1$ in $\Div^{\sqR}(X)$ by $E^{\na}_{1},E^{\na}_{2},E'^{\na}_{1}$ and $E'^{\na}_{2}=0$. One checks that the set $\mu^{-1}(\smat{t'^{-1}_{1}}{0}{0}{1}\wt\g \smat{t_{1}}{0}{0}{t_{2}})$ is in natural bijection with the fiber of $\wt\l: \wt\frX_{D,\wt\g}\to \Div^{\sqR}(X)^{4}$ over $(E^{\na}_{1},E^{\na}_{2},E'^{\na}_{1},E'^{\na}_{2})$, or equivalently the fiber of $\l:\frX_{D,\g}\to \Div^{\sqR}(X)^{3}$. Moreover, we have
\begin{eqnarray*}
|t_{i}|=q^{-\deg E_{i}},\quad |t'_{i}|=q^{-\deg E'_{i}}, \quad i,j\in\{1,2\}.
\end{eqnarray*}
Hence the integrand in  \eqref{J GL2'} descends to the following function on $\Div^{\sqR}(X)^{3}$ (with $E^{\na}_{2}=0$)
\begin{equation}\label{integrand Div}
\l_{!}\one_{\frX_{D,\g}}q^{-\deg(E_{1}-E_{2}+E'_{1}-E'_{2})s_{1}}q^{-\deg(-E_{1}+E_{2}+E'_{1}-E'_{2})s_{2}}\y(E^{\na}_{1}-E^{\na}_{2}).
\end{equation}

To finish the argument we need some general remarks about integrating a function over a groupoid: 
\begin{enumerate}
\item[(i)] If $\cG$ is a groupoid with finite automorphisms, and $f$ is a function on $\cG$ with finite support, then define
\begin{equation*}
\int_{\cG}f:=\sum_{g\in \cG}\frac{1}{\#\Aut(g)}f(g).
\end{equation*}

\item[(ii)] The integration above is compatible with push-forward of functions. If $\ph: \cG\to \cG'$ is a map of groupoids with finite automorphisms, and $f$ is a function with finite support on $\cG'$, then
\begin{equation*}
\int_{\cG}f=\int_{\cG'}\ph_{!}f
\end{equation*}
where $(\ph_{!}f)(g')=\int_{\ph^{-1}(g')}f|_{\ph^{-1}(g')}$, where $\ph^{-1}(g')$ is the  fiber groupoid of $\ph$ over $g'$.

\item[(iii)] Suppose we have a topological group $H$ with Haar measure $dh$, and a homomorphism $\ph: H_{1}\to H$ from a compact topological group $H_{1}$ such that the image of $\ph$ is open and $\ph$ has finite kernel. Then the groupoid $\cG=H/H_{1}$ has discrete topology with finite automorphism groups equal to $\ker(\ph)$. For a function $f$ on $H$ invariant under right translation by $\ph(H_{1})$, we have
\begin{equation}\label{int HH}
\int_{H}f(h)dh=\vol(\ph(H_{1}),dh)\#\ker(\ph)\cdot\int_{H/H_{1}}\ov f
\end{equation}
where $\ov f$ is the pullback of the descent of $f$ from $H/\ph(H_{1})$ to $H/H_{1}$. 
\end{enumerate}

Applying (iii) above to $H=(\AA^{\times})^{3}$ and $H_{1}=(\OO^{\times}_{\sqR})^{3}$ with the natural map $H_{1}\to (\OO^{\times})^{3}\incl H$. Note that the kernel and the cokernel of the map $\OO_{\sqR}^{\times}\to \OO^{\times}$ have the same finite cardinality $2^{\#R}$. Since $\vol(\OO^{\times})=1$ under the Haar measure on $\AA^{\times}$, the constant factor on the right side of \eqref{int HH} is $1$ in this case. Therefore by \eqref{int HH}, \eqref{J GL2'} can be written as the integration over $\Div^{\sqR}(X)^{3}$ of the function \eqref{integrand Div}. Applying (ii) above to $\l: \frX_{D,\g}\to \Div^{\sqR}(X)^{3}$, we further turn the integration over $ \Div^{\sqR}(X)^{3}$ into an integration over $\frX_{D,\g}$
\begin{equation}\label{J frX pre}
\JJ(\g,h^{\Sig_{\pm}}_{D}, s_{1}, s_{2})=\int_{\frX_{D,\g}}q^{-\deg(E_{1}-E_{2}+E'_{1}-E'_{2})s_{1}}q^{-\deg(-E_{1}+E_{2}+E'_{1}-E'_{2})s_{2}}\y(E^{\na}_{1}-E^{\na}_{2})
\end{equation}
where $(E_{1}^{\na},E_{2}^{\na}, E_{1}'^{\na}, E_{2}'^{\na})$ is the image of a variable point of $\frX_{D,\g}$ in $\Div^{\sqR}(X)^{4}/\Div^{\sqR}(X)$.  Now the formula \eqref{J count frX} follows from \eqref{J frX pre} by the definition in (i).
%
\end{proof}

\sss{Proof of Theorem \ref{th:J tr} for $u\notin\{0,1,\infty\}$}\label{sss:u gen}
For $u\notin\{0,1,\infty\}$, let $\wt\g(u)=\smat{1}{u}{1}{1}$, which represents the unique $\wt A(F)$ double coset in $\GL_{2}(F)$ with invariant $u$. We define a map
\begin{eqnarray*}
\l: \frX_{D,\wt\g(u)}&\to & \cN_{d}(k)\\
(E^{\na}_{1},E^{\na}_{2},E'^{\na}_{1},E'^{\na}_{2},\psi_{R})&\mapsto& (\cL^{\na}_{1},\cL^{\na}_{2},\cL'^{\na}_{1}, \cL'^{\na}_{2}, \ph, \psi_{R})
\end{eqnarray*}
where $\cL^{\na}_{i}$ (resp. $\cL'^{\na}_{i}$) is the  image of $-E^{\na}_{i}$ (resp. $-E'^{\na}_{i}$) under $\AJ^{\sqR}(k): \Div^{\sqR}(X)\to \Pic^{\sqR}_{X}(k)$; the definition of $\ph$ is contained in the definition of $\wt\frX_{D,\wt\g}$.   
If $\L$ is in the image of $\l$, then $a:=g^{\fl}_{\un d}(\L)\in \cA^{\fl}_{D}(k)$ and $\inv_{D}(a)=u$. In particular, if $u$ is not in the image of $\inv_{D}$, $\frX_{D,\wt\g(u)}=\vn$ hence $J(u,h^{\Sig_{\pm}}_{D},s_{1},s_{2})=0$ by Lemma \ref{l:J count frX}. 

Now we assume $u=\inv_{D}(a)$ for some (unique) $a\in\cA^{\fl}_{D}(k)$. Let $\cN_{\un d,a}=g^{\fl,-1}_{\un d}(a)$ and $\cN_{d,a}=\coprod_{\un d\in Q_{d}}\cN_{\un d,a}$. Then we can write
\begin{equation*}
\l: \frX_{D,\wt\g(u)}\to  \cN_{d,a}(k).
\end{equation*}
Let us define an inverse to $\l$.  Let $(\cL^{\na}_{1},\dotsc, \cL'^{\na}_{2},\ph,\psi_{R})\in\cN_{d,a}(k)$. Since the $(\cL^{\na}_{1},\dotsc, \cL'^{\na}_{2})$ are up to simultaneous tensoring with $\Pic^{\sqR}_{X}(k)$, we may fix $\cL'^{\na}_{2}$ to be $\dot{\cO}_{X}$, the identity object in $\Pic^{\sqR}_{X}(k)$. Since $\inv_{D}(a)=u\ne 0,\infty$, the maps $\ph_{ij}$ are all nonzero. Then $\ph_{21}:\cL_{1}\to \cO_{X}=\cL'_{2}$ allows us to write $\cL_{1}=\cO_{X}(-E_{1})$ for an effective divisor $E_{1}$. The lifting $\cL^{\na}_{1}$ of $\cL_{1}$ gives a canonical lifting $E_{1}^{\na}\in \Div^{\sqR}(X)$ of $E_{1}$, so that $\AJ^{\sqR}(k)(-E_{1}^{\na})\cong\cL^{\na}_{1}$ canonically. Similarly, using $\ph_{22}$ we get $E^{\na}_{2}\in \Div^{\sqR}(X)$ whose inverse represents $\cL^{\na}_{2}$. Using $\ph_{11}$ and $E_{1}^{\na}$, we further get $E'^{\na}_{1}\in \Div^{\sqR}(X)$ whose inverse represents $\cL'^{\na}_{1}$. Then $(E_{1}^{\na}, E_{2}^{\na}, E'^{\na}_{1}, 0, \psi_{R})$ ($0$ denotes the identity in $\Div^{\sqR}(X)$) gives an element in $\frX_{D, \wt\g(u)}$. It is easy to check that this assignment is inverse to $\l$, hence $\l$ is an isomorphism of groupoids.

Under $\l$, we have
\begin{eqnarray}
\label{deg d12} -\deg(E_{1}-E_{2}+E'_{1} -E'_{2})&=&d_{12}-d_{21}=2d_{12}-d-\r,\\
\label{deg d11} -\deg(-E_{1}+E_{2}+E'_{1} -E'_{2})&=&d_{11}-d_{22}=2d_{11}-d-\r, \\
\label{eta EL} \y(E^{\na}_{1}-E^{\na}_{2})&=&\y(\cL^{\na}_{11})\y(\cL^{\na}_{12})=\y(\cL^{\na}_{21})\y(\cL^{\na}_{22}),
\end{eqnarray}
where $\cL^{\na}_{ij}=\cL^{\na,\ot-1}_{j}\ot\cL'^{\na}_{i}$ and $\deg\cL_{ij}=d_{ij}$. Therefore we may rewrite \eqref{J count frX} as
\begin{eqnarray*}
&&\JJ(\g(u), h^{\Sig_{\pm}}_{D}, s_{1},s_{2})\\
&=&\sum_{\L= (\cL^{\na}_{1},\cdots, \cL'^{\na}_{2}, \ph, \psi_{R})\in \cN_{d,a}(k)}\frac{1}{\#\Aut(\L)}q^{(2d_{12}-d-\r)s_{1}+(2d_{11}-d-\r)s_{2}}\y(\cL^{\na}_{11})\y(\cL^{\na}_{12}).
\end{eqnarray*}
By Prop. \ref{p:LPic fun}, the trace function given by $L^{\Pic}$ is the character $\y$ on $\Pic^{\sqR}_{X}(k)$. The formula \eqref{J tr} then follows from the Lefschetz trace formula for Frobenius:
\begin{equation*}
\sum_{\L=(\cL^{\na}_{1},\cdots, \cL'^{\na}_{2}, \ph, \psi_{R})\in \cN_{\un d,a}(k)}\frac{1}{\#\Aut(\L)}\y(\cL^{\na}_{11})\y(\cL^{\na}_{12})=\Tr(\Fr_{a}, (\bR g^{\fl}_{\un d,!}L_{\un d})_{\ov a}).
\end{equation*}

\sss{Proof of Theorem \ref{th:J tr} for $u=0$} There are three $A(F)$ double cosets with invariant $0$:
\begin{equation*}
1=\mat{1}{0}{0}{1}, \quad n_{+}=\mat{1}{1}{0}{1}, \quad n_{-}=\mat{1}{0}{1}{1}.
\end{equation*}

We first consider the case when $\Sig_{-}=\vn$.  Then $a_{0}=(\cO_{X}(D+R), 1,0)\in \cA^{\fl}_{D}(k)$ is the unique point satisfying $\inv_{D}(a_{0})=0=u$. Let $\wh Q_{d}\subset \ZZ^{4}$ be the set defined similarly as $Q_{d}$ except we drop the condition that $d_{ij}\ge0$. For any $\un d\in \wh Q_{d}$, we define $\wh\cN_{\un d}$ in the same way as $\cN_{\un d}$ except that we drop the condition \eqref{ph degree} in Definition \ref{defn Nd}, but requiring at most one of $\ph_{ij}$ is zero. We still have a map $\wh g^{\fl}_{d}: \wh\cN_{\un d}\to \cA_{d}\to \cA^{\fl}_{d}$, and we denote the fiber over $a_{0}$ by $\wh\cN_{\un d, a_{0}}$. Let $\wh\cN_{d, a_{0}}=\coprod_{\un d\in \wh Q_{d}}\wh\cN_{\un d, a_{0}}$. We have a decomposition $\wh\cN_{d,a_{0}}=\wh\cN^{+}_{d,a_{0}}\sqcup \wh\cN^{-}_{d,a_{0}}$, where $\wh\cN^{+}_{d,a_{0}}$ consists of those $(\cL^{\na}_{1},\dotsc, \cL'^{\na}_{2},\ph, \psi_{R})$ such that $\ph_{21}=0, \ph_{12}\ne0$; $\wh\cN^{-}_{d,a_{0}}$ consists of those $(\cL^{\na}_{1},\dotsc, \cL'^{\na}_{2},\ph, \psi_{R})$ such that $\ph_{12}=0, \ph_{21}\ne0$.

The same argument as in \S\ref{sss:u gen} gives canonical isomorphisms of groupoids $\l_{\pm}: \frX_{D,n_{\pm}}\isom \wh\cN^{\pm}_{d,a_{0}}(k)$. Using the isomorphism $\l_{\pm}$, \eqref{deg d12}, \eqref{deg d11} and \eqref{eta EL}, Lemma \ref{l:J count frX} implies
\begin{eqnarray}
\label{J n+} &&\JJ(n_{+}, h_{D}^{\Sig_{\pm}},s_{1},s_{2})\\
\notag&=&\sum_{\L=(\cL^{\na}_{1},\cdots, \cL'^{\na}_{2}, \ph, \psi_{R})\in \wh\cN^{+}_{d,a_{0}}(k)}\frac{1}{\#\Aut(\L)}q^{(2d_{12}-d-\r)s_{1}+(2d_{11}-d-\r)s_{2}}\y(\cL^{\na}_{11})\y(\cL^{\na}_{12})\\
\notag&=&\sum_{\un d\in \wh Q_{d}}q^{(2d_{12}-d-\r)s_{1}+(2d_{11}-d-\r)s_{2}}\sum_{\L=(\cL^{\na}_{1},\cdots, \cL'^{\na}_{2}, \ph, \psi_{R})\in \wh\cN^{+}_{\un d,a_{0}}(k)}\frac{1}{\#\Aut(\L)}\y(\cL^{\na}_{11})\y(\cL^{\na}_{12})
\end{eqnarray}
Similarly,
\begin{eqnarray}
\label{J n-} &&\JJ(n_{-}, h_{D}^{\na},s_{1},s_{2})\\
\notag&=&\sum_{\un d\in \wh Q_{d}}q^{(2d_{12}-d-\r)s_{1}+(2d_{11}-d-\r)s_{2}}\sum_{\L=(\cL^{\na}_{1},\cdots, \cL'^{\na}_{2}, \ph, \psi_{R})\in \wh\cN^{-}_{\un d,a_{0}}(k)}\frac{1}{\#\Aut(\L)}\y(\cL^{\na}_{21})\y(\cL^{\na}_{22}).
\end{eqnarray}
On the other hand, by the Lefschetz trace formula for Frobenius, we have
\begin{eqnarray*}
&&\sum_{\un d\in Q_{d}}q^{(2d_{12}-d-\r)s_{1}+(2d_{11}-d-\r)s_{2}}\Tr(\Fr_{a_{0}}, (\bR g^{\fl}_{\un d, !}L_{\un d})_{a_{0}})\\
&=&\sum_{\un d\in Q_{d}}q^{(2d_{12}-d-\r)s_{1}+(2d_{11}-d-\r)s_{2}}\sum_{\L=(\cL^{\na}_{1},\cdots)\in \cN_{\un d,a_{0}}(k)}\frac{1}{\#\Aut(\L)}\y(\cL^{\na}_{11})\y(\cL^{\na}_{12})\\
&=&\sum_{\un d\in Q_{d}}q^{(2d_{12}-d-\r)s_{1}+(2d_{11}-d-\r)s_{2}}\left(\sum_{\L\in \cN^{+}_{\un d,a_{0}}(k)}\frac{1}{\#\Aut(\L)}\y(\cL^{\na}_{11})\y(\cL^{\na}_{12})+\sum_{\L\in \cN^{-}_{\un d,a_{0}}(k)}\frac{1}{\#\Aut(\L)}\y(\cL^{\na}_{21})\y(\cL^{\na}_{22})\right).
\end{eqnarray*}
Here $\cN^{\pm}_{\un d,a_{0}}$ is defined as $\wh\cN^{\pm}_{\un d,a_{0}}\cap \cN_{\un d,a_{0}}$. By the condition \eqref{ph degree} in Definition \ref{defn Nd}, we have $\cN^{-}_{\un d,a_{0}}=\vn$ if $d_{12}<d_{21}-N$; $\cN^{+}_{\un d,a_{0}}=\vn$ if $d_{12}\ge d_{21}-N$. Therefore, the above formula equals
\begin{eqnarray}
\label{J trace RHS}
&&\sum_{\un d\in Q_{d}, d_{12}<d_{21}-N}q^{(2d_{12}-d-\r)s_{1}+(2d_{11}-d-\r)s_{2}}\sum_{\L\in \cN^{+}_{\un d,a_{0}}(k)}\frac{1}{\#\Aut(\L)}\y(\cL^{\na}_{11})\y(\cL^{\na}_{12})\\
&+&\sum_{\un d\in Q_{d}, d_{12}\ge d_{21}-N}q^{(2d_{12}-d-\r)s_{1}+(2d_{11}-d-\r)s_{2}}\sum_{\L\in \cN^{-}_{\un d,a_{0}}(k)}\frac{1}{\#\Aut(\L)}\y(\cL^{\na}_{21})\y(\cL^{\na}_{22}).
\end{eqnarray}

Comparing the RHS of \eqref{J n+}, \eqref{J n-} and \eqref{J trace RHS}, the only difference is the range of $\un d$ in the summation; however, many $\un d$'s do not contribute as the following lemma shows.

\begin{lemma}\label{l:van d12} Let $\un d\in \wh Q_{d}$. 
\begin{enumerate}
\item If $d_{12}\ge 2g-1+\r$ then
\begin{equation*}
\sum_{\L=(\cL^{\na}_{1},\cdots, \cL'^{\na}_{2}, \ph, \psi_{R})\in \wh\cN^{+}_{\un d,a_{0}}(k)}\frac{1}{\#\Aut(\L)}\y(\cL^{\na}_{11})\y(\cL^{\na}_{12})=0.
\end{equation*}
\item If $d_{21}-N_{+}\ge 2g-1+\r$ then
\begin{equation*}
\sum_{\L=(\cL^{\na}_{1},\cdots, \cL'^{\na}_{2}, \ph, \psi_{R})\in \wh\cN^{-}_{\un d,a_{0}}(k)}\frac{1}{\#\Aut(\L)}\y(\cL^{\na}_{21})\y(\cL^{\na}_{22})=0.
\end{equation*}
\item We have
\begin{equation*}
\JJ\left(\mat{1}{0}{0}{1},h_{D}^{\Sig_{\pm}},s_{1},s_{2}\right)=0.
\end{equation*}
\end{enumerate}
\end{lemma}
\begin{proof}
(1) Let $ (X^{\sqR}_{d_{11}}\times X^{\sqR}_{d_{22}})_{D+R}$ be the fiber over $D+R$ of the map 
\begin{equation*}
X^{\sqR}_{d_{11}}\times X^{\sqR}_{d_{22}}\xr{\add^{\sqR}}X^{\sqR}_{d+\r}\xr{\om^{\sqR}_{d+\r}}X_{d+\r}.
\end{equation*}
We have an isomorphism
\begin{equation}\label{N+ a0}
\wh\cN^{+}_{\un d,a_{0}}\isom (X^{\sqR}_{d_{11}}\times X^{\sqR}_{d_{22}-N_{-}})_{D+R}\times X^{\sqR}_{d_{12}}
\end{equation}
by recording $(\cL^{\na}_{ij}, \ph_{ij}, \psi_{ij,R})$ for $(i,j)=(1,1),(2,2)$ and $(1,2)$ (then $\cL^{\na}_{21}$ is determined uniquely and $\ph_{21}=0$). Using this isomorphism we can write
\begin{eqnarray}
\label{sep eta 12}&&\sum_{\L=(\cL^{\na}_{1},\cdots, \cL'^{\na}_{2}, \ph, \psi_{R})\in \wh\cN^{+}_{\un d,a_{0}}(k)}\frac{1}{\#\Aut(\L)}\y(\cL^{\na}_{11})\y(\cL^{\na}_{12})\\
\notag &=&\sum_{\L'=(\cL^{\na}_{11},\cdots)\in (X^{\sqR}_{d_{11}}\times X^{\sqR}_{d_{22}-N_{-}})_{D+R}(k)}\frac{1}{\#\Aut(\L')}\y(\cL^{\na}_{11})\sum_{\L''=(\cL^{\na}_{12},\cdots)\in X^{\sqR}_{d_{12}}(k)}\frac{1}{\#\Aut(\L'')}\y(\cL^{\na}_{12}).
\end{eqnarray}
Since $d_{12}\ge 2g-1+\r$, the fibers of the map $\AJ^{\sqR}_{d_{12}}(k): X^{\sqR}_{d_{12}}(k)\to \Pic^{\sqR,d_{12}}_{X}(k)$ have the same cardinality. Since the character $\y$ is nontrivial on $\Pic^{\sqR,d_{12}}_{X}(k)$, the last sum in \eqref{sep eta 12} vanishes. 

The proof of (2) is similar to (1), using the isomorphism $\wh\cN^{-}_{\un d,a_{0}}\isom (X^{\sqR}_{d_{11}}\times X^{\sqR}_{d_{22}-N_{-}})_{D+R}\times X^{\sqR}_{d_{21}-N_{+}}$ instead of \eqref{N+ a0}.

(3) The restriction of the character $(t,t')\mapsto |tt'|^{s_{1}}|t'/t|^{s_{2}}\y(t)$ on the stabilizer of $1$ under $A(\AA)\times A(\AA)$ (the diagonal $A(\AA)$) is nontrivial, therefore the integral vanishes.
\end{proof}

By Lemma \ref{l:van d12}(3), we have
\begin{equation}\label{J0}
\JJ(0,h^{\Sig_{\pm}}_{D},s_{1},s_{2})=\JJ(n_{+},h^{\Sig_{\pm}}_{D},s_{1},s_{2})+\JJ(n_{-},h^{\Sig_{\pm}}_{D},s_{1},s_{2}),
\end{equation}
which is calculated in \eqref{J n+} and \eqref{J n-}. Using Lemma \ref{l:van d12}(1), we may restrict the summation in the RHS of \eqref{J n+} to those $\un d\in \wh Q_{d}$ such that $0\le d_{12}\le 2g-2+\r$ ($d_{12}\ge0$ for otherwise $\wh\cN^{+}_{\un d,a_{0}}=\vn$). Since $d\ge 4g-3+N+\r$, we have $d_{12}+(d_{21}-N_{+})\ge 2(2g-2+\r)+1$. Therefore we may alternatively restrict the summation in the RHS of \eqref{J n+} to those $\un d\in Q_{d}$ such that $d_{12}< d_{21}-N_{+}$. Therefore, the RHS of \eqref{J n+} matches the first term in the RHS of \eqref{J trace RHS}. Similarly, the RHS of \eqref{J n-} matches the second term in the RHS of \eqref{J trace RHS}.  We thus get \eqref{J tr} by combining \eqref{J0}, \eqref{J n+}, \eqref{J n-} and \eqref{J trace RHS}.

Finally, we consider the case $\Sig_{-}\ne\vn$. Then $u$ is not in the image of $\inv_{D}$. In this case,  $\frX_{D, n_{\pm}}=\vn$, hence $\JJ(n_{\pm}, h^{\Sig_{\pm}}_{D},s_{1},s_{2})=0$ by Lemma \ref{l:J count frX}. Together with Lemma \ref{l:van d12}(3), we get $\JJ(0, h^{\Sig_{\pm}}_{D},s_{1},s_{2})=0$. 

\sss{Proof of Theorem \ref{th:J tr} for $u=\infty$}  There are three $A(F)$ double cosets with invariant $\infty$:
\begin{equation*}
w_{0}=\mat{0}{1}{1}{0}, \quad n_{+}w_{0}=\mat{1}{1}{1}{0}, \quad n_{-}w_{0}=\mat{0}{1}{1}{1}.
\end{equation*}
The argument is the same as in the case $u=0$, which we do not repeat.

%
%

\section{Proof of the main theorem}

\subsection{Comparison of sheaves}

\sss{The perverse sheaf $K_{d}$} Let $d\ge0$ be an integer and consider the direct image complex  $\nu^{\sqR}_{d,!}\Ql$ under $\nu_{d}^{\sqR}:  X'_{d}\to X^{\sqR}_{d}$ defined in \eqref{norm for sym}.  Let $X^{\circ}_{d}\subset X_{d}$ be the open locus of multiplicity-free divisors, and let $X^{\sqR,\c}_{d}$ (resp. $X'^{\c}_{d}$) be its preimage in $X^{\sqR}_{d}$ (resp. $X'_{d}$). Restricting $\nu_{d}^{\sqR}$ to $X^{\sqR, \circ}_{d}$ we get a finite \'etale Galois cover $X'^{\c}_{d}\to X^{\sqR,\c}_{d}$ with Galois group $\Gamma_{d}=(\ZZ/2\ZZ)^{d}\rtimes S_{d}$ ($\nu_{d}^{\sqR}$ is still \'etale when the multiplicity-free divisor meets $R$, as $X'\to X^{\sqR}_{1}$ is \'etale). As in \cite[\S8.1.1]{YZ}, for $0\le i\le d$, we consider the following representation $\r_{d,i}=\Ind_{\Gamma_{d}(i)}^{\Gamma_{d}}(\wt\chi_{i})$ of $\Gamma_{d}$, where $\Gamma_{d}(i)=(\ZZ/2\ZZ)^{d}\rtimes (S_{i}\times S_{d-i})$, $\chi_{i}$ is the character on $(\ZZ/2\ZZ)^{d}$ which is nontrivial on the first $i$ factors and trivial on the rest, and $\wt\chi_{i}$ is the extension of $\chi_{i}$ to $\Gamma_{d}(i)$ which is trivial on $S_{i}\times S_{d-i}$. As we noted towards the end of the proof of \cite[Prop 8.2]{YZ}, there is a canonical isomorphism of $\Gamma_{d}$-representations.
\begin{equation}\label{decomp ind}
\Ind^{\Gamma_{d}}_{S_{d}}(\one)\cong\bigoplus_{i=0}^{d}\r_{d,i}.
\end{equation}

Then $\r_{i}$ gives rise to a local system $L(\r_{d,i})$ on $X^{\sqR,\c}_{d}$ (which is smooth over $k$). Let $j_{d}: X^{\sqR,\c}_{d}\incl \hX^{\sqR}_{d}$ be the inclusion. Let 
\begin{equation*}
K_{d,i}=j_{d,!*}(L(\r_{d,i})[d])[-d]
\end{equation*}
be the middle extension perverse sheaf on $\hX^{\sqR}_{d}$.

We first study the direct image complex of $f_{d}: \cM_{d}\to \cA_{d}$. By Prop. \ref{p:M}, for $d\ge  2g'-1+N$, $\dim \cM_{d}=m=\cA_{d}$.

\begin{prop}\label{p:Rfd} Let $d\ge  2g'-1+N$.
\begin{enumerate}
\item The complex $\bR f_{d,!}\Ql[m]$ is a perverse sheaf on $\cA_{d}$, and it is the middle extension of its restriction to any non-empty open subset of $\cA_{d}$.
\item  We have a canonical isomorphism
\begin{equation}\label{fd decomp}
\bR f_{d,!}\Ql\cong \bigoplus_{i=0}^{d+\r-N_{-}}\bigoplus_{j=0}^{d+\r-N_{+}}(K_{d+\r-N_{-},i}\boxtimes K_{d+\r-N_{+},j})|_{\cA_{d}}.
\end{equation}
Here we are identifying $\cA_{d}$ with an open substack of $\hX^{\sqR}_{d+\r-N_{-}}\times_{\Pic_{X}^{\sqR;\sqR,d+\r}} \hX^{\sqR}_{d+\r-N_{+}}$ using \eqref{emb Ad}. 
\end{enumerate}
\end{prop}
\begin{proof}
(1) We observe that the base $\cA_{d}$ is irreducible (because both maps $\nu_{a}$ and $\nu_{b}$ are vector bundles when $d\ge2g-1+N$). By Prop. \ref{p:M}\eqref{M smooth}, $\cM_{d}$ is smooth and equidimensional. By Prop. \ref{p:M}\eqref{f proper}\eqref{f small}, $f_{d}$ is proper and small.  Therefore,  $\bR f_{d!}\Ql[m]$ is a middle extension perverse sheaf from any non-empty open subset of $\cA_{d}$.

(2) In fact this part holds under a weaker condition $d\ge 3g-2+N$. By Prop. \ref{p:M}\eqref{MA Cart} and the K\"unneth formula, we have
\begin{equation*}
\bR f_{d!}\Ql\cong (\bR\wh\nu^{\sqR}_{d+\r-N_{-}, !}\Ql\boxtimes\bR\wh\nu^{\sqR}_{d+\r-N_{+}, !}\Ql)|_{\cA_{d}}.
\end{equation*}
Therefore it suffices to show that for $d'\ge 2g'-g=3g-2+\r$ (note that $d+\r-N_{\pm}\ge 3g-2+\r$),
\begin{equation*}
\bR \wh\nu^{\sqR}_{d'!}\Ql\cong \bigoplus_{i=0}^{d'}K_{d',i}.
\end{equation*}
We claim that $\wh\nu^{\sqR}_{d'}:\hX'_{d'}\to \hX^{\sqR}_{d'}$ is small when $d'\ge 2g'-g$. In fact, the only positive dimensional fibers are over the zero section $\Pic^{\sqR,d'}_{X}\incl \hX^{\sqR}_{d'}$, which has codimension $d'-g+1$ (provided that $d'\ge g-1$). The restriction of $\wh\nu^{\sqR}_{d'}$ over $\Pic^{\sqR,d'}_{X}$ is the norm map $\Pic_{X'}^{d'}\to \Pic^{\sqR,d'}_{X}$, whose fibers have dimension $g'-g$. Since $d'\ge 2g'-g$, we have $d'-g+1\ge 2(g'-g)+1$, which implies that $\wh\nu^{\sqR}_{d'}$ is small. 

Since the source of $\wh\nu^{\sqR}_{d'}$ is smooth and geometrically connected of dimension $d'$, and $\wh\nu^{\sqR}_{d'}$ is proper, $\bR\wh\nu^{\sqR}_{d'!}\Ql[d]$ is a middle extension perverse sheaf from its restriction to $X^{\sqR, \c}_{d'}$. The rest of the argument is the same as \cite[Prop. 8.2]{YZ}, using \eqref{decomp ind}.
\end{proof}

Recall from \S\ref{sss:Hk Md} that we have endomorphisms $f_{d,!}[\ov\cH^{\dm}_{+}]$ and $f_{d,!}[\ov\cH^{\dm}_{-}]$ of $\bR f_{d,!}\Ql$.

\begin{prop}\label{p:H pm action} Suppose $d\ge 2g'-1+N$. Then the action of $f_{d,!}[\ov\cH^{\dm}_{+}]$ (resp. $f_{d,!}[\ov\cH^{\dm}_{-}]$ ) preserves each direct summand in the decomposition \eqref{fd decomp}, and acts on the summand $(K_{d+\r-N_{-},i}\boxtimes K_{d+\r-N_{+},j})|_{\cA_{d}}$ by the scalar $d+\r-N_{+}-2j$ (resp. $d+\r-N_{-}-2i$).
\end{prop}
\begin{proof}
By Prop. \ref{p:Rfd}(1), any endomorphism of the middle extension perverse sheaf $\bR f_{d!}\Ql$ (up to a shift) is determined by its restriction to any non-empty open subset of $\cA_{d}$. Therefore it suffices to prove the same statements over $\cA^{\dm}_{d}$, over which $\cH^{\dm}_{+}$ (resp. $\cH^{\dm}_{-}$) is the pullback of the incidence correspondence $I'_{d+\r-N_{+}}$ (resp. $I'_{d+\r-N_{-}}$), see \S\ref{sss:Hk Md}. The rest of the argument is the same as \cite[Prop. 8.3]{YZ}.
\end{proof}

Now we turn to the direct image complex of $g_{\un d}: \cN_{\un d}\to \cA_{d}$. By Prop. \ref{p:N}, when $d\ge 2g'-1+N$ and $\cN_{\un d}\ne\vn$, $\dim \cN_{d}=\dim\cA_{d}=m$.

\begin{prop}\label{p:Rgd} Let $d\ge2g'-1+N$ and $\un d\in Q_{d}$. 
\begin{enumerate}
\item The complex $\bR g_{\un d,!}L_{\un d}[m]$ is a perverse sheaf on $\cA_{d}$, and it is the middle extension of its restriction to any non-empty open subset of $\cA_{d}$.
\item We have a canonical isomorphism
\begin{equation}\label{gund}
\bR g_{\un d,!}L_{\un d}\cong (K_{d+\r-N_{-}, d_{11}}\boxtimes K_{d+\r-N_{+}, d_{12}})|_{\cA_{d}}.
\end{equation}
\end{enumerate}
\end{prop}
\begin{proof}
(1) As in the proof of \cite[Prop. 8.5]{YZ}, $g_{\un d}$ is not small; however, by Prop. \ref{p:N}\eqref{N smooth}\eqref{g proper}, we know that $\bR g_{\un d, !}L_{\un d}[m]$ is Verdier self-dual. Since $g_{\un d}$ is finite over $\cA^{\dm}_{d}$,  $\bR g_{\un d, !}L_{\un d}[m]$ is a middle extension perverse sheaf on $\cA^{\dm}_{d}$. To prove $\bR g_{\un d, !}L_{\un d}[m]$ is a middle extension perverse sheaf on the whole $\cA_{d}$, we only need to show that  the restriction $\bR g_{\un d, !}L_{\un d}[m]|_{\partial \cA_{d}}$ lies in strictly negative perverse degrees, where $\partial \cA_{d}=\cA_{d}-\cA_{d}^{\dm}$. 

We have $\cA_{d}=\cA^{a=0}_{d}\sqcup\cA^{b=0}_{d}$ (see notation in the proof of Prop. \ref{p:M}\eqref{f small}). Below we will show that $\bR g_{\un d, !}L_{\un d}[m]|_{\cA^{b=0}_{d}}$ lies in negative perverse degrees, and the argument for $\cA^{a=0}_{d}$ is similar. 

When $d_{12}<d_{21}-N_{+}$, we have a Cartesian diagram
\begin{equation*}
\xymatrix{   g_{\un d}^{-1}(\cA^{b=0}_{d})   \ar[r]\ar[d]^{g_{\un d}}  &(X^{\sqR}_{d_{11}}\times X^{\sqR}_{d_{22}-N_{-}})\times_{\Pic^{\sqR;\sqR, d+\r}_{X}} (X^{\sqR}_{d_{12}}\times \Pic^{\sqR,d_{21}-N_{+}}_{X})\ar[d]^{\add^{\sqR}_{d_{11},d_{22}-N_{-}}\times h}\\
\cA^{b=0}_{d}\ar[r] & X^{\sqR}_{d+\r-N_{-}}\times_{\Pic^{\sqR;\sqR, d+\r}_{X}} \Pic^{\sqR,d+\r-N_{+}}_{X}
}
\end{equation*}
where the map $h$ is the composition
\begin{equation*}
X^{\sqR}_{d_{12}}\times \Pic^{\sqR,d_{21}-N_{+}}_{X}\xr{\AJ^{\sqR}_{d_{12}}\times\id}  \Pic^{\sqR,d_{12}}_{X}\times\Pic^{\sqR,d_{21}-N_{+}}_{X}\xr{\mult}\Pic^{\sqR,d+\r-N_{+}}_{X}
\end{equation*}
We have
\begin{equation*}
\bR g_{\un d, !}L_{\un d}|_{\cA^{b=0}_{d}}\cong \left(\bR \add^{\sqR}_{d_{11},d_{22}-N_{-},!}(L_{d_{11}}\boxtimes\Ql)\boxtimes \bR h_{!}(L_{d_{12}}\boxtimes\Ql)\right)|_{\cA_{d}^{b=0}}.
\end{equation*}
The first factor $\bR \add^{\sqR}_{d_{11},d_{22}-N_{-},!}(L_{d_{11}}\boxtimes\Ql)$ is concentrated in degree $0$ since $\add^{\sqR}_{d_{11},d_{22}-N_{-}}$ is finite. The second factor is the constant sheaf on $\Pic^{\sqR,d+\r-N_{+}}_{X}$ with geometric stalk isomorphic to $\cohog{*}{X^{\sqR}_{d_{12}}\ot\kbar, L_{d_{12}}}$. By Lemma \ref{l:coho Xd}, $\cohog{*}{X^{\sqR}_{d_{12}}\ot\kbar, L_{d_{12}}}$ always lies in degrees $\le\dim \cohog{1}{X^{\sqR}_{1}\ot\kbar, L}=\dim \cohoc{1}{(X-R)\ot\kbar, L}=2g-2+\r$. Therefore, $\bR g_{\un d, !}L_{\un d}|_{\cA^{b=0}_{d}}$ lies in degrees $\le 2g-2+\r$. Since $\codim_{\cA_{d}}(\cA_{d}^{b=0})=d+\r-N_{+}-g+1$ (see the proof of Prop. \ref{p:M}\eqref{f small}), which is $\ge(2g-2+\r)+1$ (for this we only need the weaker condition $d\ge 3g-2+N_{+}$),  we conclude that $\bR g_{\un d, !}L_{\un d}[m]|_{\cA^{b=0}_{d}}$ lies in cohomological degrees strictly less than $-\dim\cA^{b=0}_{d}$, hence in strictly negative perverse degrees.

When $d_{12}\ge d_{21}-N_{+}$, the argument is similar. The role of the map $h$ is now played by
\begin{equation*}
h': \Pic^{\sqR,d_{12}}_{X}\times X^{\sqR}_{d_{21}-N_{+}}\xr{\id\times\AJ^{\sqR}_{d_{21}-N_{+}}}\Pic^{\sqR,d_{12}}_{X}\times\Pic^{\sqR,d_{21}-N_{+}}_{X}\xr{\mult}\Pic^{\sqR,d+\r-N_{+}}_{X}.
\end{equation*}
Using the isomorphism
\begin{equation*}
\g=(h', \pr_{2}): \Pic^{\sqR,d_{12}}_{X}\times X^{\sqR}_{d_{21}-N_{+}}\isom \Pic^{\sqR,d+\r-N_{+}}_{X}\times X^{\sqR}_{d_{21}-N_{+}}
\end{equation*}
the map $h'\g^{-1}$ becomes the projection to the first factor of $\Pic^{\sqR,d+\r-N_{+}}_{X}\times X^{\sqR}_{d_{21}-N_{+}}$. By Prop. \ref{p:char sh}, $\mult^{*}L^{\Pic}_{d+\r-N_{+}}\cong L^{\Pic}_{d_{12}}\boxtimes L^{\Pic}_{d_{21}-N_{+}}$. Therefore we have $(\g^{-1})^{*}(L_{d_{12}}\boxtimes \Ql)\cong L^{\Pic}_{d+\r-N_{+}}\boxtimes L^{-1}_{d_{21}-N_{+}}\cong L^{\Pic}_{d+\r-N_{+}}\boxtimes L_{d_{21}-N_{+}}$, and hence
\begin{equation*}
h'_{!}(L_{d_{12}}\boxtimes \Ql)\cong L^{\Pic}_{d+\r-N_{+}}\ot\cohog{*}{X^{\sqR}_{d_{21}-N_{+}}\ot\kbar, L_{d_{21}-N_{+}}}.
\end{equation*}
Then we use Lemma \ref{l:coho Xd} again to conclude  that $\bR g_{\un d, !}L_{\un d}[m]|_{\cA^{b=0}_{d}}$ lies in strictly negative perverse degrees.

(2) By (1),  we only need to check \eqref{gund} over the open subset $\cA_{d}^{\dm}$. By Prop. \ref{p:N}\eqref{N comm}, the diagram \eqref{N g comm} is Cartesian over $\cA^{\dm}_{d}$, we have
\begin{equation*}
\bR g_{\un d,!}L_{\un d}|_{\cA^{\dm}_{d}}\cong \left(\add^{\sqR}_{d_{11},d_{22}-N_{-},!}(L_{d_{11}}\boxtimes\Ql)\boxtimes\add^{\sqR}_{d_{12},d_{21}-N_{+},!}(L_{d_{12}}\boxtimes\Ql)\right)|_{\cA^{\dm}_{d}}.
\end{equation*}
Here $\add^{\sqR}_{i,j}$ is the addition map \eqref{add sqR}. Therefore it suffices to show that for any $i,j\ge0$, there is a canonical isomorphism over $X^{\sqR}_{i+j}$
\begin{equation}\label{add sqR K}
\add^{\sqR}_{i,j,!}(L_{i}\boxtimes \Ql)\cong K_{i+j, i}|_{X^{\sqR}_{i+j}}.
\end{equation}
Now both sides are middle extension perverse sheaves (because $\add^{\sqR}_{i,j}$ is finite surjective with smooth irreducible source). The isomorphism \eqref{add sqR K} then follows from the same isomorphism between the restrictions of both sides to $(X-R)^{\c}_{i+j}$, and the latter was proved in \cite[Prop. 8.5]{YZ}.
\end{proof}

\subsection{Comparison of traces}
For $\mu,\mu'\in\frT_{r,\Sig}$, recall the definition of $r_{\pm}$ from \eqref{defn r pm}. For $f\in \sH^{\Sig}_{G}$, with $f^{\Sig_{\pm}}$ defined in \eqref{def f Sig pm}, let
\begin{equation*}
\JJ^{\mu,\mu'}(f)=\left(\frac{\partial}{\partial s_{1}}\right)^{r_{+}}\left(\frac{\partial}{\partial s_{2}}\right)^{r_{-}}\left(q^{N_{+}s_{1}+N_{-}s_{2}}\JJ(f^{\Sig_{\pm}},s_{1}, s_{2})\right)\Big|_{s_{1}=s_{2}=0}.
\end{equation*}

\begin{theorem}\label{th:IJ hD} Suppose $D$ is an effective divisor on $U$ of degree $d\ge \max\{2g'-1+N,2g\}$, then
\begin{equation}\label{IJ hD}
(-\log q)^{-r}\JJ^{\mu,\mu'}(h_{D})=\II^{\mu,\mu'}(h_{D}).
\end{equation}
\end{theorem}
\begin{proof}
By Theorem \ref{th:J tr}, we have
\begin{eqnarray*}
q^{N_{+}s_{1}+N_{-}s_{2}}\JJ(h^{\Sig_{\pm}}_{D},s_{1},s_{2})&=&\sum_{\un d\in Q_{d}}q^{(2d_{12}-d-\r+N_{+})s_{1}+(2d_{11}-d-\r+N_{-})s_{2}}\\
&\cdot&\sum_{a\in \cA^{\fl}_{D}(k)}\Tr(\Fr_{a}, (\bR g^{\fl}_{\un d,!}L_{\un d})_{a})
\end{eqnarray*}
Using $\bR g^{\fl}_{\un d,!}L_{\un d}=\bR \Om_{!}\bR g_{\un d,!}L_{\un d}$, we have 
\begin{equation*}
\sum_{a\in \cA^{\fl}_{D}(k)}\Tr(\Fr_{a}, (\bR g^{\fl}_{\un d,!}L_{\un d})_{a})=\sum_{\wt a\in \cA_{D}(k)}\frac{1}{\#\Aut(\wt a)}\Tr(\Fr_{\wt a}, (\bR g_{\un d,!}L_{\un d})_{\wt a}).
\end{equation*}
Here $\cA_{D}\subset \cA$ is the preimage of $\cA^{\fl}_{D}$.  Using Prop. \ref{p:Rgd}, we can rewrite the above as
\begin{equation*}
\sum_{\wt a\in \cA_{D}(k)}\frac{1}{\#\Aut(\wt a)}\Tr(\Fr_{\wt a}, (K_{d+\r-N_{-}, d_{11}}\boxtimes K_{d+\r-N_{+},d_{12}})_{\wt a}).
\end{equation*}
Therefore we get
\begin{eqnarray*}
q^{N_{+}s_{1}+N_{-}s_{2}}\JJ(h^{\Sig_{\pm}}_{D},s_{1},s_{2})&=&\sum_{i=0}^{d+\r-N_{-}}\sum_{j=0}^{d+\r-N_{+}}q^{(2j-d-\r+N_{+})s_{1}+(2i-d-\r+N_{-})s_{2}}\\
&\cdot&\sum_{\wt a\in \cA_{D}(k)}\frac{1}{\#\Aut(\wt a)}\Tr(\Fr_{\wt a}, (K_{d+\r-N_{-}, i}\boxtimes K_{d+\r-N_{+},j})_{\wt a}).
\end{eqnarray*}
Taking derivatives, we get
\begin{eqnarray}
\notag
(\log q)^{-r}\JJ^{\mu,\mu'}(h_{D})&=&\sum_{i=0}^{d+\r-N_{-}}\sum_{j=0}^{d+\r-N_{+}}(2j-d-\r+N_{+})^{r_{+}}(2i-d-\r+N_{-})^{r_{-}}\\
\label{J tr KK}&&\cdot\sum_{\wt a\in \cA_{D}(k)}\frac{1}{\#\Aut(\wt a)}\Tr(\Fr_{\wt a}, (K_{d+\r-N_{-}, i}\boxtimes K_{d+\r-N_{+},j})_{\wt a}).
\end{eqnarray}

On the other hand, by Theorem \ref{th:Ir} we have
\begin{eqnarray*}
&&\II^{\mu,\mu'}(h_{D})\\
&=&\sum_{a\in \cA^{\fl}_{D}(k)}\Tr\left((f^{\fl}_{d,!}[\ov\cH^{\dm}_{+}])^{r_{+}}_{a}\circ (f^{\fl}_{d,!}[\ov\cH^{\dm}_{-}])^{r_{-}}_{a}\circ \Fr_{a}, (\bR f^{\fl}_{d,!}\Ql)_{a}\right)\\
&=&\sum_{\wt a\in \cA_{D}(k)}\frac{1}{\#\Aut(\wt a)}\Tr\left((f_{d,!}[\ov\cH^{\dm}_{+}])^{r_{+}}_{\wt a}\circ (f_{d,!}[\ov\cH^{\dm}_{-}])^{r_{-}}_{\wt a}\circ \Fr_{\wt a}, (\bR f_{d,!}\Ql)_{\wt a}\right)
\end{eqnarray*}
By Prop. \ref{p:Rfd} and Prop. \ref{p:H pm action}, for $\wt a\in \cA_{d}(k)$ we have
\begin{eqnarray*}
&&\Tr\left((f_{d,!}[\ov\cH^{\dm}_{+}])^{r_{+}}_{\wt a}\circ (f_{d,!}[\ov\cH^{\dm}_{-}])^{r_{-}}_{\wt a}\circ \Fr_{\wt a}, (\bR f_{d,!}\Ql)_{\wt a}\right)\\
&=&\sum_{i=0}^{d+\r-N_{-}}\sum_{j=0}^{d+\r-N_{+}}(d+\r-N_{+}-2j)^{r_{+}}(d+\r-N_{-}-2i)^{r_{-}}\\
&\cdot&\Tr\left(\Fr_{\wt a}, (K_{d+\r-N_{-},i}\boxtimes K_{d+\r-N_{+},j})_{\wt a}\right).
\end{eqnarray*}
Therefore
\begin{eqnarray}\label{I tr KK}
\II^{\mu,\mu'}(h_{D})&=&\sum_{i=0}^{d+\r-N_{-}}\sum_{j=0}^{d+\r-N_{+}}(d+\r-N_{+}-2j)^{r_{+}}(d+\r-N_{-}-2i)^{r_{-}}\\
\notag&&\cdot\sum_{\wt a\in \cA_{D}(k)}\frac{1}{\#\Aut(\wt a)}\Tr\left(\Fr_{\wt a}, (K_{d+\r-N_{-},i}\boxtimes K_{d+\r-N_{+},j})_{\wt a}\right).
\end{eqnarray}
Comparing \eqref{J tr KK} and \eqref{I tr KK}, we get \eqref{IJ hD}. The extra sign $(-1)^{r}$ in \eqref{IJ hD} comes from the fact that $(d+\r-N_{+}-2j)^{r_{+}}(d+\r-N_{-}-2i)^{r_{-}}=(-1)^{r}(2j-d-\r+N_{+})^{r_{+}}(2i-d-\r+N_{-})^{r_{-}}$.
\end{proof}

\sss{}\label{sss:aut spec} Fix $\xi\in\frSi'(\kbar)$. Let $V'(\xi)=\cohoc{2r}{\Sht'^{r}_{G}(\Sig;\xi)\ot\kbar,\Ql}(r)$. By the discussion in \S\ref{sss:var coho decomp}, the finiteness results proved in \S\ref{sss:finiteness} for the cohomology of $\Sht^{r}_{G}(\Sii)$ as a $\sH^{\Sig}_{G}$-module are also valid for $V'$, hence for its summand $V'(\xi)$.

Let 
\begin{equation*}
K=\prod_{x\notin \Sig}G(\cO_{x})\times\prod_{x\in \Sig}\Iw_{x}.
\end{equation*}
Denote by $\cA(K)$ the space of compactly supported, $\QQ$-valued functions on the double coset $G(F)\bs G(\AA)/K$. The moduli stack $\Sht^{0}_{G}(\Sig)$ is exactly the discrete groupoid $G(F)\bs G(\AA)/K$, therefore, $\cA(K)\ot\Ql$ is identified with $\cohoc{0}{\Sht^{0}_{G}(\Sig)\ot\kbar,\Ql}$. Corollary \ref{c:H fin} implies that the image of the action map $\sH^{\Sig}_{G}\to \End(\cA(K))$ is a finitely generated $\QQ$-algebra with Krull dimension one. Theorem \ref{th:spec decomp} allows us to write
\begin{equation*}
\cA(K)\ot\Qlbar=\cA(K)_{\Eis}\ot\Qlbar\oplus (\oplus_{\pi\in \Pi_{\Sig}(\Qlbar)} \cA(K)_{\pi}).
\end{equation*} 
Here $\Pi_{\Sig}(\Qlbar)$ is the set of cuspidal automorphic representations (with $\Qlbar$-coefficients) of $G(\AA)$ with level $K$. Each $\pi$ determines a character $\l_{\pi}: \sH^{\Sig}_{G}\to \Qlbar$. By strong multiplicity one for $G$, the character $\l_{\pi}$ determined $\pi$. Therefore we may identify $\Pi_{\Sig}(\Qlbar)$ as a subset of $\Spec \sH^{\Sig}_{G}\ot\Qlbar$. 

Let
\begin{equation*}
\wt \sH^{\Sig}_{\ell}=\Im(\sH^{\Sig}_{G}\ot\Ql\to \End_{\Ql}(V'(\xi))\times \End_{\Ql}(\cA(K)\ot\Ql)\times \Ql[\Pic_{X}(k)]^{\io_{\Pic}})
\end{equation*}
Then by Corollary \ref{c:H fin}, $\wt \sH^{\Sig}_{\ell}$ is again a finitely generated $\Ql$-algebra with Krull dimension one.

\begin{theorem}\label{th:IJ} Let $\mu,\mu'\in\{\pm1\}^{r}$. Then for all $f\in \sH^{\Sig}_{G}$, we have the identity
\begin{equation*}
(-\log q)^{-r}\JJ^{\mu,\mu'}(f)=\II^{\mu,\mu'}(f).
\end{equation*}
\end{theorem}
The proof is the same as that of \cite[Theorem 9.2]{YZ}, using the finiteness property of $\wt \sH^{\Sig}_{\ell}$ and \cite[Lemma 9.1]{YZ}.

\subsection{Conclusion of the proofs}

\sss{Proof of Theorem \ref{th:main}} Both $\II^{\mu,\mu'}(h)$ and $\JJ^{\mu,\mu'}(h)$ depend only on the image of $h$ in $\wt \sH^{\Sig}_{\ell}$.

Let $\cY=\Spec \wt \sH^{\Sig}_{\ell}$. By Theorem \ref{th:spec decomp}, we have a decomposition
\begin{equation*}
\cY^{\red}=Z_{\Eis, \Ql}\coprod \cY_{0}
\end{equation*}
where $\cY_{0}$ is a finite set of closed points. Under this decomposition, we have a corresponding decomposition of $\wt \sH^{\Sig}_{\ell}$
\begin{equation}\label{prod decomp H}
\wt \sH^{\Sig}_{\ell}=\wt \sH^{\Sig}_{\ell,\Eis}\times \wt \sH^{\Sig}_{\ell,0}
\end{equation}
such that $\Spec \wt \sH^{\Sig,\red}_{\ell,\Eis}=Z_{\Eis, \Ql}$ and $\Spec \wt \sH^{\Sig,\red}_{\ell,0}=\cY_{0}$. We have a decomposition
\begin{equation*}
V'(\xi)\ot\Qlbar=V'(\xi)_{\Eis}\ot\Qlbar\op(\oplus_{\fkm\in \cY_{0}(\Qlbar)}V'(\xi)_{\fkm})
\end{equation*}
where $\Supp(V'(\xi)_{\Eis})\subset Z_{\Eis,\Ql}$ and $V'(\xi)_{\fkm}$ is the generalized eigenspace of $V'(\xi)\ot\Qlbar$ under the character $\fkm$ of $\wt\sH^{\Sig}_{\ell}$. Under this decomposition, let $Z^{\mu}_{\fkm}(\xi)$ be the projection of $Z^{\mu}(\xi)\in V'(\xi)$ (the cycle class of $\th'^{\mu}_{*}[\Sht^{\un\mu}_{T}(\mi\cdot\xi)]$) to the direct summand $V'(\xi)_{\fkm}$.

Let $h\in \wt \sH^{\Sig}_{\ell,0}$, viewed as $(0,h)\in \wt \sH^{\Sig}_{\ell}$ under the decomposition \eqref{prod decomp H}. Since the $\sH^{\Sig}_{G}$-action on $V'(\xi)$ is self-adjoint with respect to the cup product pairing, we have
\begin{equation}\label{Ih}
\II^{\mu,\mu'}(h)=\sum_{\fkm\in \cY_{0}(\Qlbar)}(Z^{\mu}_{\fkm}(\xi), h*Z^{\mu'}_{\fkm}(\xi)).
\end{equation}
On the other hand, we have
\begin{equation}\label{Jh}
\JJ^{\mu,\mu'}(h)=\sum_{\pi\in \Pi_{\Sig}(\Qlbar)} \l_{\pi}(h)\left(\frac{\partial}{\partial s_{1}}\right)^{r_{+}}\left(\frac{\partial}{\partial s_{2}}\right)^{r_{-}}\left(q^{N_{+}s_{1}+N_{-}s_{2}}\JJ_{\pi}(h^{\Sig_{\pm}},s_{1},s_{2})\right)\Big|_{s_{1}=s_{2}=0}.
\end{equation}

By the discussion in \S\ref{sss:aut spec}, $\Pi_{\Sig}(\Qlbar)$ can be viewed as a subset of $\cY_{0}(\Qlbar)$. Now let $\pi$ be as in the statement of Theorem \ref{th:main}. Let $h=e_{\pi}$ be the idempotent in $\wt \sH^{\Sig}_{\ell,0}\ot\Qlbar$ corresponding to $\pi\in \Pi_{\Sig}(\Qlbar)\subset \cY_{0}(\Qlbar)$. In \eqref{Ih} and \eqref{Jh} we plug in $h=e_{\pi}$, we get
\begin{eqnarray*}
\II^{\mu,\mu'}(e_{\pi})&=&(Z^{\mu}_{\pi}(\xi), Z^{\mu}_{\pi}(\xi)).\\
\JJ^{\mu,\mu'}(e_{\pi})&=&\left(\frac{\partial}{\partial s_{1}}\right)^{r_{+}}\left(\frac{\partial}{\partial s_{2}}\right)^{r_{-}}\left(q^{N_{+}s_{1}+N_{-}s_{2}}\JJ_{\pi}(h^{\Sig_{\pm}},s_{1},s_{2})\right)\Big|_{s_{1}=s_{2}=0}.
\end{eqnarray*}
Applying Theorem \ref{th:IJ} to $e_{\pi}$,
\begin{equation*}
(-\log q)^{-r}\left(\frac{\partial}{\partial s_{1}}\right)^{r_{+}}\left(\frac{\partial}{\partial s_{2}}\right)^{r_{-}}\left(q^{N_{+}s_{1}+N_{-}s_{2}}\JJ_{\pi}(h^{\Sig_{\pm}},s_{1},s_{2})\right)\Big|_{s_{1}=s_{2}=0}=(Z^{\mu}_{\pi}(\xi), Z^{\mu}_{\pi}(\xi)).
\end{equation*}
By Prop. \ref{p:global char}, the left side above is the left side of \eqref{main formula}. The proof of Theorem \ref{th:main} is complete.

\sss{Proof of Theorem \ref{th:int Sht}} 
Make a change of variables $t_{1}=s_{1}+s_{2}$, $t_{2}=s_{1}-s_{2}$, we have
\begin{eqnarray*}
\left(\frac{\partial}{\partial t_{1}}\right)^{r_{1}}\left(\frac{\partial}{\partial t_{2}}\right)^{r-r_{1}}
&=&\frac{1}{2^{r}}\left(\frac{\partial}{\partial s_{1}}+\frac{\partial}{\partial s_{2}}\right)^{r_{1}}\left(\frac{\partial}{\partial s_{1}}-\frac{\partial}{\partial s_{2}}\right)^{r-r_{1}}\\
&=&\frac{1}{2^{r}}\sum_{I\subset \{1,2,\cdots, r\}}(-1)^{\#(I\cap\{r_{1}+1,\cdots, r\})}\left(\frac{\partial}{\partial s_{1}}\right)^{r-\#I}\left(\frac{\partial}{\partial s_{2}}\right)^{\#I}.
\end{eqnarray*}
Therefore,
\begin{eqnarray*}
&&\sL^{(r_{1})}(\pi,\ha)\sL^{(r-r_{1})}(\pi\otimes\y, \ha)\\
&=&\left(\frac{\partial}{\partial t_{1}}\right)^{r_{1}}\left(\frac{\partial}{\partial t_{2}}\right)^{r-r_{1}}(\sL(\pi,t_{1}+\ha)\sL(\pi\otimes\y,t_{2}+\ha))\Big|_{t_{1}=t_{2}=0}\\
&=&\frac{1}{2^{r}}\sum_{I\subset \{1,2,\cdots, r\}}(-1)^{\#(I\cap\{r_{1}+1,\cdots, r\})}\left(\frac{\partial}{\partial s_{1}}\right)^{r-\#I}\left(\frac{\partial}{\partial s_{2}}\right)^{\#I}\sL_{F'/F}(\pi,s_{1},s_{2})\Big|_{s_{1}=s_{2}=0}.
\end{eqnarray*}
For $I\subset\{1,2,\dotsc, r\}$, let $\s_{I}\in \{\pm1\}^{r}$ be the element which is $-1$ on the $i$-th coordinate if $i\in I$ and $1$ elsewhere. We may view $\s_{I}$ as an element in $\frA_{r,\Sig}$. Let $\mu\in \frT_{r,\Sig}$. By Theorem \ref{th:main}
\begin{equation*}
\left(\frac{\partial}{\partial s_{1}}\right)^{r-\#I}\left(\frac{\partial}{\partial s_{2}}\right)^{\#I}\sL_{F'/F}(\pi,s_{1},s_{2})\Big|_{s_{1}=s_{2}=0}=\left(Z^{\mu}_{\pi}(\xi), Z^{\s_{I}\cdot \mu}_{\pi}(\xi)\right)=\left(Z^{\mu}_{\pi}(\xi), \s_{I}\cdot Z^{\mu}_{\pi}(\xi)\right).
\end{equation*}
where the second equality follows from Lemma \ref{l:cycle mu mu'}. Therefore
\begin{eqnarray*}
&&\sL^{(r_{1})}(\pi,\ha)\sL^{(r-r_{1})}(\pi\otimes\y, \ha)\\
&=&\frac{1}{2^{r}}\sum_{I\subset \{1,2,\cdots, r\}}(-1)^{\#(I\cap\{r_{1}+1,\cdots, r\})}\left(Z^{\mu}_{\pi}(\xi), \quad \s_{I}\cdot Z^{\mu}_{\pi}(\xi)\right)\\
&=&\left(Z^{\mu}_{\pi}(\xi), \quad \frac{1}{2^{r}}\sum_{I\subset \{1,2,\cdots, r\}}(-1)^{\#(I\cap\{r_{1}+1,\cdots, r\})}\s_{I}\cdot Z^{\mu}_{\pi}(\xi)\right)\\
&=&\left(Z^{\mu}_{\pi}(\xi), \quad \prod_{i=1}^{r_{1}}\frac{1+\s_{i}}{2}\prod_{j=r_{1}}^{r}\frac{1-\s_{j}}{2}\cdot Z^{\mu}_{\pi}(\xi)\right)=\left(Z^{\mu}_{\pi}(\xi), \ep_{r_{1}}\cdot Z^{\mu}_{\pi}(\xi)\right).
\end{eqnarray*}
Since $\ep_{r_{1}}$ is an idempotent in $\QQ[(\ZZ/2\ZZ)^{r}]$ which is self-adjoint with respect to the intersection pairing on $\Sht'^{r}_{G}(\Sig;\xi)$, we have $\left(Z^{\mu}_{\pi}(\xi), \ep_{r_{1}}\cdot Z^{\mu}_{\pi}(\xi)\right)=\left(\ep_{r_{1}}\cdot Z^{\mu}_{\pi}(\xi), \ep_{r_{1}}\cdot Z^{\mu}_{\pi}(\xi)\right)$. The theorem is proved.

\appendix
\section{Picard stack with ramifications}\label{A:Pic}

In this appendix we record some constructions in the geometric class field theory with ramifications of order two, which will be used in the descriptions of the moduli spaces in \S\ref{s:M} and \S\ref{s:N}.

\subsection{The Picard stack and Abel-Jacobi map with ramifications} 

Let $R\subset X$ be a reduced finite subscheme. 

\begin{defn} Let $\Pic^{\sqR}_{X}$ be the functor on $k$-schemes whose $S$-valued points is the groupoid of triples $\cL^{\na}=(\cL, \cK_{R}, \io)$ where
\begin{itemize}
\item $\cL$ is a line bundle over $X\times S$;
\item $\cK_{R}$ is a line bundle over $R\times S$;
\item $\io: \cK^{\otimes2}_{R}\isom \cL|_{R\times S}$ is an isomorphism of line bundles over $R\times S$. 
\end{itemize}
\end{defn}

We have a decomposition $\Pic^{\sqR}_{X}=\sqcup_{d\in\ZZ}\Pic^{\sqR,d}_{X}$, where $\Pic^{\sqR,d}_{X}$ is the subfunctor defined by imposing that $\deg(\cL_{s})=d$ for each geometric point $s\in S$.

\sss{} We present $\Pic^{\sqR}_{X}$ as a quotient stack. Let $\Pic_{X,R}$ be the moduli stack classifying $(\cL,\g)$ where $\cL$  is a line bundle over $X$ and $\gamma$ is a trivialization of $\cL_{R}$. The Weil restriction $\Res^{R}_{k}\Gm$ acts on $\Pic_{X,R}$ by changing the trivialization $\g$, whose quotient is naturally isomorphic to $\Pic_{X}$. From the definition of $\Pic^{\sqR}_{X}$ we see there is a natural isomorphism of stacks
\begin{equation*}
\Pic^{\sqR}_{X}\cong [\Pic_{X,R}/_{[2]}\Res^{R}_{k}\Gm]
\end{equation*}
Here the quotient is obtained by making $\Res^{R}_{k}\Gm$ act on $\Pic_{X,R}$ via the square of the usual action, and the notation $/_{[2]}$ is to emphasize the square action. When $R=\vn$, $\Res^{R}_{k}\Gm=\Spec k$ by convention, and the above discussion is still valid.

The forgetful map $(\cL,\cK_{R},\io)\mapsto \cL$ gives a morphism of stacks
\begin{equation*}
\Pic^{\sqR}_{X}\to \Pic_{X}
\end{equation*}
which is a $\Res^R_{k}\mu_{2}$-gerbe.

\sss{Variant of $\Pic^{\sqR}_{X}$}\label{sss:wtPic} We shall also need the following variant of $\Pic^{\sqR}_{X}$. Let $\Pic^{\sqR;\sqR}_{X}$ be the stack whose $S$-points consist of $(\cL, \cK_{R}, \io, \a_{R})$, where $(\cL,\cK_{R}, \io)\in \Pic^{\sqR}_{X}(S)$ and $\a_{R}$ is a section of $\cK_{R}$. Then we have
\begin{equation*}
\Pic^{\sqR;\sqR}_{X}\cong \Pic_{X,R}\twtimes{[2], \Res^{R}_{k}\Gm}\Res^{R}_{k}\AA^1
\end{equation*}
Here the action of $\Res^{R}_{k}\Gm$ on $\Pic_{X,R}$ is the square action and its action on $\Res^{R}_{k}\AA^1$ is by dilation.

\begin{defn} For each integer $d\geq0$, let $\hX^{\sqR}_{d}$ be the $k$-stack whose $S$-points is the groupoid of tuples $(\cL^{\na}, a, \a_{R})$ where
\begin{itemize}
\item $\cL^{\na}=(\cL, \cK_{R}, \io)\in \Pic^{\sqR,d}_{X}(S)$; in particular, $\io$ is an isomorphism $\cK^{\otimes2}_{R}\isom \cL_{R}$.
\item $a$ is a global section of $\cL$;
\item $\a_{R}$ is a section of $\cK_{R}$ such that $\io(\a^{\otimes2}_{R})=a_{R}$, where $a_{R}$ is the restriction of $a_{R}$ to $R\times S$.
\end{itemize}
We let $X^{\sqR}_{d}\subset \hX^{\sqR}_{d}$ be the open substack defined by requiring that $a$ is nonzero along the geometric fiber $X\times \{s\}$, for all geometric points $s\in S$.
\end{defn}

\sss{} Forgetting the square roots $(\cK_{R},\io,\a_{R})$ we get a morphism to the stack $\hX_{d}$ defined in \cite[\S3.2.1]{YZ}
\begin{equation*}
\wh\om^{\sqR}_{d}: \hX^{\sqR}_{d}\to \hX_{d}.
\end{equation*}
Over a geometric point $(\cL,a\in \Gamma(X_{K}, \cL))\in \hX_{d}(K)$, the fiber of $\wh\om^{\sqR}_{d}$ is a product $\prod_{x\in R(K)}\cP_{x}$, where $\cP_{x}\cong\Spec K$ if $a(x)\neq0$, and $\cP_{x}\cong[\Spec K/\mu_{2,K}]$ if $a(x)=0$. In particular, the restriction of $\wh\om^{\sqR}_{d}$ to $X^{\sqR}_{d}$
\begin{equation*}
\om^{\sqR}_{d}: X^{\sqR}_{d}\to X_{d}.
\end{equation*}
realizes $X_{d}$ as the coarse moduli scheme of $X^{\sqR}_{d}$. When $d=1$, $X^{\sqR}_{1}$ is the DM curve with coarse moduli space $X$ and automorphic group $\mu_{2}$ along $R$.

\begin{defn} For an open subset $U\subset X$, we define $U^{\sqR}_{d}$ to be the subset of $X^{\sqR}_{d}$ which is the preimage of $U_{d}$ under the map $\om^{\sqR}_{d}$.
\end{defn}

We have another description of $\hX^{\sqR}_{d}$ as follows. Evaluating  a section of a line bundle along $R$ gives a morphism
\begin{equation*}
\ev^{R}_{d}: \hX_{d}\to [\Res^{R}_{k}\AA^1/\Res^{R}_{k}\Gm].
\end{equation*} 
From the construction of $\hX^{\sqR}_{d}$ we get a Cartesian diagram
\begin{equation}\label{XR Cart}
\xymatrix{\hX^{\sqR}_{d}\ar[d]^{\wh\om^{\sqR}_{d}}\ar[rr]^-{\ev^{\sqR}_{d}} && [\Res^{R}_{k}\AA^1/\Res^{R}_{k}\Gm]\ar[d]^{[2]}\\
\hX_{d}\ar[rr]^-{\ev^{R}_{d}} && [\Res^{R}_{k}\AA^1/\Res^{R}_{k}\Gm]}
\end{equation}
Here the vertical map $[2]$ is the square map on both $\Res^{R}_{k}\AA^1$ and $\Res^{R}_{k}\Gm$.

\begin{lemma}\label{l:ev sm}
\begin{enumerate}
\item The map $\ev^{R}_{d}$ is smooth when restricted to $X_{d}$. 
\item $X^{\sqR}_{d}$ is a smooth DM stack over $k$.
\end{enumerate}
\end{lemma}
\begin{proof}
(1) We may argue by base changing to $\kbar$. We have $[\Res^{R}_{k}\AA^1/\Res^{R}_{k}\Gm]_{\kbar}\cong \prod_{x\in R(\kbar)}[\AA^{1}/\Gm]$, and the map $\ev^{R}_{d,\kbar}:X_{d,\kbar}\to \prod_{x\in R(\kbar)}[\AA^{1}/\Gm]$ is the product of the evaluation maps $\ev_{x}$ for $x\in R(\kbar)$.  The following general statement follows from an easy calculation of tangent spaces.

\begin{claim} Let $Z$ be a smooth and irreducible $\kbar$-scheme and $f_{i}: Z\to [\AA^{1}/\GG_{m}]$ be a collection of morphisms, $1\leq i\leq n$. Assume the image of each $f_{i}$ does not lie entirely in $[\{0\}/\Gm]$, so the scheme-theoretic preimage of $[\{0\}/\Gm]$ under $f_{i}$ is a divisor $D_{i}\subset Z$. Let $f:Z\to \prod_{i=1}^{n}[\AA^{1}/\Gm]\cong [\AA^{n}/\Gm^{n}]$ be the fiber product of the $f_{i}$'s. Then $f$ is a smooth morphism if and only if the divisors $D_{1},\dotsc, D_{n}$ are smooth and intersect transversely. 
\end{claim}

We apply this claim to $Z=X_{d,\kbar}$ and the maps $\ev_{x}$ for $x\in R(\kbar)$. The divisor $D_{x}$ in this case is the locus in $X_{d,\kbar}$ classifying those degree $d$ divisors $D$ of $X$ containing $x$. For a subset  $I\subset R(\kbar)$, the intersection $D_{I}=\cap_{x\in I} D_{x}$ is the locus classifying those degree $d$ divisors $D$ of $X$ containing all points in $I$. This is non-empty only if $\#I\leq d$. When this is the case, we have an isomorphism $X_{d-\#I}\cong D_{I}$ given by $D\mapsto D+\sum_{x\in I}x$ (the fact that this is an isomorphism can be checked by an \'etale local calculation, reducing to the case $X$ is $\AA^{1}$). In particular, $D_{I}\subset X_{d,\kbar}$ is smooth of codimension $\#I$. This shows that the divisors $\{D_{x}\}_{x\in R(\kbar)}$ intersect transversely. By the Claim above, the map $\ev^{R}_{d,\kbar}$ is smooth when restricted to $X_{d,\kbar}$.

(2) Since $\ev^{R}_{d}|_{X_{d}}$ is smooth by part (1), so is $\ev^{\sqR}_{d}|_{X^{\sqR}_{d}}$ by the Cartesian diagram \eqref{XR Cart}. Therefore $X^{\sqR}_{d}$ is a smooth algebraic stack over $k$. Since the square map $[\Res^{R}_{k}\AA^1/\Res^{R}_{k}\Gm]\to [\Res^{R}_{k}\AA^1/\Res^{R}_{k}\Gm]$ is relative DM and $X_{d}$ is a scheme, we see that $X^{\sqR}_{d}$ is a DM stack again from \eqref{XR Cart}.
\end{proof}

\sss{The addition map} Suppose $d_{1},d_{2}\in\ZZ_{\ge0}$, then we have a map
\begin{equation*}
\wh\add_{d_{1},d_{2}}^{\sqR}: \hX_{d_{1}}^{\sqR}\times \hX_{d_{2}}^{\sqR}\to \hX^{\sqR}_{d_{1}+d_{2}}
\end{equation*}
sending $(\cL_{1}^{\na},a_{1}, \a_{R,1}, \cL_{2}^{\na},a_{2}, \a_{R,2})$ to $(\cL_{1}^{\na}\ot\cL^{\na}_{2},a_{1}\ot a_{2}, \a_{R,1}\ot\a_{R,2})$. It restricts to a map
\begin{equation}\label{add sqR}
\add_{d_{1},d_{2}}^{\sqR}:X_{d_{1}}^{\sqR}\times X_{d_{2}}^{\sqR}\to X^{\sqR}_{d_{1}+d_{2}}.
\end{equation}

In particular, applying this construction iteratively, we get a map (for $d\ge0$)
\begin{equation}\label{symmR}
p^{\sqR}_{d}: (X^{\sqR}_{1})^{d}\to X^{\sqR}_{d}.
\end{equation}
which is $S_{d}$-invariant with respect to the permutation action on the source.

\sss{The Abel-Jacobi map}\label{sss:AJ}  Forgetting the sections $a$ we get a morphism
\begin{equation*}
\wh\AJ^{\sqR;\sqR}_{d}: \hX^{\sqR}_{d}\to \Pic^{\sqR;\sqR,d}_{X}.
\end{equation*}
We also get a map
\begin{equation*}
\wh\AJ^{\sqR}_{d}: \hX^{\sqR}_{d}\to \Pic^{\sqR,d}_{X}
\end{equation*}
by further forgetting $\a_{R}$. Let $\AJ^{\sqR;\sqR}_{d}$ and $\AJ^{\sqR}_{d}$ be the restrictions of $\wh\AJ^{\sqR;\sqR}_{d}$ and $\wh\AJ^{\sqR}_{d}$ to $X^{\sqR}_{d}$. When $R=\vn$, $\AJ^{\sqR}_{d}$ reduces to the usual Abel-Jacobi map.

\sss{Presentation of $\Pic^{\sqR}_{X}(k)$}\label{sss:OsqR} 
For $x\in R$, let
\begin{equation*}
\cO_{\sqx}=\cO_{x}\times_{k(x)}k(x), \quad \cO^{\times}_{\sqx}=\cO_{x}^{\times}\times_{k(x)^{\times}}k(x)^{\times}
\end{equation*}
where the second projections $k(x)\to k(x)$ and $k(x)^{\times}\to k(x)^{\times}$ are the square maps. Let $\OO^{\times}_{\sqR}=\prod_{x\in R}\cO^{\times}_{\sqx}\times\prod_{x\in|X-R|}\cO^{\times}_{x}$. We have a homomorphism $\OO^{\times}_{\sqR}\to \OO^{\times}=\prod_{x\in |X|} \cO_{x}^{\times}\to \AA_{F}^{\times}$.

\begin{lemma} There is a canonical isomorphism of Picard groupoids
\begin{equation}\label{idele Pic R}
F^{\times}\bs \AA_{F}^{\times}/\OO^{\times}_{\sqR}\isom\Pic^{\sqR}_{X}(k)
\end{equation}
sending $\vp^{-1}_{x}$ (where $\vp_{x}$ is a uniformizer at $x\in |X-R|$) to the point $\cO_{X}(x)^{\na}=(\cO_{X}(x), \cO_{R}, 1)\in \Pic^{\sqR}_{X}(k)$.

\end{lemma}
\begin{proof} Consider the groupoid $\wh\Pic^{\sqR}_{X}(k)$ classifying $(\cL,\t_{\y},\{\t_{x}\}_{x\in |X|}, \cK_{R},\io, t_{R}=\{t_{x}\}_{x\in R})$, where $(\cL,\cK_{R},\io)\in\Pic^{\sqR}_{X}(k)$, $\t_{\y}: \cL|_{\Spec F}\cong F$ is a trivialization of $\cL$ at the generic point, and $\t_{x}: \cL|_{\Spec \cO_{x}}\cong \cO_{x}$ is a trivialization of $\cL$ in the formal neighborhood of $x$, $t_{x}: \cK_{x}\isom k(x)$ is a trivialization of $\cK_{x}$ for every $x\in R$, such that the following diagram is commutative
\begin{equation*}
\xymatrix{      \cK_{x}^{\ot2}\ar[r]^{\io_{x}}\ar[d]^{t_{x}^{\ot2}} & \cL_{x}\ar[d]^{\t_{x}|_{x}}     \\
k(x)^{\ot2}\ar@{=}[r] & k(x)}
\end{equation*}
Similarly, we define $\wh\Pic_{X}(k)$ to classify part of the data $(\cL,\t_{\y},\{\t_{x}\}_{x\in |X|})$ as above. The forgetful map $\wh\Pic^{\sqR}_{X}(k)\to \wh\Pic_{X}(k)$ is an equivalence: the choices of the extra data $(\cK_{R}, \iota, \t_{R})$ are unique up to a unique isomorphism. 

We have an isomorphism $\wh\Pic_{X}(k)\isom \AA^{\times}_{F}$ sending $(\cL,\t_{\y},\{\t_{x}\}_{x\in |X|})$ to $(\t_{x}\circ \t_{\y}^{-1})_{x\in|X|}\in \AA^{\times}$. Therefore we get a canonical isomorphism
\begin{equation*}
\a: \AA^{\times}_{F}\isom \wh\Pic^{\sqR}_{X}(k).
\end{equation*}
It is easy to see that for $x\in |X-R|$, $\a(\vp_{x}^{-1})$ has image  $\cO_{X}(x)^{\na}$ in $\Pic^{\sqR}_{X}(k)$.

There is an action of $F^{\times}$ on $\wh\Pic^{\sqR}_{X}(k)$ by changing $\t_{\y}$. For $x\in |X-R|$, there is an action of $\cO_{x}^{\times}$ on $\wh\Pic^{\sqR}_{X}(k)$ by changing $\t_{x}$. For $x\in R$, there is an action of $\cO^{\times}_{\sqx}=\cO_{x}^{\times}\times_{k(x)^{\times}}k(x)^{\times}$ on $\wh\Pic^{\sqR}_{X}(k)$ by changing $\t_{x}$ and $t_{x}$ compatibly. Therefore we get an action of $F^{\times}\times \OO^{\times}_{\sqR}$ on $\wh\Pic^{\sqR}_{X}(k)$. The isomorphism $\a$ is equivariant with respect to these actions. The forgetful map $\wh\Pic^{\sqR}_{X}(k)\to\Pic^{\sqR}_{X}(k)$ is a torsor for the action of $F^{\times}\times \OO^{\times}_{\sqR}$. Therefore $\a$ induces the equivalence \eqref{idele Pic R}.
\end{proof}

\subsection{Geometric class field theory}\label{ss:geom CFT}
In this subsection, we fix $L$ to be a rank one $\Qlbar$-local system on $X^{\sqR}_{1}$.  Since $X^{\sqR}_{1}$ is a smooth DM curve with coarse moduli space $X$ and automorphic group $\mu_{2}$ along $R$, such a local system is the same datum as a rank one $\Qlbar$-local system on  $X-R$ with monodromy of order at most $2$ at the $x\in R$. 

Starting from $L$, we will give a canonical construction of local systems $L_{d}$ on $X^{\sqR}_{d}$ for $d\ge0$ and show that it descends to $\Pic^{\sqR,d}_{X}$. In the case $R=\vn$, such a construction goes back to Deligne.

\sss{The local system $L_{d}$}
Consider the $S_{d}$-invariant map $p^{\sqR}_{d}$ in \eqref{symmR}. The complex $p^{\sqR}_{d,!}L^{\boxtimes d}$ is a middle extension perverse sheaf on $X^{\sqR}_{d}$ (i.e., it is the middle extension of a local system from a dense open subset of $X^{\sqR}_{d}$) because $p^{\sqR}_{d}$ is a finite map from a smooth and geometrically connected DM stack. Therefore the $S_{d}$-invariant part 
\begin{equation*}
L_{d}:=(p^{\sqR}_{d,!}L^{\boxtimes d})^{S_{d}}
\end{equation*}
is also a middle extension perverse sheaf on $X^{\sqR}_{d}$. 

\begin{lemma}\label{l:coho Xd}
Suppose the local system $L$ is geometrically nontrivial. Then
\begin{equation*}
\cohog{i}{X^{\sqR}_{d}\ot\kbar,L_{d}}=\begin{cases}\wedge^{d}\left(\cohog{1}{X_{1}^{\sqR}\ot\kbar, L}\right) & i=d,\\ 0 & i\ne d. \end{cases}
\end{equation*}
\end{lemma}
\begin{proof}
By construction, the graded vector space $\cohog{*}{X^{\sqR}_{d}\ot\kbar,L_{d}}$ is the $S_{d}$-invariants of the graded vector space $\cohog{*}{X^{\sqR}_{1},L}^{\ot d}$ ($S_{d}$ acts by permuting the factors with the Koszul sign convention). Since $L$ is geometrically nontrivial, $\cohog{*}{X^{\sqR}_{1},L}$ is concentrated in degree $1$. Hence $\cohog{*}{X^{\sqR}_{d}\ot\kbar,L_{d}}$  is concentrated in degree $d$, and is equal to $\wedge^{d}\left(\cohog{1}{X_{1}^{\sqR}\ot\kbar, L}\right)$ in that degree.
\end{proof}

\begin{lemma}\label{l:Ld}
The perverse sheaf $L_{d}$ is a local system of rank one on $X^{\sqR}_{d}$.
\end{lemma}
\begin{proof} Since $L_{d}$ is a  middle extension perverse sheaf on $X^{\sqR}_{d}$, to show it is a local system of rank one, it  suffices to check the stalks of $L_{d}$ at any geometric point of $X_{d}^{\sqR}$ is one-dimensional. Consider a geometric point $(\cL^{\na}, a,\a_{R})\in X^{\sqR}_{d}$ with $\div(a)=D$. By factorizing the situation according to the points in $D$, we reduce to show that for $x\in R(\kbar)$, $L_{d}$ has one-dimensional stalk at the geometric point $dx\in X^{\sqR}_{d}(\kbar)$. The point $dx$ has automorphism $\mu_{2}$, and the restriction of $p^{\sqR}_{d}$ to the preimage of this orbifold point is
\begin{equation*}
p_{dx}: [\pt/\mu_{2}]^{d}\to [\pt/\mu_{2}]
\end{equation*}
induced by the multiplication map $m: \mu^{d}_{2}\to \mu_{2}$. The restriction of $L$ to $x=[\pt/\mu_{2}]\in X^{\sqR}_{1}$ is given by either the trivial or the sign representation of $\mu_{2}$ on $\Qlbar$. Therefore $p_{dx,!}L^{\boxtimes d}_{x}$ is the $K_{d}=\ker(m:\mu^{d}_{2}\to \mu_{2})$-coinvariants on $L^{\boxtimes d}_{x}$, which is $L^{\boxtimes d}_{x}$ itself since $K_{d}$ always acts trivially on it. Therefore, the stalk of $L_{d}$ at $dx$ is one-dimensional.
\end{proof}

\begin{lemma}\label{l:add} For $d_{1},d_{2}\ge0$ there is a canonical isomorphism of local systems on $X^{\sqR}_{d_{1}}\times X^{\sqR}_{d_{2}}$
\begin{equation*}
\a_{d_{1},d_{2}}: \add_{d_{1},d_{2}}^{\sqR, *}L_{d_{1}+d_{2}}\cong L_{d_{1}}\boxtimes L_{d_{2}}.
\end{equation*}
which is commutative and associative in the obvious sense.
\end{lemma}
\begin{proof} Let $d=d_{1}+d_{2}$.
Since both $\add^{\sqR,*}_{d_{1},d_{2}}L_{d}$ and $L_{d_{1}}\boxtimes L_{d_{2}}$ are local systems, it suffices to give such an isomorphism over a dense open substack of $X^{\sqR}_{d_{1}}\times X^{\sqR}_{d_{2}}$. Let $U=X-R$. Let  $U_{d}^{\c}\subset X^{\sqR}_{d}$ be the open subscheme consisting of multiplicity-free divisors on $U$. Let $(U_{d_{1}}\times U_{d_{2}})^{\c}\subset X^{\sqR}_{d_{1}}\times X^{\sqR}_{d_{2}}$ be the preimage of $U_{d}^{\c}$ under $\add_{d_{1},d_{2}}^{\sqR}$. 

The monodromy representation of the local system $L|_{U}$ is given by a homomorphism
\begin{equation*}
\chi: \pi_{1}(U)\to \{\pm1\}.
\end{equation*}
For any $n\in\ZZ_{\ge0}$, there is a canonical homomorphism 
\begin{equation*}
\ph_{n}: \pi_{1}(U_{n}^{\c})\to \pi_{1}(U)^{n}\rtimes S_{n}
\end{equation*}
given by the branched $S_{n}$-cover $U^{n}\to U_{n}$. 

The monodromy representation of the local system $L_{d_{1}}\boxtimes L_{d_{2}}|_{(U_{d_{1}}\times U_{d_{2}})^{\c}}$ is given by
\begin{eqnarray}\notag
\pi_{1}((U_{d_{1}}\times U_{d_{2}})^{\c})&\xr{(p_{1*},p_{2*})}& \pi_{1}(U_{d_{1}}^{\c})\times\pi_{1}(U_{d_{2}}^{\c})\xr{\ph_{d_{1}}\times\ph_{d_{2}}} (\pi_{1}(U)^{d_{1}}\rtimes S_{d_{1}})\times (\pi_{1}(U)^{d_{2}}\rtimes S_{d_{2}})\\
\label{first pi rep}&=&\pi_{1}(U)^{d}\rtimes(S_{d_{1}}\times S_{d_{2}})\xr{(\chi,\cdots,\chi)\times 1} \{\pm1\}
\end{eqnarray}
The last map is $\chi$ on all the $\pi_{1}(U)$-factors and trivial on $S_{d_{1}}\times S_{d_{2}}$.

On the other hand, the local system $\add_{d_{1},d_{2}}^{*}L_{d}|_{U_{d}^{\c}}$ is given by the character
\begin{equation}\label{sec pi rep}
\pi_{1}((U_{d_{1}}\times U_{d_{2}})^{\c})\xr{\add_{*}}\pi_{1}(U_{d}^{\c})\xr{\ph_{d}}\pi_{1}(U)^{d}\rtimes S_{d}\xr{(\chi,\cdots,\chi)\times 1} \{\pm1\}.
\end{equation}
Observe that \eqref{first pi rep} and \eqref{sec pi rep} are the same homomorphisms. This gives the desired isomorphism $\a_{d_{1},d_{2}}$. We leave the verification of the commutativity and associativity properties of $\a_{d_{1},d_{2}}$ as an exercise.
\end{proof}

\begin{lemma} For $d\ge \r+\max\{2g-1,1\}$, the local system $L_{d}$ on $X^{\sqR}_{d}$ descends to $\Pic^{\sqR,d}_{X}$ via the map $\AJ^{\sqR}_{d}$.
\end{lemma}
\begin{proof}
The case $R=\vn$ is well-known; we treat only the case $R\ne\vn$.

When $d\ge 2g-1+\r$, by Riemann-Roch, $\AJ^{\sqR}_{d}$ is a locally trivial fibration, therefore it suffices to show that the restriction of $L_{d}$ to geometric fibers of $\AJ^{\sqR}_{d}$ are trivial. 

Fix a geometric point $\cL^{\na}=(\cL,\cK_{R}, \io)\in \Pic^{\sqR,d}_{X}(K)$ for some algebraically closed field $K$. We base change the situation from $k$ to $K$ without changing notation. The fiber of $\AJ^{\sqR}_{d}$ over $\cL^{\na}$ is
\begin{equation*}
M=\cohog{0}{X, \cL}^{\circ}\times_{\cohog{0}{R, \cL_{R}}}\cohog{0}{R,\cK_{R}}
\end{equation*}
where $\cohog{0}{X, \cL}^{\circ}=\cohog{0}{X, \cL}-\{0\}$, and the map $\cohog{0}{R,\cK_{R}}\to \cohog{0}{R, \cL_{R}}$ is the square map via $\io$. The torus $\Gm$ acts on $M$ by weight 2 on $\cohog{0}{X, \cL}$ and weight 1 on $\cohog{0}{R, \cK_{R}}$. Then the map $M\to X_{d}^{\sqR}$ factors through the quotient $[M/\Gm]$. The triviality of $L_{d}|_{[M/\Gm]}$ follows from the Claim below.

\begin{claim} $[M/\Gm]$ is simply-connected. 
\end{claim}

It remains to prove the Claim. Choosing a basis for $\cohog{0}{R, \cL_{R}}$ and extending it to $\cohog{0}{X, \cL}$, we may identify $M$ with a punctured affine space $\AA^{n}-\{0\}$ and the action of $\Gm$ has weights $2$ (on the first $n-\r$ coordinates) and 1 (on the last $\r$ coordinates). Since $n=d-g+1\ge \r+1$, the weight $2$ appears at least once.

Suppose $Y\to [M/\Gm]$ is a finite \'etale map with $Y$ connected. Consider the map $\pi: \PP^{n-1}\to [M/\Gm]$ given by $[x_{1},\dotsc, x_{n-\r}, y_{1},\dotsc, y_{\r}]\mapsto [x_{1}^{2},\dotsc, x_{n-\r}^{2}, y_{1},\dotsc, y_{\r}]$. Then $\pi$ is a branched Galois cover with Galois group $\mu_{2}^{n-\r}$. Since $\PP^{n-1}$ is simply-connected, $\pi$ lifts to $\wt\pi: \PP^{n-1}\to Y$. Therefore the function field $K(Y)\subset K(\PP^{n-1})$ corresponds to a subgroup $\Gamma\subset \mu^{n-\r}_{2}$ so that $Y$ is the normalization of $[M/\Gm]$ in $\Spec K(Y)$. We consider the open subset $M^{\c}$ where the last coordinate $y_{\r}\ne0$, then $M^{\c}/\Gm\cong \AA^{n-1}$. Let $Y^{\c}$ be the preimage of $M^{\c}/\Gm$ in $Y$, and let $(\PP^{n-1})^{\c}\cong \AA^{n-1}$ be the preimage in $\PP^{n-1}$. Then $Y^{\c}$ is the GIT quotient of $(\PP^{n-1})^{\c}$ by $\Gamma$. If $\Gamma\ne\mu^{n-\r}_{2}$, then there is a non-empty subset $I\subset \{1,\dotsc, n-\r\}$ such that $\Gamma$ is contained in the kernel of  $e^{*}_{I}: \mu_{2}^{n-\r}\to \mu_{2}$ given by $e^{*}_{I}(\ep_{i})=\ep_{i}$ if $i\in I$ and $1$ is $i\notin I$. In this case, $x_{I}=\prod_{i\in I}x_{i}$ is fixed by $\Gamma$ hence $x_{I}\in \cO(Y^{\c})$. However, $x_{I}\notin \cO(M^{\c}/\Gm)$ (only $x^{2}_{I}\in \cO(M^{\c}/\Gm)$). This implies that $Y^{\c}\to M^{\c}/\Gm$ is ramified along the divisor $x_{I}=0$ in $Y^{\c}$, contradiction. Therefore $\Gamma=\mu^{n-\r}_{2}$ and $Y=[M/\Gm]$. 
\end{proof}

\sss{Construction of $L^{\Pic}_{d}$ for all $d\in\ZZ$}\label{sss:cons LPic}
Let $L^{\Pic}_{d}$ be the descent of $L_{d}$ to $\Pic^{\sqR,d}_{X}$ when $d\ge \r+\max\{2g-1,1\}$. 
Next we extend the local systems $\{L^{\Pic}_{d}\}$ to all components of $\Pic^{\sqR}_{X}$.

Fix any integer $d$. For any divisor $D=\sum_{x\in |X-R|} n_{x} \cdot x\in \Div(X-R)$ of degree $d'$, we have a canonical line $L_{D}=\ot L_{x}^{\ot n_{x}}$. Tensoring with $\cO_{X}(D)^{\na}$ (the canonical lift of $\cO_{X}(D)$ to $\Pic_{X}^{\sqR}$) defines an isomorphism $t_{D}: \Pic^{\sqR,d}_{X}\to \Pic^{\sqR,d+d'}_{X}$. If $d'+d\ge\max\{2g-1,1\}+\r$, $L^{\Pic}_{d+d'}$ is already defined, and we define $L^{\Pic}_{d}$ to be the local system
$t_{D}^{*}L^{\Pic}_{d+d'}\ot L^{\ot-1}_{D}$ on $\Pic^{\sqR,d}_{X}$. We claim that $L^{\Pic}_{d}$ thus defined is canonically independent of the choice of $D$, as long as the degree$d'$ of $D$ satisfies $d'\ge \max\{2g-1,1\}+\r-d$. To show this, it suffices to show that for any $n,n'\ge \max\{2g-1,1\}+\r$ (so that $L^{\Pic}_{n}$ and $L^{\Pic}_{n'}$ are both defined as the descent of $L_{n}$ and $L_{n'}$) and any $D\in\Div^{n'-n}(X-R)$, there is a canonical isomorphism $t_{D}^{*}L^{\Pic}_{n'}\cong L^{\Pic}_{n}\ot L_{D}$ as local systems on $\Pic^{\sqR,n}_{X}$. It is easy to reduce to the case $D$ effective. Since $\AJ^{\sqR}_{n}$ has connected geometric fibers, it is enough to give such an isomorphism after pulling back to $X^{\sqR}_{n}$, i.e., we need to give a canonical isomorphism of local systems on $X^{\sqR}_{n}$
\begin{equation}\label{TDL}
T_{D}^{*}L_{n'}\cong L_{n}\otimes L_{D}
\end{equation}
where $T_{D}: X^{\sqR}_{n}\to X^{\sqR}_{n'}$ is the addition by $D$. Such an isomorphism is given by Lemma \ref{l:add} by taking restricting $\a_{n,n'-n}$ to $X^{\sqR}_{n}\times\{D\}$.

We have therefore defined a canonical local system $L^{\Pic}_{d}$ on $\Pic^{\sqR}_{d}$ for each $d\in\ZZ$. Let $L^{\Pic}$ be the local system on $\Pic^{\sqR}_{X}$ whose restriction to $\Pic^{\sqR}_{d}$ is $L_{d}^{\Pic}$.

\begin{lemma}\label{l:Ld pullback} For $d\ge0$, we have a canonical isomorphism of local systems on $X^{\sqR}_{d}$
\begin{equation*}
\AJ^{\sqR,*}_{d}L^{\Pic}_{d}\cong L_{d}.
\end{equation*}
\end{lemma}
\begin{proof}
Let $D$ be a divisor on $X-R$ of degree $d'\ge \max\{2g-1,1\}+\r-d$. By construction we have $L^{\Pic}_{d}=t_{D}^{*}L^{\Pic}_{d+d'}\ot L_{D}^{\ot-1}$. Pulling back both sides to $X^{\sqR}_{d}$, and noting $\AJ^{\sqR}_{d+d'}\circ T_{D}=t_{D}\circ \AJ^{\sqR}_{d}$, we get
\begin{equation*}
\AJ^{\sqR*}_{d}L^{\Pic}_{d}=\AJ^{\sqR*}_{d}t_{D}^{*}L^{\Pic}_{d+d'}\ot L_{D}^{\ot-1}=T_{D}^{*}\AJ^{\sqR*}_{d+d'}L^{\Pic}_{d+d'}\ot L_{D}^{\ot-1}=T_{D}^{*}L_{d+d'}\ot L_{D}^{\ot-1}.
\end{equation*}
which is canonically isomorphic to $L_{d}$ by \eqref{TDL}.
\end{proof}

\begin{prop}\label{p:char sh} The local system $L^{\Pic}$ is a character sheaf on $\Pic_{X}^{\sqR}$. More precisely, this means the following
\begin{enumerate}
\item There is a canonical trivialization $\io:L^{\Pic}|_{e}\cong\Qlbar$, where $e$ is the origin of $\Pic^{\sqR}_{X}$.
\item There is a canonical isomorphism of local systems on $\Pic_{X}^{\sqR}\times \Pic_{X}^{\sqR}$
\begin{equation*}
\mu: \mult^{*}L^{\Pic}\cong L^{\Pic}\boxtimes L^{\Pic}
\end{equation*}
where $\mult: \Pic_{X}^{\sqR}\times \Pic_{X}^{\sqR}\to \Pic_{X}^{\sqR}$ is the multiplication map. 
\item The isomorphism $\mu$ is commutative and associative in the obvious sense, and its restrictions to $\{e\}\times \Pic_{X}^{\sqR}$ and $\Pic_{X}^{\sqR}\times \{e\}$ are the identity maps on $L^{\Pic}$ (after using $\io$ to trivialize $L^{\Pic}|_{e}$). 
\end{enumerate}
\end{prop}
\begin{proof}
By construction, $L^{\Pic}|_{e}\cong L^{\Pic}_{d}|_{\cO(D)^{\na}}\ot L_{D}^{\ot-1}\cong L_{d}|_{D}\ot L_{D}^{\ot-1}$ for any effective divisor $D\in \Div(X-R)$ of large degree $d$ (we are viewing $D$ as a $k$-point of $(X-R)_{d}\subset X_{d}^{\sqR}$, so $L_{d}|_{D}$ means the stalk of $L_{d}$ at this $k$-point $D$). If we write $D=\sum_{x\in |X-R|}n_{x}\cdot x$, then by construction we have a canonical isomorphism $L_{d}|_{D}\cong \ot_{x\in |X-R|} L_{x}^{\ot n_{x}}=L_{D}$, which gives a trivialization $\io_{D}: L^{\Pic}|_{e}\cong \Qlbar$. We leave it as an exercise to check that $\io_{D}$ is independent of the choice of $D$.

Now we construct the isomorphism $\mu$, i.e., a system of isomorphisms
\begin{equation*}
\mu_{d_{1},d_{2}}: \mult^{*}_{d_{1},d_{2}}L_{d_{1}+d_{2}}^{\Pic}\cong L_{d_{1}}^{\Pic}\boxtimes L_{d_{2}}^{\Pic}
\end{equation*}
for all $d_{1},d_{2}\in\ZZ$. When $d_{1},d_{2}\ge\r+\max\{2g-1,1\}$, $L^{\Pic}_{d_{i}}$ and $L^{\Pic}_{d_{1}+d_{2}}$ come by descent from $L_{d_{i}}$ and $L_{d_{1}+d_{2}}$. Since $\AJ^{\sqR}_{d_{1}+d_{2}}$ has connected geometric fibers, it suffices to give $\mu_{d_{1},d_{2}}$ after pulling back both sides to $X^{\sqR}_{d_{1}+d_{2}}$, in which case the desired isomorphism is given by $\a_{d_{1},d_{2}}$ constructed in Lemma \ref{l:add}. 

For general $d_{1},d_{2}$, let $D_{1},D_{2}\in\Div(X-R)$ with degrees $\deg D_{i}=n_{i}$ such that $n_{i}+d_{i}\ge\r+\max\{2g-1,1\}$ for $i=1,2$. Then by construction,
\begin{equation}\label{L1L2}
L_{d_{1}}^{\Pic}\boxtimes L^{\Pic}_{d_{2}}\cong (t^{*}_{D_{1}}L^{\Pic}_{d_{1}+n_{1}}\boxtimes t^{*}_{D_{2}}L^{\Pic}_{d_{2}+n_{2}})\ot (L_{D_{1}}^{\ot-1}\ot L_{D_{2}}^{\ot-1}).
\end{equation} 
On the other hand, $L_{d_{1}+d_{2}}^{\Pic}\cong t^{*}_{D_{1}+D_{2}}L_{d_{1}+d_{2}+n_{1}+n_{2}}\ot L_{D_{1}+D_{2}}^{\ot-1}$, hence
\begin{eqnarray}
\label{mult L12}\mult^{*}_{d_{1},d_{2}}L_{d_{1}+d_{2}}^{\Pic}&\cong& \mult^{*}_{d_{1},d_{2}}t^{*}_{D_{1}+D_{2}}L_{d_{1}+d_{2}+n_{1}+n_{2}}\ot L_{D_{1}+D_{2}}^{\ot-1}\\
\notag &\cong & ((t_{D_{1}}\times t_{D_{2}})^{*}\mult^{*}_{d_{1}+n_{1},d_{2}+n_{2}}L_{d_{1}+d_{2}+n_{1}+n_{2}})\ot( L_{D_{1}}^{\ot-1}\ot L_{D_{2}}^{\ot-1}).
\end{eqnarray}
Comparing the RHS of \eqref{L1L2} and \eqref{mult L12}, the desired isomorphism $\mu_{d_{1},d_{2}}$ is induced from the already-constructed $\mu_{d_{1}+n_{1},d_{2}+n_{2}}$. Again we leave it as an exercise to check that $\mu_{d_{1},d_{2}}$ is independent of the choices of $D_{1},D_{2}$, and it satisfies commutativity, associativity, and compatibility with $\io$. 
\end{proof}

\subsection{Ramified double cover}\label{ss:ram cover}
Let $\nu: X'\to X$ be a double cover with ramification locus $R\subset X$, where $X'$ is also a smooth projective and geometrically connected curve over $k$.  Let $\s:X'\to X'$ be the nontrivial involution over $X$. Let $R'\subset X'$ be the reduced preimage of $R$, then $\nu$ induces an isomorphism $R'\isom R$.


\sss{The norm map on Picard} Let $i_{R}:R\incl X$ be the inclusion. We consider the \'etale sheaf $\GG_{m,R}$ on $R$ as an \'etale sheaf on $X$ via $i_{R,*}$. There is a restriction map $\GG_{m,X}\to \GG_{m,R}$. Consider the following \'etale sheaf on $X$
\begin{equation*}
\GG^{\sqR}_{m,X}=\GG_{m,X}\times_{\GG_{m,R}, [2]} \GG_{m,R}
\end{equation*}
where the map $\GG_{m,R}\to \GG_{m,R}$ is the square map. By construction, $\Pic^{\sqR}_{X}$ is the moduli stack of $\GG^{\sqR}_{m,X}$-torsors over $X$.

We have the sheaf homomorphism induced by the norm map $\Nm: \nu_{*}\GG_{m,X'}\to \GG_{m,X}$ and the restriction map $r_{R'}: \nu_{*}\GG_{m,X'}\to \nu_{*}\GG_{m,R'}=\GG_{m,R}$. Computing with local coordinates at $R$, we see that the composition $\nu_{*}\GG_{m,X'}\xrightarrow{\Nm} \GG_{m,X}\xrightarrow{r_{R}} \GG_{m,R}$ (the latter $r_{R}$ is given by restriction) is the square of the restriction map $r_{R'}$. Therefore $(\Nm,r_{R'})$ induces a sheaf homomorphism
\begin{equation*}
\un\Nm^{\sqR}_{X'/X}:\nu_{*}\GG_{m,X'}\to \GG^{\sqR}_{m,X}
\end{equation*}
which is easily seen to be surjective by local calculation at $R$. The map $\un\Nm^{\sqR}_{X'/X}$ on sheaves induces a morphism of Picard stacks
\begin{equation*}
\Nm^{\sqR}_{X'/X}: \Pic_{X'}\to \Pic^{\sqR}_{X}
\end{equation*}
which lifts the usual norm map $\Nm_{X'/X}:\Pic_{X'}\to \Pic_{X}$. 

\sss{The norm map on symmetric powers}
There is also a natural lifting of the norm map $\wh\nu_{d}: \wh X'_{d}\to \wh X_{d}$ 
\begin{equation}\label{norm for sym}
\wh\nu^{\sqR}_{d}: \hX'_{d}\to \hX^{\sqR}_{d}.
\end{equation}
In fact, for $(\cL', a')\in \wh X'_{d}(S)$, where $\cL'$ is a line bundle over $X'\times S$ and $a'$ a global section of $\cL'$, $\cL=\Nm_{X'/X}(\cL')$ is a line bundle over $X\times S$, and $a=\Nm(a')$ is a section of $\cL$. We have a canonical isomorphism $\io: (\cL'|_{R'\times S})^{\otimes 2}\cong (\cL'\otimes\s^{*}\cL')|_{R'\times S}\cong \cL|_{R\times S}$. Under $\io$, $a'|_{R'\times S}$ gives a square root of the restriction $a|_{R\times S}$. We then send $(\cL',a')$ to $(\cL,\cL'|_{R'\times S},\io,a,a'|_{R'\times S})\in  \hX^{\sqR}_{d}(S)$. 

By construction, we have a commutative diagram
\begin{equation*}
\xymatrix{    \hX'_{d}\ar[rr]^-{\wh\nu^{\sqR}_{d}}\ar[d]^{\wh\AJ'_{d}} && \hX^{\sqR}_{d}   \ar[d]^{\wh\AJ^{\sqR}_{d}}    \\
\Pic_{X'}\ar[rr]^-{\Nm^{\sqR}_{X'/X}} && \Pic^{\sqR}_{X}
}
\end{equation*}
where $\wh\AJ'_{d}$ is the Abel-Jacobi map for $X'$.

\sss{Descent of line bundles} A local calculation shows that the image of $1-\s: \nu_{*}\GG_{m,X'}\to \nu_{*}\GG_{m,X'}$ is  equal to the kernel of $\un\Nm^{\sqR}_{X'/X}$. Therefore we have an exact sequence of \'etale sheaves on $X$:
\begin{equation*}
1\to \GG_{m,X}\to\nu_{*}\GG_{m,X'}\xrightarrow{1-\s}\nu_{*}\GG_{m,X'}\xr{\un\Nm^{\sqR}_{X'/X}} \GG^{\sqR}_{m,X}\to1.
\end{equation*}
Taking the corresponding Picard stacks we get an exact sequence of Picard stacks
\begin{equation}\label{Pic exact}
1\to \Pic_{X}\xrightarrow{\nu^{*}}\Pic_{X'}\xrightarrow{1-\s}\Pic_{X'}\xrightarrow{\Nm^{\sqR}_{X'/X}}\Pic^{\sqR}_{X}\to 1.
\end{equation}

\sss{The local system $L$} The direct image sheaf $\nu_{*}\Ql$ has a decomposition $\nu_{*}\Ql=\Qlbar\oplus L_{X'/X}$ into $\s$-eigensheaves with eigenvalues $1$ and $-1$. Then $L_{X'/X}|_{X-R}$ is a $\Ql$-local system of rank one with monodromy in $\{\pm1\}$ ramified exactly along $R$. Let $L$ be the local system on $X^{\sqR}_{1}$ corresponding to $(L_{X'/X}\ot\Qlbar)|_{X-R}$.  Associated to $L$ is a local system $L^{\Pic}$ on $\Pic^{\sqR}_{X}$ constructed  in \S\ref{ss:geom CFT}.

Let $F'=k(X')$, a quadratic extension of $F$ unramified away from $R$.  By class field theory, $F'/F$ gives rise to an id\`ele class character
\begin{equation*}
\y_{F'/F}: F^{\times}\bs \AA^{\times}_{F}/\OO^{\times}_{\sqR}\to \{\pm1\}.
\end{equation*}
For the notation $\OO^{\times}_{\sqR}$, see \S\ref{sss:OsqR}.

\begin{prop}\label{p:LPic fun}
Under the sheaf-to-function correspondence, the function on $\Pic_{X}^{\sqR}(k)$ given by $L^{\Pic}$ is the id\`ele class character $\y_{F'/F}$ under the isomorphism \eqref{idele Pic R}.
\end{prop}
\begin{proof} Let $f_{L}:\Pic^{\sqR}_{X}(k)\to \Qlbar^{\times}$ be the function attached to $L^{\Pic}$. 
By Prop. \ref{p:char sh}, $f_{L}$ is a group homomorphism. We know that $\y_{F'/F}$ is characterized by the property that for a uniformizer $\vp_{x}$ at $x\in |X-R|$,
\begin{equation*}
\y_{F'/F}(\vp^{-1}_{x})=\begin{cases} 1, & \mbox{ if $x$ is split in $F'$}; \\ -1, & \mbox{ if $x$ is inert in $F'$.}\end{cases}
\end{equation*}
Now $x$ is split (resp. inert) in $F'$ if and only if $\Tr(\Fr_{x}, L_{x})=1$ (resp. $\Tr(\Fr_{x}, L_{x})=-1$). Therefore
\begin{equation*}
\y_{F'/F}(\vp^{-1}_{x})=\Tr(\Fr_{x}, L_{x}).
\end{equation*}
We only need to check that $f_{L}$ enjoys the same property as $\y_{F'/F}$. Since $\vp^{-1}_{x}$ corresponds to $\cO(x)^{\na}\in\Pic^{\sqR,d_{x}}_{X}(k)$ under \eqref{idele Pic R}, we need to show
\begin{equation*}
\Tr(\Fr_{\cO(x)^{\na}}, L^{\Pic}|_{\cO(x)^{\na}})=\Tr(\Fr_{x}, L_{x}),\quad \forall x\in |X-R|.
\end{equation*}
Let $d=d_{x}$. By Lemma \ref{l:Ld pullback}, $L^{\Pic}_{d}$ pulls back to $L_{d}$ on $X^{\sqR}_{d}$;  viewing $x$ as a divisor of degree $d$ on $X-R$ (and denoted $[x]$), it maps to $\cO(x)^{\na}$ via $\AJ^{\sqR}_{d}$,  hence the left side above is equal to $\Tr(\Fr_{k}, L_{d}|_{[x]})$. Therefore it suffices to show
\begin{equation}\label{Tr Lx}
\Tr(\Fr_{k}, L_{d}|_{[x]})=\Tr(\Fr_{x}, L_{x}).
\end{equation}
By the construction of $L_{d}$, there is an isomorphism $L_{d}|_{[x]}\cong L_{x}^{\ot d}$ such that the $\Fr_{k}$-action on $L_{d}|_{[x]}$ corresponds to the automorphism $\ell_{1}\ot\ell_{2}\ot\cdots \ot\ell_{d}\mapsto \ell_{2}\ot\cdots\ot\ell_{d}\ot\Fr_{x}(\ell_{1})$ on $L_{x}^{\ot d}$. This shows \eqref{Tr Lx} and finishes the proof of the proposition. 
\end{proof}


\end{document}